\documentclass[12pt]{memoir}

\usepackage[latin1]{inputenc}
\usepackage{graphicx}
\usepackage{amsmath}
\usepackage{amsfonts}
\usepackage{amssymb}
\usepackage{amsthm}
\usepackage{enumerate}
\usepackage{mathrsfs}
\usepackage{cite}
\usepackage[all]{xy}
\usepackage{verbatim}
\usepackage[pagebackref]{hyperref}     % backtracking til ref's i bibliografi. Skal kommenteres ud i endelige version(?)
\usepackage[top=1.5cm,bottom=1.5cm,left=1.5cm,right=1.5cm]{geometry}
\usepackage{appendix}

\usepackage{dsfont}
\newcommand{\bbb}[1][\mathds{1}]{\mathds{#1}}

\newtheoremstyle{primedtheorem}{}{}{\itshape}{}{\scshape}{}{3pt}{#2 {#1}'.}
\newtheoremstyle{swappedplain}{}{}{\itshape}{}{\scshape}{}{3pt}{#2 #1. #3}
\newtheoremstyle{swappedroman}{}{}{\rmfamily}{}{\scshape}{}{3pt}{#2 #3 #1.}
\newtheoremstyle{swappedromancontents}{}{}{\rmfamily}{}{\scshape}{}{3pt}{#2 #3 #1. \addcontentsline{toc}{subsubsection}{#2 #3}}
\newtheoremstyle{questionstyle}{}{}{\rmfamily}{}{\scshape}{}{3pt}{* #1 #2 *}

\theoremstyle{swappedplain}
\newtheorem{theorem}[section]{Theorem}
\newtheorem*{theorem*}{Theorem}
\newtheorem{lemma}[section]{Lemma}
\newtheorem*{lemma*}{Lemma}
\newtheorem{proposition}[section]{Proposition}
\newtheorem*{proposition*}{Proposition}
\theoremstyle{primedtheorem}

\theoremstyle{swappedplain}

\newtheorem{corollary}[section]{Corollary}
\newtheorem*{corollary*}{Corollary}

\newtheorem{definition}[section]{Definition}
\newtheorem*{definition*}{Definition}
\newtheorem{notation}[section]{Notation}
\newtheorem*{notation*}{Notation}
\newtheorem{remark}[section]{Remark}
\newtheorem*{remark*}{Remark}
\newtheorem{example}[section]{Example}
\newtheorem*{example*}{Example}
\newtheorem{observation}[section]{Observation}
\newtheorem*{observation*}{Observation}
\newtheorem{porism}[section]{Porism}

\newtheorem{problem}[section]{Problem}
\newtheorem{conjecture}[section]{Conjecture}
\newcommand{\sign}{\mathrm{sign}}
\newcommand{\image}{\mathrm{Im}\,}

\newcommand{\transpose}{\mathrm{T}}
\newcommand{\lcsu}{lcsu}
\newcommand{\proj}[1]{\operatorname{Proj}(#1)}
\newcommand{\extprod}[1]{\Lambda^{#1}}

\chapterstyle{lyhne}

\newcommand{\ruledesc}[1]{\scshape{#1}} %distances seem compeltely foobar.
\setsecheadstyle{\ruledesc}
\setaftersecskip{5pt}
\renewcommand{\subsection}[1]{\section{#1}}

\renewcommand{\span}{\mathrm{span}}

\newcommand{\todo}[1]{}

\usepackage{ccfonts}
\usepackage[T1]{fontenc}

\title{$L^2$-Betti Numbers of Locally Compact Groups}
\author{Henrik Densing Petersen}
\date{November 2012}

\begin{document}

\frontmatter
\maketitle

\thispagestyle{empty}

\newpage
\thispagestyle{empty}

\underline{$L^2$-Betti Numbers of Locally Compact Groups}

\begin{flushright}
Henrik Densing Petersen

Department of Mathematical Sciences,

University of Copenhagen,

Universitetsparken 5,

DK-2100 K{\o}benhavn {\O}.

hdp (at) math.ku.dk
\end{flushright}

\vfill

\underline{PhD thesis submitted to:} \hfill \underline{Assessment commitee:}

PhD School of Science, \hfill Erik Christensen (K{\o}benhavns Universitet)

Faculty of Science, \hfill Andreas Thom (Universit{\"a}t Leipzig)

University of Copenhagen, \hfill Nadia Larsen (Universitetet i Oslo)

November 30, 2012. \hfill \underline{Academic advisor:}

ISBN 978-87-7078-993-6 \hfill Ryszard Nest (K{\o}benhavns Universitet)

\newpage

%\thispagestyle{empty}
%\phantom{blablabla}

%\newpage

\thispagestyle{empty}
\begin{flushright}
\textit{Dedicated to Henrik Fischer and Thomas Pedersen.}
\end{flushright}

\plainbreak{13}

%\begin{figure}
%\begin{center}
%\includegraphics[width=8.0cm]{never.jpg}
%\end{center}
%\end{figure}

%\vfill

%\includegraphics[width=2.0cm]{snowflake-icon.jpg}

\newpage

\phantom{bla}
\plainbreak{9}

\subsubsection*{Abstract}
We introduce a notion of $L^2$-Betti numbers for locally compact, second countable, unimodular groups. We study the relation to the standard notion of $L^2$-Betti numbers of countable discrete groups for lattices. In this way, several new computations are obtained for countable groups, including lattices in algebraic groups over local fields, and Kac-Moody lattices.

We also extend the vanishing of reduced $L^2$-cohomology for countable amenable groups, a well known theorem due to Cheeger and Gromov, to cover all amenable, second countable, unimodular locally compact groups.

\subsubsection*{Resum{\'e}}
Vi introducerer $L^2$-Betti tal for lokalkompakte, anden t{\ae}llelige, unimodul{\ae}re grupper. Disses relation med $L^2$-Betti tal for gitre i lokalkompakte grupper unders{\o}ges. S{\aa}ledes f{\aa}s flere nye udregninger for t{\ae}llelige grupper, for eksempel for gitre i algebraiske grupper over lokale legemer, og Kac-Moody gitre.

Vi viser ogs{\aa} at den reducerede $L^2$-kohomologi forsvinder for alle amenable anden t{\ae}llelige, unimodul{\ae}re lokalkompakte grupper. Dette udvider et velkendt resultat af Cheeger og Gromov for t{\ae}llelige grupper.

\subsubsection*{Preface}
This is an updated version of the version of my thesis, defended on January 18, 2013, taking comments, corrections, and discussion of the committee into account, and fixing a couple of additional howlers as well. Of course, any remaining inaccuracies are entirely my fault.

\phantom{ } \hfill - HDP, February 2013.

%%%%%%%%%%%%%%%%%%%%%%%%%%
%\begin{comment}
\newpage

\phantom{bla}
\plainbreak{7}

\subsubsection*{Acknowledgements}
I owe a special debt of gratitude to Flemming Tops{\o}e and Erik Christensen. As a student, I was fortunate to follow Flemming's course on Information Theory, and after that to be able to work with Flemming on 'Universal prediction and optimal coding'. During this time and ever since, Flemming's enthusiasm and personal kindness has been a reliable source of strength for me. My first forray into operators on Hilbert space was in Erik's course on Operator Theory, where we (Adam and I) each week got to enjoy Erik's friendly tutelage and sharp insights. Hilbert spaces and operators on them hold a special place in my heart of hearts thanks to this experience.

During my time as a Ph.D.~student I have been fortunate to have opportunities to meet and interact with many extraordinary mathematicians, and it is a pleasure thank them here.

It has been a pleasure for me to visit with Roman Sauer, Antoine Gournay, Alain Valette, Pierre de la Harpe, Stefaan Vaes, David Kyed, Thomas Schick, Kate Juschenko, Nicolas Monod, and Damien Gaboriau. During my stays at their various institutions of higher learning, I had opportunity to discuss $L^2$-Betti numbers and enjoy the benefit of insightful comments and questions.

A very special opportunity for me to live in Paris was during the very interesting program on von Neumann algebras and ergodic theory of group actions held in the spring of 2011 and I want to thank the organisers, Damien Gaboriau, Sorin Popa, and Stefaan Vaes for making it happen. I also want to thank the special people I met there for the good times.

During the spring of 2012 I stayed at EPFL for two months on the invitation of Kate Juschenko and Nicolas Monod. I am deeply grateful to Kate and Nicolas for their varm hospitality.

I thank Pierre de la Harpe, Roman Sauer, and Nicolas Monod for bringing interesting references to my attention, Thomas Danielsen and David Kyed for reading and commenting on parts of the present text, and especially Stefaan Vaes for spotting a critical error in an early version.

Finally, it is with the deepest affection and admiration that I thank Ryszard Nest for all his help and all our discussions. It is hard to explain to the outsider exactly what it is like to work with Ryszard; the best I can do is that it feels a bit like having a direct connection to 'the source', whatever that means. It has been truly an inspiration.

%\end{comment}
%%%%%%%%%%%%%%%%%%%%%%%%%%%%%%%%%%%%%%%%%

\newpage
\settocdepth{chapter}
%\settocdepth{subsection}
\tableofcontents*

%\newpage
%\listoftodos

\mainmatter

\chapter{Introduction}
%\epigraph{I think it's nice to be nomadic. It's not even like we really live in Toronto. I mean I'm sub-letting at the moment, but I'm being evicted. I'm nocturnal. The walls are thin.}{Alice Glass}
\epigraph{It is good to have an end to journey toward; but it is the journey that matters, in the end.}{Ernest Hemingway}

The theory of $L^2$-Betti numbers has proved tremendously useful, particularly in group theory and geometry (see Section \ref{sec:history} for a quick panoramic glance). For a discrete group $\Gamma$, the $L^2$-Betti numbers are positive extended-real numbers $\beta^n_{(2)}(\Gamma) \in [0,\infty], \; n\in \mathbb{N}_0$. They are coarse enough to be computable in many interesting cases, yet reflect enough properties of $\Gamma$ that their computation is not a meaningless exercise.

The goal of the present text is to suggest a definition of $L^2$-Betti numbers for locally compact (unimodular) groups. Let me indicate what I think are some good reasons to engage in such a task.

\subsubsection*{Lattices and orbit equivalence}
Measurable group theory is the study of groups, for instance countable discrete groups, via their actions on measure spaces. The standard setup is a countable discrete group $\Gamma$ acting on a standard probability space $(X,\mu)$ freely, and preserving the measure. Many striking methods and results have been developed in this area, and in particular dealing with \emph{rigidity}, that is, how much information is retained by the orbit equivalence relation $\mathcal{R}(\Gamma \curvearrowright X)$ induced by the action? See e.g.~\cite{FurmanMERigidity,FurmanOERigidity,MoSh06,KidaRigidity,PopaCocycle}.

D.~Gaboriau proved in \cite{Ga02} that $\mathcal{R}(\Gamma \curvearrowright X)$ completely remembers the $L^2$-Betti numbers of $\Gamma$. In fact, a stronger result holds allowing one to conclude that the $L^2$-Betti numbers are proportional for groups which are \emph{measure equivalent} in the sense of Gromov. The prototypical examples of measure equivalent groups come from considering two lattices $\Gamma,\Lambda \subseteq G$ in a locally compact group. Then the actions $\Gamma\curvearrowright G/\Lambda$ and $\Lambda \curvearrowright G/\Gamma$ are measure equivalent, and the conclusion is that $\beta^n_{(2)}(\Gamma) = c\cdot \beta^n_{(2)}(\Lambda)$ for all $n\geq 0$, and where $c$ is in fact the ratio of covolumes.

In particular, when $G$ is discrete and $\Gamma \subseteq G$ is a finite index subgroup one has $\beta^n_{(2)}(\Gamma) = [G:\Gamma]\cdot \beta^n_{(2)}(G)$. This was well known since the beginning, and the fastest proof is via cohomology: the $L^2$-Betti numbers can be defined as $\beta^n_{(2)}(\Gamma) := \dim_{L\Gamma} H^n(\Gamma,\ell^2\Gamma)$ where the dimension $\dim_{L\Gamma}$ is the extended von Neumann dimension function of L{\"u}ck. By the Shapiro lemma, $H^n(\Gamma,\ell^2\Gamma) \simeq H^n(G,\ell^2G)$ for all $n$ and this immediately implies the claim by an easy observation relating $L\Gamma$-dimension to $LG$-dimension.

This suggests a cohomological proof of Gaboriau's theorem in the special case of lattices by establishing similar results when the ambient group $G$ is a locally compact group and $\Gamma$ is a lattice in $G$.

\subsubsection*{Higher $L^2$-Betti numbers}
The ideas just mentioned are part of a more general program concerning the (co)homological algebra point of view of $L^2$-Betti numbers. Namely, if one can establish relations between the cohomology of lattices and that of the ambient group in some sense, then the hope would be that one can bypass the oftentimes very wild behaviour of discrete groups and consider instead more mildly mannered locally compact groups.

In the following this manifests in particular in several non-vanishing results for \emph{higher} $L^2$-Betti numbers. By 'higher' we mean typically $\beta^n_{(2)}(-)$ for $n\geq 3$, or maybe even just $n\geq 2$. While several results are known for vanishing of $L^2$-Betti numbers, also the higher ones, non-vanishing results are, to the best of my knowledge, only known for lattices in Lie groups (due to Borel \cite{Borel85sym}), and for product groups via the K{\"u}nneth formula (and groups measure equivalent to these). On the other hand, many more interesting non-vanishing results are known for the first $L^2$-Betti number. 

The following argument is a little doctored, since it relies on results proved below instead of intuition. Nevertheless, we obtain below several \emph{natural} non-vanishing results for higher $L^2$-Betti numbers of lattices in locally compact groups.

\subsubsection*{Not all graphs are Cayley graphs}
For a finitely generated group $\Gamma$ and a (finite) symmetric generating set $S$, the Cayley graph $\mathcal{G}(\Gamma,S)$ has vertex set $\Gamma$ and an edge $(\gamma,\gamma')$ whenever $\gamma^{-1}\gamma'\in S$. The first $L^2$-Betti number of $\Gamma$ can be described in terms of the space of harmonic Dirichlet functions on $\mathcal{G}(\Gamma,S)$. Recall that a function $f\colon \Gamma\rightarrow \mathbb{C}$ is a harmonic Dirichlet function if it satisfies (given no $s\in S$ has order two, $\bbb\notin S$)
\begin{equation}
\forall \gamma\in \Gamma : f(\gamma) = \frac{1}{\sharp S}\sum_{s\in S} f(\gamma s), \quad and \quad \sum_{\gamma\in \Gamma, s\in S} \lvert f(\gamma)-f(\gamma s) \rvert^2 < \infty. \nonumber
\end{equation}

In particular, $\beta^1(\Gamma) = 0$ if and only if the only harmonic Dirichlet functions are the constant ones. The definition of harmonic Dirichlet functions makes sense on any locally finite countable graph, and Gaboriau in \cite{Ga05} defines a notion of first $L^2$-Betti number, $\beta^1_{(2)}(\mathcal{G})$ for any locally finite, vertex-transitive, unimodular graph $\mathcal{G}$, i.e.~there is a closed, unimodular subgroup $G\subseteq \operatorname{Aut}(\mathcal{G})$ which acts transitively on the vertex set of $\mathcal{G}$. As in the group case, the definition fits such that $\beta^1_{(2)}(\mathcal{G})$ vanishes if and only if the only harmonic Dirichlet functions are the constant ones, and further, $\beta^1_{(2)}(\Gamma) = \beta^1_{(2)}(\mathcal{G}(\Gamma,S))$ for a finitely generated group $\Gamma$. Using the machinery for standard equivalence relations, the (non-)vanishing of the first $L^2$-Betti number has important percolation-theoretic implications for the graph.

If one could extend the equality of first $L^2$-Betti numbers for $\Gamma$ and its Cayley graphs to an equality $\beta^1_{(2)}(\mathcal{G}) = \beta^1_{(2)}(G)$ this would further strenghten the interplay between probability and group theory.

\subsubsection*{Locally compact groups are interesting!}
Finally I would offer what I consider the most compelling reason to define and study a notion of $L^2$-Betti numbers for locally compact groups: locally compact groups are interesting in their own right; and this is \textbf{self-evident}. Hence no excuses are really needed: it is OK to study locally compact groups as an end in itself, and given the relevance of $L^2$-Betti numbers for the special class of discrete groups, generalizing the definition with that in mind is a natural thing to do.

\plainbreak{2}

Presently I want to discuss some aspects of $L^2$-Betti numbers of locally compact groups in an informal manner. The main results will be stated more formally below.

\subsubsection{The homological algebra approach}
Whereas the original definition of $L^2$-Betti numbers (Atiyah and Cheeger-Gromov \cite{AtiyahL2,ChGr86}) was geometric / analytic in nature and intent, L{\"u}ck in the nineties recast it completely within the framework of homological algebra \cite{Lu02}. Considering, for instance, the cohomologies $H^n(\Gamma,\ell^2\Gamma)$ these carry naturally a right-action of the group von Neumann algebra $L\Gamma$ of $\Gamma$, which allows to define the $L^2$-Betti numbers as the \emph{$L\Gamma$-dimension}
\begin{equation}
\beta^n_{(2)}(\Gamma) := \dim_{L\Gamma} H^n(\Gamma,\ell^2\Gamma) \nonumber
\end{equation}
in the sense of L{\"u}ck. The dimension function $\dim_{L\Gamma}$ is defined for any $L\Gamma$-module, in the purely algebraic sense that we consider $L\Gamma$ just as a ring. Recall (see e.g.~Appendix \ref{app:prelims}) that $L\Gamma$ comes equiped with a canonical trace $\tau$, on which the dimension function depends, so that we may write $\dim_{(L\Gamma,\tau)}$ when we want to emphasize this. L{\"u}ck's approach allows to bring to bear all the usual tools of classical homological algebra, including spectral sequences.

The definition of $L^2$-Betti numbers we make below is in this spirit. That is, given a locally compact unimodular group $G$ (below we have also the blanket assumption that $G$ be second countable), we want to construct a dimension function $\dim_{LG}$ on the category of $LG$-modules and then define the $L^2$-Betti numbers as the $LG$-dimensions of suitable cohomology spaces.

In the general locally compact group case, the group von Neumann algebra comes equiped with a canonical weight (see Appendix \ref{app:prelims}), which we usually denote $\psi$. Further, $\psi$ is a trace exactly when $G$ is unimodular, and extends the construction of the trace on the von Neumann algebra of a discrete group. In general $\psi$ depends on the choice of scaling of the Haar measure $\mu$ on $G$, whence so will the dimension function.

Whereas the trace is finite in the discrete case, one only has a semi-finite trace in general. Even in the case where $G$ is a compact group, so that $LG$ has a finite trace, the \emph{canonical trace} which we use is not finite unless $G$ is discrete.

The first goal of this text is to make a definition (see Section \ref{sec:elltwolocallycompactdef})
\begin{equation}
\beta^n_{(2)}(G,\mu) := \dim_{(LG,\psi)} H^n(G,L^2G) \nonumber
\end{equation}
of $L^2$Betti numbers for any locally compact (second countable) unimodular $G$. If we take for granted that a nice dimension always exists, this still leaves open the choice of a "suitable" cohomology theory.

\subsubsection{Which cohomology, exactly?}
While one could in principle take a locally compact group $G$, forget the topology, and consider cohomology $H^*(G,-):=\operatorname{Ext}^*_{\mathbb{C}G}(\mathbb{C},-)$ where $\mathbb{C}G:=\span_{\mathbb{C}} \{ \delta_g \mid g\in G\}$ is the usual group algebra of a discrete group, this is not exactly very useful. Indeed, considering for instance the second degree cohomology one ideally wants this to classify some class of extensions; but if we forget the topology, we are only classifying extensions of the discrete group $G$. This is not necessarily bad in any objective sense, but of course it is natural to consider a construction which retains some topological information about $G$.

Constructing such a theory is not as straight-forward as one might like. Indeed, the way to take into account the topology of $G$ is to consider only modules with some additional structure, for instance one can consider topological vector spaces on which $G$ acts continuously. This presents new obstacles. Indeed, the category of topological vector spaces admitting a continuous action of $G$ is not abelian, and so does not fall within the classical framework for homological algebra.

Nonetheless, several useful cohomology theories have been contructed for locally compact groups. In particular, the measurable cohomology groups of C.C.~Moore \cite{Mo63I,Mo63II,Mo76I,Mo76II}, and continuous cohomology \cite{BorelWallachBook,Guichardetbook}. The choice of measurable cohomology is not so crazy, recalling Weil's result that the Haar measure essentially determines the topology. Fortunately, recent results \cite{AuMo11} show that the two theories coincide for Fr{\'e}chet modules, in particular $L^2G$; however, we do choose to work in the framework of continuous cohomology since this seems the more studied of the two, with many useful results already in the literature.

The continuous cohomology $H^n(G,L^2G)$ is itself a vector space with a natural (quotient) topology, which is not necessarily Hausdorff. We will also consider the largest Hausdorff quotient $\underline{H}^n(G,L^2G)$ and the associated \emph{reduced} $L^2$-Betti numbers of $G$. In Proposition \ref{prop:reducedelltwolimit} we show that $\dim_{LG} \underline{H}^n(G,L^2G) = 0$ if and only if $\underline{H}^n(G,L^2G) = 0$. The same is not true for $H^n(G,L^2G)$, even for $G$ countable discrete: in that case $\beta^n_{(2)}(G) = 0$ for all $n$ when $G$ is amenable, but it is known that $H^1(G,L^2G)$ is Hausdorff if and only if $G$ is non-amenable whence in particular $H^1(G,L^2G)\neq 0$. In particular this remark also should indicate that the question of whether the cohomology is Hausdorff is interesting in and of itself.

\subsubsection*{Scope of the definition}
As just mentioned, we define the $L^2$-Betti numbers for locally compact unimodular groups, also assuming second countability. Here are some remarks to justify working with specifically these objects.

At first glance we might take the word 'group' for granted here, but in fact it is the most arbitrary of all the choices pertaining to the scope. The only reason for restricting attention to groups (instead of say, groupoids) is that this was the most immediate concern, and of course conceptually a necessary first step in any case. We leave further generalizations to future work.

Local compactness is an obvious requirement. The class of groups admitting a Haar measure essentially correspond to locally compact groups by a classical theorem of Weil, and without a Haar measure and the group von Neumann algebra one eventually construct from that, I just don't know how to define a suitable dimension function. (Nor is it clear what exactly the right coefficient module would be, given that we approach the problem from the point of view of cohomology.)

Unimodularity seems equally obvious then. The canonical weight on the von Neumann algebra is a trace exactly when the group is unimodular, and again this is needed for the dimension function (to be canonical given the group).

Finally, there is the requirement that the group $G$ should be second countable, i.e.~that it has a countable neighbourhood basis. This ensures in particular that $L^2G$ is separable. It also implies that $G$ is $\sigma$-compact, which crucially lets us write spaces of cochains in continuous cohomology as \emph{countable} projective limits. Countability is needed here because the dimension function does not in general behave "continuously" under uncountable projective limits, just like the Haar measure will not be continuous under countable unions/intersections.

But we also want $G$ to have a countable neighbourhood basis at the identity. Indeed, we develop the dimension function only in the context of $\sigma$-finite von Neumann algebras since certain arguments to show (non-)vanishing of dimension boil down essentially to $\varepsilon /2^n$ type arguments on an orthogonal decomposition $\bbb = \sum_{i\in I} p_i$ of the identity in $LG$. In order to do this we need the index set $I$ to be at most countable, and in certain situations I see no other way to ensure this but to force it as a blanket assumption.

By the Birkhoff-Kakutani theorem \cite[Theorem 1.22]{MoZi55}, having a countable neighbourhood basis at the identity implies that $G$ is metrizable  (with a right-invariant metric). Then $\sigma$-compactness implies that $G$ is separable whence in fact second countable. Thus the assumption.

I do want to comment that many things should still hold even if we drop the assumption of $\sigma$-compactness, and I think that this is a reasonable thing to do. For instance, I do not see any a priori reasons why it is not interesting to consider uncountable discrete groups, but there are as mentioned above some problems of analysis associated with this. For instance, it is not clear to me, and indeed may even fail, that $\dim_{L\Gamma} H^n(\Gamma,\ell^2\Gamma) = \dim_{L\Gamma} H_n(\Gamma,\ell^2\Gamma)$ when $\Gamma$ is an uncountable discrete group. Thus there may be some loss of flexibility in the definition. Secondly, I do not at the time of writing know any examples of say, locally compact groups $G$ which are metrizable but not separable, and would be obviously interesting from the point of view of $L^2$-invariants.

\subsubsection*{Cocompact lattices}
In general, establishing an equality
\begin{equation} \label{eq:latticewant}
\beta^n_{(2)}(G,\mu) = \operatorname{covol}_{\mu}(\Gamma)\cdot \beta^n_{(2)}(\Gamma)
\end{equation}
when $\Gamma$ is a lattice in $G$ turns out to be more subtle than in the case of finite index inclusions of discrete groups. In Chapter \ref{chap:cocompact} we establish equation \eqref{eq:latticewant} in the special case where $\Gamma$ is a cocompact lattice in $G$. This turns out to be much easier than the general case since the Shapiro lemma in continuous cohomology provides an isomorphism $H^n(\Gamma,\ell^2\Gamma) \simeq H^n(G,L^2G)$ when $\Gamma$ is cocompact, but not in general.

Another subtlety compared to the finite index case is the comparison of the dimension functions. For a finite index inclusion $\Lambda \subseteq \Gamma$ of discrete groups we can restrict the $L\Gamma$ action on \emph{any} module $E$ to an action of $L\Lambda$, and we have $\dim_{L\Lambda} E = [\Gamma:\Lambda]\cdot \dim_{L\Gamma}E$.

For an inclusion $\Gamma\subseteq G$ where $\Gamma$ is a lattice in $G$, this is not the case, and indeed one can produce examples where $\dim_{LG}E = 0$ but $\dim_{L\Gamma}E = \infty$ (see Example \ref{ex:dimrestrictionineqsharp}). In Section \ref{sec:dimlattice} we show that equality holds on Hilbert modules, and we then change dimensions by approximating the cohomology by such, in the sense of projective limits.

\subsubsection*{Connected groups; Lie groups}
By the solution of Hilbert's fifth problem \cite{MoZi55}, any connected locally compact group is a compact extension of a Lie group. We will show that extensions by compact groups do not change the $L^2$-Betti numbers, and so computation in this case reduces to (connected) Lie groups, at least in principle.

A theorem of van Est allows to relate the continuous cohomology of a Lie group $G$ with the (relative) cohomology of its Lie algebra $\mathfrak{g}$, significantly simplifying computations. Then, when $G$ is semi-simple, one decomposes $L^2G$ using detailed knowledge of the representation theory, analyzing the cohomology of each discrete series representation individually. In the first serious test of our definition, we compute in Section \ref{sec:sltwoRexample} the $L^2$-Betti numbers of $SL_2(\mathbb{R})$ in a very direct and explicit manner along the outline just given.

Passing to the Lie algebra and the (relative) cohomology of this, one loses a priori all topological information about $H^n(G,L^2G)$. An important consequence of the $LG$-module structure is that one can still use the Lie algebra to show that the \textbf{reduced} cohomology $\underline{H}^n(G,L^2G)$ vanishes. Indeed by the remarks above, this follows if the $LG$-dimension of the Lie algebra cohomology vanishes. Thus in this important special case, the dimension function implicitly allows to keep track of topological / analytic information.

\subsubsection{Totally disconnected groups}
From a technical perspective, there is really no reason to restrict attention to discrete groups instead of (unimodular) totally disconnected groups.

Recall that any totally disconnected group $G$ contains a neighbourhood basis at the identity consisting of compact open subgroups $(K_n)_{n\in \mathbb{N}}$. While a compact open subgroup $K$ need not be normal (in which case we would consider the countable discrete group $G/K$), the inclusion $K\subseteq G$ still in many ways resembles a discrete group. In particular, convolution by the indicator function $\frac{1}{\mu(K)}\bbb_{K}$ on $L^2G$ is a projection with range $L^2(K\backslash G)$, and whereas in general $\dim_{LG}L^2G = \infty$, we have $\dim_{LG} L^2(K \backslash G) = \frac{1}{\mu(K)}$. Further, the projections $\frac{1}{\mu(K_n)}\bbb_{K_n}$ increase to the identity in $G$. This adds a layer to every argument where one argues to use this approximation, after which everything proceeds as normal.

In particular, the proof that one can define $L^2$-Betti numbers via homology or cohomology as one pleases has a very direct generalization to totally disconnected groups. In fact, early on in the project I worked exclusively with continuous cohomology since that was what I knew from the monographs \cite{BorelWallachBook,Guichardetbook}, but as work progressed so did the realization dawn that many things could be written down easier and in a more direct manner, using homology. There are (to my eyes) surprisingly few research papers involving continuous homology, compared to the substantial literature devoted to cohomology.

We also establish equation \eqref{eq:latticewant} for any lattice in any totally disconnected group. There is still no direct isomorphism of cohomologies but there is a map, and along the lines of the remarks just made the property which allows us to conclude that this map is an isomorphism "up to dimension" is exactly that everything happens on finite direct sums of spaces of the form $L^2(K\backslash G)$ for $K\subseteq G$ a compact open subgroup.

\subsubsection*{Amenable groups}
A very important result is the vanishing of all the $L^2$-Betti numbers of any (infinite) amenable discrete group. The main result in the present text is an extension of this to all (unimodular, second countable) locally compact groups.

This was already established in degree one for connected groups, essentially by Delorme \cite{De77}, the same paper as the property $(T)$ result. The proof in that paper is very explicit and concrete. However, wanting to establish the vanishing in all degrees, we are forced to take a more general approach, using structure theory (for connected amenable groups) and spectral sequences to reduce the computation to explicit groups - in fact we can reduce it all the way down to an explicit computation for $\mathbb{R}$.

In the other extreme we have totally disconnected amenable groups. Backing up earlier discussion, we provide a proof of vanishing directly generalizing a proof for discrete groups. In fact we provide also a second proof, generalizing L{\"u}ck's proof for discrete groups via \emph{dimension flatness} of the group von Neumann algebra, though without so many details.

\subsubsection*{Spectral sequences}
One consequence of the fact that the category of topological modules over a locally compact group $G$ is not abelian is that the Hochschild-Serre spectral sequence cannot be constructed in full generality.

To get around this we need to develop several versions of the Hochschild-Serre spectral sequence to account for all the different possibilities. The obvious first step is that we need to use a spectral sequence for the inclusion $G_0\unlhd G$, where $G_0$ is the connected component of the identity in $G$. The quotient $G/G_0$ is a totally disconnected group, and with this in view we develop a notion of "quasi-continuous cohomology", which allows that the coefficient modules are Hausdorff only up to mod'ing out by a zero-dimensional (in the sense of $LG$-dimension) submodule.

While one can get around using the Hochschild-Serre spectral sequence developed in this context by a more direct argument, since we are only after vanishing, the idea is that the quasi-continuous cohomology is a natural theory to consider for totally disconnected groups in complete generality.

\subsubsection*{Simple examples}
We compute along the way the $L^2$-Betti numbers of several examples and classes of groups, including $SL_2(\mathbb{R}$, $Sp_{2n}(\mathbf{F}_q((t)))$, and all amenable groups. He is one example not touched upon:

\begin{problem}
Compute the $L^2$-Betti numbers of a (topologically) simple, totally disconnected group without lattices, e.g.~the group contructed in \cite{nolattice}.
\end{problem}

\subsubsection*{Exact categories}
Thanks in large part to the author's lack of knowledge, the framework for homological algebra used throughout is somewhat haphazard. Everything should fall within the framework of exact categories and be expressed as such, though we only make a tiny digression on this in Appendix \ref{chap:QC}. Too bad!

\plainbreak{3}

\subsubsection*{Statement of results}
The main text is divided in two parts: the first, consisting of five chapters (\ref{chap:bettidef} through \ref{chap:Ramen}) develops the central theme, namely the $L^2$-Betti numbers of locally compact groups. The second, consisting of five appendices (\ref{app:prelims} through \ref{app:cohomLie}) develops various auxiliary results. In addition there is the introduction you are presently reading, and a brief 'prologue' chapter, containing some general references of interest for $L^2$-Betti numbers.

Of these, Chapters \ref{chap:bettidef} through \ref{chap:Ramen} represent original research and ideas carried out during my time as a Ph.D.~student at the University of Copenhagen, as do Appendices \ref{app:dimension} and \ref{chap:QC}. The remaining chapters and appendices are mainly expository in nature, even if they may contain material that has not been published elsewhere, or maybe just not exactly in the form presented here. While the text is logically structured in such a way that one should start at Appendix \ref{app:prelims}, read through all of the appendices and then wrap around to Chapter \ref{chap:bettidef}, I expect that expert readers can safely start at Chaper \ref{chap:bettidef} and refer to the appendices as needed.

Appendix \ref{app:prelims} recalls various standard results and fixes notation to be used throughout.

\subsubsection*{Results on the dimension function}
In Appendix \ref{app:dimension} we consider a $\sigma$-finite von Neumann algebra $\mathscr{A}$ with a semi-finite, faithful normal tracial weight $\psi$. On the category of right-$\mathscr{A}$-modules we then construct a dimension function $\dim_{(\mathscr{A},\psi)}$ such that:
\begin{itemize}
\item For any projection $p\in \mathscr{A}$ we have
\begin{equation}
\dim_{(\mathscr{A},\psi)} p(L^2\psi) = \psi(p),\nonumber
\end{equation}
where $L^2\psi$ is the Hilbert space in the GNS contruction for $(\mathscr{A},\psi)$. In other words, $\dim_{(\mathscr{A},\psi)}$ extends the well known \emph{von Neumann dimension},
\item When $\psi$ is a finite trace, our definition reduces exactly to that of L{\"u}ck \cite{Lu02}.
\end{itemize}

The dimension $\dim_{(\mathscr{A},\psi)} E$ is defined as the supremum of "von Neumann dimensions" $(Tr\otimes \psi)(p)$ of modules $p\mathscr{A}^n$, where $p\in M_n(\mathscr{A})$ is a projection, embedding in $E$.

We refer to the dimension function $\dim_{(\mathscr{A},\psi)}$ as the extended von Neumann dimension, L{\"u}ck's dimension function, or simply $\mathscr{A}$-dimension. It has a a number of nice properties that one would expect from a "dimension", i.e.~it satisfies continuity properties for increasing unions and decreasing intersections, and a version of the rank theorem from linear algebra. See Theorem \ref{thm:dimensionsummary} for a summary. The first part of the appendix is devoted to proving all these basic properties. 

One very important difference between the extended von Neumann dimension and the classical dimension of vector spaces is that one can have $\mathscr{A}$-modules $E\neq 0$ such that $\dim_{(\mathscr{A},\psi)}E = 0$. In Section \ref{sec:rankdensity} we prove a very useful criterion for vanishing of $\mathscr{A}$-dimension, due to Sauer \cite{Sau03} for finite trace:

\begin{lemma*}[(See Lemma \ref{lma:sauerslocalcriterion})]
Let $E$ be a right-$\mathscr{A}$-module. Then $\dim_{(\mathscr{A},\psi)}E = 0$ if and only if for every $x\in E$ there is a sequence of projections $p_n\nearrow \bbb$ in $\mathscr{A}$ such that $x.p_n = 0$ for all $n$.
\end{lemma*}

This result, referred to as the (or Sauer's) \emph{local criterion}, is useful since it enables one to show vanishing of $\mathscr{A}$-dimension not by analyzing all possible embeddings of projective modules, but by considering the action on individual given elemenents, where analysis is conceptually much easier and more direct.

When the trace is finite, one can consider a subcategory of the category of all modules consisting of \emph{rank complete} modules, and there is a (idempotent) \emph{completion} functor into this. In Section \ref{sec:rankhahnbanach} we recall this and make some further remarks. In particular, we advocate an analogy between rank complete $\mathscr{A}$-modules and vector spaces (in particular, every vector space is a rank complete $\mathbb{C}$-module), with $\mathscr{A}$-equivariant linear maps corresponding to linear maps. With this in mind we prove an extension theorem for morphisms between rank complete modules:

\begin{theorem*}[(See Theorem {\ref{thm:rankhahnbanach}})]
Suppose that $\psi$ is a finite, faithful, normal trace on $\mathscr{A}$ Let $E\subseteq F$ and $Y$ be rank complete $\mathscr{A}$-modules. Then any $\mathscr{A}$-module homomorphism $\varphi\in \operatorname{hom}_{\mathscr{A}}(E,Y)$ extends to an $\mathscr{A}$-morphism $\bar{\varphi}\in \operatorname{hom}_{\mathscr{A}}(F,Y)$.
\end{theorem*}

While the proof is a standard application of Zorn's lemma, this observation has a number of useful consequences. An abstract way of stating the result is that, on the category of rank complete modules, the functor $\operatorname{hom}_{\mathscr{A}}(-,Y)$ is exact. 

By definition, semi-finiteness of $\mathscr{A}$ means that the identity is an orthogonal sum of finite projections. This allows us to compare the dimension function for $(\mathscr{A},\psi)$ with dimension functions for finite traces, defined on corners of $\mathscr{A}$. Through this, the above result leads to "dimension exactness" results for the hom-functor $\operatorname{hom}_{\mathscr{A}}(-,\mathscr{A})$ and, by adjointness, for the induction functor $-\otimes_{\mathscr{A}}\mathscr{B}$ where $\mathscr{A}\subseteq \mathscr{B}$ is an inclusion of semi-finite von Neumann algebras. See Theorem \ref{thm:homdimexactpreserv}.

These results circumvent the unfortunate fact that in the semi-finite case, one does not seem to have a suitable notion of rank completion. In Section \ref{sec:quasimod} we consider an alternative construction of a localization of the category of all modules wrt.~morphisms that are isomorphisms in dimension (i.e.~have zero $\mathscr{A}$-dimensional kernels and cokernels) which does not require the completion. This section is by intention less formal, and no proofs are given.

We also show in Section \ref{sec:dimlattice} the following result. In the statement, $\mu$ is a Haar measure on $G$, $\psi$ the corresponding canonical tracial weight, and $\tau$ the canonical, normalized trace on the von Neumann algebra of the discrete group $H$.

\begin{theorem*}[(See \ref{lma:dimrestrictionclosure})]
Let $G$ be a second countable unimodular locally compact group and $H$ a lattice in $G$. Then for every projection $p\in M_n(LG), n\geq 1$
\begin{equation}
\dim_{(LG,\psi)} pL^2\psi^n = \frac{1}{\operatorname{covol}_{\mu}(H)}\cdot \dim_{(LH,\tau)} pL^2\psi^n. \nonumber
\end{equation}
\end{theorem*}

\subsubsection*{Topological modules and relative homological algebra}
The rest of the appendices establish various results in continuous cohomology of locally compact groups, and cohomology of Lie algebras. These theories are both "relative" cohomology theories in that one restricts the short exact sequences under consideration in the definition of "injective" and "projective" module \cite{HochschildRelative}.

In the first of these, Appendix \ref{chap:cohom}, we recall the construction of continuous cohomology. This is well known, see e.g.~\cite{BorelWallachBook, Guichardetbook}. Thus we only briefly recall the main points, and note that the cohomology spaces $H^n(G,L^2G)$ are $LG$-modules for any locally compact, unimodular group $G$, and more generally $H^n(G,E)$ are right-$\mathscr{A}$-modules whenever $E$ is a topological $G$-$\mathscr{A}$-module, meaning a topological (Hausdorff) vector space with a continuous action of $G$ and a commuting action of $\mathscr{A}$ such that each $a\in \mathscr{A}$ acts continuously.

We also recall the construction of continuous homology, in particular $H_n(G,E)$, which we do in slightly greater detail since there appears to be no comprehensive account available.

In Appendix \ref{chap:QC} we discuss a more general notion of continuous cohomology which takes the dimension function into account, termed \emph{quasi-continuous cohomology} and denoted $H^*_{\mathfrak{Q}}$. A bit loosely, we consider modules that still have a topology, but now this is assumed to be Hausdorff only up to taking quotients by zero-dimensional (in the sense of $\mathscr{A}$-dimension) submodules.

This allows a more general construction of the Hochschild-Serre spectral sequence than is possible in continuous cohomology, in particular:

\begin{theorem*}[(See Theorem \ref{thm:QCHSSS})]
Let $H\unlhd G$ be a closed normal subgroup and assume that $G$ is totally disconnected (whence so are $H$ and $G/H$). Then there is a Hochschild-Serre Spectral sequence abutting to $H_{\mathfrak{Q}}^*(G,L^2G)$ with $E_2$-term
\begin{equation}
E_2^{p,q} = H_{\mathfrak{Q}}^p(G/H,H_{\mathfrak{Q}}^q(H,L^2G)). \nonumber
\end{equation}
\end{theorem*}

This is very much in accordance with the slogan that we should try to put totally disconnected groups within the same framework as discrete groups whenever possible.

Finally in Appendix \ref{app:cohomLie} we recall the notion of smooth cohomology for Lie groups. This coincides in fact with continuous cohomology, but has the advantage that there is a direct relation with the cohomology of Lie algebras, due to van Est (see Theorem \ref{thm:vanEst}).

As we have already remarked upon above, keeping track of the dimension function allows in certain cases a substitute for having a topology on the Lie algebra cohomology, and we give the definitions with this in mind.

For the reader's convenience we also give proofs of two versions of the Hochschild-Serre spectral sequence in this setting (see Theorems \ref{thm:HSLiealgebra} respectively \ref{thm:HSmixed}). The first is a well known version for inclusions of Lie algebras, appearing also in the standard monographs \cite{BorelWallachBook,Guichardetbook} in various guises, and we make only minor changes in the exposition. The second is a "mixed case" version, which allows us to consider the case of a discrete subgroup of a Lie group $G$, in particular the case where $G$ has infinite centre.

\plainbreak{2}

\subsubsection*{The definition of $L^2$-Betti numbers}
Chapter \ref{chap:bettidef} contains the definition of $L^2$-Betti numbers for locally compact, second countable unimodular groups,
\begin{equation}
\beta^n_{(2)}(G,\mu) := \dim_{(LG,\psi)} H^n(G,L^2G), \nonumber
\end{equation}
where $\psi$ is the canonical weight on the group von Neumann algebra $LG$, corresponding to the Haar measure $\mu$ on $G$, $L^2G$ is a right-$LG$-modules via the anti-isomorphism of $LG$ with its commutant, and $H^n(G,-)$ is continuous cohomology.

We generalize two basic results from discrete groups, namely the computation of the zero'th $L^2$-Betti number
\begin{equation}
\beta^0_{(2)}(G,\mu) = \left\{ \begin{array}{cl} 0 & , G \; \textrm{non-compact} \\ \frac{1}{\mu(G)} & , G \; \textrm{compact} \end{array} \right. ,
\end{equation}
and a vanishing result for abelian groups:
\begin{theorem*}[(See \ref{thm:ltwoabelianlcgroups})]
$\beta^n_{(2)}(G,\mu) = 0$ for all $n\geq 0$ for any non-compact abelian group.
\end{theorem*}

Finally, we compute the $L^2$-Betti numbers of $SL_2(\mathbb{R})$ by analyzing the representation theory. See \ref{thm:sltwocalc}.

\subsubsection*{Results on lattices}
Using Lemma \ref{lma:dimrestrictionclosure}, relating the dimension functions $\dim_{LG}$ and $\dim_{LH}$ when $H$ is a lattice in $G$, we show in Chapter \ref{chap:cocompact} one of our main results:

\begin{theorem*}[(See Theorem \ref{cor:finitecovolcocompact})]
Let $G$ be a second countable, unimodular locally compact group with Haar measure $\mu$ and suppose that $G$ contains a cocompact lattice $H_0$. Then for every lattice (not neccesarily cocompact) $H$ in $G$ and every $n$ we have
\begin{equation}
\beta^n_{(2)}(H) = \mathrm{covol}_{\mu}(H)\cdot \beta^n_{(2)}(G, \mu). \nonumber
\end{equation}
\end{theorem*}

\noindent
The cocompactness ensures that we have an (homeomorphic) isomorphism $H^n(H_0,\ell^2H_0) \simeq H^n(G,L^2G)$ of $LH$-modules, by the Shapiro lemma. That this is a homeomorphism implies, by results on discrete groups explained in the prologue, that we can, essentially, assume that the cohomology $H^n(G,L^2G)$ is Hausdorff. This allows us, along the lines explained above, to approximate the cohomology by a limit of Hilbert spaces, on which the two dimension functions in play coincide.

In the general case of a lattice $\Gamma$ in $G$, the Shapiro lemma does not give an isomorphism of $L^2$-cohomology, but only an $L\Gamma$-equivariant (continuous) linear map
\begin{displaymath}
\xymatrix{ H^n(G,L^2G) \ar[r]^>>>>>{\phi} & H^n(G, \operatorname{Coind}_{\Gamma}^G\ell^2\Gamma) \simeq H^2(\Gamma,\ell^2\Gamma) }
\end{displaymath}
where the coinduced $L\Gamma$-module is $\operatorname{Coind}_{\Gamma}^G\ell^2\Gamma \simeq L^2_{loc}(G/\Gamma,\ell^2\Gamma)$, the space of \emph{locally square integrable} $\ell^2\Gamma$-valued functions on the quotient $G/\Gamma$; by definition this is the space of functions $\xi\colon G/\Gamma \rightarrow \ell^2\Gamma$ which are measurable and such that $\int_{K}\lVert \xi \rVert_{\ell^2\Gamma}^2 \mathrm{d}\nu < \infty$ for all compact subsets $K\subseteq G/\Gamma$, where $\nu$ is the left-invariant finite measure on $G/\Gamma$.

In particular, $L^2G\simeq L^2(G/\Gamma,\ell^2\Gamma)$ embeds continuously in $\operatorname{Coind}_{\Gamma}^G\ell^2\Gamma$, inducing the map $\phi$ above. For totally disconnected groups, everything can be considered relative to some compact open subgroup, so in spirit the situation here is closer to
\begin{displaymath}
\xymatrix{ H^n(G,L^2(K\backslash G)) \ar[r]^<<<<<{\phi} & H^n(G, (\operatorname{Coind}_{\Gamma}^G\ell^2\Gamma )^K)}
\end{displaymath}
and we can then argue that the $L\Gamma$-modules $L^2(K\backslash G)$ respectively $(\operatorname{Coind}_{\Gamma}^G\ell^2\Gamma )^K$ are isomorphic up to dimension, in the sense that the natural inclusion of the former in the latter has cokernel with $L\Gamma$-dimension zero. This leads to:

\begin{theorem*}[(See Theorem \ref{thm:totdisclattice})]
Let $G$ be a totally disconnected, second countable, unimodular locally compact group with Haar measure $\mu$, and suppose that $\Gamma$ is a lattice in $G$. Then for all $n\geq 0$
\begin{equation}
\beta^n_{(2)}(\Gamma) = \operatorname{covol}_{\mu}(\Gamma)\cdot \beta^n_{(2)}(G,\mu). \nonumber
\end{equation}
\end{theorem*}

\subsubsection*{The examples $Sp_{2n}(\mathbf{F}_q((t)))$}
In Section \ref{sec:Sp2n} we apply the results relating the $L^2$-Betti numbers of lattices in totally disconnected groups to the following situation: Let $\mathbf{F}_q$ be a finite field with cardinality $q$ and consider the non-Archimedean local field of formal Laurent series $\mathbf{F}_q((t))$.

It is known that the symplectic groups $Sp_{2n}(\mathbf{F}_q((t)))$ for $n\geq 2$ contain lattices, but that no such lattice is cocompact.

Using the action of $Sp_{2n}(\mathbf{F}_q((t)))$ on its Bruhat-Tits building, we show that once the residue field $\mathbf{F}_q$ is sufficiently large, the top $L^2$-Betti number $\beta^n_{(2)}(Sp_{2n}(\mathbf{F}_q((t))),\mu) \neq 0$. This implies then:

\begin{theorem*}[(See \ref{cor:Sp2nlattice})]
For any lattice $\Gamma \subseteq Sp_{2n}(\mathbf{F}_q((t)))$ we have, when $q$ is sufficiently large, $\beta^n_{(2)}(\Gamma) > 0$.
\end{theorem*}

This result is interesting since the lack of cocompact lattices means exactly that there are no discrete groups acting "nicely" on the Bruhat-Tits buildings, in particular with finite fundamental domain. This means that the analysis we do to compute the $L^2$-Betti numbers does not a priori pass to the action of a discrete group, so that the approach of passing through the ambient locally compact group and using the change of dimension theorem seems to be the most natural way to compute the $L^2$-Betti numbers of such lattices.

\subsubsection*{Amenable groups}
Recall that a (second countable) locally compact group $G$, where we fix a left-Haar measure $\mu$, is amenable if it admits a left-invariant mean on $L^{\infty}(G,\mu)$, that is, a positive linear functional $m\colon L^{\infty}(G)\rightarrow \mathbb{C}$ such that $m(\bbb) = 1$ where $\bbb\colon g\mapsto 1$ for all $g\in G$, and $m(g.f) = m(f)$ for all $g\in G$ and $f\in L^{\infty}(G,\mu)$. Equivalently, $G$ is amenable if and only if it admits a F{\o}lner sequence, that is a sequence $(F_n)$ consisting of (Borel) subsets $F_n\subseteq G$ of finite measure, such that for any (Borel) $F\subseteq G$ of finite measure,
\begin{equation}
\frac{\mu(F_n.F\Delta F)}{\mu(F_n)} \rightarrow_n 0 \nonumber
\end{equation}
where $\Delta$ denotes the symmetric difference. 

Our main result is an extension of Cheeger-Gromov's vanishing result for $L^2$-Betti numbers of amenable countable discrete groups \cite{ChGr86} to locally compact groups.

\begin{theorem*}[(See Theorems \ref{thm:Ramen0noncompact} and \ref{thm:RamenGeneral})]
Let $G$ be a non-compact, amenable, second countable, unimodular locally compact group. Then for all $n\geq 0$
\begin{equation}
\beta^n_{(2)}(G,\mu) = 0. \nonumber
\end{equation}
\end{theorem*}

The result itself is joint with D. Kyed and S. Vaes \cite{JointDKSV}, where it is established by entirely different methods.

Here, our approach is based on structure theory for locally compact groups, and the various spectral sequences developed in the appendices. In fact, these reductions allow us to split the proof essentially in two cases:
\begin{enumerate}[(i)]
\item First, we show that the theorem holds when $G$ is totally disconnected. It is not surprising that the proof is a direct generalization of a proof for discrete groups.
\item Second, we have to show that when $G$ is a second countable, unimodular locally compact group and $\mathbb{R}$ embeds in $G$ as a closed subgroup, then $\dim_{LG} H^n(\mathbb{R},L^2G) = 0$ for all $n\geq 0$. See Lemma \ref{lma:abelianbycompact}.
\end{enumerate}

\subsubsection*{Product groups}
In chapter \ref{chap:Kunneth} we establish the following K{\"u}nneth formula for locally compact groups.

\begin{theorem*}[(See Theorem \ref{thm:Kunnethtotdisc})]
Let $G,H$ be totally disconnected, second countable, unimodular, locally compact groups with Haar measures $\mu$ respectively $\nu$. For all $n\geq 0$
\begin{equation}
\beta^n_{(2)}(G\times H, \mu\times \nu) = \sum_{k=0}^n \beta^k_{(2)}(G,\mu)\cdot \beta^{n-k}_{(2)}(H,\nu). \nonumber
\end{equation}
\end{theorem*}

Using the "structural result" in Theorem \ref{thm:RamenGeneral} this implies essentially that, since we can compute the $L^2$-Betti numbers of any connected Lie group, the computation of $\beta^n_{(2)}(G,\mu)$ for any group $G$ reduces to the computation of the $L^2$-Betti numbers of the totally disconnected group $G/G_0$.

This leads also to non-vanishing results of (higher) $L^2$-Betti numbers of certain \emph{Kac-Moody lattices}, see Theorem \ref{thm:KacMoodylattice}. This provides in particular non-vanishing results for a large class of \emph{simple, finitely generated} (or even finitely presented) discrete groups by results in \cite{CaRe09}.

\chapter{Prologue}

The purpose of this chapter is to give a \emph{brief} account of various results which, besides their own intrinsic interest, when taken together shows that the theory of $L^2$-Betti numbers is interesting in a broad mathematical sense.

Needless to say, omissions and unsuppressed bias runs rampant. Sorry.

\section{{$L^2$}-Betti numbers for countable discrete groups} \label{sec:history} \todo{sec:history}

Let $M$ be a compact Riemannian manifold with universal covering $\tilde{M}$. Denote by $L^2\Lambda^n(\tilde{M})$ the Hilbert space of square-integrable complex-valued exterior $n$-forms on $\tilde{M}$. These are Hilbert modules over the group von Neumann algebra $L(\pi_1(M))$ of the fundamental group of $M$.

Atiyah defines in \cite{AtiyahL2} the $L^2$-Betti numbers $\beta_{(2)}^n(\tilde{M})$ as the von Neumann dimensions of the space of square-integrable harmonic $n$-forms $\mathcal{H}^n(\tilde{M})$, where $L^2\Lambda^n(\tilde{M}) = \overline{d\Lambda_c^{n-1}(\tilde{M})}\oplus \mathcal{H}^n(\tilde{M})\oplus \overline{d^*\Lambda_c^{n+1}(\tilde{M})}$ is the Hodge-Kodaira decomposition.\todo{find reference. Maybe Shubin's notes.}

\begin{theorem}{(Dodziuk {\cite{Dodziuk77}})}
The $L^2$-Betti numbers are homotopy invariants of $M$.
\end{theorem}

While the "$L^2$-Euler characteristic" defined by $\chi_{(2)}(M) := \sum_{n=0}^{\dim M} (-1)^n\cdot \beta_{(2)}^n(\tilde{M})$ coincides with the usual Euler characteristic, the $L^2$-Betti numbers have the advantage that they behave well under finite coverings, i.e.~if $N\rightarrow M$ is a $k$-sheeted covering, then
\begin{equation}
\beta^n_{(2)}(\tilde{M}) = k\cdot \beta^n_{(2)}(\tilde{N}), \nonumber
\end{equation}
which is not true for the classical Betti numbers.

Similarly, one may define the $L^2$-Betti numbers of (the universal covering of) a finite simplicial (or CW) complex.\todo{give reference}

In \cite{ChGr86}, a general notion of $L^2$-Betti numbers $\beta^n_{(2)}(\Gamma)$ is defined for any countable discrete group $\Gamma$, in such a way that for a compact Riemannian manifold $M$ (respectively finite simplicial- or CW-complex $X$), $\beta^n_{(2)}(\pi_1(M)) = \beta^n_{(2)}(\tilde{M})$ (repsectively $\beta^n_{(2)}(\tilde{X})$). A good introduction and survey is \cite{EckmannL2intro}.

In fact, the notion is defined more generally in \cite{ChGr86}: via singular cohomology one defines for any (free\todo{check}) action of a $\Gamma$ on a topological space $X$ a sequence of $L^2$-Betti numbers $\beta^n_{(2)}(X;\Gamma)$, coinciding with $\beta^n_{(2)}(\Gamma)$ if the space is contractible.

\begin{theorem}[(\cite{ChGr86})] \label{thm:CGamenable} \todo{thm:CGamenable}
For $\Gamma$ infinite amenable, $\beta^n_{(2)}(\Gamma) = 0$ for all $n\geq 0$.
\end{theorem}

This theorem shows in particular that if $X$ is a $K(\Gamma,1)$-space for $\Gamma$ amenable, then $\chi(X) = 0$. I do not know of a proof of this fact which doesn't go through $L^2$-Betti numbers.

Let $\tilde{X}$ be the universal cover of the finite $K(\Gamma,1)$-space $X:=\tilde{X}/\Gamma$. That is, $X$ is a finite CW-complex and $\Gamma=\pi_1(X)$ is its fundamental group. While the $L^2$-Betti numbers $\beta^n_{(2)}(\Gamma) = \beta^n_{(2)}(\tilde{X};\Gamma)$ do not coincide with the classical Betti numbers $\beta^n(X)$ (that is the whole point!) it is nevertheless an important part of the theory to establish comparison results. The following result, due to W. L{\"u}ck, is central, and has motivated many further developments. Most results of this type are called 'L{\"u}ck approximation theorems' in homage. See e.g.~\cite{BeGa04,Elek08,Schick01,DLMSY03} for further developments.

\begin{theorem}[(L{\"u}ck's Approximation Theorem {\cite{Lu94}})]
Suppose that $\Gamma=\pi_1(X)$ is residually finite and let $(\Gamma_k)$ be a decreasing sequence of normal subgroups with trivial intersection. Then
\begin{equation}
\beta^n_{(2)}(\Gamma) = \lim_k \frac{\beta^n(\tilde{X}/\Gamma_k)}{[\Gamma:\Gamma_k]}. \nonumber
\end{equation}
\end{theorem}

In contrast to Theorem \ref{thm:CGamenable}, \cite{ChGr86} shows also that for the free groups $\mathbb{F}_k$ on $k$ generators, $\beta^1_{(2)}(\mathbb{F}_k) = k-1$ and all other $\beta^n_{(2)}(\mathbb{F}_k) = 0$. There are many interesting results using the first $L^2$-Betti number; see the discussion below.

\begin{conjecture}[(Chern-Hodge)]\todo{find reference}
Let $M$ be a closed connected Riemannian manifold of even dimenion $\dim M = 2m$, and suppose $M$ has (strictly) negative sectional curvature. Then $(-1)^n\chi(M) > 0$.
\end{conjecture}

For $m=1$ the conjecture follows from the Gaus-Bonnet theorem, and for $m=2$ was proved by Milnor, see \cite{ChernChern}.

\begin{theorem}[(Gromov {\cite{Gr91}})]
The Chern-Hodge conjecture holds for $M$ K{\"a}hler hyperbolic.
\end{theorem}

Related to the Chern-Hodge conjecture is also the following conjecture (we state only parts of it). See \cite[Chapter 11]{Lu02}.

\begin{conjecture}[(Hopf-Singer)]
Let $M$ be a closed, connected Riemannian manifold of evendimension $\dim M = 2m$, and suppose $M$ has (strictly) negative sectional curvature. Then $\beta^m_{(2)}(\tilde{M}) > 0$ and all other $\beta^n_{(2)}(M) = 0$.

Let $M$ be be a closed, aspherical manifold of even dimension $\dim M = 2m$. Then $\beta^n_{(2)}(M) = 0$ for $n\neq m$
\end{conjecture}

A groundbreaking work in the theory of $L^2$-Betti numbers is the paper{Ga02} of Gaboriau. There, Gaboriau proves that the $L^2$-Betti numbers are measure equivalence invariants. Recall that two countable groups $\Gamma,\Lambda$ are measure equivalent (ME) with compression constant $c$ if there is a standard measure space $(X,\mu)$, in general with $\mu(X) = \infty$, admitting commuting actions of $\Gamma$ and $\Lambda$ each with a fundamental domain $F_{\Gamma}$ respectively $F_{\Lambda}$ with finite measure, and $c:= \mu(F_{\Gamma})/\mu(F_{\Lambda})$ (caveat: ME is an equivalence only omitting $c$, which is not symmetric in $\Gamma,\Lambda$).

\begin{theorem}[(Gaboriau \cite{Ga02})]
If $\Gamma$ is ME to $\Lambda$ with compression constant $c$, then for all $n$,
\begin{equation}
\beta^n_{(2)}(\Gamma) = c\cdot \beta^n_{(2)}(\Lambda). \nonumber
\end{equation}
\end{theorem}

Gaboriau's proof goes through a definition of $L^2$-Betti numbers in general for a standard measure-preserving equivalence relation on a standard probability space. He then shows that $\beta^n_{(2)}(\Gamma) = \beta^n_{(2)}(\mathcal{R}(\Gamma\curvearrowright (X,\mu)))$ for any essentially free measure-preserving action of $\Gamma$ on a standard probability space $X$.

This gives in particular a more conceptual proof of the vanishing theorem of Cheeger-Gromov for amenable groups, since the orbit equivalence relation generated by any free measure-preserving action of an amenable group is hyperfinite.\todo{reference}

Here is another application of this fact.
\begin{theorem}[(Gaboriau \cite{Ga02})]
Let $\xymatrix{0\ar[r] & N \ar[r] & \Gamma \ar[r] & \Lambda \ar[r] & 0}$ be a short exact sequence of groups and suppose that $N,\Lambda$ are infinite and that $\beta^1_{(2)}(N)<\infty$ (e.g.~$N$ is finitely generated). Then $\beta^1_{(2)}(\Gamma) = 0$.
\end{theorem}

This implies in particular that if $\Gamma$ is finitely presented, then the deficiency of $\operatorname{def}(\Gamma) := \max\{ g-r \mid \Gamma = \langle \gamma_1,\dots ,\gamma_g \mid w_1,\dots ,w_r\rangle \}$ satisfies $\operatorname{def}(\Gamma) \leq 1$.

Another breakthrough was the work of L{\"u}ck who put the theory properly inside the framework of homological algebra, using his extended von Neumann dimension \cite{Lu97,Lu98I,Lu98II}. See also his comprehensive book \cite{Lu02}.

the extended von Neumann dimension $\dim_{(\mathscr{A},\tau)}$ is defined, given a finite von Neumann algebra $\mathscr{A}$ with a fixed faithful normal trace $\tau$, on the category of all modules over $\mathscr{A}$, in the algebraic sense of a module over the ring $\mathscr{A}$. It enjoys many nice properties that one would want in a "dimension function".

Using this, L{\"u}ck can define directly the $L^2$-Betti numbers as
\begin{equation}
\beta_n^{(2)}(\Gamma) := \dim_{(L\Gamma,\tau)} \operatorname{Tor}_n^{\mathbb{C}\Gamma} (L\Gamma,\mathbb{C}). \nonumber
\end{equation}

This was unified with Gaboriau's work in \cite{Sau03} where Sauer gives an algebraic definition of the $L^2$-Betti numbers of a discrete measurable groupoid. In general, the homological algebra framework of L{\"u}ck is generalized by Connes and Shlyakhtenko in \cite{CoSh03} who define, for $R$ a weakly dense $*$-subalgebra of a finite tracial von Neumann algebra $\mathscr{A}$, the $L^2$-Betti numbers of $R$ as
\begin{equation}
\beta_n^{(2)}(R,\tau) := \dim_{(\mathscr{A}\bar{\otimes}\mathscr{A}^{op}, \tau\otimes \tau)} \operatorname{Tor}_n^{R\otimes R^{op}} (\mathscr{A}\bar{\otimes} \mathscr{A}^{op}, R). \nonumber
\end{equation}

This provides a general framework and allows e.g.~a definition of $L^2$-Betti numbers for quantum groups \cite{Ky08homology}.

The attentive reader will have noticed that we have discussed both $L^2$-Betti numbers with the index in superscript, and $L^2$-Betti numbers with indix in the subscript. The convention is that we write $\beta^n_{(2)}(-)$ or $\beta_n^{(2)}(-)$ depending on whether the exact definition of "$L^2$-Betti number" is based on cohomology respectively homology.

Generally speaking it is clear that all definitions agree under suitable finiteness assumptions, e.g.~it is obvious that the definitions of L{\"u}ck versus Cheeger-Gromov coincide for groups with classifying spaces admitting a finite model. However, for contable groups in general there are some subtleties since e.g.~the definition of Cheeger-Gromov relies on projective limits here and in that case one does not necessarily "see" all non-trivial cocycles in the algebraic sense. We show in an addendum, see Section \ref{sec:dualitydiscrete} that there is in fact no difference; probably these results are well-known to the experts, but I did not manage to find an exact statement of this in the litterature.

The main step in the argument is the following result

\begin{theorem}[(Thom, see {\cite{Thom06a,Thom06b,PeTh06}})]
For any countable group $\Gamma$ and all $n\geq 0$,
\begin{equation}
\dim_{(L\Gamma,\tau)} \operatorname{Tor}_n^{\mathbb{C}\Gamma}(L\Gamma,\Gamma) = \dim_{(L\Gamma,\tau)} \operatorname{Ext}^n_{\mathbb{C}\Gamma}(\mathbb{C},\ell^2\Gamma). \nonumber
\end{equation}
\end{theorem}

In particular, one can describe the first $L^2$-Betti number of $\Gamma$ as
\begin{equation}
\beta^1_{(2)}(\Gamma) = \dim_{(L\Gamma,\tau)} Z^1(\Gamma,\ell^2\Gamma)/B^1(\Gamma,\ell^2\Gamma), \nonumber
\end{equation}
where $Z^1(\Gamma,\ell^2\Gamma)$ is the space of $1$-cocycles, i.e.~maps $\xi\colon \Gamma\rightarrow \ell^2\Gamma$ such that $\xi(gh)=g.\xi(h)+\xi(g)$ for all $g,h\in \Gamma$. Next we will describe a number of results revolving around the first $L^2$-Betti number, arguably the coolest and most popular of all the $L^2$-Betti numbers, but first we state the following important problem:

\begin{problem}
It is well known that for any countable group $\Gamma$ there is an inequality
\begin{equation}
\beta^1_{(2)}(\Gamma) \leq \operatorname{cost}(\Gamma) - 1, \nonumber
\end{equation}
where $\operatorname{cost}(\Gamma)$ is the cost as defined and studied in \cite{Levitt95} and \cite{Ga00}. Does equality hold in general?
\end{problem}

A countable discrete group $\Gamma$ is called unitarizable if every uniformly bounded representation of $\Gamma$ in Hilbert space is similar to a unitary distribution. It is known that every amenable group is unitarizable \cite{Nagy47,Dix50,Day50}, and is an open question whether the converse is also true. A partial result in this direction is the following result, due to Epstein and Monod, keeping in mind that the $L^2$-Betti numbers all vanish for amenable groups.

\begin{theorem}[({\cite{EpMo09}})]
Every countable discrete, residually finite group $\Gamma$ with $\beta^1_{(2)}(\Gamma) > 0$ is non-unitarizable.
\end{theorem}

It was well-known that $\mathbb{F}_2$, the free group on two generators is non-unitarizable, and it is easy to deduce from this that any discrete group containing $\mathbb{F}_2$ is non-unitarizable as well. However, using results from \cite{PeTh06}, an example of a residually finite \emph{torsion} group $\Gamma$ with $\beta^1_{(2)}(\Gamma) > 0$ was constructed in \cite{Osin09}. See also \cite{LuOs11}.

On the other hand, \cite{PeTh06} contains a quite striking \emph{Freiheitssatz}, stipulating that any torsion-free countable discrete group $\Gamma$ such that every non-zero element in the complex group ring acts on $\ell^2\Gamma$ without kernel contains a free non-abelian subgroup if $\beta^1_{(2)}(\Gamma)>0$. It is an open problem whether there exists a torsion-free group not satisfying the first part of the hypothesis.

Returning to the Hopf-Singer conjecture, Sauer and Thom were able to prove the following partial result using homological algebra machinery and measured groupoids.

\begin{theorem}[(Sauer-Thom {\cite{SaTo10}})]
The Hopf-Singer conjecture holds for any closed, aspherical manifold with poly-surface fundamental group.
\end{theorem}

Finally let us mention

\begin{theorem}[(Sauer-Thom {\cite{SaTo10}})]
Let $H\unlhd G$ with $G$ countable discrete. Suppose that both groups $H$ and $G/H$ are infinite, and that $\beta^m_{(2)}(H) = 0$ for $0\leq m < n$. If $\beta^n_{(2)}(H) < \infty$ then $\beta^n_{(2)}(G) = 0$.
\end{theorem}

\section{Some Duality Results for Discrete Groups} \label{sec:dualitydiscrete}

\begin{lemma} \label{lma:moddualineq} \todo{lma:moddualineq}
Let $\mathscr{A}$ be a $\sigma$-finite, semi-finite von Neumann algebra with a faithful, normal, tracial weight $\psi$, and let $M$ be a right-$\mathscr{A}$-module. Let $Q$ be a submodule of the dual $M^{'} := \hom_{\mathscr{A}}(M,\mathscr{A})$ which separates points on $M$.

\begin{enumerate}[(i)]
\item Then, considering $Q$ as a left-$\mathscr{A}$-module (with $\mathscr{A}$ acting by post-multiplication) we have
\begin{equation}
\dim_{\psi} M \leq \dim_{\psi} Q. \nonumber
\end{equation}

\item Further, the same statement holds with $\hom_{\mathscr{A}}(M,L^2\psi)$ in place of $M^{'}$.
\end{enumerate}
\end{lemma}

\begin{proof}
$(i)$: Let $P$ be a $\psi$-fg. projective submodule of $M$, say $P\simeq p\mathscr{A}^n$. Then we have $P^{'} \simeq \mathscr{A}^np$ and from the hypothesis it follows readily that the restriction map $r:Q\rightarrow P^{'}$ has (algebraically) dense image.

It follows then that
\begin{equation}
\dim_{\psi} r(Q) = \dim_{\psi} P^{'} = \psi (p) = \dim_{\psi} P. \nonumber
\end{equation}
This proves $(i)$. The second claim is entirely analogous, or even better it follows directly from the observation that $M^{'}$ is rank dense in $\hom_{\mathscr{A}}(M,L^2\psi)$.
\end{proof}

\begin{corollary}
For any right-$\mathscr{A}$-module $M$,
\begin{equation}
\dim_{\psi} M^{'} = \dim_{\psi} \mathbf{P}M. \nonumber
\end{equation}
\end{corollary}

\begin{proof}
Recall that $\mathbf{P}M$ is that quotient of $M$ by the algebraic closure of $\{0\}$ in $M$. Thus $M^{'} = (\mathbf{P}M)^{'}$ and separates points on this. On the other hand, $\mathbf{P}M$ embeds in the dual of $M^{'}$ and separates points on this by definition.
\end{proof}

If $\Gamma$ is a countable discrete group we define a duality between $\ell^2$-cohomology $H^*(\Gamma,\ell^2\Gamma)$ and homology with coefficients in the group von Neumann algebra $H_*(\Gamma, L\Gamma)$, where we consider $L\Gamma$ a left-$L\Gamma$-right-$\Gamma$-module, as follows. For $\xi:\Gamma^n\rightarrow \ell^2\Gamma$ and $f\in \mathbb{C}\Gamma^n\otimes L\Gamma$, which we identify with finitely supported functions into $L\Gamma$ we define
\begin{equation}
\langle f,\xi \rangle := \sum_{\gamma \in \Gamma^n} f(\gamma).\xi(\gamma) \in \ell^2\Gamma. \nonumber
\end{equation}
Note that the sum is finite so that this is well defined.

By a straight-forward computation (which we leave out) one then sees the following result.

\begin{proposition}
For all $f\in \mathbb{C}\Gamma^n\otimes L\Gamma$ and $\xi:\Gamma^n\rightarrow \ell^2\Gamma$,
\begin{equation}
\langle d_n f, \xi \rangle = \langle f, d^n\xi \rangle. \nonumber
\end{equation}
\end{proposition}

\begin{flushright}
\qedsymbol
\end{flushright}

\begin{theorem} \label{thm:dualityalldiscrete} \todo{thm:dualityalldiscrete}
Let $\Gamma$ be a countable discrete group. Then we have
\begin{eqnarray}
\dim_{L\Gamma} \underline{H}^n(\Gamma, \ell^2\Gamma) & = & \dim_{L\Gamma} H^n(\Gamma, \ell^2\Gamma) \nonumber \\
 & = & \dim_{L\Gamma} H_n(\Gamma, L\Gamma) = \dim_{L\Gamma} \mathbf{P}H_n(\Gamma, L\Gamma). \nonumber
\end{eqnarray}
\end{theorem}

This theorem should be seen as a generalization of \cite[Corollary 2.4]{PeTh06}. The middle equality, which is by far the most substantial, is proved in \cite[p. 6]{PeTh06} (for general refence, see also \cite{Thom06a, Thom06b}). We will give a more direct proof of this equality below as well.

\begin{proof}
The first and third equalities are consequences of the following two equalities:
\begin{eqnarray}
\label{eq:ltwodualityone}
\dim_{L\Gamma} \underline{H}^n(\Gamma,\ell^2\Gamma) = \dim_{L\Gamma} \mathbf{P}H_n(\Gamma, L\Gamma). \\
\label{eq:ltwodualitytwo}
\dim_{L\Gamma} \underline{H}^n(\Gamma,\ell^2\Gamma) = \dim_{L\Gamma} H_n(\Gamma, L\Gamma).
\end{eqnarray}

The first of these, (\ref{eq:ltwodualityone}), follows straight from the previous proposition once we note that since $\Gamma$ is countable discrete, $L^2_{loc}(\Gamma^n, \ell^2\Gamma) = \{ \xi:\Gamma^n\rightarrow \ell^2\Gamma \}$ is isomorphic as a right-$L\Gamma$-module to $\hom_{L\Gamma}(L^2_c(\Gamma^n,L\Gamma), \ell^2\Gamma)$.

The second will take some more work but the basic observation is that while one cannot detect whether a cocycle $\xi:\Gamma^n\rightarrow \ell^2\Gamma$ is a coboundary by considering its restriction to finite sets, once a cycle is a boundary it stays a boundary, so to speak. This means we can actually get precisely $H_n(\Gamma,L\Gamma)$ as an inductive limit of finite-dimensional modules.

For all $m,n\in \mathbb{N}$ let $S_n^{(m)}\subseteq \Gamma$ be finite sets, all containing the identity, incresing in $m$ to $\Gamma$, and such that $(S_{n+1}^{(m)})^2 \subseteq S_n^{(m)}$. Denote $K_n^{(m)}= \prod_{i=1}^n S_n^{(m)}$ the $n$-fold direct product of $S_n^{(m)}$ with itself. It is then easy to see that the (co)boundary maps give well-defined maps $d_{(m)}^n: \mathcal{F}(K^{(m)}_n, \ell^2\Gamma)\rightarrow \mathcal{F}(K^{(m)}_{n+1},\ell^2\Gamma)$ and similarly the boundary maps so that we get complexes
\begin{equation}
0\rightarrow \ell^2\Gamma \xrightarrow{d^0_{(m)}} \mathcal{F}(K_1^{(m)},\ell^2\Gamma) \rightarrow \dots \nonumber
\end{equation}
and
\begin{equation}
0 \leftarrow L\Gamma \xleftarrow{d_0^{(m)}} \mathbb{C}K_1^{(m)}\otimes L\Gamma \leftarrow \dots . \nonumber
\end{equation}

Identifying $\mathbb{C}K^{(m)}_n\otimes L\Gamma$ with functions $K^{(m)}_n\rightarrow L\Gamma$ we get again by restriction a duality $\langle \cdot , \cdot \rangle_m$, for brevity usually we drop the subscript, by
\begin{equation}
\langle f,\xi \rangle_m := \sum_{\gamma \in \Gamma^n} f(\gamma).\xi(\Gamma), \quad \begin{array}{l} f\in \mathbb{C}K^{(m)}_n\otimes L\Gamma \\ \xi \in \mathcal{F}(K^{(m)}_n,\ell^2\Gamma) \end{array}. \nonumber
\end{equation}
This is $L\Gamma$-bimodular and by a direct calculation\todo{placemarker} satisfies the analogue of the previous proposition: for all $m\in \mathbb{N}, n\geq 0$ we have
\begin{equation}
\label{eq:restrictionduality}
\langle d_n^{(m)} f,\xi \rangle = \langle f, d_{(m)}^n\xi \rangle.
\end{equation}

Denote $B^n_{(m)} := \image d_{(m)}^{n-1}, Z^n_{(m)} := \ker d_{(m)}^n, B_n^{(m)} := \image d_{n}^{(m)}, Z_n^{(m)} := \ker d^{(m)}_{n-1}$. Then we have the following
\begin{lemma} \label{lma:restrictionduality} \todo{lma:restrictionduality}
For all $\xi\in \mathcal{F}(K_n^{(m)}, \ell^2\Gamma)$
\begin{enumerate}[(i)]
\item $\xi \in Z^n_{(m)} \Leftrightarrow \forall f\in B_n^{(m)}: \langle f,\xi \rangle = 0$.
\item $\xi \in \overline{B^n_{(m)}}^{\lVert \cdot \rVert_2} \Leftrightarrow \forall f\in Z_n^{(m)}: \langle f,\xi \rangle = 0$.
\end{enumerate}
Similarly, for all $f\in \mathbb{C}K_n^{(m)} \otimes L\Gamma$
\begin{enumerate}[(i)]
\item[(iii)] $f\in Z_n^{(m)} \Leftrightarrow \forall \xi \in B^n_{(m)}: \langle f,\xi\rangle = 0$.
\item[(iv)] $f\in \overline{B_n^{(m)}}^{(alg)} \Leftrightarrow \forall \xi\in Z^n_{(m)} : \langle f,\xi \rangle = 0$.
\end{enumerate}
\end{lemma}
We postpone the proof of this lemma until after we finish the proof of the theorem. 

Denoting $H_n^{(m)} := Z_n^{(m)} / B_n^{(m)}$ and $\underline{H}^n_{(m)} := Z^n_{(m)} / \overline{B^n_{(m)}}^{\lVert \cdot \rVert_2}$ the lemma then tells us that
\begin{equation}
\dim_{L\Gamma} \underline{H}^n_{(m)} = \dim_{L\Gamma} \mathbf{P}H_n^{(m)} = \dim_{L\Gamma} H_n^{(m)} \nonumber
\end{equation}
where the final equation holds since everything is finitely generated.

To finish the proof we need to show that $H_n(\Gamma,L\Gamma) = \lim_{\rightarrow} H_n^{(m)}$ and $\underline{H}^n(\Gamma, \ell^2\Gamma) = \lim_{\leftarrow} \underline{H}^n_{(m)}$.

The inductive limit is clear: There is are inclusion maps $\varphi_m \colon H_n^{(m)}\rightarrow H_n(\Gamma,L\Gamma)$ whence a map from the inductive limit. This is obviously bijective.

In the projective limit case we get (by restriction) maps $\kappa_m \colon \underline{H}^n(\Gamma,\ell^2\Gamma)\rightarrow \underline{H}^n_{(m)}$ whence a map into the projective limit $\lim_{\leftarrow_m} \underline{H}^n_{(m)}$ and this has rank dense image, e.g.~since we have a commutating square
\begin{displaymath}
\xymatrix{ H^n(\Gamma,\ell^2\Gamma) \ar@{->>}[r] \ar[d] & \lim_{\leftarrow_m} H^n_{(m)} \ar[d]^{\pi} \\ \underline{H}^n(\Gamma,\ell^2\Gamma) \ar[r] & \lim_{\leftarrow_m} \underline{H}^n_{(m)} }
\end{displaymath}
where $\pi$ has rank dense image by an easy $\varepsilon/2^n$ argument.
%%%% WOOOOPS!
%and this is clearly surjective, since for $\xi \in \mathcal{F}(\Gamma^{n},\ell^2\Gamma)$ and $\gamma \in K_{(m)}^{n+1}$, the value of $(d^n\xi)(\gamma)$ depends only on the values of $\kappa_m(\xi)$.

For injectivity we have to show that $\xi \in \overline{B^n(\Gamma,\ell^2\Gamma)}$ if and only if $\xi_m := \kappa_m(\xi) \in \overline{B^n_{(m)}}$ for all $m$. This follows the same observation as in surjectivity: For the left-to-right implication we note that $d^{n-1}\eta_k \rightarrow_k \xi$ in $\mathcal{F}(\Gamma^n,\ell^2\Gamma)$ if and only if for all $m$ we have $d^{n-1}_{(m)}(\kappa_m(\eta_k)) = \kappa_m(d^{n-1}_{(m)}\eta_k)\rightarrow_k \xi_m$.

For the converse implication simply take, for given finite set $K\subseteq \Gamma^n$ and $\varepsilon > 0$, an $m$ such that $K\subseteq K_n^{(m)}$ and an $\eta\in \mathcal{F}(K^{(m)}_{n-1}, \ell^2 \Gamma)$ such that $\lVert \xi_m - d_{(m)}^n\eta \rVert_2 < \varepsilon$ and then again this is the same as $\lVert \xi_m - \kappa_m(d^n\eta_0)\rVert_2 < \varepsilon$ where $\eta_0$ is the extension of $\eta$ by zero.

this finishes the proof of the theorem.
\end{proof}

\begin{proof}[Proof of the lemma]
The equivalences $(i)$ and $(iii)$ are obvious by the remark just before the lemma, i.e. that the duality is compatible with the (co)boundary maps. So are the left-to-right implications of $(ii)$ and $(iv)$.

For the right-to-left implication of $(ii)$ suppose that $\xi\notin \overline{B^n_{(m)}}^{\lVert \cdot \rVert_2}$ and let $P\in M_{\sharp K_n^{(m)}}\otimes L\Gamma$ be the projection onto the orthogonal complement of $\overline{B^n_{(m)}}^{\lVert \cdot \rVert_2}$. Then one of the rows $(p_{i*})$ is non-zero and letting $f(k) = p_{ik}, k\in K_n^{(m)}$ we get a non-zero cycle (since $\langle f,B^n_{(m)} \rangle=0$ by construction of $P$) and we may choose $i$ such that $\langle f,\xi \rangle \neq 0$ since $P\xi \neq 0$. This proves $(ii)$.

For $(iv)$ we do essentially the same thing: Let $f\notin \overline{B_n^{(m)}}^{(alg)}$ and recall that there is a projection $Q\in M_{\sharp K_n^{(m)}}\otimes L\Gamma$ such that $\overline{B_n^{(m)}}^{(alg)} = (\mathbb{C}K_n^{(m)} \otimes LG)Q$. Then again this means $f.(\bbb-Q)\neq 0$ so that we may take a $\xi \in (\bbb-Q)(\mathbb{C}K_n^{(m)} \otimes \ell^2G)$ such that $\langle f,\xi \rangle \neq 0$. For instance we may again take $\xi$ an appropriate non-zero column of $\bbb-Q$.

As for $(ii)$ we see by $(i)$ that $\xi$ is a cocycle since $Q.\xi=0$ implies that $\langle B_n^{(m)},\xi \rangle = 0$.
\end{proof}

Finally we give a proof of the middle equality which is, in our restricted setting, more direct than the general proof given in \cite{Thom06b}.

\begin{proof}[Proof of (*)]
Consider the complex of inhomogeneous chains for computing $\ell^2$-homology:
\begin{equation}
0\leftarrow L\Gamma \leftarrow \mathcal{F}_c(\Gamma, L\Gamma) \leftarrow \cdots
\end{equation}
Since rank-completion is a dimension-exact functor, the $\ell^2$-Betti numbers can equally well be computed as dimensions of the complex
\begin{equation}
0 \leftarrow \mathfrak{c}(L\Gamma) \leftarrow \mathfrak{c}(\mathcal{F}_c(\Gamma, L\Gamma) \leftarrow \cdots
\end{equation}

Then we make the observation that the completions are, as $L\Gamma$-modules,
\begin{equation}
\mathfrak{c}(\mathcal{F}_c(\Gamma^n,L\Gamma)) = \{ f\colon \Gamma^n \rightarrow \mathcal{U}(\Gamma) \mid \forall 0\neq p\in \proj{L\Gamma} \exists F\subseteq \Gamma^n \; \mathrm{finite} \; \forall \gamma\in \Gamma^n\setminus F : (\bbb-p).f(\gamma) = 0 \}. \nonumber
\end{equation}

It is easy to see that the dual of this is exactly $\mathcal{F}(\Gamma^n,\mathcal{U}(\Gamma) )$. Hence, since taking duals is dimension-preserving exact functor on the category of rank-complete modules, we get that the $\ell^2$-Betti numbers are the dimensions of homology spaces of the dual complex
\begin{equation}
0 \rightarrow \mathcal{U}(\Gamma) \rightarrow \mathcal{F}(\Gamma,\mathcal{U}(\Gamma)) \rightarrow \cdots
\end{equation}

By an $\varepsilon / 2^i$ type argument, $\mathcal{F}(\Gamma^n,\mathcal{U}(\Gamma))$ is the rank-completion of $\mathcal{F}(\Gamma^n,\ell^2\Gamma)$, and then appealing to dimension-exactness of rank-completion we are done, recalling from above that the coboundary maps on inhomogeneous cochains are indeed dual to the boundary maps in inhomogeneous chains.
\end{proof}

 %OKOKOKOKOK

% THIS CHAPTER NEEDS AN INTRO
\chapter{{$L^2$}-Betti numbers of locally compact groups} \label{chap:bettidef}
%\epigraph{It's important to have good definitions}{George A. Elliot}

\section{The general definition of $L^2$-Betti numbers} \label{sec:elltwolocallycompactdef} \todo{sec:elltwolocallycompactdef}

Recall from Appendix \ref{chap:cohom} the setup for continuous cohomology as a derived functor from the category $\mathfrak{E}_{G,\mathscr{A}}$ of topological $G$-$\mathscr{A}$-modules to the category of $\mathscr{A}$-modules, where $G$ is a locally compact group and $\mathscr{A}$ a semi-finite tracial von Neumann algebra. In particular, given a choice of Haar measure $\mu$ on a {\lcsu} group $G$ we can define the $n$'th (cohomological) $L^2$-Betti number 
\begin{equation}
\beta^{n}_{(2)}(G,\mu) := \dim_{\psi} H^n(G,L^2G) \nonumber
\end{equation}
where $\psi$ is the canonical weight on $LG$ corresponding to $\mu$. In particular, our definition coincides with the usual one in case $G$ is a countable discrete group and $\mu$ the counting measure.

By uniqueness of the Haar measure up to scaling by strictly positive reals, the sequence $(\beta^n_{(2)}(G,\mu))_n$ of $L^2$-Betti numbers is unique up to a scaling constant not depending on $n$. (See Proposition \ref{prop:bettihaarscaling} below.)

We also define a reduced version of the $L^2$-Betti numbers. For countable discrete groups these coincide with the (non-reduced) $L^2$-Betti numbers defined above (see Section \ref{sec:dualitydiscrete}) but the situation is not so clear in the non-discrete case, so we introduce notation to formally distinguish these.
\begin{equation}
\underline{\beta}^{n}_{(2)}(G,\mu) := \dim_{\psi} \underline{H}^{n}(G,L^2G), \nonumber
\end{equation}

where $\underline{H}^{n}(G,L^2G) := H^{n}(G,L^2G)/\overline{\{ 0 \}}$, the closure taken in the topological space $H^n(G,L^2G)$, which is not necessarily Hausdorff. In other words, $\underline{H}^n(G,L^2G) \simeq \ker d^n / \overline{\operatorname{im} d^{n-1}}$ is the largest Hausdorff space which is a continuous image of $H^n(G,L^2G)$.

%\todo{Homological definition commented out here}
%\begin{comment}
Alternatively we could take L{\"u}ck's definition \cite[Definition 6.50]{Lu02} and extend it by recalling also from Appendix \ref{chap:cohom} that the continuous homology $H_n(G,LG)$ is a left-$LG$-module. We define the $n$'th (homological) $L^2$-Betti number as
\begin{equation}
\beta_{n}^{(2)}(G,\mu) := \dim_{\psi} H_n(G,LG). \nonumber
\end{equation}

As noted in Section \ref{sec:dualitydiscrete} it is shown in \cite{PeTh06} that the cohomological and homological $L^2$-Betti numbers coincide for countable discrete groups. In Chapter \ref{chap:totdisc} we show that all three definitions coincide on the more general class of totally disconnected {\lcsu} groups.

Also note that we have not defined a reduced version of the homological $L^2$-Betti numbers since it is not clear what that would be. It is natural to consider the reduced cohomology and associated version of the $L^2$-Betti numbers in light of the use of projective limits, and also because Hausdorff cohomology is inherently interesting \cite[Section IX.3]{BorelWallachBook} - recall for instance that amenability of $G$ is characterized by $H^1(G,L^2G)$ being non-Hausdorff. On the other hand the situation seems more complicated for inductive topologies where it appears natural to consider closures in weak topologies. In particular, inductive topologies can be harder to work with than projective ones.

%\end{comment}

\begin{remark}
Generally it would be interesting to explore the options offered by the comparison theorems for different cohomology theories in the recent \cite{AuMo11}. In particular, there is an isomorphism $\operatorname{id}\colon H^n(G,L^2G) \xrightarrow{\sim} H^n_{Borel}(G,L^2G)$ by \cite[Theorem A]{AuMo11}. 
\end{remark}

We finish this section with some basic results.

\begin{proposition} \label{prop:bettihaarscaling} \todo{prop:bettihaarscaling}
Let $G$ be a {\lcsu} group and $c\in (0,\infty)$. Let $\mu$ be a Haar measure on $G$ and denote by $\mu_c$ the scaled Haar measure satisfying $\mu_c(A)=c\cdot \mu(A)$ for every measurable subset $A$ of $G$. Then for all $n\geq 0$
\begin{equation}
\beta^n_{(2)}(G,\mu_c) = \frac{1}{c}\beta^n_{(2)}(G,\mu) \nonumber
\end{equation}
and similarly for the reduced and homological $L^2$-Betti numbers. In particular, the vanishing of $L^2$-Betti numbers is independent of the choice of Haar measure
\end{proposition}

\begin{proof}
Denote by $\psi_{c}$ respectively $\psi$ the canonical weights on $LG$ corresponding to $\mu_c$ respectively $\mu$. Then it is easy to check that $\psi_c(x^*x) = \frac{1}{c}\psi(x^*x)$ for all $x\in LG^2_{\psi}$, and the proposition follows immediately from this.
\end{proof}

\begin{proposition}
If $G$ is compact, then $\beta^{n}_{(2)}(G,\mu)=0$ for all $n\geq 1$.
\end{proposition}
\begin{proof}
The cohomology vanishes by \cite[Chapter III, Corollary 2.1]{Guichardetbook}.
\end{proof}

\begin{lemma} \label{lma:Kinvariant} \todo{lma:Kinvariant}
Let $G$ be a {\lcsu} group. For every symmetric subset $K$ of $G$, denote $\langle K\rangle := \cup_{n\in \mathbb{N}} K^{n}$. Then.
\begin{equation}
\dim_{\psi} \{ f\in L^2G \mid (d^0f)\vert_K = 0\} = \frac{1}{\mu( \langle K\rangle )} . \nonumber
\end{equation}
In particular, for every (closed) subgroup $K$ of $G$, the projection $P_K$ in $LG$ onto the right-submodule of $L^2G$ consisting of functions constant on left-$K$-cosets has trace $\psi(P_K)=\frac{1}{\mu(K)}$.
\end{lemma}

\begin{proof}
Indeed fix $K$ denote by $F$ the module in the left-hand side above and note that it is closed in $L^2G$. Let $P$ be the projection onto $F$ in $L^2G$. Now let $P_0$ be a fixed but arbitrary subprojection of $P$ with $\psi(P_0)<\infty$.

Then $P_0$ is given by left convolution by a left-bounded function $f_0\in L^2G$. We have then for all $h\in L^2G$ and all $\gamma \in G, s\in K$
\begin{eqnarray}
\int_{G} f_0(t)h(t^{-1}s^{-1}\gamma)d\mu(t) & = & (f_0*h)(s^{-1}\gamma) \nonumber \\
 & = & (f_0*h)(\gamma) = \int_{G} f_0(t)h(t^{-1}\gamma)d\mu(t). \nonumber
\end{eqnarray}

Substituting $s^{-1}t$ for $t$ on the left-hand side we conclude that in fact for all $s\in K$,
\begin{equation}
\lambda_s f_0 = f_0. \nonumber
\end{equation}
It follows that $f_0$ is constant on $\langle K\rangle$ so that this has finite measure if $f_0$ is non-zero. In particular, if $P$ has non-zero trace, we may always find a non-zero $f_0$, proving the claim in the case $\mu (\langle K\rangle ) = \infty$.

On the other hand suppose $\mu( \langle K\rangle ) < \infty$. Then since $P=P_{\langle K \rangle}$ is the projection of $L^2G$ onto the submodule of functions constant on cosets of $\langle K\rangle$, given by convolution by $\frac{1}{\mu( \langle K\rangle)} \bbb_{\langle K\rangle}$ it has trace
\begin{equation}
\psi \left( \lambda \left( \frac{1}{\mu( \langle K\rangle)} \bbb_{\langle K\rangle} \right) \right) = \frac{1}{\mu(\langle K\rangle)}. \nonumber
\end{equation}

\end{proof}

\begin{proposition} \label{prop:zerothelltwo} \todo{prop:zerothelltwo}
Let $G$ be a {\lcsu} group. If $G$ is not compact then (for any choice of Haar measure)
\begin{equation}
\beta^{0}_{(2)}(G,\mu) = 0. \nonumber
\end{equation}
On the other hand, if $G$ is compact and we normalize the Haar measure such that $\mu(G)=1$, then $\beta^0_{(2)}(G,\mu)=1$.
\end{proposition}

\begin{proof}
This follows directly from Lemma \ref{lma:Kinvariant}.
\end{proof}

%%%%%%%%%%%%%%%%%%%%%%%%%%%%%%%%%%%%%%%%%%%%%%%%%%%%%%%%%%%
\begin{comment}
\begin{lemma} \label{lma:reducedelltwolimit} \todo{lma:reducedelltwolimit}
Let $G$ be a {\lcsu} group and $n\in \mathbb{N}$. Let $(K_i)_{i\in \mathbb{N}}$ be an increasing sequence of compact subsets of $G^n$, cofinal in the net of compact subsets.

Denoting by $Z_i$ respectively $B_i$ the closures in $L^2(K_i,L^2G)$ of the images of $Z^n(G,L^2G)$ respectively $B^n(G,L^2G)$ under restriction to $K_i$, we have

\begin{equation}
\underline{\beta}_{(2)}^{n}(G,\mu) = \lim_{i\rightarrow \infty} \dim_{\psi} Z_i\ominus B_i, \nonumber
\end{equation}
the limit of an increasing sequence.
\end{lemma}

\begin{proof}
Indeed letting $\phi_i:L^2_{loc}(G^n,L^2G)\rightarrow L^2(K_i,L^2G)$ and $\phi_{ij}: L^2(K_j,L^2G)\rightarrow L^2(K_i,L^2G)$ be restriction maps, this induces a projective system $(Z_i / B_i, \phi_{ij})$ and
\begin{equation}
\underline{H}^n(G,L^2G) \cong \lim_{\leftarrow} Z_i/B_i. \nonumber
\end{equation}
Further the maps induced on this by the $\phi_i$ as well as the $\phi_{ij}$ are all surjective so that there are two possible situations.

First, if $\dim_{\psi} Z_i\ominus B_i=\infty$ for some $i$ then this holds for all $j > i$ as well as for $\dim_{\psi} \underline{H}^n(G,L^2G)$, by surjectivity. Hence the statement is true in this case.

The other possibility then is that $\dim_{\psi} Z_i\ominus B_i < \infty$ for all $i\in \mathbb{N}$. In this case the claim follows directly from Theorem \ref{thm:dimfinality}.
\end{proof}
\end{comment}
%%%%%%%%%%%%%%%%%%%%%%%%%%%%%%%%%%%%%%%%%%%%%%%%%%

The following lemma will be used in the next chapter to allow a change of dimension from a group to a subgroup. It shows, roughly speaking, that the reduced $L^2$-cohomology of a locally compact group is, up to dimension, a projective limit of Hilbert spaces in a sufficiently nice way.

\begin{lemma} \label{lma:reducedelltwolimit} \todo{lma:reducedelltwolimit}
Let $G$ be a locally compact, $2$nd countable group and $n\in \mathbb{N}$. Let $(K_i)_{i\in \mathbb{N}}$ be an increasing sequence of compact subsets of $G^n$, cofinal in the net of compact subsets. Let $H$ be a closed, unimodular subgroup of $G$ and $\psi$ the tracial weight on $LH$ induced by some choice of Haar measure.

Denoting by $Z_i$ respectively $B_i$ the closures in $L^2(K_i,L^2G)$ of the images of $Z^n(G,L^2G)$ respectively $B^n(G,L^2G)$ under restriction to $K_i$, we have

\begin{equation}
\dim_{LH} \underline{H}^n(G,L^2G) = \lim_{i\rightarrow \infty} \dim_{(LH,\psi)} Z_i\ominus B_i, \nonumber
\end{equation}
an increasing limit of dimensions of closed, invariant subspaces of $\mathcal{H}\bar{\otimes} L^2G \simeq \mathcal{K}\bar{\otimes} L^2H$ with $\mathcal{H,K}$ Hilbert spaces.

In particular, if $G$ is also unimodular and $\dim_{(LH,\psi)} E = \dim_{(LG,\tilde{\psi})} E$ for every closed $LG$-submodule $E$ of $\mathcal{H}\bar{\otimes} L^2G$ then
\begin{equation}
\underline{\beta}^n_{(2)}(G,\tilde{\mu}) = \dim_{(LH,\psi)} \underline{H}^n(G,L^2G). \nonumber
\end{equation}
\end{lemma}

\begin{proof}
Letting $\phi_i:L^2_{loc}(G^n,L^2G)\rightarrow L^2(K_i,L^2G)$ and $\phi_{ij}: L^2(K_j,L^2G)\rightarrow L^2(K_i,L^2G)$ be restriction maps, this induces a projective system $(Z_i / B_i, \phi_{ij})$ and there is an injective map of $LH$-modules
\begin{equation}
\iota \colon \underline{H}^n(G,L^2G) \rightarrow \lim_{\leftarrow} Z_i/B_i. \nonumber
\end{equation}

Denote $H_i:=Z^n(G,L^2G)/\phi_i^{-1}(B_i)$. The $\phi_i$ induce injective morphisms of $LH$-modules $\bar{\phi}_i\colon H_i \rightarrow Z_i\ominus B_i$ with dense image, whence by Lemma \ref{lma:dimbyTr} we have
\begin{equation}
\dim_{(LH,\psi)} H_i = \dim_{(LH,\psi)} Z_i\ominus B_i. \nonumber
\end{equation}

Further, we have $\cap_i \phi_i^{-1}(B_i) = \overline{B^n(G,L^2G)}$. Hence the claim follows by Theorem \ref{thm:dimfinality} in case $\dim_{(LH,\psi)} \underline{H}^n(G,L^2G) < \infty$. (The projective limits theorem is applied here both to the projective system $(Z_i/B_i)_i$ and the images of the decreasing $LH$-modules $\phi_i^{-1}(B_i)$ in the reduced cohomology.)

On the other hand, the claim is also clear if $\dim_{(LH,\psi)} \underline{H}^n(G,L^2G) = \infty$ by injectivity of $\iota$.
\end{proof}

\begin{proposition} \label{prop:reducedelltwolimit} \todo{prop:reducedelltwolimit}
Let $G$ be a {\lcsu} group. Then for all $n\geq 0$
\begin{equation}
\underline{H}^n(G,L^2G) = 0 \quad \textrm{ if and only if } \quad \underline{\beta}^n_{(2)}(G,\mu)=0. \nonumber
\end{equation}
In particular the vanishing of $\beta^n_{(2)}(G,\mu)$ implies the vanishing of reduced $L^2$-cohomology in degree $n$.
\end{proposition}

There is a caveat here that vanishing of $\beta^n_{(2)}(G,\mu)$ does \emph{not} imply that the cohomology $H^n(G,L^2G)$ vanishes. An example is provided e.g.~by any amenable, countably infinite discrete group and $n=1$.

\begin{proof}
This is a direct consequence of the previous lemma and the faithfulness of the dimension function on standard Hilbert space.
\end{proof}

\section{Abelian groups}

In this paragraph we exploit the following observation to show that the $L^2$-Betti numbers all vanish for abelian (non-compact) groups. The point is that if $Q\in LG$ is a central projection then as right-$LG$-modules we have an isomorphism $H^n(G,Q(L^2G)) \simeq H^n(G,L^2G).Q$, which is just given by the inclusion map of $L^2_{loc}(G^n,Q(L^2G))$ in $L^2_{loc}(G^n,L^2G)$.

\begin{proposition} \label{prop:dimcentralprojections} \todo{prop:dimcentralprojections}
Let $(\mathscr{A},\psi)$ be a semi-finite, $\sigma$-finite, tracial von Neumann algebra and suppose that $(Q_k)$ is an increasing sequence of central projections in $\mathscr{A}$ with limit the identity. Then for any right-$\mathscr{A}$-module $M$
\begin{equation}
\dim_{(\mathscr{A},\psi)} M = \lim_k \dim_{(\mathscr{A},\psi)} M.Q_k = \lim_k \dim_{(Q_k\mathscr{A},\psi(Q_k \cdot))} M.Q_k. \nonumber
\end{equation}
Further, the limit is increasing.
\end{proposition}

\begin{proof}
The second equality and the inequality '$\geq$' in the first are clear. To prove that the left-hand term is at most equal the middle term let $P\leq M$ be a $\psi$-finite projective submodule, $P\simeq p\mathscr{A}^n$ for some $n$ and $p\in M_n(\mathscr{A})$.

Then also $P.Q_k$ is a $\psi$-finite projective submodule of $M.Q_k$ and $P.Q_k \simeq (\bbb_n \otimes Q_k)p\mathscr{A}^n$. This gives
\begin{equation}
\dim_{\psi} P.Q_k = (\mathrm{Tr}_n\otimes \psi)( (\bbb_n\otimes Q_k)p) \nearrow (\mathrm{Tr}_n \otimes \psi)(p). \nonumber
\end{equation}
The proposition follows since $P$ was arbitrary.
\end{proof}

\begin{theorem} \label{thm:ltwoabelianlcgroups} \todo{thm:ltwoabelianlcgroups}
Let $G$ be a {\lcsu}, non-compact abelian group. Then for all $n\geq 0$,
\begin{equation}
\beta_{(2)}^n(G,\mu) = 0. \nonumber
\end{equation}
\end{theorem}

Before the proof we recall some facts about the Fourier transform.

Let $\hat{G}$ be the unitary dual of the abelian locally compact group $G$, and denote by $\mathcal{F}:L^2G\rightarrow L^2\hat{G}$ the unitary extension of the Fourier transform. Recall that this is an isomorphism, and that it sets up a spatial isomorphism between the action of $L^1G$ on $L^2G$ by convolution and that of (a weak-operator dense subalgebra of) $L^{\infty}\hat{G}$ acting on $L^2\hat{G}$, extending to a spatial isomorphism of $LG$ and $L^{\infty}\hat{G}$.

By the characterization of the canonical weight $\psi$ on $LG$ (that $\psi(x^{*}x) < \infty \Leftrightarrow x=\lambda(f), f\in L^2G$ and in that case $\psi(x^{*}x)=\lVert f\rVert_2^2$) we see that $\psi$ corresponds just to integration against $\hat{\mu}$, the Haar measure on $\hat{G}$. (This is normalized such that $\mathcal{F}$ is an isometry.)

\begin{proof}
We show that in fact there is a sequence $(Q_k)$ of central projections in $LG$ increasing to the identity and such that for all $k$ the cohomology $H^n(G,Q_k(L^2G))=0$. By the previous proposition this implies the theorem.

By \cite[Chapter III, Proposition 3.1(i)]{Guichardetbook} it is enough to find the $Q_k$ such that for every $k$ there is a $g\in G$ such that $(\bbb - \lambda(g))\vert_{Q_k(L^2G)}$ is invertible.

To this end let $\{ g_i\}_{i\in \mathbb{N}}$ be a countable dense subset of $G$ and define for $i,j\in \mathbb{N}$ subsets $B_{i,j}$ of $\hat{G}$ by
\begin{equation}
B_{i,j} = \{ \chi \in \hat{G} \mid \lvert \chi (g_i) - 1\rvert \geq \frac{1}{j} \}. \nonumber
\end{equation}
Then we get $(\cup_{i,j} B_{i,j})^{\complement} = \{ 1\} \subseteq \hat{G}$. Since $G$ is non-compact the dual is non-discrete so that this has measure zero. It follows that if we denote by $C_{i,j,l}\subseteq B_{i,j}$ pairwise disjoint sets with finite measure and $B_{i,j}=\cup_l C_{i,j,l}$, then $\bbb - \lambda(g_i)$ is invertible when restricted to $V_{i,j,l} = \mathcal{F}^{-1}(\bbb_{C_{i,j,l}}.L^2\hat{G})$. Hence if we let the $Q_k$ be (exhausting) finite sums of the projections in $LG$ (corresponding to) multiplication by $\bbb_{C_{i,j,l}}$ the claim follows since for any such finite sum we get
\begin{equation}
H^n(G,(\oplus_{fin} \bbb_{C_{i,j,l}})(L^2G)) = \oplus_{fin} H^n(G,\bbb_{C_{i,j,l}}(L^2G)) = 0. \nonumber
\end{equation}
\end{proof}

Actually the proof clearly yields the following stronger statement:

\begin{porism} \label{por:ltwononcptcentre} \todo{por:ltwononcptcentre}
Let $G$ be a {\lcsu} group with non-compact centre. Then $\beta^n_{(2)}(G,\mu) = 0$ for all $n\geq 0$. {\phantom{aa} \hfill \qedsymbol}
\end{porism}

\todo{remark about diffuse center commented out. It may not be entirely justified. (It's not obvious there would be a group element/measure in the center?)}
\begin{comment}
\begin{remark} \label{rmk:diffcenter} \todo{rmk:diffcenter}
We note briefly that the proof is of course more general, applying to all $G$ such that the center of $LG$ has no type $\mathrm{I}$ summand.
\end{remark}
\end{comment}

 % OK

\section{Compact normal subgroups}

\begin{theorem} \label{thm:compactnormal} \todo{thm:compactnormal}
Let $G$ be a {\lcsu} group and $K$ a compact, normal subgroup. Denote by $\nu = \mu_*$ the push-forward of the Haar measure $\mu$ on $G$ to $H:=G/K$.

Then for all $n\geq 0$
\begin{equation}
\beta^{n}_{(2)}(G,\mu) = \beta^n_{(2)}(H,\nu). \nonumber
\end{equation}
\end{theorem}

This theorem is proved in two parts: first we identify the relevant cohomology spaces, then we justify the change of dimension funtion.

The first part is in fact clear: by \cite[Chapter III, Corollary 2.2]{Guichardetbook}, see also Proposition \ref{prop:totdiscbar}, we get immediately an isomorphism of $LG$-modules

\todo{eq:compactnormal}
\begin{equation} \label{eq:compactnormal}
H^n(G,L^2G) \simeq H^n(H,L^2H).
\end{equation}
Alternatively this follows from an application of the Hochschild-Serre spectral sequence in continuous cohomology. See e.g.~\cite[Chapter III, Section 5]{Guichardetbook}.

Here the $LG$-action on $L^2H$ is the usual one via.~the identification $L^2H = L^2(K\backslash G)$.

To see that the dimension functions fit, we note the following (surely well-known)

\begin{lemma}
The orthogonal projection $P_H \colon L^2G \rightarrow L^2H\otimes \bbb_K$ is central in $LG$, and $P_HLG \simeq LH$.

Further, the inclusion of $LH \subseteq LG$ in this way, is given on $LH^2_{\psi_{\nu}}$ by extension of functions $\xi \mapsto \bar{\xi}$ with $\bar{\xi}(hk) = \xi(h)$. In particular, it is trace-preserving. {\phantom{aa} \hfill \qedsymbol}
\end{lemma}

\begin{proof}[Proof of the Theorem]
By equation \eqref{eq:compactnormal}, all we need to show is that
\begin{equation}
\dim_{(LH,\psi_{\nu})} H^n(H,L^2H) = \dim_{(LG,\psi)} H^n(H,L^2H). \nonumber
\end{equation}
This follows directly from the lemma: Any f.g.~projective $LH$-module is also a f.g.~projective $LG$-module, with identical trace. On the other hand, the action on the right-hand side is nothing but $(\xi.T)(h) = \xi(h).(P_HT)$ for $h\in H^n$ and, say, $\xi$ an inhomogeneous cocycle.

Then clearly any f.g.~projective $LG$-submodule of the cohomology, is identically a f.g.~projective $LH$-submodule.
\end{proof}

        % OK

\section{The example $SL_2(\mathbb{R})$} \label{sec:sltwoRexample} \todo{sec:sltwoRexample}

We now compute $\beta_{(2)}^{1}(SL_2(\mathbb{R}),\mu)$ by exploiting knowledge of the representation theory to give a precise description of $H^1(SL_2(\mathbb{R}),L^2(SL_2(\mathbb{R})))$. The end result, Theorem \ref{thm:sltwocalc}, follows in fact already from Theorem \ref{cor:finitecovolcocompact} and the well-known fact that $\mathbb{F}_2$ is a lattice in $SL_2(\mathbb{R})$, and that the latter has cocompact lattices, e.g. by the uniformization theorem, or more generally by \cite[Theorem C]{Bor63}. Alternatively, and we emphasize this point, one can view Theorem \ref{thm:sltwocalc}, through Corollary \ref{cor:finitecovolcocompact}, as a new proof that $\beta^1_{(2)}(SL_2(\mathbb{Z})) = \frac{1}{12}$ whence $\beta^1_{(2)}(\mathbb{F}_2)=1$, by entirely different means than in \cite{ChGr86}.

The latter theorem, \cite[Theorem C]{Bor63} coupled with Theorem \ref{cor:finitecovolcocompact} and the methods in \cite{Borel85sym} gives a more general version of equation \eqref{eq:sltwococycle} below than presented here, but the concrete example $SL_2(\mathbb{R})$ already illustrates the approach. Note that while the emphasis in \cite{Borel85sym} is on the $L^2$-cohomology of symmetric spaces, we focus on group cohomology and the computation of von Neumann dimension, which is not described in \cite{Borel85sym}. We also use the trace on the von Neumann algebra of the ambient locally compact group instead of passing to the von Neumann algebra of a lattice, using e.g. \cite[Equation 3.3]{AtSchmid77} (see also \cite[Theorem 3.3.2]{GdlHJ}), further streamlining the computation.

%As mentioned in the introduction this approach appears already in \cite{Borel85sym}. We go through the case of $SL_2(\mathbb{R})$ in detail below in part for the convenience of the uninitiated, as everything can be done, or at least stated, in a very hands on manner when one restricts ones attention just to $SL_2(\mathbb{R})$, and in part because the machinery of the semi-finite dimension function allows us to very easily read off exactly the dimension in the end.

In this example we fix
\begin{equation}
G:=SL(2,\mathbb{R}) = \left\{ \left( \begin{array}{cc} x & y \\ u & v \end{array} \right) \in M_2(\mathbb{R}) \mid xv-uy = 1 \right\}. \nonumber
\end{equation}

For convenience we recall some basic facts about $G$. Recall that the Iwasawa decomposition $G=KP^+=KAN$ is a bijection $G=K\times P^+$ as sets, where
\begin{eqnarray}
K=\left\{ \left( \begin{array}{cc} \cos \theta & \sin \theta \\ -\sin \theta & \cos \theta \end{array} \right) \mid \theta \in \mathbb{R} \right\} \simeq \mathbb{T}, \nonumber \\
%P^+ = \left\{ \left( \begin{array}{cc} a & b \\ 0 & a^{-1} \end{array} \right) \mid \mathbb{R} \owns a > 0, b\in \mathbb{R} \right\}, \nonumber \\
 A = \left\{ \left( \begin{array}{cc} e^t & 0 \\ 0 & e^{-t} \end{array} \right) \mid t\in \mathbb{R} \right\}, \; N = \left\{ \left( \begin{array}{cc} 1 & s \\ 0 & 1 \end{array} \right) \mid s\in \mathbb{R} \right\}, \nonumber
\end{eqnarray}
$A^+$ is the subset of $A$ for which $t>0$, and $P^+=AN$. We denote by $u_{\theta}$ a general element of $K$ and by $p$ a general element of $P^+$. Note that $K$ and $P^+$ are subgroups of $G$, with $K$ maximal compact. We denote general elements of $A$ by $a_t$ and of $N$ by $n_s$.  Then the Haar measure $\mu$ on $G$ is $\mathrm{d}\mu = \frac{1}{2} e^{2t} \mathrm{d}\theta\mathrm{d}t\mathrm{d}s$. Alternatively, one has the polar decomposition $G=K\dot{\cup}KA^+K = K\overline{A^+}K$, and in this setting the Haar measure has the form $\mathrm{d}\mu = \frac{1}{2}\sinh(2t)\mathrm{d}\theta_1\mathrm{d}\theta_2\mathrm{d}t$.

We write $\mathfrak{g}$ respectively $\mathfrak{k}$ for the Lie algebras of $G$ respectively $K$.

Now by \cite[p. 124]{Guichardetbook} there are exactly two simple $(\mathfrak{g},\mathfrak{k})$-modules with non-vanishing first cohomology, denoted there $E_1^{\pm}$. These are invariant submodules of $\mathcal{H}_{0,1}$ - the Hilbert space of maps $f:G\rightarrow \mathbb{C}$ satisfying (here $(B.X)$'s refer to p. 278 in \cite{Guichardetbook})
\begin{eqnarray}
(B.1) &  & f(g.p) = a^{-2} f(g), \; f(g.(-\bbb_2)) = f(g), \quad g\in G, p\in P^+ \nonumber \\
(B.2) &  & \frac{1}{2} \int_K \lvert f(k)\rvert^2 \mathrm{d}k < \infty. \nonumber
\end{eqnarray}
Clearly $\mathcal{H}_{0,1}$ is isomorphic to $L^2(-\frac{\pi}{2},\frac{\pi}{2})$ since $f\in \mathcal{H}_{0,1}$ is given by its values on $K$, but in order to write the action of $G$ in the most convenient form we consider a different realization. Following \cite[Chapter VII]{GelfandGenFunfive} we denote by $F^{+}_{-1}$ the set of functions on the closed unit disc in $\mathbb{C}$, analytic in the interior and infinitely differentiable on the boundary. Then $G$ acts on this by

\begin{equation}
(g.\xi)(w) = \xi\left( \frac{aw+b}{\bar{b}w+\bar{a}} \right) (\bar{b}w+\bar{a})^{-2} \nonumber
\end{equation}
where
\begin{equation}
g=\left( \begin{array}{cc} \alpha & \beta \\ \gamma & \delta \end{array} \right), \quad \begin{array}{l} a = \frac{1}{2}( (\alpha+\delta)+i(\gamma-\beta)) \\ b=\frac{1}{2}( (\alpha - \delta) - i(\gamma+\delta)) \end{array} . \nonumber
\end{equation}

By \cite[Chapter VII, Section 5.4]{GelfandGenFunfive} there is a $G$-invariant inner product on this given by (here $\mathrm{d}w = \mathrm{d}x+i\mathrm{d}y$)
\begin{eqnarray}
(\xi,\eta) & = & \frac{i}{2}\int_{\lvert w\rvert <1} \xi(w)\overline{\eta(w)} \mathrm{d}w\mathrm{d}\bar{w} \nonumber \\
 & = & \int_{x^2+y^2 < 1} \xi(x+iy)\overline{\eta(x+iy)} \mathrm{d}x\mathrm{d}y. \nonumber
\end{eqnarray}

Denote by $\mathcal{H}^{+}$ the Hilbert space completion. This is a unitary representation of $G$ with an orthogonal basis $\{ w^{k}\}_{k \in \{0\}\cup \mathbb{N}}$ consisting of monomials. For each $u_{\theta} \in K$ the eigenvectors of this are exactly the $w^{k}, k=0,1,\dots$ with corresponding eigenvalues $e^{2(k+1)i\theta}$. It follows in particular that the space of $K$-finite vectors is the linear span $E_1^{+}=\span \{ w^k\}$.

Next we want to determine explicitly representatives of the cocycles in $H^1(G,E_1^{+})$. This is done by applying  \cite[Chapter II, Proposition 5.1]{Guichardetbook}. Recall the Cartan decomposition $\mathfrak{g}=\mathfrak{k}\oplus \mathfrak{p}$ where $\mathfrak{k}=\mathbb{R}.X_0$ and $\mathfrak{p}=\span \{ X_1,X_2\}$ for
\begin{equation}
 X_0= \left( \begin{array}{cc} 0 & 1 \\ -1 & 0 \end{array} \right), X_1=\left( \begin{array}{cc} 1 & 0 \\ 0 & -1 \end{array} \right), X_2= \left( \begin{array}{cc} 0 & 1 \\ 1 & 0\end{array} \right). \nonumber
\end{equation}
For the complexifications $\mathfrak{p}_{\mathbb{C}}$ and $\mathfrak{g}_{\mathbb{C}}$ we have $\mathfrak{p}_{\mathbb{C}} = \span \{ X_{\pm} \}$ with $X_{\pm}= X_1\pm iX_2$, and the brackets in $\mathfrak{g}_{\mathbb{C}}$ are given by
\begin{equation}
[X_0,X_{\pm}] = \pm 2iX_{\pm} , \quad [X_+,X_-] = -4iX_0. \nonumber
\end{equation}
Denoting by $\pi$ the action of $G$ on $\mathcal{H}^{+}$ we have for the corresponding representation $\mathrm{d}\pi$ of the complex Lie algebra $\mathfrak{G}_{\mathbb{C}}$ \todo{eq:sltworalgebraactions}
\begin{eqnarray} \label{eq:sltworalgebraactions}
%\begin{split}
\mathrm{d}\pi(X_0)w^k  = & \, i2kw^k, \\
\mathrm{d}\pi(X_{\pm})w^k  = & \, \left\{ \begin{array}{cl} 0 & \textrm{if } X_{\pm} = X_- \textrm{ and } \; k=0 \\ (2\pm 2k)w^{k\pm 1} & \mathrm{otherwise} \end{array} \right. .
%\end{split}
\end{eqnarray}
Now by \cite[Chapter II, Proposition 5.1]{Guichardetbook} we have $H^1(\mathfrak{g},\mathfrak{k},E_1^{+}) = \mathrm{Hom}_{\mathfrak{k}}(\mathfrak{p}_{\mathbb{C}}, E_1^{+})$. By the above, it follows directly that
\begin{equation} \label{eq:sltwococycle}
H^1(\mathfrak{g},\mathfrak{k},E_1^+) = \mathbb{C}.\phi_+, \quad \textrm{where }\; \phi_+: \begin{array}{c} X_+\mapsto w^0=1 \\ X_-\mapsto 0 \end{array} .
\end{equation}

Recall that the way one realizes a simple, admissible discrete series module $\mathcal{H}$ inside $L^2G$ is through matrix coefficients $g\mapsto ( g.\xi,\eta)_{\mathcal{H}}, \; \xi,\eta\in \mathcal{H}$. Here we consider the matrix coefficients for $\mathcal{H}^+$, $\xi_{m,n}:g\mapsto ( \pi(g) w^m, w^n )$. Then for all $n\in \mathbb{N}_0$, $F^n:=\span \{ \xi_{m,n} \mid m\in \mathbb{N}_0 \} \subseteq L^2G$ with the right-regular representation is isomorphic to the module $E_1^{+}$. So is, for all $m\in \mathbb{N}_0$, $F_m:=\span \{ \xi_{m,n} \mid n\in \mathbb{N}_0\}$ with the left-regular representation.

To see that indeed $\mathcal{H}^+$ is square integrable, i.e. in the discrete series, and find out exactly how it embeds in $L^2 G$ we compute the \lq\lq top left\rq\rq\phantom{ }matrix coefficient
\begin{eqnarray}
\xi_{0,0}(u_{\theta_1}a_tu_{\theta_2}) & = & (\pi ( u_{\theta_1}a_tu_{\theta_2})w^0,w^0)_{\mathcal{H}^{+}} \nonumber \\
 & = & ( \pi(a_tu_{\theta_2}).w^0,\pi(u_{-\theta_1})w^0)_{\mathcal{H}^{+}} \nonumber \\
 & = & e^{2i(\theta_1+\theta_2)} ( \pi(a_t)w^0,w^0 )_{\mathcal{H}^{+}} \nonumber \\
 & = & e^{2i(\theta_1+\theta_2)} \cdot \frac{i}{2}\int_{\lvert w \rvert <1} (w \sinh t + \cosh t)^{-2} \mathrm{d}w\mathrm{d}\bar{w} \nonumber \\
 & = & e^{2i(\theta_1+\theta_2)} (\cosh t)^{-2} \cdot \frac{i}{2} \int_{\lvert w \rvert <1} (1-(-w\tanh t))^{-2} \mathrm{d}w\mathrm{d}\bar{w} \nonumber \\
 & = & \pi e^{2i(\theta_1+\theta_2)} (\cosh t)^{-2}, \nonumber
\end{eqnarray}
where the final equality follows e.g. by a power series expansion of the integrand. Then we get
\begin{equation}
\lVert \xi_{0,0} \rVert_2^{2} = 2\pi^4. \nonumber
\end{equation}

Denote by $A_1:=\overline{\span}\{ \xi_{m,n} \mid m,n\in \mathbb{N}_0 \}$. Then for the left-regular representation, the space of $K$-finite vectors in $A_1$ is exactly the (algebraic) direct sum $A_1^{\circ}:=\oplus^{alg}_{n\in \mathbb{N}_0} \overline{F^n}$. Clearly these are all smooth, so $A_1^{\circ} = (A_1^{\infty})_{(K)}$. Further, this is invariant under the right-action of $LG$.

Similarly we can denote for $E_1^-$ an $A_{-1}$ etc.

We are now in position to make the final steps of the computation. Denoting by $\hat{G}_d$ the set of discrete series representations we have for $(L^2G)_d$, the discrete part of $L^2G$,
\begin{equation}
(L^2G)_d = \oplus_{\omega\in \hat{G}_d} A(\omega) \nonumber
\end{equation}
where $A(\cdot)$ denotes the bi-module of matrix coefficients. In particular $A_1=A(\mathcal{H}^+)$. Here the complement of $(L^2G)_d$ is a direct integral wrt. a diffuse measure whence, since only finitely many admissible modules have non-vanishing cohomology (see \cite[Chapter II, Corollary 4.2]{Guichardetbook}), this does not contribute, cf. \cite[Chapter III, Proposition 2.6]{Guichardetbook}. By the same theorem $\underline{H}^1(G,(A_{\pm 1})^{\perp}\cap (L^2G)_d) = 0$.

By \cite[Chapter III, Proposition 1.6]{Guichardetbook} and van Est's theorem \cite[Chapter III, Corollary 7.2]{Guichardetbook} (See also \cite[Remark 3.5, Chapter II]{Guichardetbook}) we have $H^1( G, A_{\pm 1}) \simeq H^1(\mathfrak{g}, K, A_{\pm 1}^{\circ}) \simeq H^1(\mathfrak{g},\mathfrak{k},A^{\circ}_{\pm 1})$, and checking Guichardet's explicit formula for the van Est isomorphism (p.227 in \cite{Guichardetbook}) this is an isomorphism of right-$LG$-modules. See also Theorem \ref{thm:vanEst}.

By the explicit description in equation \eqref{eq:sltwococycle} it follows that, as right-$LG$-modules, $H^1(\mathfrak{g},\mathfrak{k},A_1^{\circ}) \simeq \overline{F^0}$. 

Thus we need to compute $\psi(P)$ for the orthogonal projection $P$ onto this subspace. We claim that $P$ is given by left-convolution by $\xi=\frac{1}{2\pi^3}\xi_{0,0}$. Indeed by the computation of $\lVert \xi_{0,0} \rVert_2$ above the formal degree of $\mathcal{H}^{+}$ is $d(\mathcal{H}^{+})=\frac{1}{2\pi^2}$, so we get

\begin{eqnarray}
(\tilde{\xi}*\xi)(\gamma) & = & \frac{1}{4\pi^6}\int_G   \overline{(g^{-1}.w^{0},w^{0})_{\mathcal{H}^+}} (g^{-1}\gamma . w^{0},w^{0})_{\mathcal{H}^{+}} \mathrm{d} \mu(g) \nonumber \\
 & = & \frac{1}{2\pi^4} (\gamma .w^0,w^0)_{\mathcal{H}^+} \overline{(w^0,w^0)_{\mathcal{H}^+}} \nonumber \\
 & = & \xi (\gamma). \nonumber
\end{eqnarray}

It follows that $\tilde{\xi}$ is left-bounded since it acts as an isometry on the range of the (in principle possibly unbounded, affiliated) operator of left-convolution by $\xi$. Then it follows from this that $\xi$ is also left-bounded since the group is unimodular. Further the calculation shows that $\lambda(\xi)$ is an orthogonal projection and clearly this is $P$ as claimed.

Thus by the same calculation as above with $\gamma= \bbb$ we get $\psi(P) = \lVert \xi \rVert_2^2 = \frac{1}{2\pi^2}$.

One gets an entirely analogous calculation for $E_1^-$, and adding the two we have shown:
\begin{theorem} \label{thm:sltwocalc} \todo{thm:sltwocalc}
With the Haar measure $\mu$ on $SL_2(\mathbb{R})$ induced by the Lebesgue measure on $\mathbb{T}$ with total mass $\pi$ (i.e. $\mathrm{d}\mu = \frac{1}{2}\sinh (2t)\mathrm{d}\theta_1\mathrm{d}\theta_2\mathrm{d}t$ as above),
\begin{equation}
\beta_{(2)}^1(SL_2(\mathbb{R}),\mu) = \underline{\beta}_{(2)}^{1}(SL_2(\mathbb{R}),\mu) = \frac{1}{\pi^2}. \nonumber
\end{equation}
\end{theorem}

\begin{proof}
The first equality holds since $G$ is non-amenable, so that $H^1(G,L^2G)$ is Hausdorff cf. \cite[Chapter III, Corollary 2.4]{Guichardetbook}.

The second follows by the discussion above.
\end{proof}
         % OK

\chapter{Cocompact lattices} \label{chap:cocompact}
%\epigraph{If there is a problem you can't solve, then there is an easier problem you can solve: find it.}{George P{\'o}lya}

In this chapter we prove (Theorem \ref{cor:finitecovolcocompact}) that whenever $H$ is a cocompact lattice in the {\lcsu} group $G$ then the $L^2$-Betti numbers of $H$ and $G$ coincide up to scaling by the covolume of $H$. In fact it then follows that this is the case for any lattice in $G$, by a result of D. Gaboriau.

In the first section we prove a more general result about changing the dimension function from $\dim_{(LG,\psi)}$ to $\dim_{(LH,\tau)}$ when $H$ is a lattice in $G$, without the assumption of cocompactness. This extends the well-known (easy) fact that for a finite index inclusion of discrete groups, the dimension functions satisfy
\begin{equation}
\dim_{(LG,\psi)} E = [G:H]\cdot \dim_{(LH,\tau)} E \nonumber
\end{equation}
for any $LG$-module $E$, when the traces are normalized such that $\psi(\bbb) = \tau(\bbb) = 1$. The question is slightly more subtle when $G$ is a general {\lcsu} group, and one cannot expect equality to hold generally (see Example \ref{ex:dimrestrictionineqsharp}).

\section{The dimension function for lattices in a locally compact group} \label{sec:dimlattice} \todo{sec:dimlattice}

In this section we compare the dimension function $\dim_{(LG,\psi)}$, where $\psi$ is the tracial weight on $LG$ for a {\lcsu} group $G$, to $\dim_{(LH,\tau)}$ where $\tau$ is the canonical trace on $LH$ for a lattice $H$ of $G$ with finite covolume. For simplicity we fix the choice of Haar measure $\mu$ on $G$ such that the covolume is $1$.

Since $LG$ acting on $L^2G$ from the left commutes with the right-action of $LH$ on $L^2G=L^2(G/H)\bar{\otimes}L^2H$ we have an inclusion $LG \subseteq \mathcal{B}(L^2(G/H))\overline{\otimes} LH$, so that there is another natural candidate for a tracial weight on $LG$, namely $Tr\otimes \tau$. In the sequel we show that these are in fact the same weight on $LG$. This is surely known, but I was not able to find a reference. Recall the construction of $\psi$ from the preliminaries.

%Recall \cite[Theorem 7.2.7]{storegert} that $\psi$ is constructed by taking a sequence $(\Psi_n)\subseteq L^1G$ such that
%\begin{equation}
%\rho(\Psi_n) \nearrow_n \bbb \quad \textrm{and} \; \; \lambda(\Psi_n) \xrightarrow{\textrm{SOT}}_n \bbb, \nonumber
%\end{equation}
%and then finding $\xi_n\in L^2G$ such that $\xi_n * \tilde{\xi_n} = \Psi_n-\Psi_{n-1}$, where $\tilde{\xi_n}(t) := \overline{\xi_n(t^{-1})}$. Then $\psi = \sum_{n=1}^{\infty} \langle \cdot \; \xi_n, \xi_n \rangle$, the sum of vector states, and $\psi$ is charaterized by $\psi(x^{*}x)$ being finite if and only if there is some left-bounded $f\in L^2G$ such that $x = \lambda(f)$, in which case $\psi(x^{*}x) = \lVert f\rVert_2^{2}$.

The idea here is to decompose $\psi$ as a sum of vector states, each implementing a copy of $\tau$, the trace on $LH$, on pairwise orthogonal right-$H$-invariant subspaces of $L^{2}G$, each isomorphic to $L^2H$. The following lemma will allow us to do this in a convenient manner. Denote by $F_r := s_r(G/H)$ a cross-section of the canonical projection $G\rightarrow G/H$ (see preliminaries).

\begin{lemma} \label{lma:torture} \todo{lma:torture}
Let $\{e_n\}_{n\in \mathbb{N}}$ be an orthonormal basis of $L^2F_r$, and let $\varepsilon, \delta >0$ and $N\in \mathbb{N}$ be given. Then there is a finite family $\{ K_i\}_{finite}$ of pairwise disjoint (relatively) compact (in G) subsets of $F_r$, and an open subset $V$ of $G$ containing the identity such that:
\begin{enumerate}[(i)]
\item Denoting by $\bbb_{K_i}$ the indicator function of $K_i$, the distance from $e_n$ to $\span\{ \bbb_{K_i} \mid \; \textrm{all} \; i \}$ is at most $\varepsilon$ for $1\leq n \leq N$.
\item $\mu(V) \leq \delta$, and $\bbb_{K_i} * \tilde{\bbb}_{K_i}$ has support contained in $V$ for all $i$.
\end{enumerate}
Further, we may take $V$ to be contained in any given neighbourhood of the identity in $G$, and we may also take $\{K_i\}$ to be consistent with any given finite partition of $F_r$ by Borel sets, in the sense that each $K_i$ lies in at most one equivalence class of the partition.
\end{lemma}

\begin{proof}
Let $V$ be any open suset of $G$, containing the identity, contained in some given open set if needed, and with measure $\mu(V)\leq \delta$. Choose a compact subset $C$ of $F_r$ such that $\mu(F_r\setminus C)\leq \frac{\varepsilon}{2}$. Let $s_n, n=1,\dots ,N$ be step functions, supported on $C$, such that $\lVert e_n -s_n\rVert_{2} \leq \varepsilon$ for all $n=1,\dots ,N$. Let $\{ S_i^{(n)}\}_{i=1,\dots,i_n}$ be the supports of the characteristic functions defining $s_n$ for $n=1,\dots ,N$.

By continuity of the group operations, there is for each $t\in C$ a relatively compact neighbourhood $U(t)$ of $t$ such that $U(t)U(t)^{-1}\subseteq V$. Then clearly for $U$ any (Borel) subset of any $U(t)$, the convolution $\bbb_{U} * \bbb_{U^{-1}}=\bbb_{U}*\tilde{\bbb}_{U}$ has support contained in $V$.

Now we finish the proof by noting that $C$ can be covered by finitely many $U(t_1),\dots,U(t_m)$. Let $\{K_i'\}_{i=1,\dots, l}$ be a relabeling of the family of intersections $\{U(t_j)\cap S^{(n)}_{j'} \mid 1\leq j\leq m, 1\leq j' \leq i_n\}$. Then we can take
\begin{equation}
K_i = K_i' \setminus \left( \cup_{j=1}^{i-1} K_j' \right). \nonumber
\end{equation}
The very final statement is clear.
\end{proof}

Now let us fix an orthonormal basis $\{ e_n \}$ of $L^2F_r$, with $e_1=\bbb_{F_r}$, and a countable, decreasing, relatively compact neighbourhood basis $\{ V_j\}_{j\in \mathbb{N}}$ around $\bbb$ in $G$. Then for each $j\in \mathbb{N}$ we choose, recursively, a family $\{ K_i^{(j)}\}_{i\in I_j}$ as in Lemma \ref{lma:torture}, say with $\varepsilon = \delta = 2^{-j}$ and $V=V_j$, such that the $j$'th family is consistent with the $(j-1)$st family. We put
\begin{equation}
\xi_i^{(j)} := \frac{\bbb_{K_i^{(j)}}}{\lVert \bbb_{K_i^{(j)}}\rVert_{2}} = \frac{\bbb_{K_i^{(j)}}}{\surd\mu (K_i^{(j)})}, \quad \eta_j := \sum_{i\in I_j} \xi_i^{(j)} * \tilde{\xi}_i^{(j)}. \nonumber
\end{equation}
Then the $(\xi^{(j)}_i)_i$ are pairwise orthogonal, $\eta_j$ is $C_0$ with support contained in $V_j$, and
\begin{eqnarray}
\lVert \eta_j \rVert_{1} & = & \sum_{i\in I_j} \frac{1}{\mu (K_i^{(j)})} \int_{G} \int_{G} \bbb_{K_i^{(j)}}(s)\tilde{\bbb}_{K_i^{(j)}}(s^{-1}t) d\mu(s) d\mu(t) \nonumber \\
 & = & \sum_{i\in I_j} \frac{1}{\mu (K_i^{(j)})} \int_{G} \bbb_{K_i^{(j)}}(s) \int_{G} \bbb_{K_i^{(j)}}(t^{-1}s) d\mu(t) d\mu(s) \nonumber \\
 & = & \sum_{i\in I_j}\frac{1}{\mu (K_i^{(j)})} \left( \int_{G} \bbb_{K_i^{(j)}}(s) d\mu(s) \right) \left( \int_{G} \bbb_{K_i^{(j)}}(t^{-1}) d\mu(t) \right) \nonumber \\
 & = & \sum_{i\in I_j} \mu(K_i^{(j)}) \nonumber \\
 & \left\{ \begin{array}{c} \geq \\ \leq \end{array} \right. & \left. \begin{array}{l} \mu(F_r)-2^{-(j+1)} \\ \mu(F_r) \end{array} \right\} = \left\{ \begin{array}{l} 1-2^{-(j+1)} \\ 1 \end{array} \right\} . \nonumber
\end{eqnarray}
Here inequality '$\geq$' in the final line follows from Lemma \ref{lma:torture}(i) since $e_1=\bbb_{F_r}$.

Denoting by $\varphi_k(\cdot) := \sum_{i\in I_k} \langle \cdot \xi_i^{(k)}, \xi_i^{(k)}\rangle$ the sum of vector states we see that for every left-bounded $f\in L^2G$
\begin{eqnarray}
\varphi_k(\lambda(\tilde{f}*f)) & = & \sum_{i\in I_k} \langle \lambda(\tilde{f}*f) \xi_i^{(k)}, \xi_i^{(k)} \rangle \nonumber \\
 & = & \int_{G} (\tilde{f}*f)(t) \cdot \left( \sum_{i\in I_k} \xi_i^{(k)}*\tilde{\xi_i^{(k)}} \right)(t) d\mu(t) \nonumber \\
 & = & \int_{G} (\tilde{f}*f)(t) \eta_k(t) d\mu(t) \nonumber \\
 & \rightarrow_k & (\tilde{f}*f)(\bbb) \nonumber \\ 
 & = & \lVert f\rVert_2^{2} = \psi (\lambda(\tilde{f}*f)). \nonumber
\end{eqnarray}

On the other hand, denoting by $M_k$ the span of $\{\xi_i^{(k)}\}_{i\in I_k}$ these are increasing finite dimensional subspaces with dense union in  $L^2(G/H)$ so that for $P_{M_k}$ the orthogonal projections of $L^2F_r$ onto these we get with $Tr$ the trace on $\mathcal{B}(L^2(G/H)) = \mathcal{B}(L^2F_r)$
\begin{eqnarray}
\varphi_k( \lambda(\tilde{f}*f) ) & = & (Tr\otimes \tau)\left( (P_{M_k}\otimes \bbb) \lambda(\tilde{f}*f) (P_{M_k}\otimes \bbb) \right) \nonumber \\
 & = & (Tr\otimes \tau)\left( ( \lambda(f)(P_{M_k}\otimes \bbb) )^{*}\lambda(f)(P_{M_k}\otimes \bbb) \right) \nonumber \\
 & = & (Tr\otimes \tau)\left( \lambda(f) (P_{M_k}\otimes \bbb) \lambda(f)^{*} \right) \nonumber \\
 & \nearrow_k & (Tr\otimes \tau)(\lambda(f)\lambda(f)^{*}) \nonumber \\
 & = & (Tr\otimes \tau)(\lambda(\tilde{f}*f)). \nonumber
\end{eqnarray}

\begin{lemma} \label{lma:dimrestrictionclosure} \todo{lma:dimrestrictionclosure}
Let $G$ be a {\lcsu} group and $H$ a countable discrete subgroup with covolume $1$. Then for every projection $p\in M_n(LG)$.
\begin{equation}
\dim_{\psi} pL^2\psi^n = \dim_{\tau} pL^2\psi^n. \nonumber
\end{equation}
\end{lemma}

\begin{proof}
By the above, $Tr\otimes \tau$ is equal to $\psi$ on the set of projections in $LG$. Hence the claim follows by this and Lemma \ref{lma:dimbyTr}.
\end{proof}

\begin{theorem}[(Restriction)] \label{thm:dimrestriction} \todo{thm:dimrestriction}
Let $G$ be a {\lcsu} group and $H$ a countable discrete subgroup with covolume $1$. Then for every $\psi$-fg. $LG$-module $M$, we have
\begin{equation}
\dim_{\psi} M = \dim_{\tau} M, \nonumber
\end{equation}
with $\psi$ the canonical weight on $LG$ corresponding to the Haar measure, $\tau$ the trace on $LH$, and where we on the right hand side consider $M$ an $LH$-module in the canonical manner, via. the embedding of $LH$ in $LG$.
\end{theorem}

\begin{proof}
Suppose that $M$ is $\psi$-fg with presentation $0\rightarrow K \rightarrow L \rightarrow M \rightarrow 0$ and $L = p(LG^n)$ for some $p\in M_n(LG)$ with finite trace. Then in fact $L\subseteq (LG^{2}_{\psi})^{n}$ so that we can consider it as a submodule of $p.(L^2G)^{n}$. Then by additivity, Lemma \ref{lma:diminfinitesum} and Lemma \ref{lma:dimrestrictionclosure} we get
\begin{eqnarray}
\dim_{\psi} M & = & \dim_{\psi} L - \dim_{\psi} K \nonumber \\
 & = & \dim_{\psi} \overline{L}^{\lVert \cdot \rVert_2} - \dim_{\psi} \overline{K}^{\lVert \cdot \rVert_2} \nonumber \\
 & = & \dim_{\tau} \overline{L}^{\lVert \cdot \rVert_2} - \dim_{\tau} \overline{K}^{\lVert \cdot \rVert_2} \nonumber \\
 & = & \dim_{\tau} L - \dim_{\tau} K \nonumber \\
 & = & \dim_{\tau} M. \nonumber
\end{eqnarray}
\end{proof}

\begin{corollary}
With notation as in the theorem, for any $LG$-module $M$,
\begin{equation}
\dim_{\psi} M \leq \dim_{\tau} M. \nonumber
\end{equation}
\end{corollary}

\begin{example} \label{ex:dimrestrictionineqsharp} \todo{ex:dimrestrictionineqsharp}
The following example shows that the inequality in the previous corollary can in fact be strict, which is not the case if one considers only countable discrete groups.

Let $G$ be a non-discrete {\lcsu} group and $H$ a lattice (cocompact or not). Then $\psi(p)=\infty$ for the canonical tracial weight $\psi$ on $LG$ and any non-zero projection $p\in LH$. It follows that the $LG$-module $E:=LG/LG^2_{\psi}$ has non-zero $LH$-dimension, by Sauer's local criterion, since $\bbb.p=p\neq 0$ in $E$ for $p\neq 0$ in $LH$.

However, clearly $\dim_{(LG,\psi)}E=0$, again by the local criterion.
\end{example}

\section{{$L^2$}-Betti numbers of cocompact lattices} \label{sec:elltwococompactlattice} \todo{sec:elltwococompactlattice}

We identify $\mathrm{Coind}_{H}^{G}\ell^2H \simeq L^2_{loc}(X,\ell^2H)$, where we denote $X:=G/H$, and under this identification the action of $G$ is (recall that $\alpha$ is the canonical cocycle for the inclusion $H\leq G$; see preliminaries for definition of $r,s_r$)
\begin{equation}
(g.\xi)(x) = r(g^{-1}.s_r(x)). \xi(g^{-1}.x) = \alpha(x,g).\xi(g^{-1}.x). \nonumber
\end{equation}
Hence $L^2G$ embeds canonically in $\mathrm{Coind}_H^G\ell^2H$ as a left-$G$-right-$LH$-modules. This yields maps\todo{eq:cohominducedmaps}
\begin{equation}
\label{eq:cohominducedmaps}
i_n:H^n(G,L^2G) = H^n(G,L^2(X,\ell^2H)) \rightarrow H^n(G,\mathrm{Coind}_{H}^{G}\ell^2H).
\end{equation}

The next theorem uses Theorem \ref{thm:dualityalldiscrete} to establish equality of $L^2$-Betti numbers of a locally compact group and all its lattices under the assumption of the existence of at least one cocompact lattice. This should be seen as a strong indication that equality holds generally, but the use of Gaboriau's machinery is somewhat unsatisfactory.

\begin{theorem} \label{cor:finitecovolcocompact} \todo{cor:finitecovolcocompact}
Let $G$ be a {\lcsu} group with Haar measure $\mu$ and suppose that $G$ contains a cocompact lattice $H_0$. Then for every lattice (not neccesarily cocompact) $H$ in $G$ and every $n$ we have
\begin{equation}
\beta^n_{(2)}(H) = \mathrm{covol}_{\mu}(H)\cdot \beta^n_{(2)}(G, \mu). \nonumber
\end{equation}
\end{theorem}

\begin{proof}
Let $n$ be given. We can assume that $H_0$ has covolume $1$.

Since $H_0$ is cocompact the morphism $H^n(G,L^2G)\rightarrow H^n(H_0,\ell^2H_0)$, given by $i_n$ and the Shapiro lemma (see Lemma \ref{lma:cohomshapiro}), is an isomorphism of right-$LH_0$-modules and a homeomorphism. Then using first Lemma \ref{lma:dimrestrictionclosure} combined with \ref{lma:reducedelltwolimit}, and then appealing to \ref{thm:dualityalldiscrete} we get
\begin{equation}
\underline{\beta}^n_{(2)}(G,\mu) = \underline{\beta}^n_{(2)}(H_0) = \beta^n_{(2)}(H_0). \nonumber
\end{equation}

If the right-hand side of this is infinite the statement now follows. If it is finite we get by Theorem \ref{thm:dimrestriction} and Theorem \ref{thm:dualityalldiscrete}

\begin{eqnarray}
\dim_{\psi} \overline{B^n(G,L^2G)}^{L^2_{loc}} / B^n(G,L^2G) & \leq & \dim_{(LH_0,\tau)} \overline{B^n(G,L^2G)}^{L^2_{loc}} / B^n(G,L^2G) \nonumber \\
 & = & \dim_{(LH_0,\tau)} \overline{B^n(H_0,\ell^2H_0)}^{L^2_{loc}} / B^n(H_0,\ell^2H_0) = 0. \nonumber
\end{eqnarray}

Then by additivity the claim follows again by the computation above. Hence we have shown the statement for $H=H_0$. For general $H$ it follows now by the theorem of Gaboriau on the measure equivalence-invariance of $\ell^2$-Betti numbers \cite[Theorem 6.3]{Ga02}.
\end{proof}

 %this is OK.

\chapter{Totally disconnected groups} \label{chap:totdisc}
%\epigraph{No one needs any of that fucking shit. I don't have anything. If this bus were to catch of fire it wouldn't make any difference to me.}{Alice Glass}

In this chapter we restrict attention to totally disconnected groups, where several simplifications can be made. The slogan is that the theory in this setting is a natural extension of the theory of $\ell^2$-Betti numbers for discrete groups. In particular, the bar resolutions for totally disconnected groups constructed in Section \ref{sec:bartotdisc}, where one relativizes with respect to a compact open subgroup, can be approximated by modules with finite dimension over the group von Neumann algebra.

This observation, along with the dimension-exactness properties of induction- and hom-functors is applied in Sections \ref{sec:dualitytotdisc} and \ref{sec:totdisclattices} to first extend the duality and change of coefficient results for discrete groups of section \ref{sec:dualitydiscrete}, and then to prove equality of $L^2$-Betti numbers of lattices and the ambient totally disconnected group, see Theorem \ref{thm:totdisclattice}.

In Section \ref{sec:locfingraph} we show that the definition of $L^2$-Betti numbers coincides with the definition due to Gaboriau \cite{Ga05} of the first $L^2$-Betti number of a locally finite, vertex-transitive unimodular graph. We also consider actions on simplicial complexes, and make some remarks about general $L^2$-Betti numbers of actions of (totally disconnected) locally compact groups, in the spirit of Cheeger-Gromov's original definition for countable groups \cite{ChGr86}.

In Section \ref{sec:totdiscamenable} we prove that all $L^2$-Betti numbers of any (non-compact) totally disconnected, amenable {\lcsu} group vanish. Consistently with the slogan of this chapter, any proof that works for countable groups should extend naturally and effortlessly to totally disconnected groups. I chose my favorite exposition, that of Lyons \cite{Ly08}, and the proof is a direct generalization.

We also sketch a second proof which is more algebraic, generalizing L{\"u}ck's ideas for dimension flatness of the inclusion $\mathbb{C}\Gamma \subseteq L\Gamma$ for countable discrete groups, in Section \ref{sec:totdisckyed}.

Note that the vanishing in degree one for amenable groups follows already from the results of Gaboriau in \cite{Ga05} and the equivalence of definitions in Theorem \ref{thm:HdRlocfingraph}. One can for instance use \cite{ElTa00} to conclude vanishing for the first $L^2$-Betti number of the graph.

    % OK

\section{Duality results for totally disconnected groups} \label{sec:dualitytotdisc} \todo{sec:dualitytotdisc}

Let $G$ be a totally disconnected {\lcsu} group, and fix a compact open subgroup $K$ and a Haar measure $\mu$ on $G$ such that $K$ has measure one. In this section we set up a duality of the complexes of inhomogeneous (co)chains in \ref{prop:totdiscbar} and \ref{thm:totdiscbarhomology} and extend the results of Section \ref{sec:dualitydiscrete} to totally disconnected groups.

\begin{theorem}[(Change of coefficients)] \label{thm:homtotdisccoeff} \todo{thm:homtotdisccoeff}
Let $G$ be a totally disconnected {\lcsu} group. For every $n\geq 0$
\begin{equation}
\beta^{(2)}_n(G,\mu) = \dim_{(LG,\psi)} H_n(G,L^2G); \quad \beta_{(2)}^n(G,\mu) = \dim_{(LG,\psi)} H^n(G,LG). \nonumber
\end{equation}
\end{theorem}

We need the following easy observation for the proof.

\begin{lemma} \label{lma:CKdescr} \todo{lma:CKdescr}
Fix $n$ and let $L$ be a fundamental domain for the action of $K$ on $G^n_K$. Let $E$ be a topological $G$-$LG$-module, dense in a quasi-complete module $\tilde{E}$. For $f\in \mathcal{F}_c(G^n_K, E)$ we have $f\in \overline{ \span}_{\mathbb{C}}\{ f'-f'.k \mid f'\in \mathcal{F}_c(G^n_K,E), k\in K\}$ if and only if $\Phi(f)=0$, where $\Phi\colon \mathcal{F}_c(G^n_K,E) \rightarrow \mathcal{F}_c(L,\tilde{E})$ is the $LG$-linear map given by
\begin{equation}
\Phi(f) := L\owns x \mapsto \int_K(f.k)(x)\mathrm{d}\mu(k) \in \tilde{E}. \nonumber
%\Phi(f) := \sum_{(g_i)\in G^n_K} \int_K(f.k)(g_i)\mathrm{d}\mu(k) \in \tilde{E}. \nonumber
\end{equation}
\end{lemma}

\begin{proof}[Proof of Theorem \ref{thm:homtotdisccoeff}]
We consider complexes
\begin{displaymath}
\xymatrix{ \mathfrak{L}_* : &\cdots & \mathcal{F}_c( G^n_K, L^2G ) \ar[l] &  \mathcal{F}_c( G^{n+1}_K, L^2G ) \ar[l]_{d_n} & \cdots \ar[l] \\ \mathfrak{M}_* : &\cdots &  \mathcal{F}_c( G^n_K, (LG_{\psi}^{2},US) ) \ar[u]^{\varphi_n} \ar[d]_{\phi_n} \ar[l] &  \mathcal{F}_c( G^{n+1}_K, (LG_{\psi}^{2},US) ) \ar[u]^{\varphi_{n+1}} \ar[d]_{\phi_{n+1}} \ar[l]_{d_n} & \cdots \ar[l] \\ \mathfrak{N}_* : &\cdots & \mathcal{F}_c( G^n_K, LG ) \ar[l] & \mathcal{F}_c( G^{n+1}_K, LG ) \ar[l]_{d_n} & \cdots \ar[l] }
\end{displaymath}
where $US$ indicates the ultra-strong topology on $LG^2_{\psi}$ and the maps $\varphi_{*},\phi_{*}$ are induced by the relevant inclusions.

Denote by $\Phi_{\mathfrak{L}_n}$ the map $\Phi$ of the lemma in the $G^n_K$-term of the top complex, etc. The maps $\varphi_{*},\phi_{*}$ commute with the $\Phi_*$ and since the $\Phi_*$ are also $LG$-morphisms it is clear that the inclusions
\begin{equation}
\varphi_n(\ker \Phi_{\mathfrak{M}_n}) \subseteq \ker \Phi_{\mathfrak{L}_n} \quad \textrm{ and } \quad \phi_n(\ker \Phi_{\mathfrak{M}_n}) \subseteq \ker \Phi_{\mathfrak{N}_n} \nonumber
\end{equation}
are rank dense.

It follows that the induced morphisms $\bar{\varphi}_* \colon \underline{C}_K \mathfrak{M}_*\rightarrow \underline{C}_K \mathfrak{L}_*$ and $\bar{\phi}_* \colon \underline{C}_K \mathfrak{M}_* \rightarrow \underline{C}_K \mathfrak{N}_*$ are rank-isomorphisms in all degrees.

This proves the statement about homology. Cohomology is handled entirely analogously.
\end{proof}

For $f\in \mathcal{F}_c(G^n_K,LG), \xi\in \mathcal{F}(G^n_K,LG)$ we define
\begin{equation}
\langle f,\xi \rangle := \sum_{(g_i)\in G^n_K} f(g_1,\dots g_n).\xi(g_1,\dots ,g_n) \in LG. \nonumber
\end{equation}
This is well-defined since $f$ is finitely supported.

\begin{lemma} \label{lma:dualitytotdisceasy} \todo{lma:dualitytotdisceasy}
The duality $\langle -,-\rangle$ induces a duality of $\underline{\mathscr{C}}_K\mathcal{F}_c(G^n_K,LG)$ and $\mathcal{F}(G^n_K,LG)^K$, and under this duality:
\begin{enumerate}[(i)]
\item $\mathcal{F}(G^n_K,LG)^K \simeq \hom_{LG}(\underline{\mathscr{C}}_K\mathcal{F}_c(G^n_K,LG), LG)$ as right-$LG$-modules.
\item $\langle d_nf,\xi \rangle = \langle f,d^n\xi \rangle$.
\end{enumerate}
\noindent
In particular,
\begin{equation}
\dim_{(LG,\psi)} \underline{H}^n(G,L^2G) = \dim_{(LG,\psi)} \mathbf{P}H_n(G,LG). \nonumber
\end{equation}
\end{lemma}

\begin{proof}
(i) follows directly from the corresponding statement
\begin{equation}
\mathcal{F}(G^n_K,LG) \simeq \hom_{LG} \left( \mathcal{F}_c(G^n_K,LG), LG \right), \nonumber
\end{equation}
which is clear.

(ii) is a direct computation, which we leave out.

By (i) and (ii) it follows that $\underline{H}^n(G,LG)\simeq \operatorname{hom}_{LG}(\mathbf{P}H_n(G,LG), LG)$ and then the final claim follows by Theorem \ref{thm:homdimexactpreserv}. (Observe that, along the lines of Theorem \ref{thm:homtotdisccoeff} one shows readily that also $\dim_{(LG,\psi)} \underline{H}^n(G,LG) = \dim_{(LG,\psi)} \underline{H}^n(G,L^2G)$.)
\end{proof}

Now let for all $m,n\in \mathbb{N}$, $S_n^{(m)}\subseteq G$ be compact subsets such that:
\begin{itemize}
\item For all $m,n$, the compact open subgroup $K\subseteq S_n^{(m)}$, and for all $n\in \mathbb{N}$, we have $S_n^{(m)}\nearrow G$.
\item For all $m,n$, $(S_{n+1}^{(m)})^2 \subseteq S_n^{(m)}$.
\end{itemize}
We may construct such a (double-)sequence as follows. Let $\{g_i\}\subseteq G$ be a countable set such that $G=\cup_i g_iK$. Put $T_m=\cup_{i=1}^{m}g_iK$.

Then let for all $m,n$, 
\begin{equation}
S_n^{(m)} = \left\{ \begin{array}{cl} T_m^{2^{m-n}} , & m\geq n, \\ K, & m < n \end{array} \right. . \nonumber 
\end{equation}

Denote by $F^{(m)}_n$ the (finite) $K$-invaiant (wrt.~the action on the first coordinate) subsets of $G^n_K$ generated by the projections of $\prod_{i=1}^n S^{(m)}_n$ in $G^n_K$.

Then one checks that, just as in Section \ref{sec:dualitydiscrete}, we get complexes as in the following diagram,
\begin{displaymath}
\xymatrix{ \mathfrak{M}_*^{m} : & \cdots \ar[r] & \mathcal{F}( F^{(m)}_n, LG )^{K} \ar[r]^{d^n} \ar@{--}[d]^{\langle -, -\rangle_{F^{(m)}_n}} & \mathcal{F}( F^{(m)}_{n+1}, LG )^{K} \ar[r] \ar@{--}[d]^{\langle -, -\rangle_{F^{(m)}_{n+1}}} & \cdots \\ \mathfrak{N}_*^{m} : &\cdots & \underline{\mathscr{C}}_K \mathcal{F}( F_n^{(m)}, LG ) \ar[l] & \underline{\mathscr{C}}_K \mathcal{F}( F^{(m)}_{n+1}, LG ) \ar[l]_{d_n} & \cdots \ar[l] }
\end{displaymath}

Then we can refine the analysis of the previous lemma to show:

\begin{lemma}
Under the duality $\langle -,-\rangle_{F^{(m)}_n}$:
\begin{enumerate}[(i)]
\item $\mathcal{F}(F^{(m)}_n,LG)^K \simeq \hom_{LG}(\underline{\mathscr{C}}_K\mathcal{F}(F^{(m)}_n,LG), LG)$ as right-$LG$-modules.
\item $\langle d_nf,\xi \rangle = \langle f,d^n\xi \rangle$.
\end{enumerate}
\noindent
In particular,
\begin{equation}
\dim_{(LG,\psi)} \underline{H}^n(G,L^2G) = \dim_{(LG,\psi)} H_n(G,LG). \nonumber
\end{equation}
\end{lemma}

\begin{proof}
Again (i) and (ii) are clear.

To see the final statement note that, since on $F^{(m)}_n$ a $K$-invariant element $\xi\in \mathcal{F}(F^{(m)}_n,LG)^K$ takes values in the module of fixed points $(LG)^{\cap_{(g_i)\in F^{(m)}_n} K_{g_1}} = \bbb_{{\cap_{(g_i)\in F^{(m)}_n} K_{g_1}}}*LG$, the $\mathcal{F}(F^{(m)}_n,LG)^K$ all have finite $LG$-dimension, whence so do the $\underline{C}_K\mathcal{F}_c(G_K^n,LG)$-spaces by (i) and Lemma \ref{lma:moddualineq}.

%Then the claim follows because $\underline{H}^n(G,L^2G)$ is the projective limit of reduced cohomologies of complexes corresponding\todo{maybe clean this up a bit.} to $\mathfrak{M}_*^{m}$ but with coefficients in $L^2G$, i.e.~rank isomorphic to $\mathfrak{M}_*^{m}$, 
Now consider complexes
\begin{displaymath}
\xymatrix{ \mathfrak{L}_*^{m} : & \cdots \ar[r] & \mathcal{F}( F^{(m)}_n, L^2G )^{K} \ar[r]^{d^n} & \mathcal{F}( F^{(m)}_{n+1}, L^2G )^{K} \ar[r] & \cdots } .
\end{displaymath}
By rank density arguments these have cohomology with dimension
\begin{equation}
\dim_{\psi} \underline{H}^n(\mathfrak{L}_*^m) = \dim_{\psi} H^n(\mathfrak{L}_*^m) = \dim_{\psi} H^n(\mathfrak{M}_*^m). \nonumber
\end{equation}

We now claim that $\underline{H}^n(G,L^2G)$ is rank-isomorphic to the projective limit $\lim_{\leftarrow} \underline{H}^n(\mathfrak{L}_*^m)$. Indeed, there is a map $\iota$ into the projective limit by restriction of inhomogeneous cocycles.

Injectivity is straight-forward: If $\xi \in \mathcal{F}(G^n_K,L^2G)^K$ is in the closure of the space of cocycles, then clearly it maps to zero in all $\underline{H}^n(\mathfrak{L}_*^m)$ by restriction. Conversely, if $\xi$ is an inhomogeneous cocycle mapping to zero by restriction for all $m$, we take for each fixed $m$ a sequence $\eta^m_i$ in $d^{n-1}(\mathcal{F}(F_{n-1}^{(m)},L^2G)^K)$ converging to the restriction of $\xi$, and extend it by zero to a sequence $\eta^{m,0}_i\in \mathcal{F}(G^{n-1}_K,L^2G)^K$. Clearly the net $d^{n-1}(\eta^{m,0}_i)$, ordered lexicographically, converges to $\xi$.

To see surjectivity, let $\Xi^k$ be the submodule of the projective limit consisting of elements that have representative sequences $(\xi_m)_m$ of inhomogeneous cocycles such that for $m\leq k$ we have $\xi_m\vert_{F_n^{m-1}} - \xi_{m-1} \in d^{n-1}(\mathcal{F}(F_{n-1}^{(m-1)}, L^2G)^K$. Then we have $\cap_k \Xi_k \subseteq \operatorname{im}(\iota)$ since for such an element we can recursively define an equivalent representative which is an inhomogeneous cocycle on all of $G^n_K$. But since everything is finite dimensional each $\Xi_k$ is rank dense in $\Xi_{k-1}$ whence by induction in the projective limit. The claim follows now by the countable annihilation lemma.

By the isomorphism up to rank of $\underline{H}^n(\mathfrak{L}_*^m)$ and $H^n(\mathfrak{M}_*^m)$, compatible with restriction maps, the projective limits have the same dimension, so that, since everything is finite dimensional
\begin{equation}
\dim_{\psi} \underline{H}^n(G,L^2G) = \dim_{(LG,\psi)} \lim_{\leftarrow m} H^n(\mathfrak{M}_*^m). \nonumber
\end{equation}

By (i) and (ii) we can apply the dimension-preservation of the hom-functor, Theorem \ref{thm:homdimexactpreserv} to get
\begin{equation}
\dim_{(LG,\psi)} H^n(\mathfrak{M}_*^m) = \dim_{(LG,\psi)} H_n(\mathfrak{N}_*^m), \nonumber
\end{equation}
and this is compatible with restriction, repectively inclusion maps, i.e.~these are dual maps also, whence
\begin{equation}
\dim_{(LG,\psi)} \operatorname{im}(H^n(\mathfrak{M_*^m}) \rightarrow H^n(\mathfrak{M}_*^{m-1})) = \dim_{(LG,\psi)} \operatorname{im}( H_n(\mathfrak{N}_*^{m-1})\rightarrow H_n(\mathfrak{N}_*^m)). \nonumber
\end{equation}

Finally, the homology $H_n(G,LG)$ the direct limit of homology of $\mathfrak{N}_*^{m}$, whence the statement follows by the projective respectively injective limit formulas, see Theorem \ref{thm:dimensionsummary}.
\end{proof}

\begin{theorem} \label{thm:dualitytotdisc} \todo{thm:dualitytotdisc}
Let $G$ be a totally disconnected {\lcsu} group. Then for all $n\geq 0$,
\begin{equation}
\underline{\beta}_{(2)}^n(G,\mu) = \beta^{(2)}_n(G,\mu) = \beta_{(2)}^n(G,\mu). \nonumber
\end{equation}
\end{theorem}

\begin{proof}
In view of the previous lemma, it is sufficient to show the second equality. It follows from Lemma \ref{lma:dualitytotdisceasy}(i),(ii) and the dimension exactness and -preserving properties of the hom-functor $\operatorname{hom}_{LG}(-,LG)$ on countably generated modules, Theorem \ref{thm:homdimexactpreserv} that
\begin{equation}
\dim_{(LG,\psi)} H_n(G,LG) = \dim_{(LG,\psi)} H^n(G,LG). \nonumber
\end{equation}

Then by Theorem \ref{thm:homtotdisccoeff} the statement follows.
\end{proof}

We observe that the proofs in this section do not rely on the fact the coefficients are specifically the group von Neumann algebra of the group in question. We use only the rank isomorphism of $LG$ and $L^2G$ in Theorem \ref{thm:homtotdisccoeff}, and the dimension exactness and -preserving properties of $\operatorname{hom}_{LG}(-,LG)$ on countably generated modules. Hence we single out the following result for easy reference.

\begin{porism} \label{por:dualitytotdisc} \todo{por:dualitytotdisc}
Let $G$ be a totally disconnected group and $\tilde{G}$ a {\lcsu} group such that $G\leq \tilde{G}$. Then for all $n\geq 0$
\begin{equation}
\dim_{(L\tilde{G},\tilde{\psi})} H^n(G,L^2\tilde{G}) = \dim_{(L\tilde{G},\tilde{\psi})} H_n(G,L\tilde{G}). \nonumber
\end{equation}
\end{porism}

         % OK

%\input{ThesisTotDiscBounds}

\section{Lattices in totally disconnected groups} \label{sec:totdisclattices} \todo{sec:totdisclattices}

Denote in this section $\bar{\iota}_n\colon H_n(G,\operatorname{Ind}_{H}^{G}L^2H) \rightarrow H_n(G,L^2G)$ the maps induced by the inclusion $\operatorname{Ind}_H^G L^2H \xrightarrow{\iota} L^2G \cong L^2H\bar{\otimes} L^2Y$, where $Y:=H\backslash G$, see Definition \ref{def:Indmodule}. Let us be more specific:

By definition, an element in $\operatorname{Ind}_H^G L^2H$ is represented by an equivalence class of functions in $L^2_c(G,L^2H)$. Recall that this means functions $f$ which are compactly supported Borel maps into the Borel structure on $L^2H$ given by the norm topology, and such that $\int_G \lVert f\rVert^2 \mathrm{d}\mu < \infty$.

Further, by the condition that functions be compactly supported, we have a canonical inclusion $L^2_c(G,L^2H)\subseteq L^1_c(G,L^2H)$, and since functions in the latter allow an integral $\int_G f\mathrm{d}\mu \in L^2H$ we get an induced map $\iota(\bar{f}) := \int_G (f(g)\otimes \bbb_Y).g\mathrm{d}\mu$ of $\operatorname{Ind}_H^G L^2H$ into $L^2G$, and this is an injective morphism of left-$LH$-modules.

Recall from the previous section the construction of the sets $F_n^{(m)}\subseteq G_K^n$ and consider the complexes
\begin{displaymath}
\xymatrix{ \mathfrak{M}_*^m: & \cdots \ar[r]^>>>>>{d_{n}^{(m)}} & \underline{\mathscr{C}}_K\mathcal{F}(F^{(m)}_n, L^2G) \ar[r]^<<<<<{d_{n-1}^{(m)}} & \cdots \ar[r]^>>>{d_0^{(m)}} & \underline{\mathscr{C}}_KL^2G \ar[r] & 0 }.
\end{displaymath}
Recall that $H_n(G,L^2G) \simeq \lim_{\rightarrow} H_n(\mathfrak{M}_*^{m},d_*^{(m)})$. Similarly, we get a sequence of complexes 
\begin{displaymath}
\xymatrix{\mathfrak{N}_*^{m}: & \cdots \ar[r]^>>>>>{d_{n}^{(m)}} & \underline{\mathscr{C}}_K\mathcal{F}(F^{(m)}_n, \operatorname{Ind}_H^G L^2H) \ar[r]^<<<<<{d_{n-1}^{(m)}} & \cdots \ar[r]^>>>{d_0^{(m)}} & \underline{\mathscr{C}}_K \operatorname{Ind}_H^G L^2H \ar[r] & 0 }
\end{displaymath}
with coefficients in $\operatorname{Ind}_H^G\ell^2H$ instead of $L^2G$, and $H_n(G,\operatorname{Ind}_H^G\ell^2H)\simeq \lim_{\rightarrow} H_n(\mathfrak{N}_*^m,d_*^{(m)})$.

\begin{theorem} \label{thm:totdisclattice} \todo{thm:totdisclattice}
Let $G$ be a totally disconnected {\lcsu} group and $H$ a lattice in $G$. Then for all $n\geq 0$,
\begin{equation}
\beta^n_{(2)}(H) = \operatorname{covol}_{\mu}(H)\cdot \beta^n_{(2)}(G,\mu). \nonumber
\end{equation}
\end{theorem}

\begin{proof}
We can assume without loss of generality that the covolume is one. See Proposition \ref{prop:bettihaarscaling}.

Since the inclusion $\iota_n^{(m)} \colon \mathfrak{N}_n^m \rightarrow \mathfrak{M}_n^m$ has dense image and the ambient modules are (isomorphic to) $p^{(n)}(L^2G^{n'})$ for some projections $p^{(n)}$ with finite trace, it follows by Lemma \ref{lma:dimbyTr} and the local criterion that $\iota_n^{(m)}$ has rank dense image for the $LH$-module structure.

Then by the inductive limit formula it follows that
\begin{eqnarray}
\dim_{(LG,\psi)} H_n(G,L^2G) & = & \sup_m \inf_{m':m'\geq m} \dim_{(LG,\psi)} \operatorname{im}(H_n(\mathfrak{M}_*^m) \rightarrow H_n(\mathfrak{M}_*^{m'})) \nonumber \\
 & = & \sup_m \inf_{m':m'\geq m} \dim_{(LH,\tau)} \operatorname{im}(H_n(\mathfrak{M}_*^m) \rightarrow H_n(\mathfrak{M}_*^{m'})) \nonumber \\
 & = & \sup_m \inf_{m':m'\geq m} \dim_{(LH,\tau)} \operatorname{im}(H_n(\mathfrak{N}_*^m) \rightarrow H_n(\mathfrak{N}_*^{m'})) \nonumber \\
 & = & \dim_{(LH,\tau)} H_n(G,\operatorname{Ind}_H^G L^2H), \nonumber
\end{eqnarray}
where the change of dimension in the second equality follows e.g.~by Lemma \ref{lma:dimrestrictionclosure} and additivity.

Now the theorem follows by changing the coefficients, cf.~Theorem \ref{thm:homtotdisccoeff}.
\end{proof}

        % OK

%\input{NotesLtwolcHodgedeRahm}

\section{Simplicial actions of totally disconnected groups} \label{sec:totdiscactions} \todo{sec:totdiscactions}

In this section we give the definitions and some basic results for $L^2$-Betti numbers for \emph{actions} of totally disconnected groups on simplicial (or CW-) complexes. These are equivariant homotopy invariants. When the totally disconnected {\lcsu} group $G$ acts on the contractible simplicial complex $\Delta$ with compact stabilizers, the $L^2$-Betti numbers of the action coincide with those of the group. For our examples this is actually the only case we need; regardless, I find it more natural to give the general definitions.

We give only quite brief proofs and indications. Generally speaking, the results and proofs are very well known in the discrete case, and not much is gained from laboriously expounding on the details; on the other hand, we study in Section \ref{sec:locfingraph} below the special case of actions on graphs, where we give exhaustive details.

\begin{definition}
Let $\Delta$ be a countable simplicial (respectively CW) complex and $G$ a totally disconnected, {\lcsu} group acting continuously and with compact stabilizers on $\Delta$. We define the simplicial (respectively cellular) $L^2$-cohomology of the action as the cohomology of the complex (where $\mathcal{F}_{alt}(\Delta_i,-)$ denotes spaces of alternating functions wrt.~the action of $S_n$ on the $n$-skeleton $\Delta_n$ of $\Delta$.)
\begin{displaymath}
\xymatrix{ \mathfrak{M}^*: & 0 \ar[r] & \mathcal{F}_{alt}(\Delta_0, L^2G)^G \ar[r]^{d^0} & \mathcal{F}_{alt}(\Delta_1, L^2G)^G \ar[r]^<<<<<{d^1} & \cdots }
\end{displaymath}
where the coboundary maps are given by (in the simplicial case - the cellular case is similar)
\begin{equation}
(d^n\xi)(v_0,\dots ,v_n) = \sum_{i=0}^{n} (-1)^i \xi(v_0,\dots ,\hat{v}_i,\dots ,v_n), \nonumber
\end{equation}
and the spaces of functions on the $n$-skeletons are endowed with the topology of pointwise convergence. We denote this
\begin{equation}
H^n_{(2)}(\Delta; G) = H^n_{cell}(\Delta; G, L^2G) = H^n(\mathfrak{M}^*). \nonumber
\end{equation}

Then we define $L^2$-Betti numbers of the action as
\begin{equation}
\beta^n_{(2)}(\Delta; G, \mu) := \dim_{\psi} H^n_{(2)}(\Delta;G). \nonumber
\end{equation}

Similarly we define the $L^2$-homology as the homology of the complex
\begin{displaymath}
\xymatrix{ \mathfrak{N}_* : & \cdots \ar[r]^>>>>>{d_1} & \underline{C}_G\mathcal{F}_{c,alt}(\Delta_1,LG) \ar[r]^{d_0} & \underline{C}_G \mathcal{F}_{c,alt}(\Delta_0,LG) \ar[r] & 0 },
\end{displaymath}
where $G$ acts on $LG$ by $T.g = Tg$ for $T\in LG$, and the boundary maps are
\begin{equation}
(d_nf)(v_1,\dots ,v_{n+1}) = \sum_{\delta\in \Delta_{n+1}:\delta=(v,v_1,\dots,v_{n+1})} f(\delta). \nonumber
\end{equation}

Here the spaces are endowed with their inductive topologies, and we denote
\begin{equation}
H_n^{(2)}(\Delta;G) = H_n^{cell}(\Delta;G,L^2G) = H_n(\mathfrak{N}_*). \nonumber
\end{equation}

The (homological) $L^2$-Betti numbers are then defined as
\begin{equation}
\beta_n^{(2)}(\Delta;G,\mu) = \dim_{(LG,\psi)} H_n^{(2)}(\Delta;G). \nonumber
\end{equation}
\end{definition}

\begin{theorem} \label{thm:homotopyinvariance} \todo{thm:homotopyinvaraince}
Let $G$ be a totally disconnected {\lcsu} group acting continuously and with compact stabilizers on the countable simplicial complex $\Delta$. Then for all $n\geq 0$ the homological and cohomological $L^2$-Betti numbers of the action coincide,
\begin{equation}
\beta^n_{(2)}(\Delta;G,\mu) = \beta_n^{(2)}(\Delta;G,\mu). \nonumber
\end{equation}

Further, the $L^2$-Betti numbers of the action are invariant under $G$-homotopy of $\Delta$, and if $\Delta$ is contractible then for all $n\geq 0$\todo{eq:elltwocontractible}
\begin{equation} \label{eq:elltwocontractible} 
\beta^n_{(2)}(\Delta; G,\mu) = \beta^n_{(2)}(G,\mu).
\end{equation}
\end{theorem}

\begin{proof}
We prove the final part first. Suppose that $s^*$ is a contraction of $\Delta$, i.e.~a contraction of the complex\todo{BEND THE s ARROWS}

\begin{displaymath}
\xymatrix{ 0 \ar[r] & \mathbb{C} \ar[r]^>>>>{\varepsilon} & \mathcal{F}_{alt}(\Delta_0, \mathbb{C}) \ar@/^/[l]^<<<{s^0} \ar[r]^{d^0} & \mathcal{F}_{alt}(\Delta_1, \mathbb{C}) \ar@/^/[l]^{s^1} \ar[r]^<<<<{d^1} & \cdots }
\end{displaymath}
Observe that, by definition of the coboundary maps, we can arrange that on any given finite subset $F$ of $\Delta_{i-1}$, the values of $s^i\xi\vert_F, \xi\in \Delta_i$ depend only on the values of $\xi$ on some finite subset of $\Delta_i$. (For instance, start with some contraction of the complex to compute homology $\mathcal{F}_{c,alt}(\Delta_*,\mathbb{C}) \rightarrow \mathbb{C} \rightarrow 0$ and let $s^*$ be the duals.)

Consider the spaces $\mathcal{F}_{alt}(\Delta_*,\mathbb{C})\otimes_{alg} L^2G$ as subspaces of $\mathcal{F}_{alt}(\Delta_*,L^2G)$. Then it is clear that the induced maps $\bar{s}^i\colon \mathcal{F}_{alt}(\Delta_i,\mathbb{C})\otimes_{alg} L^2G \rightarrow \mathcal{F}_{alt}(\Delta_{i-1}, \mathbb{C})\otimes_{alg} L^2G$ are bounded, whence they extend continuously to a contraction of the injective resolution $0 \rightarrow L^2G \rightarrow \mathcal{F}_{alt}(\Delta_*,L^2G)$.

We leave out the entirely analogous prove of equality for homology $L^2$-Betti numbers.

We leave out the proof of the homotopy invariance. For the equality of homological and cohomological $L^2$-Betti numbers, we just remark that the proof is entirely analogous to the duality results of Section \ref{sec:dualitytotdisc}. In any case, since we only use the case where $\Delta$ is contractible, this follows from Theorem \ref{thm:dualitytotdisc} and the part already proved.
\end{proof}

We note that any totally disconnected ($2$nd countable) group acts on a contractible countable simplicial complex with compact stabilizers. In fact there is a universal model for such a complex, $\underline{E}G$, constructed as follows \cite{bch}.

Denote $W:= \dot{\cup}_K G/K$, the disjoint union of coset spaces over all compact open subgroups of $G$. Then we let $\underline{E}G$ be the infinite Milnor join $W*W*\cdots$. Note this is countable since there are only countably many compact open sets in $G$.

\begin{corollary}
For any totally disconnected {\lcsu} group $G$ and all $n\geq 0$,
\begin{equation}
\beta^n_{(2)}(G,\mu) = \beta^n_{(2)}(\underline{E}G; G,\mu). \nonumber
\end{equation}
\end{corollary}

\begin{flushright}
\qedsymbol
\end{flushright}
We note for reference that the equalities of the previous theorem and corollary actually follow by equality of the cohomology spaces:

\begin{porism} \label{por:totdiscactionscohom} \todo{por:totdiscactionscohom}
Let $G$ be a (totally disconnected) {\lcsu} group acting continuously on the contractible, countable simplicial complex $\Delta$. Then for all $n\geq 0$,
\begin{equation}
H^n(G,L^2G) = H_{(2)}^n(\Delta; G) \quad \textrm{and} \quad H_n(G,LG) \simeq H^{(2)}_n(\Delta;G). \nonumber
\end{equation}
\end{porism}

The next results show how to compute the $L^2$-Betti numbers of a simplicial action of $G$ in terms of the spaces of $\ell^2$-chains and -cochains on the simplicial complex. For the statements, denote by $\Delta_n$ the $n$-skeleton of the simplicial complex $\Delta$, i.e. the set of $n$-simplices. There is a canonical action of $S_{n+1}$ on this, and  we denote by $\ell^{2}_{alt}(\Delta_n)$ the space of $\ell^2$-functions $f$ on $\Delta_n$ such that for all $v\in \Delta_n$ and all $\sigma\in S_{n+1}$ we have $f(\sigma.v) = \sign (\sigma) f(v)$. We then get a complex
\begin{displaymath}
\xymatrix{ 0 \ar[r] & \ell^2\Delta_0 \ar[r]^{\partial^0} & \ell^2_{alt}(\Delta_1) \ar[r]^{\partial^1} & \ell^2_{alt}(\Delta_2) \ar[r]^<<<<{\partial^2} & \cdots }
\end{displaymath}
where the coboundary maps $\partial^n$ are given by
\begin{equation}
(\partial^nf)(v_0,\dots ,v_n) = \sum_{i=0}^{n} (-1)^{i}f(v_0,\dots ,\hat{v}_i,\dots ,v_n). \nonumber
\end{equation}

We denote by $\ell^{2}_{\circ}(\Delta_n)$ closure of the space spanned by finite cycles, equivalently the orthogonal complement of $\ker \partial_n$ and by $\ell^{2}_{\star}(\Delta_n)$ the closure of the image of $\partial_{n-1}$. We might also denote by $Z^n_{(2)}(\Delta)$ and $Z_n^{(2)}(\Delta)$ the spaces of $\ell^2$ $n$-cocycles, respectively -cycles, i.e.~the kernels of $\partial^n$ respectively $\partial_n:=(\partial^n)^*$; we denote then also $B^n_{(2)}(\Delta)$ and $B_n^{(2)}(\Delta)$ the spaces of $\ell^2$ $n$-(co)boundaries, where we do not automatically take the closure.

We identify $\ell^2_{alt}(\Delta_n)$ with a subspace of the direct sum of right-$LG$-modules $\ell^2(G_{s}\backslash G)$ where $s$ runs over a fundamental domain for the action of $G$ on $\Delta_n$ and $G_s$ is the stabilizer. This in particular gives a right-$LG$-module structure on $\ell^2_{alt}(\Delta_n)$, and the coboundary maps are $LG$-equivariant.

\begin{proposition} \label{prop:elltwoactioncompute} \todo{prop:elltwoactionscompute}
Let $\Delta$ be a locally finite countable simplicial complex with a continuous action of the totally disconnected {\lcsu} group $G$, such that the stabilizers are compact.

Suppose that the action is cofinite, i.e.~in each $n$-skeleton $\Delta_n$ there is a finite fundamental domain for the ation. Then the $L^2$-Betti numbers of the action can be computed as the $LG$-dimensions of cohomology spaces $H^n_{(2)}(G;\Delta) \simeq H^n(\mathfrak{M}^*,\partial^*)$ of the complex
\begin{displaymath}
\xymatrix{ \mathfrak{M}^*: & 0 \ar[r] & \ell^2\Delta_0 \ar[r]^{\partial^0} & \ell^2_{alt}\Delta_1 \ar[r]^{\partial^1} \ar[r] & \cdots }
\end{displaymath}
where each $\ell^2_{alt}\Delta_i$ is a finite-$LG$-dimensional Hilbert module, isomorphic as such that $\oplus_i^{fin}L^2(K_i\setminus G)$ with the $K_i$ compact open.

The $L^2$-homology can be computed as $H_n^{(2)}(\Delta;G)\simeq H_n(\mathfrak{N}_*,\partial_*)$, using the complex
\begin{displaymath}
\xymatrix{ \mathfrak{N}_* : & \cdots \ar[r]^{\partial_1} & \ell^2_{alt}\Delta_1 \ar[r]^{\partial_0} & \ell^2_{alt}\Delta_0 \ar[r] & 0 },
\end{displaymath}
where the boundary maps $\partial_i := (\partial^i)^*$ are the adjoints of the coboundary maps.

Furthermore, the reduced and unreduced $L^2$-cohomology have the same $LG$-dimension, coinciding also with the dimension of the $L^2$-homology:
\begin{eqnarray}
\beta_n^{(2)}(\Delta ; G,\mu) = \beta^n_{(2)}(\Delta ;G,\mu) & = & \dim_{\psi} \underline{H}^n(\mathfrak{M}^*,\partial^*) = \dim_{\psi} \ker \partial^n \left/ \overline{\operatorname{im} \partial^{n-1}} \right. \nonumber \\
 & = & \dim_{\psi} \ell^2_{alt} \Delta_n \ominus ( \ell^2_{\circ}\Delta_n \oplus \ell^2_{\star} \Delta_n ) <\infty. \nonumber
\end{eqnarray}
\end{proposition}

\begin{proof}
The first equality is clear. The equality of dimension for reduced and non-reduced cohomology of the complex $\mathfrak{M}^*$ follows by additivity and Lemma \ref{lma:dimbyTr}.

Denote for each $n\geq 0$ by $L_n$ a fundamental domain for the action of $G$ on $\Delta_n$. Then there are isomorphisms of $LG$-modules fitting into a morphism of complexes
\begin{displaymath}
\xymatrix{  0 \ar[r] & \mathcal{F}_{alt}(\Delta_0, L^2G)^G \ar[d]_{\iota_0}^{\sim} \ar[r]^{d^0} & \mathcal{F}_{alt}(\Delta_1, L^2G)^G \ar[d]_{\iota_1}^{\sim} \ar[r]^<<<<<{d^1} & \cdots \\ 0 \ar[r] & \ell^2\Delta_0 \ar[r]^{\partial^0} & \ell^2_{alt}\Delta_1 \ar[r]^{\partial^1} \ar[r] & \cdots }
\end{displaymath}
where the $\iota_n$ are essentially evaluation on the fundamental domains. Specifically, if $\delta = g.\delta_0$ for a $\delta_0\in L_n$ one puts $(\iota_n\xi)(\delta) = \xi(\delta_0)(G_{\delta_0}g^{-1})$.

We leave out the straight-forward details here. See Section \ref{sec:locfingraph} for more details in low degree.
\end{proof}

To compute the dimensions of subspaces of $\ell^2_{alt}\Delta_n$, we have the following lemma.

\begin{lemma} \label{lma:dimbyTrtotdisc} \todo{lma:dimbyTrtotdisc}
Let $G$ be a totally disconnected {\lcsu} group and fix a Haar measure $\mu$. Let $K_i, i=1,\dots,n$ be compact open subgroups of $G$ and denote for all $i$ the projections $p_i:=\frac{1}{\mu(K_i)}\lambda(\bbb_{K_i})$, where $\lambda$ is the left-regular representation of $G$. Then there are isomorphisms of right-$LG$-modules $p_i.L^2G \simeq L^2(K_i\backslash G)$ and for submodule $E\leq \oplus_i L^2(K_i\backslash G)$ we have
\begin{equation}
\dim_{(LG,\psi)} E = \sum_{i=1}^n \frac{1}{\mu(K_i)} \langle P_{\overline{E}}.\bar{\bbb}_{K_i},\bar{\bbb}_{K_i} \rangle_{L^2(K_i\backslash G)}, \nonumber
\end{equation}
where $P_{\overline{E}}$ is the orthogonal projection onto the closure of $E$ and the $\bar{\bbb}_{K_i}$ are understood as the indicator on the points $K_i\in K_i\backslash G$ in the discrete countable set $K_i\backslash G$, and $L^2(K_i\backslash G)$ is equiped with the inner product induced by counting measure on $K_i\backslash G$.
\end{lemma}

\begin{proof}
This follows from rank-density considerations and the following observation: namely, let $C_j, j\in \mathbb{N}$ be a decreasing neighbourhood basis at the identity in $G$ with each $C_j$ a compact open subgroup. Then the tracial weight $\psi$ on $LG$ is given by
\begin{equation}
\psi(T^*T) = \lim_{j\rightarrow \infty}\frac{1}{\mu(C_j)} \lVert T.\bbb_{C_j}\rVert_2^2. \nonumber
\end{equation}

Then we can take $C_j\subseteq \cap_i K_i$
\end{proof}

\begin{theorem} \label{thm:HdRdecompcocompact} \todo{thm:HdRdecompcocompact}
Let $G$ be a locally compact, $2$nd countable unimodular group acting continuously as degree zero automorphisms on a locally finite, countable simplicial complex $\Delta$. Assume that $\Delta$ is contractible, that the stabilizer of any given simplex (ordered or unodered) is compact in $G$, and that the action is cocompact (i.e. cofinite).

Denoting by $\Delta_n$ the $n$-skeleton of $\Delta$, by $L_n$ a fundamental domain for the action of $G$ on $\Delta_n$, and by $P_n$ the orthogonal projection onto $\ell^2_{alt}(\Delta_n) \ominus (\ell^{2}_{\circ}(\Delta_n) \oplus \ell^{2}_{\star}(\Delta_{n}))$ then, for all $n\in \mathbb{N}_0$
\begin{eqnarray}
\beta^{n}_{(2)}(G,\mu) = \underline{\beta}^n_{(2)}(G,\mu) & = & \dim_{\psi} \ell^2_{alt}(\Delta_n) \ominus (\ell^{2}_{\circ}(\Delta_n) \oplus \ell^{2}_{\star}(\Delta_{n})) \nonumber \\
 & = & \sum_{s\in L_n} \frac{1}{\mu(G_s)} \langle P_n \bbb_s, \bbb_s \rangle_{\ell^2\Delta_n} . \nonumber
\end{eqnarray}
\end{theorem}

\begin{proof}
This follows directly from the previous proposition and lemma, and Theorem \ref{thm:homotopyinvariance}.
\end{proof}

\begin{corollary} \label{cor:HdRsoft} \todo{cor:HdRsoft}
Keep notation and assumptions as in the theorem. Denote further by $\tilde{L}_n$ a fundamental domain for the action of $G$ on $\Delta_n/S_{n+1}$, i.e. the set of unordered n-simplices.

Suppose that $\Delta$ is an infinite tree. Then $\beta^n_{(2)}(G,\mu)=0$ for $n\neq 1$ and
\begin{eqnarray}
\beta^{1}_{(2)}(G,\mu) & = & \dim_{\psi} \ell^{2}_{alt}(\Delta_1) - \dim_{\psi} \ell^{2}(\Delta_0) \nonumber \\
 & = & \left( \sum_{e\in \tilde{L}_1} \frac{1}{\mu(G_e)}\right) - 1. \nonumber
\end{eqnarray}
More generally, if $\Delta$ has dimension $n\geq 1$ then $\beta_{(2)}^{m}(G,\mu) = 0$ for $m>n$ and
\begin{eqnarray}
\beta^n_{(2)}(G,\mu) & \geq & \dim_{\psi} \ell^2_{alt}(\Delta_n) - \dim_{\psi} \ell^2_{alt}(\Delta_{n-1}) \nonumber \\
 & = & \left( \sum_{s\in \tilde{L}_n} \frac{1}{\mu(G_s)} \right) -\left( \sum_{s\in \tilde{L}_{n-1}} \frac{1}{\mu(G_s)} \right). \nonumber
\end{eqnarray}
\end{corollary}

\begin{flushright}
\qedsymbol
\end{flushright}

The final theorem of this section shows in particular that our definition corresponds to that of \cite{ChGr86} for countable discrete groups.

\begin{theorem} \label{thm:totdiscactioncompute} \todo{thm:totdiscactioncompute}
Let $\Delta$ be a countable simplicial complex with a continuous action of the totally disconnected {\lcsu} group $G$ and suppose that the stabilizers are all compact. Let $\Delta^{(k)}, k\in \mathbb{N}$ be an (increasing) exhaustion of $\Delta$ by locally finite, $G$-invariant subcomplexes on which the action is cofinite. Such an exhaustion always exists.

For any fixed $n\geq 0$, denote for $k\leq l$ by $R_{k,l}$ the restriction of the projection onto $\ell^2_{alt} \Delta_n^{(k)} \ominus ( \ell^2_{\circ}\Delta_n^{(k)} \oplus \ell^2_{\star} \Delta_n^{(k)})$ to $\ell^2_{alt} \Delta_n^{(l)} \ominus ( \ell^2_{\circ}\Delta_n^{(l)} \oplus \ell^2_{\star} \Delta_n^{(l)})$.

Then:
\begin{eqnarray}
\beta_n^{(2)}(\Delta ;G, \mu) & = & \sup_{k\in \mathbb{N}} \inf \left\{ \dim_{\psi} \operatorname{im} \left( H_n^{(2)}(\Delta^{(k)};G) \rightarrow H_n^{(2)}(\Delta^{(l)};G) \right) \mid l\geq k \right\} \nonumber \\
 & = & \sup_{k\in \mathbb{N}} \inf \left\{ \dim_{\psi}Z_n^{(2)}(\Delta^{(k)}) - \dim_{\psi} \bar{B}_n^{(2)}(\Delta^{(l)}) \cap \ell^2_{alt}(\Delta^{(k)}) \mid l\geq k \right\} \nonumber \\ 
 & = & \sup_{k\in \mathbb{N}} \inf \left\{\dim_{\psi} \overline{ \operatorname{im} R_{k,l} } \mid l\geq k\right\} . \nonumber
\end{eqnarray}
and this also coincides with $\underline{\beta}_{(2)}^n(\Delta;G,\mu) = \beta^n_{(2)}(\Delta ;G, \mu)$.
\end{theorem}

\begin{proof}
Clearly we can exhaust $\Delta$ by subcomplexes on which the action is cofinite. Then all we have to see is that a stabilizer of any $n$-simplex must act with finite orbits on the set of $n+1$ dimensional simplices containing this as a face. But this is clear since stabilizers are compact by hypothesis and open by continuity of the action.

Now consider the complexes
\begin{displaymath}
\xymatrix{ \mathfrak{N}_* : & \cdots \ar[r] & \underline{C}_G\mathcal{F}_{c,alt}(\Delta_n,LG) \ar[r] & \cdots \ar[r]^>>>>>{d_1} & \underline{C}_G\mathcal{F}_{c,alt}(\Delta_1,LG) \ar[r]^{d_0} & \underline{C}_G \mathcal{F}_{c,alt}(\Delta_0,LG) \ar[r] & 0 \\ \mathfrak{N}_*^k : & \cdots \ar[r] & \underline{C}_G\mathcal{F}_{c,alt}(\Delta_n^{(k)},LG) \ar[u]^{\iota} \ar[r] & \cdots \ar[r]^>>>>>{d_1} & \underline{C}_G\mathcal{F}_{c,alt}(\Delta_1^{(k)},LG) \ar[u]^{\iota} \ar[r]^{d_0} & \underline{C}_G \mathcal{F}_{c,alt}(\Delta_0^{(k)},LG) \ar[u]^{\iota} \ar[r] & 0 }
\end{displaymath}

Then $H_n(\mathfrak{N}_*) \simeq lim_{\rightarrow_k} H_n(\mathfrak{N}_*^k)$, and since the modules $H_n(\mathfrak{N}_*^k)$ all have finite dimension by Proposition \ref{prop:elltwoactioncompute} whence the claim follows by the inductive limits formula, see Theorem \ref{thm:dimensionsummary}.

The coincidence of homological and cohomological $L^2$-Betti numbers was already noted in Theorem \ref{thm:homotopyinvariance}.
\end{proof}

         % OK

\subsection{Groups acting on locally finite graphs} \label{sec:locfingraph} \todo{sec:locfingraph}

In this section we show explicitly how to take an action of a group $G$ on a graph and construct a partial injective resolution of $L^2G$ from this. The end result is Theorem \ref{thm:HdRlocfingraph} which shows that the first $L^2$-Betti number of a unimodular, vertex-transitive closed subgroup of automorphisms of a locally finite graph coincides with the first $L^2$-Betti number of the graph as defined by Gaboriau \cite[Definition 2.10]{Ga05}.

This is written to be quite independent of Section \ref{sec:totdiscactions}. One can also read it as a more detailed proof of a slightly more general version Theorem \ref{thm:HdRdecompcocompact}, in degree one. This version says that one can compute the $L^2$-Betti numbers up to degree $n+1$ using an $n$-connected simplicial complex.

%We leave out the straight-forward proofs of the auxilliary results (the construction is well-known in the case of Cayley graphs, see Remark \ref{rmk:HdRCayley}). %but do go through the construction in an explicit, elementary way so that we can give a detailed and self-contained proof of Theorem \ref{thm:HdRlocfingraph}.

By a graph we mean a one-dimensional simplicial complex, that is, a graph $\mathcal{G}$ convists of a set of vertices $\mathcal{V}$ and a set of edges $\mathcal{E} \subseteq \mathcal{V}\times \mathcal{V}$ such that $\mathcal{E}$ is disjoint from the diagonal and symmetric. For an edge $e=(v,u)\in \mathcal{E}$ we denote the opposite edge $\bar{e}=(u,v)$. For vertices $u,v\in \mathcal{V}$ we write $v\sim u$ if $(v,u)\in \mathcal{E}$ (equivalently $(u,v)\in \mathcal{E}$) and say that $v$ and $u$ are neighbours in this case. When we consider e.g.~a path in the graph $\mathcal{G}$ we will say, slightly abusing language, that the path is the sum of its edges.

%\begin{notation} \label{not:HdR}
Throughout this section (unless explicitly stated otherwise) let $G$ be a {\lcsu} group acting continuously as a vertex-transitive group of automorphisms on a countable, connected, locally finite graph $\mathcal{G}=(\mathcal{V},\mathcal{E})$, and assume that the stabilizer of any given simplex is compact in $G$. We fix once and for all a \emph{basepoint} $\rho\in \mathcal{V}$. 
%\end{notation}

Let $G_{\rho}$ be the stabilizer of $\rho$. Since the action is continuous this is a compact, open subgroup of $G$. We fix the Haar measure $\mu$ on $G$ such that $\mu(G_{\rho})=1$. We freely identify $\mathcal{V}$ with $G_{\rho}\backslash G$ as well as with $G/G_{\rho}$, the former by $g.\rho \leftrightarrow G_{\rho}g^{-1}$ so that the action by $G$ is inverse right multiplication in this case. We fix sections $s_l$ respectively $s_r$ of the canonical projections $G\rightarrow G_{\rho}\backslash G$ resp. $G/G_{\rho}$.

%All graphs considered are undirected, or rather, for every edge $e\in \mathcal{E}$, the reverse edge $\overline{e}$ is also in $\mathcal{E}$. For $v,u\in \mathcal{V}$ we write $v\sim u$ if $(v,u)\in \mathcal{E}$.

\begin{remark} \label{rmk:HdRCayley}
If $G$ is discrete and $G_{\rho}=\{\bbb\}$, this implies that $\mathcal{G}$ is a Cayley graph for $G$ with (finite) symmetric generating set $S=\{ g\in G \mid g.\rho \sim \rho \}$.

Note that by \cite[Proposition 2.9 / Definition 2.10]{Ga05}, the definition of $\beta^1(\mathcal{G})$ in \cite{Ga05} satisfies $\beta^1(\mathcal{G})=\beta^1_{(2)}(\Gamma)$ when $\mathcal{G}$ is a Cayley graph of the finitely generated group $\Gamma$. Theorem \ref{thm:HdRlocfingraph} is then the natural extension of this to locally compact groups.
\end{remark}

In order to state Theorem \ref{thm:HdRlocfingraph}, denote
\begin{equation}
\ell^2_{alt}(\mathcal{E})=\{f\in \ell^2\mathcal{E} \mid f(e)=-f(\overline{e})\} \nonumber
\end{equation}
and consider the coboundary map $\partial \colon \ell^2\mathcal{V} \rightarrow \ell^2_{alt}\mathcal{E}$ given by $(\partial \xi)(u,v) = \xi(v)-\xi(u)$. Then we denote $\ell^2_{\star}(\mathcal{E}) := \overline{\partial(\ell^2\mathcal{V})}^{\lVert \cdot \rVert_2}$. 

Also we consider the "cycle space" $\ell^2_{\circ}(\mathcal{E})$, i.e. the closed span of alternating charateristic functions of cycles,
\begin{equation}
\ell^2_{\circ}(\mathcal{E}) := \overline{\span} \left\{ \sum_{i=0}^{n} (\delta_{(v_i,v_{i+1})}-\delta_{(v_{i+1},v_i)}) \mid \forall i: (v_i,v_{i+1})\in \mathcal{E}, \; v_{n+1}=v_{0} \right\}. \nonumber
\end{equation}

Note that this coincides with notation in Section \ref{sec:totdiscactions} for the simplicial complex $\Delta=\mathcal{G}$.

\begin{theorem} \label{thm:HdRlocfingraph} \todo{thm:HdRlocfingraph}
Let $G$ be a {\lcsu} group acting continuously as a vertex-transitive group of automorphisms on a countable, connected, locally finite graph $\mathcal{G}=(\mathcal{V},\mathcal{E})$, and assume that the stabilizer of any given simplex is compact in $G$. Let $\rho\in \mathcal{V}$ and fix the Haar measure $\mu$ on $G$ such that $\mu(G_{\rho})=1$.

Denote by $P$ the orthogonal projection onto $(\ell^2_{\star}(\mathcal{E})\oplus \ell^2_{\circ}(\mathcal{E}))^{\perp}\subseteq \ell^2_{alt}(\mathcal{E})$ and for $e\in \mathcal{E}$, $\tilde{\delta}_e := \frac{1}{2}(\delta_e - \delta_{\overline{e}})$. Then
\begin{eqnarray}
\beta^1_{(2)}(G,\mu) & = & \frac{1}{\surd 2} \sum_{v\sim \rho} \langle P\tilde{\delta}_{(\rho,v)}, \tilde{\delta}_{(\rho,v)}\rangle_{\ell^2\mathcal{E}} \nonumber \\
 & = & \sum_{v\sim \rho} \langle P\delta_{(\rho,v)},\delta_{(\rho,v)}\rangle_{\ell^2\mathcal{E}}. \nonumber
\end{eqnarray}
\end{theorem}

The proof is given below, after all the auxilliary results.

In the interest of generality we will presently consider coefficients in a general quasi-complete topological $G$-$LG$-module $E$ instead of $L^2G$.

Consider modules
\begin{eqnarray}
F^0 & := & \{ f:\mathcal{V}\rightarrow E \}, \nonumber \\
F^1 & := & \{ f:\mathcal{E}\rightarrow E \mid \forall e\in \mathcal{E}: f(e)=-f(\overline{e}) \}, \nonumber
\end{eqnarray}
both with the topology of pointwise convergence. These are left $G$-modules by $(g.f)(\cdot) = g.f(g^{-1}\cdot)$ and right-$LG$-modules by post-multiplication.

Fix once and for all an unoriented spanning tree $\mathcal{T}$ of $\mathcal{G}$. Recall that adding an (oriented) edge $(v_0,v_1)$ to $\mathcal{T}$, the resulting graph contains exactly one oriented cycle $c_{(v_0,v_1)}$, the \emph{fundamental cycle} of $(v_0,v_1)$. Note that this is a simple cycle, i.e.~it does not intersect itself, and that any cycle is a sum of simple cycles, in the obvious sense.

We denote by $\mathcal{C^{(T)}}$ the set of (oriented) fundamental cycles, endowed with discrete topology. We may leave out the superscript if this causes no confusion. We then set

\begin{equation}
F^{2} := \{ f:\mathcal{C} \rightarrow E \mid \forall c\in \mathcal{C}: f(c) = -f(\overline{c}) \}, \nonumber
\end{equation}
the bar as usual denoting reversal of direction. We endow also $F^2$ with the topology of pointwise convergence.

\begin{remark}
It is well-known, and straight-forward to check, that the set of alternating characteristic functions on undirected fundamental cycles, i.e. functions of the form $\bbb_{c}-\bbb_{\bar{c}}$ where $c$ is an oriented fundamental cycle, forms a Banach space basis of the cycle space $\ell^2_{\circ}(\mathcal{E})$.
\end{remark}

By the preceding remark, we can identify $F^2$ with the set of alternating functions $f:\{ all \; cycles\} \rightarrow L^2G$ such that if $c=\sum_i c_i$ is the unique decomposition of $c$ into fundamental cycles, possibly with repetitions, then $f(c)=\sum_i f(c_i)$ and still $f(c)=-f(\bar{c})$. In particular this identification gives the $G$-action as 
\begin{equation}
(g.f)(c) = g.f(g^{-1}.c) := g.\left( \sum_i f(g^{-1}.c_i) \right). \nonumber
\end{equation}

\begin{proposition} \label{prop:HdRFonerelinjective} \todo{prop:HdRFonerelinjective}
The $F^i, i=0,1,2$, are injective.
\end{proposition}

See also \cite[Chapter X, Section 2.4]{BorelWallachBook}.

%\todo{proof commented out}
%\begin{comment} %This is proof that F^1 is injective.

\begin{proof}
We prove the claim for $F^1$, the two others being entirely analogous. Consider a diagram
\begin{displaymath}
\xymatrix{  & F^1 \\ B \ar@{-->}[ur]^{\exists ?w} & A \ar[l]_{u} \ar[u]^v & 0\ar[l] }
\end{displaymath} 
where $u$ is strengthened injective with left-inverse $s:B\rightarrow A$. We have to show the existence of a $G$-map $w:B\rightarrow F^1$ making the diagram commute. To this end we define for $b\in B$.
\begin{equation}
(wb)(v_0,v_1) = \frac{1}{2}\sum_{i=0}^1\int_{G_{\rho}} v\left( s_r(v_i)k.s(k^{-1}s_r(v_i)^{-1}.b)\right)(v_0,v_1)\mathrm{d}\mu(k), \nonumber
\end{equation}

For $a\in A$ we then compute
\begin{eqnarray}
(w\circ u)(a)(v_0,v_1) & = & \frac{1}{2}\sum_{i=0}^1\int_{G_{\rho}} v\left( s_r(v_i)k.s(k^{-1}s_r(v_i)^{-1}.u(a))\right)(v_0,v_1)\mathrm{d}\mu(k) \nonumber \\
 & = & \frac{1}{2}\sum_{i=0}^1\int_{G_{\rho}} v\left( s_r(v_i)k.s(u(k^{-1}s_r(v_i)^{-1}.a))\right)(v_0,v_1)\mathrm{d}\mu(k) \nonumber \\
 & = & \frac{1}{2}\sum_{i=0}^1\int_{G_{\rho}} v\left( s_r(v_i)k.k^{-1}s_r(v_i)^{-1}.a)\right)(v_0,v_1)\mathrm{d}\mu(k) \nonumber \\
 & = & \frac{1}{2}\sum_{i=0}^1\int_{G_{\rho}} v(a)(v_0,v_1) \mathrm{d}\mu(k) = v(a)(v_0,v_1). \nonumber
\end{eqnarray}
Thus the diagram commutes. Clearly $w$ is linear and, noting that $g^{-1}s_r(v_0) = s_r(g^{-1}.v_0)k'$ for some $k'$ depending on $g$ and $v_0$, we get
\begin{eqnarray}
w(g.b)(v_0,v_1) & = & \frac{1}{2}\sum_{i=0}^1\int_{G_{\rho}} v\left( s_r(v_i)k.s(k^{-1}s_r(v_i)^{-1}g.b)\right)(v_0,v_1)\mathrm{d}\mu(k) \nonumber \\
 & = & \frac{1}{2}\sum_{i=0}^1\int_{G_{\rho}} v\left( gg^{-1}s_r(v_i)k.s(k^{-1}(g^{-1}s_r(v_i))^{-1}.b)\right)(v_0,v_1)\mathrm{d}\mu(k) \nonumber \\
 & = & \frac{1}{2}\sum_{i=0}^1\int_{G_{\rho}} v\left( gs_r(g^{-1}v_i)k'k.s((k'k)^{-1}s_r(g^{-1}.v_i)^{-1}.b)\right)(v_0,v_1)\mathrm{d}\mu(k) \nonumber \\
 & = & \frac{1}{2}\sum_{i=0}^1\int_{G_{\rho}} v\left( gs_r(g^{-1}v_i)k.s(k^{-1}s_r(g^{-1}.v_i)^{-1}.b)\right)(v_0,v_1)\mathrm{d}\mu(k) \nonumber \\
 & = & \frac{1}{2}\sum_{i=0}^1\int_{G_{\rho}} g.v\left( s_r(g^{-1}v_i)k.s(k^{-1}s_r(g^{-1}.v_i)^{-1}.b)\right)(g^{-1}.v_0,g^{-1}.v_1)\mathrm{d}\mu(k) \nonumber \\
 & = & g.(wb)(g^{-1}.v_0,g^{-1}.v_1). \nonumber
\end{eqnarray}
Thus $w$ intertwines the $G$ actions. This proves the claim for $i=1$.
\end{proof}

Define $\epsilon_{\mathcal{G}}:E\rightarrow F^0$ and the coboundary maps $d^{i}_{\mathcal{G}}:F^i\rightarrow F^{i+1}, i=0,1$, by
\begin{eqnarray}
(\epsilon_{\mathcal{G}}\eta)(v) = \eta, \quad v\in \mathcal{V}, \nonumber \\
(d^{0}_{\mathcal{G}}f)(v_0,v_1) = f(v_1)-f(v_0), \quad (v_0,v_1)\in \mathcal{E} \nonumber
\end{eqnarray}
and
\begin{equation}
(d^1_{\mathcal{G}}f)(c) = f(v_0,v_1) + f(v_1,v_2) + \cdots + f(v_{n-1},v_n) + f(v_n,v_0), \nonumber
\end{equation}
where $c$ is an oriented fundamental cycle,
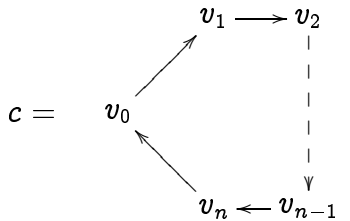
\begin{displaymath}
c = \quad \xygraph{ []{v_0} : [ur]{v_1} : [r]{v_2} : @{-->} [dd]{v_{n-1}} : [l]{v_n} : "v_0" }.
\end{displaymath} 

Thus we have a (truncated) complex
\begin{displaymath}
\xymatrix{ 0 \ar[r] & E \ar[r]^{\epsilon} & F^0 \ar[r]^{d_{\mathcal{G}}^0} & F^1 \ar[r]^{d_{\mathcal{G}}^1} & F^2 }
\end{displaymath}
with the $F^i$ injective. Next we define a (truncated) contraction as follows. Define $s^{0}_{\mathcal{G}}:F^0\rightarrow E$ by
\begin{equation}
s^{0}_{\mathcal{G}}f = f(\rho), \nonumber
\end{equation}
$s^{1}_{\mathcal{G}}:F^1\rightarrow F^0$ by
\begin{equation}
(s^{1}_{\mathcal{G}})(v) = \left\{ \begin{array}{ll} 0  & , v=\rho \\ f(\rho, v_1^{(v)}) + f(v_1^{(v)}, v_2^{(v)}) + \cdots + f(v_{t(v)-1}^{(v)}, v) & , v\neq \rho \end{array} \right. , \nonumber
\end{equation}
where $(\rho, v_1^{(v)}, \dots ,v)$ is the unique path in $\mathcal{T}$ from $\rho$ to $v$. And finally, we define $s^{2}_{\mathcal{G}}:F^{2} \rightarrow F^{1}$ by
\begin{equation}
(s^{2}_{\mathcal{G}}f)(e) = \left\{ \begin{array}{ll} 0 & , e\in \mathcal{T} \\ f(c_{e}) & ,e\notin \mathcal{T} \end{array} \right. . \nonumber
\end{equation}

\begin{proposition}
For the maps defined above we have $s^{0}_{\mathcal{G}} \circ \epsilon_{\mathcal{G}} = \bbb_E$ and
\begin{eqnarray}
\epsilon_{\mathcal{G}} \circ s^{0}_{\mathcal{G}} + s^{1}_{\mathcal{G}} \circ d^{0}_{\mathcal{G}} = \bbb_{F^{0}}, \nonumber \\
d^{0}_{\mathcal{G}}\circ s^{1}_{\mathcal{G}} + s^{2}_{\mathcal{G}}\circ d^{1}_{\mathcal{G}} = \bbb_{F^1}. \nonumber
\end{eqnarray}
\end{proposition}

%\todo{proof commented out}
%\begin{comment}
\begin{proof}
Let $f\in F^0$ and $v\in \mathcal{V}$. For $v=\rho$ we get
\begin{equation}
(\epsilon_{\mathcal{G}} \circ s^{0}_{\mathcal{G}} + s^{1}_{\mathcal{G}} \circ d^{0}_{\mathcal{G}})(f)(\rho) = (\epsilon_{\mathcal{G}} \circ s^{0}_{\mathcal{G}})(f)(\rho) + 0 = f(\rho). \nonumber
\end{equation}
For $v\neq \rho$, with $(\rho, v_1^{(v)}, \dots , v)$ the path in $\mathcal{T}$ from $\rho$ to $v$ as above, we get
\begin{eqnarray}
(\epsilon_{\mathcal{G}} \circ s^{0}_{\mathcal{G}} + s^{1}_{\mathcal{G}} \circ d^{0}_{\mathcal{G}})(f)(v) & = & (\epsilon_{\mathcal{G}} \circ s^{0}_{\mathcal{G}})(f)(v) + (s^{1}_{\mathcal{G}} \circ d^{0}_{\mathcal{G}})(f)(v) \nonumber \\
 & = & f(\rho) + (d^{0}_{\mathcal{G}}f)(\rho, v_1^{(v)}) + \cdots + (d^{0}_{\mathcal{G}}f)(v_{t(v)-1}^{(v)},v) = f(v). \nonumber
\end{eqnarray}
%\end{proof}
%\end{comment}

%\todo{proof commented out}
%\begin{comment}
%\begin{proof}
Next, we have if $(v_0,v_1)\in \mathcal{T}$
\begin{eqnarray}
(d^{0}_{\mathcal{G}}\circ s^{1}_{\mathcal{G}} + s^{2}_{\mathcal{G}}\circ d^{1}_{\mathcal{G}})(f)(v_0,v_1) & = & (s^{1}_{\mathcal{G}}f)(v_1) - (s^{1}_{\mathcal{G}}f)(v_0) + 0 \nonumber \\
 & = & \left\{ \begin{array}{ll} f(v_0,v_1) & ,\textrm{ in case } \; v_0 = v^{(v_1)}_{t(v_1)-1} \\ -f(v_1,v_0) & ,\textrm{ in case not, i.e. } \; v_1 = v^{(v_0)}_{t(v_0)-1} \end{array} \right. \nonumber \\
 & = & f(v_0,v_1). \nonumber
\end{eqnarray}

For $(v_0,v_1)\notin \mathcal{T}$ we get that
\begin{eqnarray}
(d^{0}_{\mathcal{G}}\circ s^{1}_{\mathcal{G}} + s^{2}_{\mathcal{G}}\circ d^{1}_{\mathcal{G}})(f)(v_0,v_1) & = & (s^{1}_{\mathcal{G}}f)(v_1) - (s^{0}_{\mathcal{G}}f)(v_0) + (d^{1}_{\mathcal{G}}f)(c_{(v_0,v_1)}) \nonumber \\
 & = & f(\rho , v_1^{(v_1)}) + \cdots + f(v_{t(v_1)-1}^{(v_1)},v_1) - \nonumber \\
 & & - f(\rho, v_1^{(v_0)}) - \cdots - f(v_{t(v_0)-1},v_0) + \nonumber \\
 & & + f(v_0,v_1) + f(v_1,v_2) + \cdots + f(v_n,v_0) \nonumber \\
 & = & f(v_0,v_1). \nonumber
\end{eqnarray}
Here the final equality is by the picture below
\begin{displaymath}
\xygraph{ []{\rho} : [u]{v_1^{(v_1)}=v_1^{(v_0)}} _{\in \mathcal{T}} : @{-->} [ul]{\bullet} _{\in \mathcal{T}} : [l]{v_j} _{\in \mathcal{T}} : @{-->} [dl]{v_n} ^{\in \mathcal{T}} : [ul]{v_0} ^{\in \mathcal{T}} : [ur]{v_1} ^{\notin \mathcal{T}} : [r]{v_2} ^{\in \mathcal{T}} : @{-->} "v_j" ^{\in \mathcal{T}} [l]{c_{(v_0,v_1)}} }
\end{displaymath}
from which one notes that for instance $v_2=v_{t(v_1)-1}^{(v_1)}$ and $v_n=v_{t(v_0)-1}^{(v_0)}$.
\end{proof}
%\end{comment}

\begin{remark}
Note that all maps $d^i_{\mathcal{G}},s^i_{\mathcal{G}}, i=0,1,2$ are continuous.
\end{remark}

It follows now automatically that the complex $0\rightarrow E \xrightarrow{\epsilon_{\mathcal{G}}} F^0 \xrightarrow{d^{0}_{\mathcal{G}}} F^1 \xrightarrow{d^{1}_{\mathcal{G}}} F^2$ is exact, and that the coboundary maps are strengthened. (Compare \cite[p. 332]{Guichardetbook}.)

By standard homological algebra arguments, the simply connectedness of this complex implies that
\begin{equation}
H^1(G,E) \simeq H^1((F^{*})^{G}). \nonumber
\end{equation}

From hereon in we again focus on the case where $E=L^2G$ in order to relate the right-hand side of the previous equation to the $\ell^2$-spaces appearing in Theorem \ref{thm:HdRlocfingraph}. First note that
\begin{equation}
h: (F^{0})^{G} \xrightarrow{\sim} \{ f\in L^2G \mid \forall g\in G: f\vert_{G_{\rho}g} = \mathrm{const.} \, \} \simeq L^2(G_{\rho}\backslash G) \nonumber
\end{equation}
where the isomorphism $h$ is evaluation in $\rho$. By the identification of $G_{\rho}\backslash G$ with $\mathcal{V}$ (recall: $G_{\rho}g^{-1} \leftrightarrow g.\rho$) we identify $(F^0)^G$ with $\ell^2\mathcal{V}$ where now the right-action of $G$ on $L^2(G_{\rho}\backslash G)$ corresponds to the usual left-action on $\ell^2\mathcal{V}$. Hence there is a correspondence between closed $LG$-submodules of $L^2(G_{\rho}\backslash G)$ and closed invariant subspaces of $\ell^2\mathcal{V}$.

Fix a fundamental domain $L_1$ for the action of $G$ on $\mathcal{E}$. Note that $L_1$ consists of \emph{oriented} edges. We can and will take the edges in $L_1$ to have the form $e=(\rho,v)$ and we denote by $G_e$ (or $G_{(\rho,v)}$) the stabilizer of (such an edge) $e$. Then again we get an isomorphism by evaluation at the edges in $L_1$:
\begin{equation}
(F^1)^{G} \xrightarrow{\sim} F^1_{alt} \subseteq \bigoplus_{e\in L_1} \{ f\in L^2G \mid \forall g\in G: f\vert_{G_eg} = \mathrm{const.} \,\} = \bigoplus_{e\in L_1} L^2(G_e\backslash G), \nonumber
\end{equation}
where $F^1_{alt}$ is the submodule consisting of $\xi=(\xi_e)_{e\in L_1}$ such that $\xi_{e}(g) = -\xi_{e'}(g')$ where $g.e=\overline{g'.e'}$.

In particular, if $G$ does not flip any edges (discrete case: no generators of order $2$), $L_1$ splits as $L_1=\{e_1,\dots  ,e_{\sharp L_1/2}\} \dot{\cup} \{\overline{e_1}, \dots ,\overline{e_{\sharp L_1/2}}\}$ and the summands corresponding to the latter term may then be ignored. One should keep in mind though that this will change calculations by a factor $2$ at some point.

Note also that each $G_e$ is a compact open subgroup.

Denote for $e\in L_1$ by $\iota_e$ the map on $\{ f\in L^2G\mid \forall g\in G: f\vert_{G_eg}=\mathrm{const.} \, \} = L^2(G_e\backslash G)$ into $\ell^2_{alt}(\mathcal{E})$ given by $(\iota_ef)(g.e) = f(G_eg^{-1})$. Then the map $\iota:=\oplus_{e\in L_1}i_e$ is an ismorphism
\begin{equation}
\iota: F^1_{alt} \xrightarrow{\sim} \ell^2_{alt}(\mathcal{E}). \nonumber
\end{equation}

\begin{lemma} \label{lma:HdRdiag}
With notation as above, the diagram
\begin{displaymath}
\xymatrix{ (F^0)^G \ar[d]^{\sim}_{h\circ ev} \ar[r]^{d^{0}_{\mathcal{G}}} & (F^1)^G \ar[d]^{\sim}_{\iota\circ ev} \\ \ell^2\mathcal{V} \ar[r]^{\partial} & \ell^2_{alt}(\mathcal{E})  }
\end{displaymath}
commutes.

Further, under the isomorphism $\iota\circ ev$, the kernel of the coboundary map $d^1_{\mathcal{G}}\vert_{(F^1)^G}$ is exactly $\ell^2_{\circ}(\mathcal{E})^{\perp} \subseteq \ell^2_{alt}(\mathcal{E})$.
\end{lemma}

%\todo{proof commented out}
%\begin{comment}
\begin{proof}
We first prove commutativity of the diagram. We have for $e=(\rho,v)\in L_1$,
\begin{equation}
(i\circ ev \circ d^0_{\mathcal{G}})(f)(g.e) = (d^0_{\mathcal{G}}f)(\rho ,v)(g^{-1}) = f(v)(g^{-1})-f(\rho)(g^{-1}). \nonumber
\end{equation}
On the other hand, letting $v=h.\rho$ we get
\begin{eqnarray}
(\partial\circ h\circ ev)(f)(g.e) & = & (h\circ ev)(f)(gh.\rho) - (h\circ ev)(f)(g.\rho) \nonumber \\
 & = & f(\rho)(h^{-1}g^{-1}) - f(\rho)(g^{-1}) \nonumber \\
 & = & (h.f(\rho))(g^{-1}) -f(\rho)(g^{-1}) = f(h.\rho)(g^{-1}) - f(\rho)(g^{-1}). \nonumber
\end{eqnarray}
Thus we have shown the first part of the lemma. Let $c=(e_1,\dots ,e_n)$ be a fundamental cycle and $f\in (F^1)^G$. Then
\begin{equation}
(d^1_{\mathcal{G}}f)(c)(g) = \sum_{i=1}^n f(e_i)(g) = \sum_{i=1}^n (i\circ ev)(f)(g^{-1}.e_i). \nonumber
\end{equation}
If the left-hand side vanishes, we see that $(i\circ ev)(f)$ sums to zero around every fundamental cycle, and since these span the cycle space, around every cycle. Thus $f\in \ell^2_{\circ}(\mathcal{E})^{\perp}$.

Running the equations in reverse gives the other inclusion, showing the second part.
\end{proof}
%\end{comment}

\begin{lemma}
For $e\in L_1$, the adjoint of $\iota_e$ is given by
\begin{equation}
\iota_e^{*} = \mu(G_e)\cdot \iota^{-1}\vert_{G.e}. \nonumber
\end{equation}
\end{lemma}

%\todo{"proof" commented out}
%\begin{comment}
\begin{proof}
This is trivial.
\end{proof}
%\end{comment}

\begin{proof}[Proof of Theorem \ref{thm:HdRlocfingraph}]
The equality of the right-hand sides is a direct computation, using e.g. that $\langle P\delta_e,\delta_{\overline{e}} \rangle = -\langle P\delta_e,\delta_e\rangle$ since $P\delta_e$ is an alternating function.

By Lemma \ref{lma:HdRdiag}, $H^1(G,L^2G) \simeq \ell^2_{alt}(\mathcal{E}) / (\partial(\ell^2\mathcal{V})\oplus \ell^2_{\circ}(\mathcal{E}))$ as topological vector spaces. Further, it follows by additivity and Lemma \ref{lma:diminfinitesum} that the dimension of the right-hand side is equal to that of $\ell^2_{alt}(\mathcal{E}) \ominus (\ell^2_{\star}(\mathcal{E}) \oplus \ell^2_{\circ}(\mathcal{E}))$. To see this note that $\dim_{\psi} \ell^2_{alt}(\mathcal{E}) < \infty$ since this is isomorphic, as a right-$LG$-module, to a subspace of $\oplus_{e\in L_1} \ell^2(G_e\setminus G)$, which has finite $LG$-dimension by Lemma \ref{lma:Kinvariant} (see also Lemma \ref{lma:dimbyTrtotdisc}). It follows from this that
\begin{equation}
\beta^1_{(2)}(G,\mu) = (Tr_{\sharp L_1}\otimes \psi)(\iota^{-1}P\iota). \nonumber
\end{equation}

Let $e=(\rho,v_0)\in L_1$. Then $G_{\rho}$ splits as the disjoint union $G_{\rho} = G_{\rho}(v_0) \dot{\cup} G_{\rho}(v_1) \dot{\cup} \cdots \dot{\cup} G_{\rho}(v_{n(e)})$ where $G_{\rho}(v_i).v_0=v_i$. In particular $G_{\rho}(v_0) = G_{e}$. Further, each $G_{\rho}(v_i)$ is a translate of $G_e$ whence they all have the same measure $\mu(G_{\rho}(v_i)) = \mu(G_e) = \frac{1}{n(e)+1}$.

The statement now follows directly from the observation that the restriction of $\psi$ to the corner $\lambda(\bbb_{G_e})LG\lambda(\bbb_{G_e})$ is given by $\psi (\lambda(\bbb_{G_e})x^*x\lambda(\bbb_{G_e})) = \frac{1}{\mu(G_e)^2}\cdot \lVert x.\bbb_{G_e}\rVert_2^2$. Thus
\begin{eqnarray}
\psi\left( \iota^{-1}\vert_{G.e}P\iota_e\right) & = & (n(e)+1)\cdot (\surd \mu(G_e))^2 \langle \iota^{-1}\vert_{G.e}P\iota_e . \bbb_{G_e}, \bbb_{G_e} \rangle_{L^2G} \nonumber \\
 & = & (n(e)+1) \langle P\delta_{(\rho,v_0)}, \delta_{(\rho,v_0)}\rangle_{\ell^2\mathcal{E}} \nonumber \\
 & = & \sum_{i=0}^{n(e)} \langle P\delta_{(\rho,v_i)},\delta_{(\rho,v_i)} \rangle_{\ell^2\mathcal{E}}. \nonumber
\end{eqnarray}
Summing over $e\in L_1$ finishes the proof.
\end{proof}

            % KINDA OK

\subsection{Examples from actions on buildings} \label{sec:Sp2n} \todo{sec:Sp2n}

A rich class of actions by groups on contractible simplicial complexes comes from algebraic groups acting on their associated Bruhat-Tits buildings. In \cite{DyJa02} the cohomology of algebraic groups with coefficients in unitary representations, including $L^2G$, was investigated using the associated buildings. In particular, \cite[Proposition 8.5]{DyJa02} gives the top $L^2$-Betti number of such a group $G$ given that the pair $(X,G)$, with $X$ the building, is in their class $\mathscr{B}t$. This amounts to a condition on the size of the residue field, depending on the dimension of $X$. (In the terminology of \cite{DyJa02}, the results give the $L^2$-Betti numbers of the building; this amounts to the same thing, by Theorem \ref{thm:HdRdecompcocompact}.)

In this section we investigate a special case, obtaining a non-zero lower bound on the top $L^2$-Betti number $\beta^n_{(2)}(Sp_{2n}(K),\mu)$ for $K$ a non-archimedean local field, under slightly weaker assumptions, and using a much more elementary argument than is needed for \cite[Proposition 8.5]{DyJa02}. We can take this more elementary approach, avoiding the use of almost orthogonality, since we want, in this example, to apply the result to obtain non-vanishing results for $\ell^2$-Betti numbers of non-cocompact lattices, for which the more detailed knowledge of the $L^2$-cohomology in \cite{DyJa02} is not needed.

%In this section we explain how to show this in the specific example of $Sp_{2n}(K)$ with $K$ a (non-archimedean) local field, recalling the construction of the Bruhat-Tits building, and we note how just considering the building we can weaken the requirement that $(X,G)$ be in $\mathscr{B}t$ slightly and still get a non-zero lower bound on the highest non-zero $L^2$-betti number. There is still a requirement on the residue field. If one is willing to use more representation theory this can be done away with entirely, cf. \cite{BorelWallachBook, Ca74}.

For a general reference to buildings see \cite{BrownBuildings}, and for the buildings associated to reductive groups see \cite{BruhatTitsReductifsI}. We recall, see \cite[Chapter V, Section 2A]{BrownBuildings}, that a $BN$-pair in a group $G$ consists of subgroups $B$ and $N$ of $G$ satisfying
\begin{itemize}
\item $G=\langle B,N \rangle$ and $T:=B\cap N$ is normal in $N$.
\item The quotient $W:= N/T = \langle S \rangle$ is a Coxeter group.
\item[(BN1)] $C(s)C(w) \subseteq C(w)\cup C(sw)$ for all $s\in S, w\in W$ where we write $C(w) = BwB$ and recall that this is independent of the representative of $w$.
\item[(BN2)] $sBs^{-1} \not\subseteq B, s\in S$.
\end{itemize}
Given a $BN$-pair we can construct a building with a (strongly) transitive action of $G$ by declaring the \emph{special subgroups} of $G$ to be conjugates of groups of the form $B\langle S'\rangle B$ with $S'\subseteq S$ and then $\Delta$ is the simplicial complex with simplices the special subgroups of $G$ ordered by reverse inclusion (with $G$ the empty, or $-1$-dimensional simplex). The action of $G$ is by conjugation. Recall that equivalently one can consider \emph{special cosets}, being the left-translates of the cosets of special subgroups and $G$ acting by left-translation \cite[Chapter V, Section 2B]{BrownBuildings}.

For us the point of the strong transitivity of the action is that there is a unique orbit of simplices of maximal dimension, namely the conjugates of $B$ (for this reason, when thinking of $B$ as a simplex in the building we call it the fundamental chamber), and that $B$ is the stabilizer of the fundamental chamber.

Let $K$ be a local field (we assume that it has charateristic $p\neq 2$) with a discrete valuation $\nu:K^{\times} \rightarrow \mathbb{Z}$, valuation ring $A$ and residue field $\mathfrak{k}=A/\pi A$ where $\pi\in A$ is such that $\nu(\pi) = 1$. Following usual conventions we set $\nu (0)=\infty$ as well. We then consider the linear algebraic group $G(K)=Sp_{2n}(K), 2\leq n\in \mathbb{N}$. Recall that this is the subgroup of $GL_{2n}(K)$ consisting of matrices $g\in G(K)$ such that
\begin{equation}
g^{\transpose}\begin{pmatrix} 0 & \bbb_n \\ -\bbb_n & 0 \end{pmatrix} g = \begin{pmatrix} 0 & \bbb_n \\ -\bbb_n & 0 \end{pmatrix}. \nonumber
\end{equation}
We will consider also the groups $G(A)$ and $G(\mathfrak{k})$ defined by the same relation in addition to the requirement that the entries be in $A$ respectively $\mathfrak{k}$. Then we have an embedding $\iota : G(A)\rightarrow G(K)$ and a projection $\kappa : G(A)\rightarrow G(\mathfrak{k})$, the latter given by (entrywise) reduction modulo $\pi$. Also note that since $A$ is compact open (in $K$) so is $G(A)$ (in $G(K)$).

We let $N$ be the group of symplectic monomial matrices, i.e.~those with exactly one non-zero entry in each row and each column. We let $B$ be the inverse image under $\kappa$ of the subgroup of upper triangular matrices in $G(\mathfrak{k})$, and normalize the Haar measure on $G(K)$ such that $\mu(B) = 1$.

One then verifies that this is a $BN$-pair for $G(K)$ and that the generators of the Weyl group $W$ verifying the axioms are represented in $G(K)$ by the matrices $S= \{s_1,\dots ,s_{n-1}, r_n, r_1'\}$ where the $s_i$ are permutation matrices corresponding to the (double transpositions) permutations $(i \; i+1)(2n-i-1 \; 2n-i) \in S_{2n}, i=1,\dots ,n-1$, $r_i$ the permutation matrices corresponding to $(i \; 2n-i)\in S_{2n}, i=1\dots, n$, i.e.
\begin{eqnarray}
s_1 = \begin{pmatrix} 0 & 1 \\ 1 & 0 \\  & & \bbb_{2n-4} \\ & & & 0 & 1 \\ & & & 1 & 0 \end{pmatrix}, \dots , s_n = \begin{pmatrix} 0 & & 1 \\ & \bbb_{n-2} \\ 1 & & 0 \\ & & & 0 & & 1 \\ & & & & \bbb_{n-2} \\ & & & 1 & & 0 \end{pmatrix} \nonumber \\
r_1 = \begin{pmatrix} 0 & & 1 \\ & \bbb_{2n-2} & \\ 1 & & 0 \end{pmatrix}, \dots , r_n=\begin{pmatrix} \bbb_{n-1} \\ & 0 & 1 \\ & 1 & 0 \\ & & & \bbb_{n-1} \end{pmatrix} \nonumber \\
r_1' = r_1\begin{pmatrix} \pi \\ & \bbb_{2n-2} \\ & & \pi^{-1} \end{pmatrix} = \begin{pmatrix} 0 & & \pi^{-1} \\ & \bbb_{2n-2} \\ \pi & & 0 \end{pmatrix}. \nonumber
\end{eqnarray}

\todo{old version of arg not relying on cor. commented out}
%To show that the top cohomology is non-vanishing it is enough to show that $\dim_{\psi} \ell^2_{\circ}(\Delta_{n-1}) < 1=\dim_{\psi}\ell^2_{alt}(\Delta_n)$, the space on the left-hand side being the orthogonal complement of the kernel of $d^{n-1}$ in $\ell^2_{alt}(\Delta_{n-1})$, whence having the same $\psi$-dimension as the image of $d^{n-1}$ in $\ell^2_{alt}(\Delta_n)$.

Given all this it is easy to check (by Gaussian elimination!) that the stabilizer of each $(n-1)$-dimensional face of the fundamental chamber splits into a disjoint union of $\sharp\mathfrak{k}+1$ cosets of $B$ so that labeling these faces $f_0,\dots ,f_n$ we get by Corollary \ref{cor:HdRsoft}

\begin{equation}
\beta^n_{(2)}(G,\mu) \geq 1- (n+1)\cdot \frac{1}{\sharp\mathfrak{k}+1} = \frac{\sharp\mathfrak{k}-n}{\sharp\mathfrak{k}+1}. \nonumber
\end{equation}

\begin{comment} %%%%%%%%%%%%%%%%%%%%%%%%%%%%
\begin{eqnarray}
\dim_{\psi} \ell^2_{alt}(\Delta_{n-1}) & = & \dim_{\psi} \bigoplus_{j=0}^n \ell^2(G/G_{f_j}) \nonumber \\
 & = & \sum_{j=0}^{n} \frac{1}{\sharp\mathfrak{k}+1} = \frac{n+1}{\sharp\mathfrak{k}+1}. 
\end{eqnarray}
In particular when $\sharp\mathfrak{k} > n$ we have $\dim_{\psi} d^{n-1}(\ell^2_{alt}(\Delta_{n-1})) \leq \dim_{\psi} \ell^2_{alt}(\Delta_{n-1}) < 1$ as desired. 
\end{comment} %%%%%%%%%%%%%%%%%%%%%%%%%%%%

Backtracking, we really only rely on the fact that each special subgroup $B\langle s \rangle B, s\in S$ decomposes as a union of $\sharp\mathfrak{k}+1$ cosets of $B$, and remarking without elaboration that the same is true for $SL_n(K)$ we summarize the above in the following

\begin{theorem} \label{thm:Sp2nnonvanishing} \todo{thm:Sp2nnonvanishing}
Let $n\in \mathbb{N}$ be given and let $K$ be a non-archimedean local field of characteristic $p\neq 2$ and with cardinality of the residue field $\sharp \mathfrak{k} > n$.\todo{should also be true for orthogonal group} Then for $G$ equal to either of $Sp_{2n}(K), SL_n(K)$ we have
\begin{equation}
\beta^n_{(2)}(G,\mu) > 0. \nonumber
\end{equation}
\end{theorem}

\begin{flushright}
\qedsymbol
\end{flushright}
In particular, applying Theorem \ref{thm:totdisclattice}, we get the following result.

\begin{corollary} \label{cor:Sp2nlattice} \todo{cor:Sp2nlattice}
With notation and assumptions as in the theorem, if $H$ is a lattice in $G$ then $\beta^n_{(2)}(H) > 0$. \hfill \qedsymbol
\end{corollary}

\begin{remark}
\begin{enumerate}[(i)]
\item Note that Corollary \ref{cor:Sp2nlattice} includes the case where $K$ is a local field of positive characteristic, e.g. a function field, in which case it is well-known, see \cite[Section 3.1.1]{FHT11},\cite[Chapter IX, 1.6(viii)]{MargulisBook}, that $Sp_{2n}(K)$ admits lattices, but no cocompact lattices. As far as I am aware the result was previously unknown in this case. Of course, when there is a cocompact lattice this acts on the building as well, and the action is cofinite so that the analysis above passes directly to the lattice. Then by \cite[Theorem 6.3]{Ga02} the nonvanishing passes to every other lattice. Hence the most interesting case of Corollary \ref{cor:Sp2nlattice} is the case of positive characteristic.

\item For vanishing in degree different from $n$, see e.g. \cite[Theorem 3.9]{BorelWallachBook}, or \cite[Theorem 5.1]{DyJa02}. 

\item We note also that \cite[Theorem 3.9]{BorelWallachBook} (see also \cite{Ca74}) gives a description similar to how the result in \cite{Borel85sym} gives equation \eqref{eq:sltwococycle}. In particular, with more detailed analysis, one might be able to remove the restriction on the size of the residue field in Theorem \ref{thm:Sp2nnonvanishing}.
\end{enumerate}
\end{remark}

         % OK

\section{Vanishing in the amenable case} \label{sec:totdiscamenable} \todo{sec:totdiscamenable}

In this section we show that for any totally disconnected, amenable {\lcsu} group $G$, the $L^2$-Betti numbers vanish, $\beta^n_{(2)}(G,\mu) = 0, n\geq 1$.

The proof follows that of Lyons in \cite{Ly08}, which is itself a refinement of earlier proofs \cite{DoMa98,Eck99}, and is not radically different from the proof given by Cheeger and Gromov in \cite[Theorem 0.2]{ChGr86}. The core idea, which notably also appears in Elek's observation that the analytic zero divisor conjecture is equivalent to the algebraic zero divisor conjecture in the case of amenable groups \cite{Ele03}, is to approximate the $LG$-dimension using $\dim_{\mathbb{C}}$ on finite-dimensional subspaces of $\ell^2G$ spanned by elements in a F{\o}lner sequence.

Let for the moment $\Delta$ be a simplicial complex with a cofinite, continuous action of $G$, with compact stabilizers, and suppose that $G$ is amenable. Denote by $L$ a fixed fundamental domain for the action of $G$ on $\Delta$, and by $(F_k)$ a F{\o}lner sequence in $G$. Recall this means that for any $F\subseteq G$ such that $\mu(F)<\infty$,
\begin{equation}
\frac{\mu\left( (F_k.F \setminus F) \cup (F\setminus F_k.F)\right)}{\mu(F_k)} \rightarrow_k 0. \nonumber
\end{equation}

Note that $L$ is not necessarily a simplicial complex. We denote by $\bar{L}$ the closure of $L$ in $\Delta$, i.e.~the smallest simplicial subcomplex containing $L$. More explicitly this consists of all simplices which are faces in some simplex in $L$.

Then we define an exhaustion of $\Delta$ by subcomplexes $\Gamma^{(k)}:=F_k.\bar{L}$ and define the (combinatorial) $n$-boundary, where $\Gamma^{(k)}_n$ is the $n$-skeleton of $\Gamma^{(k)}$,
\begin{equation}
\partial_n \Gamma^{(k)} := \Gamma^{(k)}_n\setminus F_k.L_n. \nonumber
\end{equation}

For any subset $S\subseteq \Delta_n$ of the $n$-skeleton of $\Delta$, we write $\mu(S):=\sum_{\delta\in S} \mu(G_{\delta})$.

%%%%%%%%%%%%%%%%%%%%%%%%%%%%%%%%%%%%%
\begin{comment}
Then there is a F{\o}lner sequence $(\Gamma^{(k)})_{k\in \mathbb{N}}$ in $\Delta$,\todo{reference / explanation!} in the following sense:
\begin{itemize}
\item The $\Gamma^{(k)}$ are finite subcomplexes of $\Delta$, and $L \subseteq \Gamma^{(k)}\subseteq \Gamma^{(k+1)}$ for some fundamental domain $L$ for the action of $G$ on $\Delta$ and all $k$,

\item The sequence is exhausting: $\Delta = \cup_k \Gamma^{(k)}$,

\item For a subcomplex $\Gamma\leq \Delta$ we write $\Gamma^{\circ}:=\{ \gamma \in \Gamma_n \; some \; n \mid \forall \delta\in \Delta_{n'}: \gamma \leq \delta \Rightarrow \delta\in \Gamma_{n'} \}$, i.e.~the interior of $\Gamma$ \todo{correct word?}wiz.~the set of all simplices whose coboundary is entirely supported in $\Gamma$. (Note: in general the interior is a subset, not a subcomplex.)

Define the $n$-dimensional boundary of $\Gamma$ as $\partial_n \Gamma := \Gamma_n \setminus \Gamma_n^{\circ}$. Then we have for all $n$: \todo{maybe the sharps should be measures?}
\begin{equation}
\frac{\mu( \partial_n \Gamma^{(k)} )}{\mu( \Gamma^{(k)}_n )} \rightarrow_k 0. \nonumber
\end{equation}
\end{itemize}
\end{comment}
%%%%%%%%%%%%%%%%%%%%%%%%%%%%%%%%%%%%%%%%%%%%%%%%%%

\noindent
The theorem of Dodziuk and Mathai then is the following.

\begin{theorem}[(\cite{DoMa98,Eck99})] \label{thm:DodziukMathai} \todo{thm:DodziukMathai}
Suppose that $G$ is an amenable countable discrete group, acting freely (on unordered simplices) and cofinitely on the simplicial complex $\Delta$, and suppose that $F_k\subseteq G$ is a F{\o}lner sequence in the sense above. Then denoting by $L$ a fundamental domain for $G\curvearrowright \Delta$,
\begin{equation}
\beta^n_{(2)}(\Delta;G) = \lim_{k} \frac{\beta_n(\Gamma^{(k)})}{\sharp F_k}, \nonumber
\end{equation}
where $\beta_n(\Gamma^{(k)})$ are the ordinary Betti numbers of the finite simplicial complex $\Gamma^{(k)}$.

The same statement holds with $\Delta$ instead a CW-complex.
\end{theorem}

Denote for $k\in \mathbb{N}$ the orthogonal projection of $\ell^2_{alt}\Delta_n$ onto $\ell^2_{alt}\Gamma_n^{(k)}$, a finite dimensional vector space, by $\Pi_{\Gamma^{(k)}}$. For $E\subseteq \ell^2_{alt}\Delta_n$ a closed subspace (not necessarily invariant) we define
\begin{equation}
\dim_{\Gamma^{(k)}} E := \operatorname{Tr}_{\ell^2\Delta_n}(\Pi_{\Gamma^{(k)}}P_E) = \sum_{\gamma\in \Gamma_n^{(k)}} \langle P_E \bbb_{\gamma}, \bbb_{\gamma} \rangle_{\ell^2\Gamma_n^{(k)}}. \nonumber
\end{equation}

The following proposition is a slight generalization of observations due to Eckmann \cite{Eck99} in the discrete case.

\begin{proposition}
Let $G$ be an amenable totally disconnected {\lcsu} group, acting continuously, cofinitely, and with compact stabilizers on a simplicial complex $\Delta$ and fix an $n\geq 0$. Let $L,(F_k)$, etc., be as above. Denote $K:=\cap_{\delta\in L_n} G_{\delta}$. We can and will take the $F_k$ to be unions of cosets of $K$.

For any closed subspace $E\subseteq \ell^2_{alt}\Delta_n$ the hollowing holds:
\begin{enumerate}[(i)]
\item For all $k\in \mathbb{N}$
\begin{equation}
0\leq \dim_{\Gamma^{(k)}} E \leq \dim_{\mathbb{C}} \Pi_{\Gamma^{(k)}}(E), \nonumber
\end{equation}
and $\dim_{\Gamma^{(k)}} E = \dim_{\mathbb{C}} \Pi_{\Gamma^{(k)}}(E)$ if and only if $E\subseteq \ell^2_{alt} \Gamma^{(k)}_n$,

\item If $F\subseteq \ell^2_{alt}\Delta_n$ is a closed subspace and $E\perp F$ then for all $k$
\begin{equation}
\dim_{\Gamma^{(k)}} E\oplus F = \dim_{\Gamma^{(k)}} E + \dim_{\Gamma^{(k)}} F. \nonumber
\end{equation}

In particular $\dim_{\Gamma^{(k)}}E \leq \dim_{\Gamma^{(k)}}\tilde{E}$ whenever $E\leq \tilde{E}$.

\item Suppose $E$ is a closed invariant subspace. Then for all $k\in \mathbb{N}$:
\begin{equation}
0\leq \dim_{\Gamma^{(k)}}E - \mu(F_k) \cdot \dim_{\psi} E \leq \sharp \partial_n \Gamma^{(k)} \leq \frac{\mu(\partial_n \Gamma^{(k)})}{\mu(K)}, \nonumber
\end{equation}

\item Denote $F:=\cup_{\delta\in L_n}\{g\in G\mid g.\delta \in \bar{L}\setminus L\}$. Then
\begin{equation}
\mu(\partial_n \Gamma^{(k)}) \leq \sharp L_n \cdot \mu(F_k.F\setminus F_k). \nonumber
%wasn't sure about this. Original said "trivial" but...
%\mu(\partial_n \Gamma^{(k)}) = \mu\left( (F_k.F\setminus F_k) \cup (F_k\setminus F_k.F) \right). \nonumber
\end{equation}
\end{enumerate}
\noindent
In particular, $\dim_{\psi}E = \lim_{k} \frac{\dim_{\Gamma^{(k)}}E}{\mu(F_k)}$ whenever $E$ is a closed invariant subspace.
\end{proposition}

\begin{proof}
Part (i) is clear: the positivity since $P_E$ is a projection onto a subspace of $\ell^2_{alt}\Delta_n \subseteq \ell^2\Delta_n$, and the upper bound by the computation
\begin{equation}
\dim_{\mathbb{C}} \Pi_{\Gamma^{(k)}}(E) = \operatorname{Tr} (\operatorname{Ran}[\Pi_{\Gamma^{(k)}}P_E]) \geq \operatorname{Tr} (\Pi_{\Gamma^{(k)}}P_E \Pi_{\Gamma^{(k)}}) = \operatorname{Tr} (\Pi_{\Gamma^{(k)}}P_E), \nonumber
\end{equation}
using that the operator norm $\lVert \Pi_{\Gamma^{(k)}}P_E \Pi_{\Gamma^{(k)}}\rVert \leq 1$ whence the positive operator $\Pi_{\Gamma^{(k)}}P_E\Pi_{\Gamma^{(k)}}$ is dominated by its range projection.

The final part of (i) is trivial, as is (ii).

To prove (iii) we compute, using invariance,
\begin{equation}
\mu(F_k)\cdot \dim_{\psi} E = \sum_{g\in F_k\subseteq G/K} \sum_{\delta \in L_n} \frac{\mu(K)}{[G_{\delta}:K]} \langle P_E.\bbb_{g.\delta},\bbb_{g.\delta} \rangle_{\ell^2\Gamma^{(k)}_n} = \sum_{\gamma\in F_k.L_n}\langle P_E.\bbb_{\gamma},\bbb_{\gamma} \rangle, \nonumber
\end{equation}
whence since $\bbb\geq P_E$,
\begin{equation}
\mu(F_k)\cdot \dim_{\psi} E + \sharp \partial_n\Gamma^{(k)} \geq \sum_{\gamma\in F_k.L_n}\langle P_E.\bbb_{\gamma},\bbb_{\gamma} \rangle + \sum_{\gamma\in \partial_n\Gamma^{(k)}} \langle P_E.\bbb_{\gamma},\bbb_{\gamma} \rangle = \dim_{\Gamma^{(k)}} E \geq 0. \nonumber
\end{equation}
This proves the first inequality and the second is trivial.

For (iv) we just note that if we denote for $\gamma\in \partial_n \Gamma^{(k)}$ in the orbit of some fixed $\gamma_0\in L_n$ the set $F_{\gamma} := \{ g\in G \mid g.\gamma_0=\gamma\}$ then $\mu(G_{\gamma})=\mu(F_{\gamma})$ and $\cup_{\gamma \in \partial_n\Gamma^{(k)}\cap G.\gamma_0} F_{\gamma} \subseteq F_kF \setminus F_k$. The statement follows directly from this.

The very final statement is clear since $F_k$ is a F{\o}lner sequence.
\end{proof}

Let now $G$ act on a contractible simplicial complex $\Delta$, not necessarily cofinitely, and let $\Delta^{(i)}$ be an increasing exhaustion of $\Delta$ by $G$-invariant subcomplexes whereupon the action is cofinite. Fix fundamental domains $L^{(i)}$, also increasing in $i$. Denote the finite exhaustions of $\Delta^{(i)}$ as above by $\Gamma^{(i,k)}$.

We prove now that $\beta_n^{(2)}(\Delta; G,\mu) = 0$. By Theorem \ref{thm:totdiscactioncompute}, this means we have to show that for every $i$,\todo{eq:Lyons4.2inclusionsZ}
\begin{equation} \label{eq:Lyons4.2inclusionsZ} 
\dim_{\psi} \left[ \cup_{j:j\geq i} ( \bar{B}_n^{(2)}(\Delta^{(j)})\cap \ell^2_{alt}(\Delta^{(i)}_n))\right]^{\perp} = \dim_{\psi} Z_n^{(2)}(\Delta^{(i)})^{\perp}.
\end{equation}

First, the proof of \cite[equation (4.2)]{Ly08} now transfers verbatim to our setting, which we record for convenience:

\begin{lemma}[(Compare {\cite[equation (4.2)]{Ly08}})] \label{lma:Lyons4.2} \todo{lma:Lyons4.2}
For each $i$,
\begin{equation}
\dim_{\psi} Z_n^{(2)}(\Delta^{(i)}) = \lim_k \frac{1}{\mu(F_k)}\dim_{\mathbb{C}} B^k(\Gamma^{(i,k)}). \nonumber
\end{equation}
\end{lemma}

Then we compute:
\begin{eqnarray}
\dim_{\psi} \left[ \cup_{j:j\geq i} \bar{B}_n^{(2)}(\Delta^{(j)})\cap \ell^2_{alt}\Delta^{(i)} \right]^{\perp} & = & \lim_k \frac{1}{\mu(F_k)} \dim_{\Gamma^{(i,k)}} \left[ \cup_{j:j\geq i} B_n^{(2)}(\Delta^{(j)})\cap \ell^2_{alt}\Delta^{(i)} \right]^{\perp} \nonumber \\
 & \leq & \liminf_k \frac{1}{\mu(F_k)} \dim_{\mathbb{C}} \Pi_{\Gamma^{(i,k)}} \left( \left[ \cup_{j:j\geq i} B_n^{(2)}(\Delta^{(j)})\cap \ell^2_{alt}\Delta^{(i)} \right]^{\perp} \right) \nonumber \\
 & \leq & \liminf_k \frac{1}{\mu(F_k)} \dim_{\mathbb{C}} \Pi_{\Gamma^{(i,k)}} \left( \left[ \cup_{j:j\geq i}\cup_l B_n(\Gamma^{(j,l)})\cap \ell^2_{alt}\Delta^{(i)} \right]^{\perp} \right). \nonumber
\end{eqnarray}

For any $l$, we have $B_n(\Gamma^{(j,l)}) \cap \ell^2_{alt}\Gamma^{(i,k)} \subseteq B_n(\Gamma^{(j,l)})\cap \ell^2_{alt}\Delta^{(i)}$, so we deduce

\begin{eqnarray}
%\liminf_k \frac{1}{\mu(F_k)} \dim_{\mathbb{C}} \Pi_{\Gamma^{(i,k)}} \left( \left[ \cup_{j:j\geq i}\cup_l B_n(\Gamma^{(j,l)})\cap \ell^2_{alt}\Delta^{(i)} \right]^{\perp} \right)  & \leq &  \nonumber \\
\dim_{\psi} \left[ \cup_{j:j\geq i} \bar{B}_n^{(2)}(\Delta^{(j)})\cap \ell^2_{alt}\Delta^{(i)} \right]^{\perp} & \leq & \liminf_k \frac{1}{\mu(F_k)} \dim_{\mathbb{C}} \Pi_{\Gamma^{(i,k)}} \left( \left[ \cup_{j:j\geq i}\cup_l B_n(\Gamma^{(j,l)})\cap \ell^2_{alt}\Gamma^{(i,k)} \right]^{\perp} \right) \nonumber \\
 & = & \liminf_k \frac{1}{\mu(F_k)} \dim_{\mathbb{C}} \left( Z_n(\Gamma^{(i,k)})^{\perp} \cap \ell^2_{alt} \Gamma^{(i,k)} \right), \nonumber
\end{eqnarray}
where the equality holds because $\Delta$ is acyclic, so that $Z_n(\Gamma^{(i,k)}) \subseteq \cup_{j:j\geq i} \cup_l B_n(\Gamma^{(j,l)})$.

But the orthogonal complement of the cycles in a finite complex is just the space of coboundaries, and then by Lemma \ref{lma:Lyons4.2} we conclude that

\begin{eqnarray}
\dim_{\psi} \left[ \cup_{j:j\geq i} \bar{B}_n^{(2)}(\Delta^{(j)})\cap \ell^2_{alt}\Delta^{(i)} \right]^{\perp} & \leq & \liminf_k \frac{1}{\mu(F_k)} \dim_{\mathbb{C}} B^n(\Gamma^{(i,k)}) \nonumber \\
 & = & \lim_k \frac{1}{\mu(F_k)} \dim_{\mathbb{C}} B^n(\Gamma^{(i,k)}) = \dim_{\psi} Z_n^{(2)}(\Delta^{(i)})^{\perp} \cap \ell^2_{alt} \Delta^{(i)}. \nonumber
\end{eqnarray}

Now \eqref{eq:Lyons4.2inclusionsZ} follows since, by the inclusion $\cup_{j:j\geq i} \bar{B}_n^{(2)}(\Delta^{(j)})\cap \ell^2_{alt}\Delta^{(i)} \subseteq Z_n^{(2)}(\Delta^{(i)})$, the other inequality is trivial. This proves the claim.

\begin{theorem} \label{thm:totdiscamenable} \todo{thm:totdiscamenable}
Let $G$ be a totally disconnected amenable {\lcsu} group. Then
\begin{equation}
\beta^n_{(2)}(G,\mu) = 0, \quad n\geq 1. \nonumber
\end{equation}
\end{theorem}

\begin{proof}
Combining the above with Theorem \ref{thm:dualitytotdisc} and Theorem \ref{thm:totdiscactioncompute} we have
\begin{equation}
\beta^n_{(2)}(G,\mu) = \beta_n^{(2)}(G,\mu) = \beta_n^{(2)}(\underline{E}G; G,\mu) = 0. \nonumber
\end{equation}
\end{proof}

For easy reference we single out the following slightly more general statement, as a consequence of the theorem.

\begin{corollary} \label{cor:totdiscamenablegen} \todo{cor:totdiscamenablegen}
Let $G$ be a totally disconnected amenable {\lcsu} group and $\tilde{G}$ a {\lcsu} group such that $G\leq \tilde{G}$. Then for all $n\geq 1$
\begin{equation}
\dim_{(L\tilde{G},\tilde{\psi})} H^n(G,L^2\tilde{G}) = 0. \nonumber
\end{equation}
\end{corollary}

\begin{proof}
By Porism \ref{por:dualitytotdisc} it is sufficient to prove the claim for homology instead of cohomology, and by Theorem \ref{thm:homtotdisccoeff} we can take coefficients in $L\tilde{G}$ instead of $L^2\tilde{G}$. That is, it suffices to show
\begin{equation}
\dim_{L\tilde{G}} H_n(G,L\tilde{G}) = 0. \nonumber
\end{equation}

But here we have isomorphisms $\pi\colon L\tilde{G} \otimes_{LG} (\underline{\mathscr{C}}_K\mathcal{F}_c(G_K^n,LG)) \rightarrow \underline{\mathscr{C}}_K\mathcal{F}_c(G_K^n,L\tilde{G})$. Indeed, this follows from the fact that we kan describe the quotients $\underline{\mathscr{C}}_K ( - ) $ algebraically as the kernels of maps $\Phi$ cf. Lemma \ref{lma:CKdescr}.

These obviously commute with boundary maps, i.e.~induce an isomorphism of complexes
\begin{displaymath}
\xymatrix{ \cdots \ar[r] & \underline{\mathscr{C}}_K\mathcal{F}_c(G_K^n,LG))\otimes_{LG}L\tilde{G} \ar[r]^<<<<<{d\otimes \bbb} \ar[d]_{\pi}^{\simeq} & \cdots \ar[r] & LG\otimes_{LG} L\tilde{G} \ar[r] \ar[d]_{\pi}^{=} & 0 \\ \cdots \ar[r] & \underline{\mathscr{C}}_K\mathcal{F}_c(G_K^n,L\tilde{G}) \ar[r]^<<<<<<{d} & \cdots \ar[r] & L\tilde{G} \ar[r] & 0} .
\end{displaymath}

Hence by the dim-exactness of the induction functor $L\tilde{G}\otimes_{LG} -$ in Theorem \ref{thm:tensordimexactness} we get the desired conclusion.
\end{proof}

\section{An algebraic approach} \label{sec:totdisckyed} \todo{sec:totdisckyed}

In this section we take an algebraic approach to the vanishing of $L^2$-Betti numbers for totally disconnected groups, based on the general operator algebraic framework for F{\o}lner sequences in \cite{AlKy12}. Since we have already one proof above, we will favor brevity over giving exhaustive details below.

Let $G$ be a totally disconnected {\lcsu} group and denote by $\mathbb{H}(G)$ the associated Hecke algebra, i.e.~the convolution $*$-algebra of compactly supported, locally constant, complex-valued functions on $G$. This is not unital, but it is idempotented, i.e.~the multiplication map $\mathbb{H}(G)\otimes_{\mathbb{C}}\mathbb{H}(G) \rightarrow \mathbb{H}(G)$ is surjective, and $\mathbb{H}(G)=\cup_n p_n\mathbb{H}(G)p_n$ is an increasing union of corners with $p_n\in \mathbb{H}(G)$ idempotents (in our case we have an almost canonical choice for the $p_n$, namely the indicator functions of compact open subgroups constituting a topological basis at the identity; in this case the $p_n$ are also self-adjoint).

As noted in \cite[Chapter XII]{BorelWallachBook}, the category of (non-topological) smooth $G$-modules $\mathfrak{E}^f_{\mathscr{A},G}$ is equivalent to the category of non-degenerate $\mathbb{H}(G)$-modules, and moreover, many of the results and methods of homological algebra usual proved under a blanket assumptions that the rings or algebras in question have a unit, actually hold more generally over idempotented algebras. Further, it is clear the two notions of "projective" in $\mathfrak{E}^f_{\mathscr{A},G}$ are equivalent. Since also the category $\mathfrak{E}^f_{\mathscr{A},G}$ is abelian we write $H^f_n(G,E):=\operatorname{Tor}_n^{\mathbb{H}(G)}( E,\mathbb{C})$ for a module $E\in \mathfrak{E}^f_{\mathscr{A},G}$.

\begin{proposition}
Let $E$ be a quasi-complete topological $\mathscr{A}$-$G$-module. Then, canonically,
\begin{equation}
H_n(G,E) = H_n(G,E^{\infty}) = H_n^f(G,E^{\infty}). \nonumber
\end{equation}
\end{proposition}

We leave out the proof, referring instead to the ideas in \cite[Chapter XII]{BorelWallachBook}.

Since the $\mathscr{A}$-module structure is canonical on $H^f_n(G,E)$, so it follows by standard theorems in homological algebra (see e.g.~\cite[Theorem 2.7.2]{Weibelbook}) that we may compute the homology by taking a projective resolution of the second variable in the $\operatorname{Tor}$-functor.

Further, since $\operatorname{Tor}$ "commutes" with direct limits, we have
\begin{equation}
\beta^{(2)}_n(G,\mu) = \dim_{(LG,\psi)} H_n(G,LG) = \dim_{(LG,\psi)} H^f_n( G,LG^{\infty}) = \dim_{(LG,\psi)} \lim_{\rightarrow_m} H^f_n( G, LGp_m ). \nonumber
\end{equation}

It follows from rank-density arguments and the inductive limit formula that it is then sufficient to show that for all $m\geq 1$ we have for the corners
\begin{equation}
0 = \dim_{(LG_{p_m},\psi(\cdot p_m))} H^f_n(G, LG_{p_m}) = \dim_{(LG_{p_m},\psi(\cdot p_m))} \operatorname{Tor}_n^{\mathbb{H}(G)_{p_m}}( LG_{p_m}, \mathbb{C}). \nonumber
\end{equation}

But now the algebras are unital, and further the corners $LG_{p_m}$ of the group von Neumann algebra, are finite, tracial von Neumann algebras, and the corners $\mathbb{H}(G)_{p_m}$ are weakly dense, unital $*$-subalgebras.

The result then follows directly from the dimension flatness result \cite[Theorem 4.4]{AlKy12} since the tower $\mathbb{C} \subseteq \mathbb{H}(G)_{p_m} \subseteq LG_{p_m}$ has the strong F{\o}lner property, as defined in \cite{AlKy12}, for all $m$ when $G$ is amenable.

 %this is OK

%\input{NotesLtwolcgroupsAlgebraicGroupList}

\chapter{Product groups} \label{chap:Kunneth}
\epigraph{I already am a product}{Lady Gaga}

K{\"u}nneth theorems relate the (co)homology of products to that of the factors. The original result \cite{Ku23} deals with the Betti numbers of a product of manifolds. For $L^2$-Betti numbers, a K{\"u}nneth formula was proved in \cite{ChGr86} for countable groups.

In the context of locally compact groups, the continuous cohomology of products is studied in \cite[Chapters X, XII]{BorelWallachBook}. In that book, the authors study principally the case where the coefficient modules are admissible in some sense, which allows, under suitable circumstances, to forget the topology on the coefficient modules and reduce the computation to a computation in the category of vector spaces.

The idea in this chapter is similar, but since the coefficient modules we are interested in are not admissible, we instead reduce the computation to a purely algebraic one "up to dimension". For this, working with homology seems more appropriate since the corespondence with an algebraic tensor product of complexes is more direct. This and the need to work with modules that have finite dimension restricts the statement of Theorem \ref{thm:Kunnethtotdisc} to cover only totally disconnected groups.\footnote{Using the results in Chapter \ref{chap:Ramen} it is not hard to show that a K{\"u}nneth formula holds in complete generality for a product of two {\lcsu} groups.}

\section{Products of totally disconnected groups}

Let $\tilde{G}=G\times H$ be a product of totally disconnected {\lcsu} groups. Let $K\leq G$ and $L\leq H$ be compact open subgroups, $\tilde{K}:=K\times L$, and for $n\geq 0$
\begin{equation}
Q_n^{tot} := \sum_{i=0}^{n} \mathcal{F}_c((G/K)^{i+1}\times (H/L)^{n+1-i}, L^2\tilde{G}). \nonumber
\end{equation}
We denote by $Q_{i,j}$ the term $\mathcal{F}_c((G/K)^{i}\times (H/L)^{j}, L^2\tilde{G})$ for $i,j\geq 0$ and note that 
\begin{displaymath}
Q_{i,j} \simeq \mathcal{F}_c( (G/K)^{i} , \mathcal{F}_c( (H/L)^{j},L^2\tilde{G} )) 
\end{displaymath}
as a $G$-module, and similarly when considering $Q_{i,j}$ as an $H$-module. (Actually the isomorphisms are as $\tilde{G}$-modules.)

Consider boundary maps $d_{i,j}'\colon Q_{i+2,j}\rightarrow Q_{i+1,j}$ and $d_{i,j}''\colon Q_{i,j+2}\rightarrow Q_{i,j+1}$ for $i,j\geq 0$, given by the usual ones restricted to the relevant arguments:
\begin{eqnarray}
(d_{i,j}'f)( g_1,\dots ,g_{i+1}, h_1,\dots ,h_{j} ) & = & \sum_{k=0}^{i+1}(-1)^{k}\left( \sum_{g\in G/K} f( g_1,\dots ,g_k,g,g_{k+1},\dots ,g_{i+1},h_1,\dots ,h_{j} ) \right) \nonumber \\
(d_{i,j}''f)( g_1,\dots ,g_{i}, h_1,\dots ,h_{j+1} ) & = & \sum_{k=0}^{j+1}(-1)^{k}\left( \sum_{h\in H/L} f( g_1,\dots ,g_{i},h_1,\dots ,h_{k},h,h_{k+1},\dots ,h_{j+1} ) \right). \nonumber
\end{eqnarray}

Hence we can fit everything into a third quadrant double complex
\begin{displaymath}
\xymatrix{  &  &  &  & Q_{0,0} = L^2\tilde{G} \\ \cdots \ar[r] & Q_{i,1} \ar[r] & \cdots \ar[r] & Q_{1,1} \ar[ur] &  \\  & \vdots \ar[u] & & \ar[u] \vdots &  \\ \cdots \ar[r] & Q_{i,j} \ar[u] \ar[r] & \cdots \ar[r] & Q_{1,j} \ar[u] &  \\ & \vdots \ar[u] & & \vdots \ar[u] &   }
\end{displaymath}
Computing the spectral sequence for linear spaces associated with this complex we see that the $E^2$-term is zero everywhere whence the complex
\begin{displaymath}
\xymatrix{ (Q_*^{tot},d'+d'') \ar[r] & L^2\tilde{G} \ar[r] & 0}
\end{displaymath}
is a resolution of $L^2\tilde{G}$. Since a direct sum of strengthened maps is strengthened\todo{admissible morphisms in exact categories if we like...} the resolution is also strengthened. (A more direct approach to show that $(Q_*^{tot},d'+d'')$ is a strengthened resolution is to write down an explicit contraction $s'+s''$ where $s'$ and $s''$ are contractions of the horizontal, respectively vertical, subcomplexes.)

By the isomorphism $Q_{i,j} \simeq \mathcal{F}_c( (G/K\times H/L)^{i}, \mathcal{F}_c( (H/L)^{j}, L^2\tilde{G}))$ for $j\geq i$, and similarly for the opposite inequality, $Q_{i,j}$ is a projective topological $L\tilde{G}$-$\tilde{G}$-module for all $i,j\geq 1$.

Let $S_m,T_m$ be finite $K$-invariant subsets of $G/K$ respectively $H/L$, increasing with union $G/K=\cup_m T_m$ and $H/L=\cup_m S_m$. Denote by $M^i_m$ the submodules
\begin{equation}
M^i_m := \overline{\span} \{ f-f.g \mid f\in \mathcal{F}_c( (G/K)^i,L^2G ), g\in G\} \cap \mathcal{F}_c( T^i_m,L^2G). \nonumber
\end{equation}

Similarly we denote $N^i_m$ for $H/L$. Then we consider complexes
\begin{displaymath}
\xymatrix{ \mathfrak{G}^m_* : & \cdots \ar[r]^>>>>>{d_{i+1}} & \mathcal{F}(T_m^{i+2},L^2G)/M^{i+2}_m \ar[r]^{d_i} & \mathcal{F}(T_m^{i+1},L^2G)/M^{i+1}_m \ar[r]^<<<<<{d_{i-1}} & \cdots \\ \mathfrak{H}^m_* : & \cdots \ar[r]^>>>>>{d_{i+1}} & \mathcal{F}(S_m^{i+2},L^2H)/N^{i+2}_m \ar[r]^{d_i} & \mathcal{F}(S_m^{i+1},L^2H)/N^{i+1}_m \ar[r]^<<<<<{d_{i-1}} & \cdots }
\end{displaymath}

The next lemma is a type of result we already used in the previous chapter, in a slightly different formulation, stating that one can write a certain complex as an inductive limit and then compute the homology as the inductive limit of homology modules.

\begin{lemma} \label{lma:totdiscprodI}
The inclusion maps $\phi_{m,m'}\colon \mathcal{F}(T_m^{i+2},L^2G) \rightarrow \mathcal{F}(T_{m'}^{i+2},L^2G)$ for $m\leq m'$ induce an inductive system $(\bar{\phi}_{m,m'}\colon \mathfrak{G}^m_i/M_m^i\rightarrow \mathfrak{G}^{m'}_i/M^i_{m'})$, and the maps $\phi_{m,m'}$ commute with the boundary maps. Hence there are maps of $LG$-modules
\begin{equation}
\lim_{\rightarrow} H_n(\mathfrak{G}_*^m) \rightarrow H_n(G,L^2G) \nonumber
\end{equation}
for all $n\geq 0$. These are isomorphisms. Similarly,
\begin{equation}
\lim_{\rightarrow} H_n(\mathfrak{H}_*^m) \rightarrow H_n(H,L^2H) \nonumber
\end{equation}
are isomorphisms.
\end{lemma}

\begin{proof}
All the needs proving is that for all $i$, the maps
\begin{equation}
\lim_{\rightarrow} \mathcal{F}_c(T_m^{i+1},L^2G)/M^{i+1}_m \rightarrow \underline{C}_G\mathcal{F}_c( (G/K)^{i+1},L^2G ) \nonumber
\end{equation}
(exist and) are isomorphisms. Existence, continuity, and surjectivity are all trivial. Injectivity is true by construction.
\end{proof}

\begin{porism} \label{por:totdiscprodI}
Denoting $Q_{i,j,m}:=\mathcal{F}_c(T_m^{i}\times S_m^{j},L^2\tilde{G})$ and $R_{i,j,m}:=M^{i}_m\bar{\otimes} N^{j}_m$ we consider the complexes
\begin{displaymath}
\xymatrix{ \mathfrak{Q}^m_* : & \cdots \ar[r] & \bigoplus_{i=0}^{n} Q_{i+1,n+1-i,m}/R_{i+1,n+1-i,m} \ar[r] & \cdots }
\end{displaymath}

Then the morphisms
\begin{equation}
\lim_{\rightarrow} H_n( \mathfrak{Q}^m_* ) \rightarrow H_n(\tilde{G},L^2\tilde{G}) \nonumber
\end{equation}
are isomorphisms. \hfill \qedsymbol
\end{porism}

Observe that the $Q_{i,j,m}$ are Hilbert spaces, and by a $3\times 3$ argument, $Q_{i,j,m}/R_{i,j,m} \simeq \mathfrak{G}^m_i\bar{\otimes} \mathfrak{H}_j^m$. We denote these quotients $\mathfrak{Q}^m_{i,j}:=Q_{i,j,m}/R_{i,j,m}$.

\begin{lemma} \label{lma:totdiscprodII}
Consider the inclusions $\iota_{i,j,m}\colon \mathfrak{G}^m_i\otimes_{alg} \mathfrak{H}^m_j \rightarrow Q_{i,j,m}/R_{i,j,m}$. Then whenever $E\leq \mathfrak{G}^m_i$ and $F\leq \mathfrak{H}^m_j$ are $LG$- respectively $LH$-submodules, we have
\begin{equation}
\dim_{L\tilde{G}} \overline{\iota_{i,j,m}(E\otimes_{alg} F)} = \dim_{LG} E\cdot \dim_{LH} F. \nonumber
\end{equation}
\end{lemma}

\begin{proof}
Since the $\mathfrak{G}^m_j$ and $\mathfrak{H}^m_i$ are finite dimensional Hilbert modules over the respective von Neumann algebras (e.g.~the $\mathfrak{G}^m_i$ are summands in direct sums $\oplus_i^{fin} L^2(G/K_i)$ for compact open subgroups $K_i$) we can assume that $E,F$ are closed, hence images of projections $p,q$, respectively, in finite matrix algebras over the resective von Neumann algebras, which can be assumed to be of the same rank, i.e.~$p\in M_k(LG), q\in M_k(LH)$ for some $k$.

The claim is then obvious since $\overline{\iota_{i,j,m}(E\otimes_{alg} F)}$ is isomorphic to the image of $p\otimes q$.
\end{proof}

\begin{theorem}[(K{\"u}nneth formula)] \label{thm:Kunnethtotdisc} \todo{thm:Kunnethtotdisc}
Let $\tilde{G}=G\times H$ be a product of totally disconnected {\lcsu} groups and fix Haar measures $\mu$ on $G$ and $\nu$ on $H$. For all $n\geq 0$,
\begin{equation}
\beta_n^{(2)}(\tilde{G},\mu\otimes \nu) = \sum_{k=0}^n \beta_n^{(2)}(G,\mu)\cdot \beta_{n-k}^{(2)}(H,\nu). \nonumber
\end{equation}
\end{theorem}

\begin{proof}
Since $\dim_{(L\tilde{G},\tilde{\psi})} \mathfrak{Q}^m_n < \infty$ for all $m,n$, Lemma \ref{lma:totdiscprodI} and the inductive limit formula imply that it is sufficient to show that for all $m,n$ we have
\begin{equation}
\dim_{(L\tilde{G},\tilde{\psi})} H_n( \mathfrak{Q}^m_* ) = \sum_{k=0}^n \dim_{(LG,\psi_{\mu})} H_n(\mathfrak{G}^m_*) \cdot \dim_{(LH,\psi_{\nu})} H_{n-k}(\mathfrak{H}_*^m). \nonumber
\end{equation}

We deduce this from Lemma \ref{lma:totdiscprodII} as follows. Consider the diagram (where the tensor products are purely algebraic)
\begin{displaymath}
\xymatrix{ \ar[r] & \mathfrak{Q}_{i+1,j-1}^m \ar[rr]^{d_{i-1,j-1}'} & & \mathfrak{Q}_{i,j-1}^m \ar[rr] & &  \\ \mathfrak{G}_{i+1}^m\otimes \mathfrak{H}_{j-1}^m \ar[ur]^{\iota_{i+1,j-1,m}} \ar[rr]^{d_{i-1}\otimes \bbb} & & \mathfrak{G}_i^m\otimes \mathfrak{H}_{j-1}^m \ar[ur]^{\iota_{i,j-1,m}} \ar[rr] & & \\ & & & \mathfrak{Q}_{i,j}^m \ar'[u][uu]_>>{d_{i,j-2}''} \ar[rr]^{d_{i-2,j}'} & & \mathfrak{Q}_{i-1,j}^m \ar[uu] \\ & & \mathfrak{G}_i^m \otimes \mathfrak{H}_j^m \ar[uu]^{\bbb \otimes d_{j-2}} \ar[ur]^{\iota_{i,j,m}} \ar[rr]^>>{d_{i-2}\otimes \bbb} & \ar[u] & \mathfrak{G}_{i-1}^m\otimes \mathfrak{H}_j^m \ar[ur]^{\iota_{i-1,j,m}} & \ar[u] }
\end{displaymath}
which is part of a larger diagram for the inclusion map of double complexes $\iota_{*,*,m}\colon (\mathfrak{G}_*^m\otimes \mathfrak{H}_*^m,d\otimes d) \rightarrow (\mathfrak{Q}_{*,*}^m, d'+d'')$. In this case we consider the image $E:=\operatorname{im} d_{i,j-2}'' + \operatorname{im} d_{i-1,j}'$. Then we have (note that we can always take the closure or not as we please, cf.~Lemma \ref{lma:dimbyTr})
\begin{eqnarray}
\dim_{(L\tilde{G},\tilde{\psi})} \bar{E} & = & \dim_{(L\tilde{G},\tilde{\psi})} \overline{\operatorname{im} d_{i,j-2}''} + \dim_{(L\tilde{G},\tilde{\psi})} \overline{\operatorname{im} d_{i-1,j-1}'} - \dim_{(L\tilde{G},\tilde{\psi})} \overline{\operatorname{im} d_{i,j-2}''} \cap \overline{\operatorname{im} d_{i-1,j-1}'} \nonumber \\
 & = & \dim_{(LG,\psi_{\mu})} \operatorname{im} d_{i-1} \cdot \dim_{(LH,\psi_{\nu})} \mathfrak{H}_{j-1}^m + \dim_{(LG,\psi_{\mu})} \mathfrak{G}_i^m \cdot \dim_{(LH,\psi_{\nu})} \operatorname{im} \otimes d_{j-2} \nonumber \\ & & \quad - \dim_{(LG,\psi_{\mu})} \operatorname{im} d_{i-1} \cdot \dim_{(LH,\psi_{\nu})} \operatorname{im} d_{j-2}. \nonumber
\end{eqnarray}

Thus, considering $\mathfrak{Q}_n^m=\sum_{i=0}^{n}\mathfrak{Q}_{i+1,n+1-i}^m$ we know the $L\tilde{G}$-dimension of the image $E_n$ of $d'+d''\colon \mathfrak{Q}_{n+1}^m\rightarrow \mathfrak{Q}_n^m$ as well as that of the image $E_{n-1}$ of $d'+d''\colon \mathfrak{Q}_{n}^m\rightarrow \mathfrak{Q}_{n-1}^m$ in terms of $LG$- and $LH$-dimensions of images of boundary maps. We also know that
\begin{equation}
\dim_{(L\tilde{G},\tilde{\psi})} \mathfrak{Q}_n^m = \sum_{i=0}^{n} \dim_{(LG,\psi_{\mu})} \mathfrak{G}_{i+1}^m \cdot \dim_{(LH,\psi_{\nu})} \mathfrak{H}^m_{n+1-i}. \nonumber
\end{equation}

Now the theorem follows (by additivity, since everything is finite) from a direct computation of
\begin{equation}
\dim_{(L\tilde{G},\psi)} H_n(\mathfrak{Q}_*^m) = \dim_{(L\tilde{G},\tilde{\psi})} \mathfrak{Q}_n^m - \dim_{(L\tilde{G},\tilde{\psi})}E_n - \dim_{(L\tilde{G},\tilde{\psi})} E_{n-1}, \nonumber
\end{equation}
substituting the expressions above.
\end{proof}

\section{Products of totally disconnected groups and semi-simple groups}

The next theorem is an observation extending the the K{\"u}nneth formula in Theorem \ref{thm:Kunnethtotdisc}, by using the classical result that semi-simple Lie groups always contain cocompact lattices. This will be used in the next chapter to show that, at least in principle, the computation of $L^2$-Betti numbers of any {\lcsu} reduces to the computation of $L^2$-Betti numbers of some totally disconnected {\lcsu} group.

\begin{theorem} \label{thm:prodLH} \todo{thm:prodLH}
Let $G,H$ be {\lcsu} groups and suppose that $G$ is a connected semi-simple Lie group with finite centre, and that $H$ is totally disconnected. Then
\begin{equation}
\beta^n_{(2)}(G\times H, \mu\times \nu) = \sum_{k=0}^n\beta^{k}_{(2)}(G,\mu)\cdot \beta^{n-k}_{(2)}(H,\nu). \nonumber
\end{equation}
\end{theorem}

\begin{proof}
By \cite{BorelHC61,BorelHC62} $G$ contains a cocompact lattice $\Gamma$ whence for all $k$ we have $\beta^k_{(2)}(G,\mu) = \operatorname{covol}_{\mu}(\Gamma)^{-1}\cdot \beta^k_{(2)}(\Gamma)$. Thus the statement follows from the K{\"u}nneth formula for totally disconnected gorups, Theorem \ref{thm:Kunnethtotdisc}, applied to $\Gamma\times H$.
\end{proof}

\section{A comment on Kac-Moody groups}

Kac-Moody groups are constructed by Tits in \cite{Tits87}. In particular, given certain \emph{root datum} one can associate to any finite field $\mathbf{F}_q$ a locally compact group $G(\mathbf{F}_q)$, which contains a $BN$-pair, leading to a strongly transitive continuous action on a building $X$, with compact stabilizers. Further, $G$ is topologically simple \cite{CaRe09}, in particular unimodular. Hence we are well within the framework of Section \ref{sec:totdiscactions}. The dimension $\dim X$ of the building does not depend on the "ground field" $\mathbf{F}_q$.

Further, while it is highly non-trivial, and in many case apparently still an open problem to decide whether $G$ has any lattices \cite[Section 3.5]{FHT11}, it is known \cite{Re99} that the product $G\times G$ contains a lattice $\Gamma$, sometimes referred to as a Kac-Moody lattice. Important questions concerning such lattices concern their similarities and differences compared to $S$-arithmetic lattices.

When $X$ is in the class $\mathscr{B}+$ of \cite{DyJa02}, i.e.~when $q\geq \frac{42^{2\dim X}}{25}$, it follows directly from \cite[Theorem B and Corollary I]{DyJa02} that $\beta^n_{(2)}(G,\mu) \neq 0$ exactly for $n=\dim X$. We conclude:

\begin{theorem} \label{thm:KacMoodylattice} \todo{thm:KacMoodylattice}
Let $\Gamma$ be a Kac-Moody lattice in $G(\mathbf{F}_q)\times G(\mathbf{F}_q)$ for some complete Kac-Moody group $G(\mathbf{F}_q)$ with $q\geq \frac{42^{2\dim X}}{25}$ where $X$ is the building associated with $G(\mathbf{F}_q)$. Then
\begin{equation}
\beta^{n}_{(2)}(\Gamma) \left\{ \begin{array}{cl} = 0 & , n\neq 2\dim X \\ > 0 & , n = 2\dim X \end{array} \right. . \nonumber
\end{equation}
\end{theorem}
\begin{flushright}
\qedsymbol
\end{flushright}

\begin{remark}
By the simplicity results for Kac-Moody lattices in \cite{CaRe09}, Theorem \ref{thm:KacMoodylattice} leads to examples of finitely generated \emph{simple} discrete groups $\Gamma$ for which $\beta^{2n}_{(2)}(\Gamma) > 0$. As remarked in \cite{CaRe09}, it even happens that some of these are finitely presented.
\end{remark}

\chapter{Killing the amenable radical} \label{chap:Ramen}

In this chapter we extend Theorem \ref{thm:totdiscamenable} to cover all amenable {\lcsu} groups (Theorems \ref{thm:Ramen0noncompact} and \ref{thm:RamenGeneral}). Using the Hochschild-Serre spectral sequence we actually prove a stronger result, namely the vanishing of $L^2$-Betti numbers given that the amenable radical, i.e.~the maximal closed amenable normal subgroup, is non-compact. The vanishing of reduced $L^2$-cohomology for amenable {\lcsu} groups with non-compact amenable radical is joint with D. Kyed and S. Vaes \cite{JointDKSV}, established there by different means.

Having previously handled the totally disconnected case, what remains is to prove vanishing of $L^2$-Betti numbers for connected amenable {\lcsu} groups, and to stitch the two results together.

The first part uses the solution of Hilbert's fifth problem by Montgomery, Zippin, and others. See \cite{MoZi55}. This reduces the problem essentially to solvable Lie groups, and then through structure theory for Lie groups and the Hochschild-Serre spectral sequence, to $\mathbb{R}$. The proof of vanishing of $L^2$-Betti numbers of abelian groups (see Theorem \ref{thm:ltwoabelianlcgroups}) does not seem to apply here since we actually have to prove vanishing with coefficients in a larger Hilbert space - the $L^2$-space of the ambient group. Hence we reprove the result in this generality also. See Lemma \ref{lma:abelianbycompact}.

A vanishing result for reduced $L^2$-cohomology of connected solvable Lie groups in degree one was obtained, essentially, in \cite[Th{\'e}oreme V.6]{De77}, and the extension to connected amenable groups carried out in \cite{Ma04}.

The second part then stitches the two vanishing results together - that for connected and that for totally disconnected amenable groups - using the Hochschild-Serre spectral sequence in quasi-continuous cohomology.

On the way we obtain a "structural result" for $L^2$-Betti numbers of locally compact groups: it is known that, modulo the amenable radical, every {\lcsu} group is a direct product of a semi-simple connected Lie group with trivial centre, and a totally disconnected group; hence the K{\"u}nneth formula, Theorem \ref{thm:prodLH}, reduces the computation of $L^2$-Betti numbers to totally disconnected groups, at least in principle.

\section{Reduction to totally disconnected groups; structure theory} \label{sec:ramenreduction} \todo{sec:ramenreduction}

Let $G$ be a locally compact group. Given any family $\{ H_i\}_{i\in I}$ of amenable closed normal subgroups of $G$, the closed subgroup $H$ generated by the $H_i$ is normal, and one sees readily enough by the extension and continuity properties of amenability (see e.g.~\cite[Proposition G.2.2]{Tbook}) that $H$ is again amenable.

Thus $H$ is the largest closed, normal, amenable subgroup of $G$ - the amenable radical, and we denote in general this by $H=\operatorname{Ramen}(G)$. See also \cite{MonodThesis,FuMonod}

In general, it is known that for any locally compact group $G$, the quotient $G/\operatorname{Ramen}(G)$ has a finite index open, normal subgroup $G^*$ such that $G^* = L\times H$ with $L$ a semi-simple, connected Lie group with trivial center, and $H$ a totally disconnected group \cite[Theorem 11.3.4]{MonodThesis}. The same is true if one considers just the "connected" amenable radical, as we now show.

\begin{definition}
For any locally compact group $G$ we denote by $\operatorname{Ramen}_0(G):= \operatorname{Ramen}(G_0)$, where $G_0$ is the connected component of the identity in $G$, the "connected amenable radical" of $G$.
\end{definition}

Observe that since $G_0$ is a characteristic subgroup in $G$, and $\operatorname{Ramen}(G_0)$ characteristic in $G_0$, it follows that $\operatorname{Ramen}_0(G)$ is characteristic, in particular normal, in $G$.

\begin{remark}
The terminology "connected amenable radical" is an abuse of language since the amenable radical in a connected group need not be connected (it might very well be countable discrete), nor is there necessarily a connected amenable normal subgroup which is maximal among all such.
\end{remark}

The following is entirely analogous to \cite[Theorem 11.3.4]{MonodThesis}. I thank N. Monod for explaining this to me.

\begin{proposition} \label{prop:Ramen0} \todo{prop:Ramen0}
Let $G$ be a locally compact $2$nd countable group. Then the quotient $G/\operatorname{Ramen}_0(G)$ contains a finite index open, normal subgroup $G^{*,0}$, canonically isomorphic to a direct product $G^{*,0} = L\times H$ where $L$ is a semi-simple connected Lie group with trivial centre and $H$ is totally disconnected.
\end{proposition}

\begin{proof}
Denote $\tilde{G}:=G/\operatorname{Ramen}_0(G)$ and let $\kappa$ be the canonical projection $\kappa \colon G \rightarrow \tilde{G}$ onto this. Since $G/G_0$ is totally disconnected it follows that $\tilde{G}_0 = \kappa(G_0)$ (we write here $\tilde{G}_0$ for the connected component of the identity in $\tilde{G}$, instead of the more explicit $(\tilde{G})_0$), and we note that this has trivial centre. Indeed, if $\tilde{Z}$ is the centre of $\tilde{G}_0$, then we have a short exact sequence

\begin{displaymath}
\xymatrix{ 0\ar[r] & \operatorname{Ramen}(G_0) \ar[r] & \kappa^{-1}(\tilde{Z}) \ar[r] & \tilde{Z} \ar[r] & 0 }
\end{displaymath}
Since the outer groups are amenable, so is the middle group, contradicting the maximality of $\operatorname{Ramen}(G_0)$ unless $\tilde{Z}$ is trivial.

Applying the same argument twice more with $\tilde{Z}$ a compact normal subgroup, respectively the solvable radical, instead of the center, we conclude the $\tilde{G}_0$ is a Lie group (by the solution to Hilbert's fifth problem \cite{MoZi55}), respectively a semi-simple Lie group.

Consider the action of $\tilde{G}$ on its connected component by conjugation and denote the kernel of the induced homomorphism $K:=\ker (\tilde{G}\rightarrow \operatorname{Out}(\tilde{G}_0))$.

Let $k\in K$. Then there is a $g_0\in \tilde{G}_0$ such that for every $h\in \tilde{G}_0$, we have $khk^{-1}=g_0hg_0^{-1}$ whence $g_0^{-1}k \in Z_{\tilde{G}}(\tilde{G}_0)$. Thus
\begin{equation}
K = \tilde{G}_0\cdot Z_{\tilde{G}}(\tilde{G}_0) := \{ gh \mid g\in \tilde{G}_0, h\in Z_{\tilde{G}}(\tilde{G}_0)\}. \nonumber
\end{equation}

Further, since the center of $\tilde{G}_0$ is trivial, this is in fact a direct product, $K = \tilde{G}_0\times Z_{\tilde{G}}(\tilde{G}_0)$.

We observe that since $\tilde{G}_0$ has trivial center, the centralizer $Z_{\tilde{G}}(\tilde{G}_0)$ embeds into $\tilde{G}/\tilde{G}_0$ whence it is totally disconnected.

Finally $K$ has finite index in $\tilde{G}$ since $\tilde{G}_0$ is a semi-simple (centre-free, without compact factors) Lie group whence $\operatorname{Out}(\tilde{G}_0)$ is finite: essentially, outer automorphisms of the simple factors correspond to automorphisms of Dynkin diagrams. See the argument and references in the proof of \cite[Theorem 11.3.4]{MonodThesis}. Thus $G^{*,0}:=K$ works with $L=\tilde{G}_0$ and $H=Z_{\tilde{G}}(\tilde{G}_0)$. 
\end{proof}

\section{The amenable radical in a connected group}

We prove some auxiliary lemmas needed in the next section.

\begin{lemma} \label{lma:abelianbycompact} \todo{lma:abelianbycompact}
Let $H$ be a non-compact connected {\lcsu} group, and $G$ {\lcsu} group such that $G_0$ is a Lie group, and $H$ is a closed subgroup of $G$. Suppose that $H/K$ is an abelian Lie group for some compact normal subgroup $K$. Then
\begin{equation}
\dim_{(LG,\psi)} H^n(H,L^2G) = 0, \quad n\geq 0. \nonumber
\end{equation}
\end{lemma}

\begin{proof}
Let $K$ be a compact normal subgroup of $H$ such that $A:=H/K$ is an abelian Lie group, which is connected since $H$ is. Note that $A$ is non-compact, and that we may take it without compact factors, whence $A\simeq \mathbb{R}^m$ for some $m\geq 1$.

Then by the Hochschild-Serre spectral sequence for groups (see Theorem \ref{thm:HSgroups}) it is sufficient to show that $\dim_{(LG,\psi)} H^n(A,(L^2G)^{K}) = 0$ for all $n\geq 0$.

First we observe (by induction on $m$) that by the Hochschild-Serre spectral sequence for Lie algebras, see Theorem \ref{thm:HSLiealgebra} and Lemma \ref{lma:technicalsmooth} it is sufficient to prove the claim for $A=\mathbb{R}$. Since this has dimension one and is non-compact, we need in fact just show that $\dim_{(LG,\psi)} H^1(A,(L^2G)^{K}) = 0$.

Let $I\subseteq A$ be the closed unit interval $I=[0,1]$ and denote by $z$ the canonical generator ($z=1$) of $\mathbb{Z}\leq A$. To avoid confusion we write the group A multiplicatively, including its subgroups.

Let $\zeta\in C(A,(L^2G)^{K})$ be an inhomogeneous (continuous) cocycle. By \cite[Chapter III, Corollary 2.5]{Guichardetbook} and Sauer's local criterion, it is sufficient to show that there is a sequence $p_k$ of projections in $LG$, increasing to the identity, and such that $\zeta(\cdot).p_k$ is bounded in $L^2$-norm for all $k$.

We can write any element $x\in A$ as $x=rz^i, i\in \mathbb{Z}, r\in I$. Then we have
\begin{equation}
\zeta(x) = \left\{ \begin{array}{ll} \zeta(r)+r(z^{i-1}+\cdots +z+\bbb).\zeta(z) & , x=rz^i, i\geq 1 \\ \zeta(r) & ,i=0 \\ \zeta(r) - rz^{-i}(z^{i-1}+\cdots + z+\bbb).\zeta(z) & , x=rz^{-i}, i\geq 1 \end{array} \right. . \nonumber
\end{equation}

Since $\zeta(I)$ is bounded by compactness of $I$, it is thus sufficient to find our $p_k$ such that, for each $k$, on the range projection of $\zeta(z).p_k$, the operators $q_i:=z^i+\cdots +z+\bbb$ are bounded in operator norm, uniformly in $i$ (but possibly depending on $k$).

Denote by $\lvert \cdot \rvert_{L\mathbb{Z}}$ the absolute value $\lvert x\rvert_{L\mathbb{Z}} = (x^*x)^{\frac{1}{2}}$ in the ring of operators affiliated with the (finite) von Neumann algebra $L\mathbb{Z}$ (recall this is isomorphic to the ring of measurable complex-valued unctions on the unit circle). Then we have an inequality of unbounded affiliated operators
\begin{equation}
\lvert q_i \rvert_{L\mathbb{Z}} \leq 2\lvert(z-\bbb)^{-1}\rvert_{L\mathbb{Z}}, \quad i\geq 1. \nonumber
\end{equation}
Let $p_k'$ be a sequence of spectral projections of the right-hand side, increasing to the identity in $L\mathbb{Z}$, and such that $\lvert(z-\bbb)^{-1}\rvert_{L\mathbb{Z}} p_k'$ is bounded for all $k$. Then also $q_ip_k'$ is bounded in norm for all $k$, and for fixed $k$ the bound is uniform in $i$.

By polar decomposition we get an increasing sequence of source projections (i.e.~right supports) $p_k\in LG$ of $p_k'.\zeta(z)\in L^2G$, increasing to the right support $p_0$ of $\zeta(z)$. If this is not the identity, add the orthogonal complement $p_0^{\perp} = \bbb - p_0$ to each projection $p_k$. This completes the proof.
\end{proof}

\begin{lemma} \label{lma:countabletechnical} \todo{lma:countabletechnical}
Let $n\geq 0$ and let $\Gamma$ be a countable discrete subgroup of a {\lcsu} group $G$. If $\beta^n_{(2)}(\Gamma) = 0$ then
\begin{equation}
\dim_{(LG,\psi)} H^n(\Gamma,L^2G) = 0. \nonumber
\end{equation}
\end{lemma}

%%%%%%%%%%%%%%%%%%%%%%%%%%%%%%%%%%%%%%%%%%%%
\begin{comment}
\begin{proof}
Let first $\Gamma$ be countable discrete. By rank density arguments we may replace the coefficients by $LG$.

Let $P_*\rightarrow \mathbb{C} \rightarrow 0$ be a projective resolution of the trivial left-$\mathbb{C}\Gamma$-module $\mathbb{C}$. Then $H^n(\Gamma,LG)$ is the $-n$'th homology of the complex
\begin{displaymath}
\xymatrix{ 0 \ar[r] & \operatorname{hom}_{\mathbb{C}\Gamma}(P_0, LG) \ar[r] & \cdots \ar[r] & \operatorname{hom}_{\mathbb{C}\Gamma}(P_i,LG)\ar[r] & \cdots }
\end{displaymath}
But $LG\simeq \operatorname{hom}_{L\Gamma}(L\Gamma,LG)$ as $\mathbb{C}\Gamma$-$LG$-modules. (On the right-hand side, $LG$ is acting by post-multiplication from the right, and $\Gamma$ by $(\gamma.f)(T) = f(T\gamma)$.)

Hence we have $\operatorname{hom}_{\mathbb{C}\Gamma}(P_i,LG)\simeq \operatorname{hom}_{L\Gamma}(L\Gamma\otimes_{\mathbb{C}\Gamma} P_i, LG)$ and one checks readily that the boundary maps $d\colon P_*\rightarrow P_*$ induce, the boundary maps on the complex coming from $\bbb\otimes d\colon L\Gamma\otimes_{\mathbb{C}\Gamma}P_* \rightarrow L\Gamma\otimes_{\mathbb{C}\Gamma}P_*$.

But we know that $\beta_n^{(2)}(\Gamma)=0$ as well, by \ref{thm:dualitydiscreteall}, whence by the dimension-exactness of the functor $\operatorname{hom}_{L\Gamma}(-,LG)$, see Theorem \ref{thm:homdimexactness}, and the fact that we may take the $P_i$ countable generated, we get the desired conclusion.
\end{proof}
\end{comment}
%%%%%%%%%%%%%%%%%%%%%%%%%%%%%%%%%%%%%%%%%%

\begin{proof}
By Porism \ref{por:dualitytotdisc} it is sufficient to show that $\dim_{LG} H_n(\Gamma, LG) = 0$. Let $(P_*,d_*)\rightarrow \mathbb{C}\rightarrow 0$ be a projective resolution of the trivial $\Gamma$-module $\mathbb{C}$. Then $H_n(\Gamma, LG)$ is the $n$'th homology of the complex
\begin{displaymath}
\xymatrix{ \cdots \ar[r]^>>>>>>>>>{\bbb\otimes d} & LG\otimes_{\mathbb{C}\Gamma}P_n \ar@{=}[d] \ar[r]^<<<<<<<<<{\bbb\otimes d} & \cdots \ar[r] & LG\otimes_{\mathbb{C}\Gamma}P_0 \ar@{=}[d] \ar[r] & 0 \\ \cdots \ar[r]^>>>>>>{\bbb\otimes \bbb\otimes d} & LG\otimes_{L\Gamma}L\Gamma\otimes_{\mathbb{C}\Gamma} P_n \ar[r]^<<<<<<{\bbb\otimes \bbb\otimes d} & \cdots \ar[r] & LG\otimes_{L\Gamma}L\Gamma\otimes_{\mathbb{C}\Gamma}P_0 \ar[r] & 0 } .
\end{displaymath}

The statement then follows by the hypothesis $\beta_n^{(2)}(\Gamma) = 0$ (see Theorem \ref{thm:dualityalldiscrete}) and the dimension-exactness of the induction functor $LG\otimes_{L\Gamma} -$, see Theorem \ref{thm:tensordimexactness}.
\end{proof}

\begin{lemma}
Let $G_0$ be a connected {\lcsu} group, let $K$ be the maximal compact normal subgroup, and denote $\kappa\colon G_0\rightarrow G_0/K$ the canonical projection. Then $\kappa(\operatorname{Ramen}(G_0)) = \operatorname{Ramen}(G_0/K)$ has connected component which is either trivial or non-compact.
\end{lemma}

\begin{proof}
Indeed, the connected component $\operatorname{Ramen}(G_0/K)$ is a characteristic subgroup of $G_0/K$. In particular, it is normal, whence it is either trivial or non-compact by maximality of $K$.
\end{proof}

\section{The vanishing result}

We now use the Hochschild-Serre spectral sequence in various guises, and structure theory for locally compact groups and Lie groups, to deduce our main result, vanishing of reduced $L^2$-cohomology for {\lcsu} groups with non-compact, normal amenable subgroups, from the results of previous sections. As mentioned in the chapter introduction, this result appears in joint work with D. Kyed and S. Vaes \cite{JointDKSV}, esteblished there by different means.

\begin{theorem}[(see also {\cite{JointDKSV}})] \label{thm:Ramen0noncompact} \todo{thm:Ramen0noncompact}
Let $G$ be a {\lcsu} group. If $\operatorname{Ramen}_0(G)$ is non-compact, then
\begin{equation}
\beta^n_{(2)}(G,\mu) = 0, \quad \textrm{ for all } n\geq 0. \nonumber
\end{equation}
\end{theorem}

\begin{proof}
It is sufficient, by Theorem \ref{thm:compactnormal} to prove the theorem for $G/K$ where $K$ is the maximal compact normal subgroup of $G_0$, the connected component of the identity in $G$.

Hence we may assume that $G_0$ is a connected Lie group with no compact normal subgroups. By the previous lemma, $\operatorname{Ramen}_0(G)$ then is either (countably-)infinite discrete, or it contains a characteristic, non-compact, amenable, connected Lie group $H$, which is then normal in $G$ since the connected amenable radical is characteristic. Thus the proof splits in two cases:

\begin{enumerate}[(i)]
\item Suppose first that $\Gamma:=\operatorname{Ramen}_0(G)$ is discrete. Then by the "mixed case" Hochshild-Serre spectral sequence (Theorem \ref{thm:HSmixed}), Lemma \ref{lma:countabletechnical}, and the countable annihilation lemma, we see that for all $n\geq 0$,
\begin{equation}
\dim_{(LG,\psi)} H^n(\mathfrak{k}\subseteq \mathfrak{g}_0, L^2G^{(\infty,G_0)}) =0. \nonumber
\end{equation}
(Caveat: Here $\mathfrak{k}$ is the Lie algebra of a maximal compact subgroup in $G_0$. Such a subgroup need of course not be normal.)

Since $G_0$ is connected, this gives directly by the van Est theorem that $\dim_{(LG,\psi)} H^n(G_0,L^2G) = 0$.

Finally, by the Hochschild-Serre spectral sequence in quasi-continuous cohomology, Theorem \ref{thm:QCHSSS}, the statement now follows in this case. Alternatively one can avoid quasi-continuous cohomology and proceed by a stadnard double-complex argument, see Section \ref{sec:QCcrutch}.

\item In the second case, $H:=\operatorname{Ramen}_0(G)_0$ is a non-compact connected amenable Lie group, normal in $G$. Hence $H$ has non-compact solvable radical by \cite[Theorem 14.3]{PierAmenable}, which is then also normal in $G$, so we assume without loss of generality that $H$ is solvable.

Let $\{\bbb\} = H^{(m)} \unlhd H^{(m-1)} \unlhd H^{(m-2)} \unlhd \cdots \unlhd H^{(1)}=H' \unlhd H^{(0)}=H$ be the derived series of $H$. Note that each $H^{(k)}$ is connected. Let $k_0$ be the minimal $k_0=k$ such that $H^{(k)}$ is compact. Then $H^{(k_0-1)}$ satisfies the hypotheses of Lemma \ref{lma:abelianbycompact} whence we conclude that
\begin{equation}
\dim_{(LG,\psi)} H^n(H^{(k_0-1)}, L^2G) = 0, \quad \textrm{ for all } n\geq 0. \nonumber
\end{equation}
By the connectedness of $H^{(k_0-1)}$ this is the same as $\dim_{(LG,\psi)} H^n ( \mathfrak{k}^{(k_0-1)}\subseteq \mathfrak{h}^{(k_0-1)}, L^2G^{(\infty,G_0)}) = 0$. See Lemma \ref{lma:technicalsmooth}.

then using the Hochschild-Serre spectral sequence for Lie algebras, see Theorem \ref{thm:HSLiealgebra}, $k_0$ times (formally this should of course be structured as an induction argument on $k_0$), along with the countable annihilation lemma, it follows that
\begin{equation}
\dim_{(LG,\psi)} H^n(\mathfrak{k}\subseteq \mathfrak{g}_0, L^2G^{(\infty,G_0)} ) = 0, \quad \textrm{ for all } n\geq 0. \nonumber
\end{equation}

Since $G_0$ is connected we get now by the van Est theorem
\begin{equation}
\dim_{(LG,\psi)} H^n(G_0,L^2G) = 0. \nonumber
\end{equation}

As in case (i) we finish the proof of this case with an appeal to the Hochschild-Serre spectral sequence in quasi-continuous cohomology, see Theorem \ref{thm:QCHSSS}.
\end{enumerate}
The two cases having now been proved, we are done.
\end{proof}

\begin{proposition} \label{prop:elltwofinidx} \todo{prop:elltwofinidx}
Let $H\leq G$ be a finite index inclusion of {\lcsu} groups. Then, letting $\nu$ be a fixed Haar measure on $H$ and $\mu$ the Haar measure on $G$ such that $L^2(G,\mu) \simeq L^2(H,\nu)\otimes \ell^2(G/H,\frac{\bbb}{[G:H]}$, we have
\begin{equation}
\dim_{(LG,\psi)} E = \dim_{(LH,\psi)} E, \nonumber
\end{equation}
for any $LG$-module $E$, where $\psi=\psi_{\mu}=\psi_{\nu}\otimes \frac{1}{[G:H]}\cdot \operatorname{Tr}$ is the canonical tracial weight on $LG$.

In particular, $\beta^n_{(2)}(G,\mu) = \beta^n_{(2)}(H,\nu)$ for all $n\geq 0$. \hfill \qedsymbol
\end{proposition}

We leave out the easy verification. We also note that a finite index subgroup is automatically open whence {\lcsu}, given that the ambient group is so.

We can now combine all this with Corollary \ref{cor:totdiscamenablegen} to show:

\begin{theorem}[(see also {\cite{JointDKSV}})] \label{thm:RamenGeneral} \todo{thm:RamenGeneral}
Let $G$ be a totally disconnected {\lcsu} group. If $G$ has non-compact amenable radical, then $\beta^n_{(2)}(G,\mu) = 0$ for all $n\geq 0$.

On the other hand, if $\operatorname{Ramen}(G)$ is compact, then $G/\operatorname{Ramen}(G)$ contains an open finite index subgroup $G^*$ which is isomorphic to a direct product of a connected semi-simple Lie group $L$, with trivial centre, and a totally disconnected group $H$ (both groups {\lcsu}), and
\begin{equation}
\beta^n_{(2)}(G,\mu^*) = \sum_{k=0}^n \beta^k_{(2)}(L,\mu)\cdot \beta^{n-k}_{(2)}(H,\nu), \nonumber
\end{equation}
where $\mu^*$ is the Haar measure on $G$ pulled back from the Haar measure on $G^*$, declaring this to have covolume one in $G/\operatorname{Ramen}(G)$.

In particular, if $G$ is non-compact amenable, then $\underline{H}^n(G,L^2G) = 0$ for all $n\geq 0$.
\end{theorem}

\begin{proof}
If $\operatorname{Ramen}_0(G)$ is non-compact, the claim follows from the previous Theorem \ref{thm:Ramen0noncompact}.

So suppose that $\operatorname{Ramen}_0(G)$ is compact. Then by Proposition \ref{prop:Ramen0} we have $G/\operatorname{Ramen}_0(G) \unrhd G_1 \simeq L\times H_1$ with $L$ as advertised, $H_1$ a totally disconnected {\lcsu} group, and $G_1$ having finite index. Hence by Theorem \ref{thm:prodLH} we have
\begin{equation}
\beta^n_{(2)}(G_1,\mu \times \nu) = \sum_{k=0}^n \beta^k_{(2)}(L,\mu)\cdot \beta^{n-k}_{(2)}(H_1,\nu). \nonumber
\end{equation}
Hence the claim follows, if $\operatorname{Ramen}(G)$ is compact, by Theorem \ref{thm:compactnormal} and Proposition \ref{prop:elltwofinidx}.

Finally, $\operatorname{Ramen}(G_1) = \bbb \times \operatorname{Ramen}(H_1)$ and so if $\operatorname{Ramen}(G)$ is non-compact, then $\operatorname{Ramen}(H_1)$ is non-compact. Hence we need to show that in this case, all the $L^2$-Betti numbers of $H_1$ vanish. But this follows by Corollary \ref{cor:totdiscamenablegen} and the Hochschild-Serre spectral sequence in quasi-continuous cohomology, Theorem \ref{thm:QCHSSS}.

If $\operatorname{Ramen}(H_1)$ is compact then we take $H:=H_1/\operatorname{Ramen}(H_1)$.

The very last claim follows by Proposition \ref{prop:reducedelltwolimit}.
\end{proof}

\section{Avoiding quasi-continuous cohomology} \label{sec:QCcrutch} \todo{sec:QCcrutch}

This section provides a work-around to the use of the Hochschild-Serre spectral sequence in quasi-continuous cohomology. It is not different in an essential way, but since we only use the spectral sequence for vanishing results, it is possible, if one wishes, to replace it by a more direct double complex argument as follows

\begin{proposition} \label{prop:QCcrutch} \todo{prop:QCcrutch}
Let $G$ be a {\lcsu} group, and $H$ a closed normal subgroup in $G$ such that the quotient $Q:=G/H$ is totally disconnected.

If $\dim_{(LG,\psi)} H^n(H,L^2G) = 0$ for all $n\geq 0$ then $\beta^n_{(2)}(G,\mu) = 0$ for all $n\geq 0$.
\end{proposition}

\begin{proof}
Fix a compact open subgroup $K$ of $Q$. We consider a first quadrant double complex of $LG$-modules (that is, we consider everything purely algebraically, with no regard for the topologies involved) given by
\begin{equation}
K^{p,q} := \mathcal{F}( (Q/K)^{p+1}, C(G^{q+1},L^2G)^H )^{Q}, \nonumber
\end{equation}
endowed with coboundary maps $d'\colon K^{p,q}\rightarrow K^{p+1,q}$ which are the usual coboundary map on the bar resolution of the topological $Q$-$LG$-module $C(G^{q+1},L^2G)^H$ of Section \ref{sec:bartotdisc}, and $d''\colon K^{p,q}\rightarrow K^{p,q+1}$ which applies the coboundary map $d_H\colon C(G^{q+1},L^2G)^H \rightarrow C(G^{q+2},L^2G)^H$ pointwise.

Then there is a spectral sequence ${}'E_2^{p,q}$ of $LG$-modules abutting to the total cohomology, and the $E_2$-terms are (as $LG$-modules)
\begin{equation}
{}'E_2^{p,q} = \left\{ \begin{array}{cl} H^q(G,L^2G) & , p=0 \\ 0 & ,p>0 \end{array} \right. , \nonumber
\end{equation}
which follows as usual by the fact that $C(G^{p+1},L^2G)^H$ is an injective topological $Q$-$LG$-module, whence the cohomology, which is computed by the bar resolution we consider here, vanishes when $p>0$.

Thus this spectral sequence collapses, and the total cohomology of our complex is just the $L^2$-cohomology of $G$, as $LG$-modules. (In fact as topological spaces as well, but we won't use this.)

On the other hand there is a spectral sequence ${}''E_2^{p,q}$ also abutting to the total cohomology, where the $E_2$-terms are given as the cohomology of complexes $H^{*,q}(d'')$. Hence it is sufficient to show that $\dim_{LG} H^{p,q}(d'') = 0$ for all $p,q$. Let $\xi\in \ker d''\cap K^{p,q}$. Then clearly $\xi(x_0,\dots ,x_p) \in \ker d_H$ for all $x_0,\dots ,x_p\in Q/K$.

First we construct a linear section
\begin{equation}
s\colon \mathcal{F}((Q/K)^{p+1}, d_H(C(G^{q+1},L^2G)^H))^{Q} \rightarrow \mathcal{F}((Q/K)^{p+1},C(G^{q+1},L^2G)^H)^Q. \nonumber
\end{equation}

Indeed, just take any linear (not necessarily continuous, $LG$-linear) section $s_0$ of the couboundary map $d_H\vert_{C(G^{q+1},L^2G)^H}$ from its image $d_H(C(G^{q+1},L^2G)^H)$. Then fix a fundamental domain $Z$ for the action of $Q$ on $(Q/K)^{p+1}$. Given $\xi\in \mathcal{F}((Q/K)^{p+1}, d_H(C(G^{q+1},L^2G)^H))^{Q}$ we define $s(\xi)$ on $Z$ by
\begin{equation}
s(\xi)(x_0,\dots ,x_p) = \frac{1}{\mu(Q_{(x_0,\dots,x_p)})}\int_{Q_{(x_0,\dots,x_p)}} x.s_0(\xi(x_0,\dots,x_p))\mathrm{d}\mu(x), \nonumber
\end{equation}
where $\mu$ is the Haar measure on $Q$ and $Q_{(x_0,\dots,x_p)}$ is the stabilizer of the point $(x_0,\dots ,x_p)\in Z$. Then $s(\xi)\vert_Z$ extends (uniquely) to a $Q$-invariant function which we call $s(\xi)$. It follows that the image of $d''$ is exactly $\mathcal{F}((Q/K)^{p+1}, d_H(C(G^{q+1},L^2G)^H))^{Q}$. (Note: it is not really important that $s$ be linear.)

By hypothesis and the countable annihilation lemma, the inclusion of $LG$-modules
\begin{equation}
\mathcal{F}((Q/K)^{p+1}, d_H(C(G^{q+1},L^2G)^H))^{Q} \leq \mathcal{F}((Q/K)^{p+1}, \ker d_H \cap C(G^{p+1},L^2G)^H)^{Q} \nonumber
\end{equation}
is rank dense. Hence we have indeed by the local criterion $\dim_{LG}H^{p,q}(d'') = 0$ for all $p,q\geq 0$.
\end{proof}

 % THE IS ALSO OK

\appendix
\chapter{Preliminaries on groups and von Neumann algebras} \label{app:prelims}

In this chapter we recall some basic notation and results in group theory and operator algebras. We give no proofs. Most of the results are standard and so we do not explicitly give attributions.

\section{Locally compact groups}

By a locally compact group we mean a group $G$ with a locally compact Hausdorff topology in which the group operations are continuous. A general reference on locally compact groups is \cite{MoZi55}. Generally we will assume, unless explicitly mentioned otherwise, that $G$ is $2$nd countable, i.e.~that its topology is generated by a countable family of open sets. Every such group has a left- and a right-\emph{Haar measure}, that is, left- respectively right-translation invariant positive Radon measures on the Borel $\sigma$-algebras. These two measures are determined uniquely up to scaling by the properties of being left- respectively right-invariant. We may sometimes abuse language and say e.g.~"the left-invariant Haar measure".

A locally compact group is called unimodular if its left- and right-Haar measures coincide. We abbreviate 'locally compact $2$nd countable unimodular' by '\lcsu'.

Given a left-invariant Haar measure $\mu$ on the locally compact group $G$, we can define for every $g\in G$ a measure $\mu_g$ on $G$ by $\mu_g(X) := \mu(Xg)$ for $X\subseteq G$ any Borel set. Clearly $\mu_g$ is left-invariant whence there exists a positive real number $\Delta_G(g)$ such that $\mu_g(X) = \Delta_G(g)\cdot \mu(X)$ for all Borel subsets $X$. The map $G\owns g\mapsto \Delta_G(g)$ does is independent of the choice of $\mu$, and is a continuous homomorphism into $(\mathbb{R}_+,\cdot)$, called the modular function. In particular we observe that any (topologically) simple group $G$ is unimodular.

Observe that a closed subgroup of a unimodular locally compact group need not be unimodular. In fact, any locally compact group $G$ embeds in a cross-product $\mathbb{R}\rtimes G$ which is unimodular. (The action is in fact induced by the modular function in order to force this.)

However, any \emph{open} subgroup is (automatically also closed and) unimodular if the ambient group is. So is any closed, \emph{normal} subgroup.

Every locally compact $2$nd countable group is a complete separable metrizable space with a right-invariant metric. (This is due to Birkhoff and Kakutani, see \cite[Theorem 1.22]{MoZi55}.)

By the Baire category theorem one then concludes the following result, allowing us to talk about short exact sequences of groups in a natural way.

\begin{proposition}
Let $\iota \colon H\rightarrow G$ be an injective, continuous homomorphism of $2$nd countable groups. If the image $\iota(H)$ is closed in $G$ then $\iota$ is a homeomorphism onto its image. {\phantom{aa} \hfill \qedsymbol}
\end{proposition}

The next theorem summarizes some useful facts about quotients.

\begin{theorem} \label{thm:grpsection} \todo{thm:grpsection}
Let $G$ be a $2$nd countable locally compact group and $H$ a closed subgroup. Then the quotient $G/H$ is a locally compact Hausdorff space, and there is a bounded Borel section (i.e.~mapping compact sets to relatively compact sets) of the canonical projection $G\rightarrow G/H$.

Further, the left-invariant Haar measure on $G$ induces a left-invariant measure on $G/H$ exactly when $\Delta_G(h)=\Delta_H(h)$ for all $h\in H$, and such a measure is always unique up to scalar multiplication.
\end{theorem}

For a general result on the existence of a bounded section, see \cite{KehletSection}. In the case where $H$ is countable discrete and $G$ is $2$nd countable there is an easier argument, see e.g.~\cite[Proposition B.2.4]{Tbook}.

We define the \emph{covolume} of $H$ as the $\operatorname{covol}_{\nu}(H) = \nu(G/H)$ whenever there is a left-invariant measure $\nu$ on the quotient $G/H$. When $H$ is discrete in $G$ (the only case we consider in the present text) $\operatorname{covol}_{\mu}(H) := \mu(s_r(G/H))$ where $s_r$ is a section as in the theorem, and $\mu$ a left-invariant Haar measure on $G$. Recall that if $G$ has a discrete subgroup with finite covolume then $G$ is unimodular \cite[Proposition B.2.2(ii)]{Tbook}.

We recall also the solution to Hilbert's fifth problem by Montgomery, Zippin, and Gleason. A topological group $G$ is called almost connected if the quotient $G/G_0$ is compact (in its quotient topology), where $G_0$ is the connected component of the identity in $G$.

\begin{theorem}[(~\cite{MoZi55})]
Let $G$ be an almost compact $2$nd countable locally compact group. Then for any neighbourhood $U$ of the identity in $G$, there is a compact, normal subgroup $K\subseteq U$ in $G$ such that $G/K$ is a Lie group.
\end{theorem}

In particular, this implies that any connected, $2$nd countable locally compact group contains a maximal compact normal subgroup, and the quotient is then a Lie group.

Recall also that if $G$ is a totally disconnected, $2$nd countable locally compact group, then $G$ contains a decreasing sequence of compact open subgroups, forming a neighbourhood basis of the identity.

\section{von Neumann algebras}

A good general textbook on von Neumann algebras is \cite{KRI,KRII}. A more comprehensive reference for specialists is \cite{TakesakiI,TakesakiII}

A von Neumann algebra is a $*$-subalgebra $\mathscr{A}$ of $\mathcal{B(H)}$, the algebra of bounded operators on a Hilbert space $\mathcal{H}$, which is closed in weak operator topology (WOT), equivalently in strong operator topology (SOT). By von Neumann's density theorem this is equivalent to $\mathscr{A}'' = \mathscr{A}$, where $\mathscr{A}'' := (\mathscr{A}')'$ is the double commutant. (By definition the commutant of a set $S$ of operators is $S' := \{ y\in \mathcal{B(H)} \mid \forall x\in S : yx=xy\}$.)

von Neumann algebras were introduced by Murray and von Neumann in a series of papers \cite{MvN1,MvN2,MvN3,MvN4}, then called rings of operators. One of the main early results was a classification into types. We say that a von Neumann algebra $\mathscr{A}$ is finite if it has a faithful (see below) finite trace $\tau\colon \mathscr{A} \rightarrow \mathbb{C}$, i.e.~a positive linear functional satisfying $\tau(xy)=\tau(yx)$ for all $x,y\in \mathscr{A}$.

We say that $\mathscr{A}$ is semi-finite if it has a faithful semi-finite tracial weight, i.e.~a weight $\psi\colon \mathscr{A}_+ \rightarrow [0,\infty]$ such that the ideal $\mathscr{A}_{\psi}^2$ of "square integrable" operators $x\in \mathscr{A}$ satisfying $\psi(x^*x)<\infty$ is weakly dense in $\mathscr{A}$. The weight $\psi$ is extended by linearity to the subspace of $\mathscr{A}$ spanned by positive elements on which $\psi$ is finite.

Finally, $\mathscr{A}$ is purely infinite if it has no non-trivial trace.

Recall that a trace $\psi$ is faithful if $\psi(x^*x) = 0$ implies that $x=0$, and normal if it is a sum of positive functionals which are ultraweakly continuous on the unit ball of $\mathscr{A}$. We call the pair $(\mathscr{A},\psi)$ a (semi-)finite tracial algebra if $\mathscr{A}$ is a von Neumann algebra and $\psi$ a (semi-)finite, faithful normal tracial weight. Unless explicitly stated otherwise, $\mathscr{A}$ is always assumed $\sigma$-finite, i.e.~any set of pairwise orthogonal non-zero projections is countable.

Let $G$ be a locally compact group and $\mu$ a left-Haar measure on $G$. The representation $\lambda\colon G\rightarrow \mathcal{B}(L^2(G,\mu))$ given by
\begin{equation*}
(\lambda(g).\xi)(h) = \xi(g^{-1}h) \nonumber
\end{equation*}
is called the left-regular representation. The closure $\overline{\span} \{ \lambda(g) \mid g\in G \}$ of the linear span of $\lambda(G)$ in the weak-operator topology is the \emph{group von Neumann algebra} $LG$ of $G$.

If $G$ is unimodular then $LG$ carries a canonical tracial weight \cite[Theorem 7.2.7]{storegert} $\psi$, which is faithful, normal, and semi-finite. One can construct this $\psi$ by taking a sequence $(\Psi_n)\subseteq L^1G$ such that
\begin{equation}
\rho(\Psi_n) \nearrow_n \bbb \quad \textrm{and} \; \; \lambda(\Psi_n) \xrightarrow{\textrm{SOT}}_n \bbb, \nonumber
\end{equation}
and then finding $\xi_n\in L^2G$ such that $\xi_n * \tilde{\xi_n} = \Psi_n-\Psi_{n-1}$, where $\tilde{\xi_n}(t) := \overline{\xi_n(t^{-1})}$. Then $\psi = \sum_{n=1}^{\infty} \langle \cdot \; \xi_n, \xi_n \rangle$, the sum of vector states, and $\psi$ is charaterized by $\psi(x^{*}x)<\infty$  if and only if there is some left-bounded $f\in L^2G$ such that $x = \lambda(f)$, in which case $\psi(x^{*}x) = \lVert f\rVert_2^{2}$.

The GNS construction $L^2\psi$ and associated representation of $LG$, with respect to the canonical trace $\psi$ on the {\lcsu} group $G$ is spatially equivalent to $L^2G$. In particular, $LG^2_{\psi}\subseteq L^2G$ and one checks that the right-action of $LG$ on the two-sided ideal $LG^2_{\psi}$ extends by continuity to a right-action on $L^2G$ in this way. Hence we often think as $L^2G$ as a right-$LG$-module and a $G$-$LG$-module in this way. An alternative way to say this is that it is well-known that $LG$ is anti-isomorphic to its commutant on $LG^2_{\psi}$ which, by the spatial equivalence, is isomorphic to the von Neumann algebra $RG := \overline{\span}\{ \rho(g) \mid g\in G \}$ where $\rho$ is the \emph{right-regular} representation of $G$ on $L^2G$. Hence we may think if the right-action of $LG$ on $L^2G$ in the natural manner, as the extension of the convolution action of $L^1G$ from the right.

We refer to standard textbooks for a treatment of functional calculus and spectral theory, but let us state the bare minimum:

\begin{theorem*}
Let $\mathscr{A}$ be a von Neumann algebra and $A\in \mathscr{A}$ a self-adjoint operator. Then the spectrum $\operatorname{sp}(A) \subseteq \mathbb{R}$ and there is a (essentially unique) $*$-homomorphism from the algebra of bounded Borel functions into $\mathscr{A}$ .
\end{theorem*}

In particular one has for any interval $I\subseteq \operatorname{sp}(A)$ a projection $e_I\in \mathscr{A}$, called a spectral projection of $A$. As $I$ increases to all of $\mathbb{R}$, the spectral projections increase to the identity in $\mathscr{A}$.

More generally the theorem holds if $A$ is any closed, densely defined (self-adjoint) operator \emph{affiliated} with $\mathscr{A}$. In particular, any function $\xi\in L^2G$ acts by (say, right-) convolution on $LG^2_{\psi}$, which we still think of as a subspace of $L^2G$, and as such is an affiliated operator. Hence there are spectral projections $e$ of $\xi$ in $\mathscr{A}$ increasing to the identity and such that $e.\xi = \xi.e$ is a bounded convolution operator. For convenience, observe that by the functional calculus we have:

\begin{proposition*}
Let $\psi$ be a normal, faithful, semi-finite trace on $\mathscr{A}$. Then any projection $p\in \mathscr{A}$ has a subprojection in $\mathscr{A}^2_{\psi}$. In particular, there is a set $(p_i)_{i\in I}$ of pairwise orthogonal projections $p_i\in \mathscr{A}^2_{\psi}$ such that $\bbb = \sum_{i\in I} p_i$.
\end{proposition*}

\chapter{Extended von Neumann dimension for semi-finite traces} \label{app:dimension}
%\epigraph{No man can step in the same river twice, for it is not the same river, and he is not the same man.}{Heraclitus}

In this chapter we generalize the framework of L{\"u}ck's extended von Neumann dimension to cover also the case of modules over a semi-finite von Neumann algebra with a fixed tracial weight.

L{\"u}ck introduced the extended von Neumann dimension and its applications to $L^2$-invariants in a series of papers \cite{Lu98I,Lu98II} (see also his comprehensive book \cite{Lu02}). This dimension function extends the usual von Neumann dimension for Hilbert modules over a finite tracial von Neumann algebra to a dimension function on any \emph{purely algebraic} module over the von Neumann algebra as a ring.

Using this, one is able to apply homological algebra methods directly to the study of $L^2$-Betti numbers of discrete groups. See also Section \ref{sec:history}.

In Section \ref{sec:dimsemifinite} we extend L{\"u}ck's definition to the case of modules over a semi-finite von Neumann algebra. For the convenience of the reader, we summarize the important properties which we will need in the theorem below.

\begin{theorem}[(compare {\cite[Theorem 6.7]{Lu02}})] \label{thm:dimensionsummary} \todo{thm:dimensionsummary}
Let $\mathscr{A}$ be a $\sigma$-finite, semi-finite von Neumann algebra and let $\psi$ be a faithful, normal, semi-finite tracial weight on $\mathscr{A}$. Then to any (right-)$\mathscr{A}$-module $M$ we can associate an extended positive real number, the $\psi$-dimension $\dim_{\psi} M$. On the category of (right-)$\mathscr{A}$-modules, this has the following properties:
\begin{enumerate}[(i)]
\item Extension of von Neumann dimension. (See Lemma \ref{lma:diminfinitesum}.)

If $p$ is a projection in $\mathscr{A}$, then $\dim_{\psi} p(L^2\psi) = \psi(p)$.

\item Additivity. (See Theorem \ref{thm:projectivedimadditivity}.)

For any short exact sequence of $\mathscr{A}$-modules
\begin{displaymath}
\xymatrix{ 0 \ar[r] & M \ar[r] & N \ar[r] & Q \ar[r] & 0 }
\end{displaymath}
the $\psi$-dimensions satisfy $\dim_{\psi}N = \dim_{\psi} M + \dim_{\psi} Q$.

\item Inductive limits. (Proof is verbatim as in \cite[Theorem 6.13]{Lu02}.)

Let $((E_i)_{i\in I},\phi_{ij})$ be an inductive system of $\mathscr{A}$-modules with connecting maps $\phi_{ij} \colon E_i\rightarrow E_j$, and suppose that for each $i\in I$ there is a $j\geq i$ such that $\dim_{\psi}\operatorname{im} \phi_{ij} < \infty$. Then
\begin{equation}
\dim_{\psi} \lim_{\rightarrow} E_i = \sup_{i\in I} \inf_{j:j\geq i} \dim_{\psi} \operatorname{im} \phi_{ij}. \nonumber
\end{equation}

\item Projective limits. (See Theorem \ref{thm:dimfinality}.)

Let $(\{E_i\}_{i\in I}, \{\phi_{ij}\}_{i,j\in I})$ be a projective system of $\mathscr{A}$-modules with connecting maps $\phi_{ij} \colon E_j\rightarrow E_i$, and suppose that there is a subsequence $(i_k)_{k\in \mathbb{N}}$ of indices such that for all $i\in I$, $i\leq i_k$ for some $k$, and that $\dim_{\psi}E_{i_k} < \infty$ for all $k$. Then
\begin{equation}
\dim_{\psi} \lim_{\leftarrow} E_i = \sup_{i\in I} \inf_{j:j\geq i} \dim_{\psi} \operatorname{im} \phi_{ij}. \nonumber
\end{equation}

\item Compression / reciprocity. (See Theorem \ref{thm:compression}.)

Let $p$ be a projection in $\mathscr{A}$ with central support the identity. Consider the semi-finite tracial algebra $(\mathscr{A}_p,\psi_p) = (p\mathscr{A}p,\psi(p\cdot p))$. For any right-$\mathscr{A}$-module $M$ we have for the right-$\mathscr{A}_p$-module $Mp$
\begin{equation}
\dim_{(\mathscr{A},\psi)} M = \dim_{(\mathscr{A}_p,\psi_p)} Mp. \nonumber
\end{equation}

\end{enumerate}
\end{theorem}

In section \ref{sec:rankdensity} we consider the notion of rank density in the semi-finite setting, following \cite{Sau03,Thom06a,Thom06b}. Unlike the finite case there seems to be no suitable notion of rank metric and -completion, but the local criterion for vanishing of dimension is still a key technical tool. It shows in particular that the vanishing of dimension is a purely algebraic property; that is, independent of the particular trace chosen on the von Neumann algebra.

We continue this discussion with some remarks about rank completion in the finite case in Section \ref{sec:rankhahnbanach}. In particular, we offer an alternative approach to "dimension exactness" results for hom- and induction functors, based on an extension theorem reminiscent of the Hahn-Banach theorem from functional analysis. Using the compression property, Theorem \ref{thm:dimensionsummary}(v), we are then able to extend the dimension exactness results to the semi-finite case.

When we consider $L^2$-Betti numbers, this will allow us to transfer methods of proof from countable discrete groups to totally disconnected groups in a very direct manner.

As an interesting side note, we obtain a new proof that for a trace-preserving inclusion $\mathscr{A} \leq \mathscr{B}$ of finite von Neumann algebras, the induction functor $\mathscr{B}\otimes_{\mathscr{A}} -$ is dimension-flat and -preserving. (See Theorem \ref{thm:tensordimexactness}.)

  % OKOK

\section{Dimension function for semi-finite tracial algebras.} \label{sec:dimsemifinite} \todo{sec:dimsemifinite}

In this section we construct L{\"u}ck's dimension function in the semi-finite case and develop some of its properties. All the results in this section are essentially due to L{\"u}ck and for finite von Neumann algebras may be found in \cite[Chapter 6]{Lu02}, though there may be some differences in the proofs owing to personal taste.

One key difference compared to the finite case is that whereas  the dimension of a finite tracial von Neumann algebra $\mathscr{A}$ as a module over itself is $1$, it is infinite in the semi-finite (non-finite) case.

\begin{notation} \label{not:dimsemifinite} \todo{not:dimsemifinite}
In this section, unless explicitly stated otherwise, $\mathscr{A}$ is a semi-finite, $\sigma$-finite von Neumann algebra and $\psi$ a fixed but arbitrary faithful, normal, semi-finite tracial weight on $\mathscr{A}$.

Recall that for $x\in \mathscr{A}^2_{\psi}$ we denote $\lVert x \rVert_2 := \sqrt{\psi(x^*x)}$.

To avoid redundancy, we restrict ourselves, unless explicitly mentioned, to consider right-modules. Every result stated below also holds for left-modules.
\end{notation}

We start with some algebraic preliminaries. Recall the following

\begin{definition}
\begin{enumerate}[(i)]
\item A ring $R$ is right (respectively left) semi-hereditary if every finitely generated right (respectively left) ideal in $R$ is projective. We say that $R$ is semi-hereditary if it is both right and left semi-hereditary.
\item A ring $R$ is left (respectively right) Rickart if the left (respectively right) annihilator of every element in $x\in R$ can be written as $Re$ (respectively $eR$) for some idempotent $e\in R$, depending on $x$. We say that $R$ is Rickart if it is both left and right Rickart.
\end{enumerate}
\end{definition}

Clearly, von Neumann algebras are Rickart since, say, the right annihilator of $x\in \mathscr{A}$ is exactly $[\ker(x)]\mathscr{A}$ where $[\ker(x)]$ is the orthogonal projection onto the kernel of $x$ acting on $L^2\psi$, and similarly the left annihilator.

\begin{proposition}
Every von Neumann algebra $\mathscr{A}$ is semi-hereditary.
\end{proposition}
\begin{proof}
This follows by \cite[Proposition 7.63]{LamModulesandRings}, since $M_n(\mathscr{A})$ is also a von Neumann algebra, hence Rickart.
\end{proof}

For a proof of the following lemma and a general introduction see \cite{LamModulesandRings} (this is p. 43).

\begin{lemma} \label{lma:semihereditaryuse} \todo{lma:semihereditaryuse}
Let $\mathscr{A}$ be a von Neumann algebra. Then every finitely generated submodule of a projective $\mathscr{A}$-module is projective. \hfill \qedsymbol
\end{lemma}

Following L{\"uck} we define an algebraic notion of closure of submodules. This measures, in dimension terms, the difference between a submodule and its annihilator in the dual, see Section \ref{sec:dualitydiscrete}.

\begin{definition} \label{def:algebraicclosure} \todo{def:algebraicclosure}
Let $N\subseteq M$ be right-$\mathscr{A}$-modules. The (algebraic) closure $\overline{N}^{(M)}$ of $N$ in $M$ is the submodule
\begin{equation}
\overline{N}^{(M)} := \{ x\in M \mid \forall f\in \hom_{\mathscr{A}}(M,\mathscr{A}) : N\subseteq \ker f \Rightarrow x\in \ker f \}. \nonumber
\end{equation}
We may also use the notation $\overline{N}^{(alg)}$ if the ambient module is clear from context.
\end{definition}

Recall also that we denote by $\mathbf{T}M$ the closure of $0$ in $M$ and call this the torsion part of $M$, and by $\mathbf{P}M$ the quotient $M/\mathbf{T}M$ and call this the projective part of $M$.

\begin{lemma} \label{lma:algebraiccontinuity} \todo{lma:algebraiccontinuity}
Let $M,N$ be $\mathscr{A}$-modules. Then for every $f\in \hom_{\mathscr{A}}(M,N)$ and every submodule $P$ of $M$, we have $f\left( \overline{P}^{(M)} \right) \subseteq \overline{f(P)}^{(N)}$. Also, if $f$ is surjective, then for every submodule $Q$ of $N$, $f^{-1}\left( \overline{Q}^{(N)}\right) = \overline{f^{-1}(Q)}^{(M)}$.
\end{lemma}
\begin{proof}
For the first part, if $f(m)\notin \overline{f(P)}^{(N)}$ then there is a $g\in \hom_{\mathscr{A}}(N,\mathscr{A})$ with $g(f(m))\neq 0$ and $f(P)\subseteq \ker g$. Then $P\subseteq \ker g\circ f$ so that $m\notin \overline{P}^{(M)}$. 

For the second part, the inclusion \lq $\supseteq$\rq\phantom{a} follows directly from the definition of algebraic closure. For the opposite inclusion let $x\in f^{-1} \left( \overline{Q}^{(N)}\right)$ and $h:M\rightarrow \mathscr{A}$ vanish on $f^{-1}(Q)$. We have to show that $h(x)=0$.

Since $\ker f \subseteq \ker h$, we get an induced $\mathscr{A}$-map $\overline{h}:N\rightarrow \mathscr{A}$ such that $(\overline{h}\circ f)(m) = h(m)$ for all $m\in M$. In particular $\overline{h}$ vanishes on $Q=f(f^{-1}(Q))$ whence on $f(x)$, as had to be shown.
\end{proof}

\begin{lemma} \label{lma:algebraicclosure} \todo{lma:algebraicclosure}
Suppose that $M$ is a submodule of $\mathscr{A}^{n}$. Then the algebraic closure of $M$ (in $\mathscr{A}^{n}$) is the largest submodule $N$ containing $M$ and such that \todo{eq:algebraicclosure}
\begin{equation}
N\cap (\mathscr{A}^{2}_{\psi})^{n} = \overline{M \cap (\mathscr{A}^{2}_{\psi})^{n}}^{\lVert \cdot \rVert_2} \cap (\mathscr{A}^{2}_{\psi})^{n}. \label{eq:algebraicclosure}
\end{equation}
\end{lemma}

\begin{proof}
If $N$ is such that '$\subseteq$' holds in equation (\ref{eq:algebraicclosure}) then for $x\in N$ and $f\in \hom_{\mathscr{A}}(\mathscr{A}^{n}, \mathscr{A})$ with $M\subseteq \ker f$, if $f(x)\neq 0$ there is a projection $p$ in $\mathscr{A}^{2}_{\psi}$ such that $xp , f(x)p\neq 0$, a contradiction. Thus $N\subseteq \overline{M}^{(\mathscr{A}^{n})}$.

If '$\not\subseteq$' holds in \eqref{eq:algebraicclosure}, say the difference containing $x$, we may build a morphism into $\mathscr{A}$ separating this from $M$, since $qx\neq 0$ with $q$ the projection onto the orthogonal complement of $\overline{M\cap (\mathscr{A}^{2}_{\psi})^{n}}^{\lVert \cdot \rVert_2}$. Then compose $q$ with the projection onto an appropriate summand in $\mathscr{A}^n$.

Hence in this case we find a non-zero $x\in N\setminus \overline{M}^{(\mathscr{A}^n)}$. In particular we see that $N=\overline{M}^{(\mathscr{A}^n)}$ satisfies '$\subseteq$' in \eqref{eq:algebraicclosure}, finishing the proof.
\end{proof}

\begin{corollary} \label{cor:algebraicclosure} \todo{cor:algebraicclosure}
With notation as in the lemma, the closure of $M$ is exactly $p\mathscr{A}^{n}$ where $p\in M_n(\mathscr{A})$ is the orthogonal projection onto $\overline{M\cap (\mathscr{A}^{2}_{\psi})^{n}}^{\lVert \cdot \rVert_2}$.
\end{corollary}
\begin{proof}
The inclusion $p\mathscr{A}^n \subseteq \overline{M}^{(\mathscr{A}^n)}$ follows directly from the lemma. Further $\bbb - p$ vanishes on $M$ since it vanishes on $M\cap (\mathscr{A}^{2}_{\psi})^{n}$ so that this inclusion is an equality.
\end{proof}

Combining this with Lemma \ref{lma:algebraiccontinuity} we get the following result, due to L{\"u}ck in the finite case (see \cite[Theorem 6.7]{Lu02}).

\begin{theorem} \label{thm:algebraicclosuresplitting} \todo{thm:algebraicclosuresplitting}
Suppose that $M$ is a finitely generated right-$\mathscr{A}$-module. Then for every submodule $P$, $M$ splits as a direct sum $M\simeq \overline{P}^{(M)}\oplus M/\overline{P}^{(M)}$. Further, $M/\overline{P}^{(M)}$ is finitely generated and projective.
\end{theorem}

We include the short proof for completeness.

\begin{proof}
Let $0\rightarrow N\rightarrow \mathscr{A}^{n}\xrightarrow{\kappa} M\rightarrow 0$ be a presentation of $M$.

Lemma \ref{lma:algebraiccontinuity} tells us that we have an exact sequence
\begin{equation}
0\rightarrow N \rightarrow \overline{\kappa^{-1}(P)}^{(\mathscr{A}^{n})} \xrightarrow{\kappa} \overline{P}^{(M)} \rightarrow 0. \nonumber
\end{equation}
By the previous corollary, $\overline{\kappa^{-1}(P)}^{(\mathscr{A}^{n})} = p\mathscr{A}^{n}$ for some orthogonal projection $p\in M_n(\mathscr{A})$, and now the claim follows by the fact that $M/\overline{P}^{(M)}\simeq (\bbb_n-p)\mathscr{A}^n$.
\end{proof}

\begin{definition} \label{def:phifinite} \todo{def:phifinite}
Let $M$ be a right-$\mathscr{A}$-module. We say that $M$ is $\psi$-finitely generated, or just $\psi$-fg, if there is an exact sequence of right-$\mathscr{A}$-modules
\begin{equation}
0\rightarrow N\rightarrow p\mathscr{A}^n \rightarrow M\rightarrow 0 \nonumber
\end{equation}
where $p$ is a projection in $M_n(\mathscr{A})$ with finite trace $Tr_n\otimes \psi$.
\end{definition}

The next definition generalizes L{\"u}ck's dimension function to the semi-finite case.

\begin{definition} \label{def:projectivedim} \todo{def:projectivedim}
Keep Notation \ref{not:dimsemifinite} and suppose that $M$ is a $\psi$-fg projective right-$\mathscr{A}$-module. Then $M\simeq q\mathscr{A}^{n}$ where $q\in M_n(\mathscr{A})$ is a projection with $(Tr_n\otimes \psi)(q)<\infty$. We define the $\psi$-dimension of $M$ as $\dim_{\psi} M := (Tr_n\otimes \psi)(q)$ with the same $q$ as above.
\end{definition}

It is implicit in the definition, but not immediately clear, that $\dim_{\psi} M$ is independent of the chosen projection $q$. That this is in fact the case is the content of the following theorem. Recall also that any finitely generated, projective $\mathscr{A}$-module has the form $p\mathscr{A}^n$ for some \emph{projection} $p$, {i.e.} $p$ is a self-adjoint idempotent.

\begin{theorem} \label{thm:dimwellposed} \todo{thm:dimwellposed}
Suppose that $M\leq N$ are $\psi$-fg projective right-$\mathscr{A}$-modules. Then for any $p,q$ such that $M\simeq p\mathscr{A}^n$ and $N\simeq q.\mathscr{A}^n$ as in Definition \ref{def:projectivedim}, $(Tr_n\otimes \psi)(p) \leq (Tr_n\otimes \psi)(q)$.
\end{theorem}

We postpone the proof to give a few corollaries and extend the domain of definition of $\dim_{\psi}$.

\begin{remark}
A priori we should consider $M\simeq p\mathscr{A}^m$ and $N\simeq q.\mathscr{A}^n$ in the statement of Theorem \ref{thm:dimwellposed}. However, if e.g. $m<n$ we have also $M\simeq (p\oplus 0_{n-m})\mathscr{A}^n$ so that the assumption that $m=n$ in the theorem is not restrictive.
\end{remark}

\begin{corollary} \label{cor:dimwellposed} \todo{cor:dimwellposed}
For $M$ a $\psi$-fg. projective module, $\dim_{\psi}M$ is well-defined.
\end{corollary}
\begin{proof}
Trivial.
\end{proof}

\begin{definition}[({\ref{def:projectivedim}} continued)]
We extend the domain of definition of $\dim_{\psi}$ to all right-modules (as in the finite) case by defining for any right-$\mathscr{A}$-module $N$ the $\psi$-dimension as
\begin{equation}
\dim_{\psi} N := \sup \{ \dim_{\psi} M \mid M\leq N \; \textrm{is}\; \psi\textrm{-fg} \; \textrm{projective submodule} \,\}. \nonumber
\end{equation}
\end{definition}

\begin{corollary} \label{cor:projectivedimmonotonic} \todo{cor:projectivedimmonotonic}
Whenever $M\leq N$ we have
\begin{equation}
\dim_{\psi} M \leq \dim_{\psi} N. \nonumber
\end{equation}
\end{corollary}
\begin{proof}
Trivial.
\end{proof}

\begin{proof}[Proof of Theorem \ref{thm:dimwellposed}.]
By the isomorphisms $M\simeq p\mathscr{A}^n$ and $N\simeq q.\mathscr{A}^n$ we consider decompositions
\begin{equation}
M\oplus \ker p = \mathscr{A}^{n} = N\oplus \ker q \nonumber
\end{equation}
and an isomorphism $\theta:M\xrightarrow{\sim} M^{(1)} \subseteq N$ given by the inclusion of $M$ in $N$. Then we define an $\mathscr{A}$-(right-)linear map $\overline{\theta}$ in $\hom_{\mathscr{A}}(\mathscr{A}^{n},\mathscr{A}^{n})$ by
\begin{equation}
\overline{\theta} = \theta\oplus 0. \nonumber
\end{equation}
By $\mathscr{A}$-linearity this is implemented by left-multiplication by a matrix in $M_n(\mathscr{A})$, so that it extends to an $\mathscr{A}$-linear map $\overline{\theta}:L^2\psi^n\rightarrow L^2\psi^n$.

Denote $M_2=\overline{M\cap (\mathscr{A}^{2}_{\psi})^{n}}^{\lVert \cdot \rVert_2} \subseteq L^2\psi^n$, similarly $N_2$, and observe that by Corollary \ref{cor:algebraicclosure}, $p$ (resp. $q$) is the orthogonal projection onto $M_2$ (resp. $N_2$). Then we have, for the operator $p\overline{\theta}^{*}\overline{\theta}p: L^2\psi^n \rightarrow L^2\psi^n$,

\begin{equation}
\overline{\operatorname{Ran}}^{\lVert \cdot \rVert_2}(p\overline{\theta}^{*}\overline{\theta}p) \leq M_2. \nonumber
\end{equation}
We show that equality does in fact hold, by showing that $\ker(\overline{\theta}p)=M_2^{\perp}$. The inclusion '$\supseteq$' is just the definition of $p:L^2\psi^n\rightarrow M_2$. Now, if this inclusion is strict, then $K=M_2\cap \ker \overline{\theta}p  \neq 0$, and this is a closed right-$\mathscr{A}$-invariant subspace of $L^2\psi^n$.

Now by Lemma \ref{lma:algebraicclosure}, $M_2\cap (\mathscr{A}^{2}_{\psi})^{n}=M\cap (\mathscr{A}^{2}_{\psi})^{n}$ since $M$, being a summand, is (algebraically) closed in $\mathscr{A}^{n}$, so that $K\cap (\mathscr{A}^{2}_{\psi})^{n}$ is empty. But this is impossible since now if $0\neq \xi\in K$, there is a projection $e \in \mathscr{A}$ such that $0 \neq \xi e \in \mathscr{A}^n$ and then we may further cut this by a sufficiently large projection in $\mathscr{A}^{2}_{\psi}$ to obtain a contradiction. To construct such an $e$, let first $e_1$ be a non-zero spectral projection of $\xi_1^{*}\xi_1$, where $\xi_1$ is the first coordinate of $\xi$. Continuing, we can take a non-zero spectral projection $e_2$ of $e_1\xi_2^*\xi_2e_1$, and so on to get $e_1\geq e_2\geq \cdots \geq e_n$ such that $\xi.e_n\neq 0$. Then we take $e = e_n$.

Thus in summary, $\overline{\theta}p$ is a bounded operator with $\overline{\operatorname{Ran}}^{\lVert \cdot \rVert_2}(p\overline{\theta}^{*})= M_2$ and range contained in $N_2$. The final step then is to apply polar decomposition to the operator (where the inclusion $\mathscr{A}^{op}\subseteq \mathcal{B}(L^2\psi)$ is the one given by right-multiplication)

\begin{equation}
\left( \begin{array}{cc} 0 & 0 \\ \overline{\theta}p & 0 \end{array} \right) \in \left( \oplus_{i=1}^{2n}\mathscr{A}^{op} \right)'\cap \mathcal{B}(L^2\psi^n\oplus L^2\psi^n). \nonumber
\end{equation}

This yields a partial isometry $v$ in $M_{2n}(\mathscr{A})$ such that $v^{*}v = p\oplus 0$ and $vv^{*} \leq 0\oplus q$, and the theorem follows.
\end{proof}

\begin{proposition} \label{prop:fgprojective} \todo{prop:fgprojective}
If the (right-)$\mathscr{A}$-module $M$ contains a finitely generated projective submodule which is not $\psi$-fg, then $\dim_{\psi}M = \infty$.
\end{proposition}
\begin{proof}
This is clear as every projection in $M_n(\mathscr{A})$ with infinite trace has subprojections with arbitrarily large trace.
\end{proof}

\begin{remark}
The previous proposition shows that the choice in Definition \ref{def:projectivedim}' to take the supremum over $\psi$-fg projective submodules instead of over all projective submodules is arbitrary and makes no difference.
%In our setting the $\psi$-fg projective modules play the role of projective modules in the finite setting. Proposition \ref{prop:fgprojective} shows that this choice does not matter. The choice was originally motivated by the fact that Theorem \ref{thm:dimensioncontinuityII} and its corollary would then be direct generalizations of the finite case.
\end{remark}

\begin{proposition}[(Compare {\cite[Assumption 6.2(2)]{Lu02}})] \label{prop:dimensioncontinuityI} \todo{prop:dimensioncontinuityI}
Let $N$ be a submodule of the finitely generated projective module $P$. Then
\begin{equation}
\dim_{\psi} N = \dim_{\psi} \overline{N}^{(P)}
\end{equation}
\end{proposition}
\begin{proof}
Let $M\leq N$ be a finitely generated submodule. Since $\mathscr{A}$ is semi-hereditary, $M$ is projective. By Corollary \ref{cor:projectivedimmonotonic}, $\dim_{\psi} M \leq \dim_{\psi} \overline{N}^{(P)}$ since also $\overline{N}^{(P)}$ is projective. We have to show that we can choose $M$ such that $\dim_{\psi} M$ is as close to $\dim_{\psi}\overline{N}^{(P)}$ as we like. 

We follow the proof of Theorem \ref{thm:dimwellposed} with $\overline{N}^{(P)}$ in place of $N$.

Denote by $v_M$ the partial isometry constructed from $M$ by applying the $2\times 2$ matrix trick as in the proof of Theorem \ref{thm:dimwellposed}. Let $\{x_i\}_{i\in I}\subset N$ be dense in $\overline{N\cap (\mathscr{A}^{2}_{\psi})^n}^{\lVert \cdot \rVert_2}$ and denote by $\mathcal{I}_0$ the set of finite subsets of $I$. Then the orthogonal projection $q$ onto $\overline{N\cap (\mathscr{A}^{2}_{\psi})^n}^{\lVert \cdot \rVert_2}$ is the least upper bound of projections $q_{I_0}$ onto the closed right-$\mathscr{A}$-invariant subspace generated by $\{x_i\}_{i\in I_0}$ over $I_0\in \mathcal{I}_0$. Now given $I_0\in \mathcal{I}_0$, if $M$ is the submodule of $N$ generated by $\{x_i\}_{i\in I_0}$. Then $(v_M^{*}v_M) \mathscr{A}^n = M$ whence
\begin{eqnarray}
\dim_{\psi} M & = & (Tr\otimes \psi)(v_M^{*}v_M) \nonumber \\
 & = & (Tr\otimes \psi)(v_Mv_M^{*}) = (Tr\otimes \psi)(q_{I_0}). \nonumber
\end{eqnarray}
The proposition then follows since $\psi$ is normal.

\begin{comment} %%%%%%%%%%%%%%
We consider again a map arising from splittings
\begin{equation}
\theta_N = i\oplus 0 : N\oplus \ker q_N = \mathscr{A}^{n} \rightarrow \mathscr{A}^{n} = P\oplus \ker q_P, \nonumber
\end{equation}
where $i$ is the inclusion of $N$ in $M$ and it is clear that the two $n$'s can indeed be taken the same, and the $q_N, q_P$ are just, say projections with ranges the subscripts.

Then $\theta_N$ is right-$\mathscr{A}$-linear so that it is really just multiplication by a matrix whence a bounded operator on Hilbert spaces. Now since $N$ is a summand in the domain, it is (algebraically) closed so that as before, $\ker (\theta_N)^{\perp} = \overline{N \cap (\mathscr{A}^{2}_{\psi})^{n}}^{\lVert \cdot \rVert_2}$. Further, $\overline{Ran}^{\lVert \cdot \rVert_2}(\theta_N) \subseteq \overline{M \cap (\mathscr{A}^{2}_{\psi})^{n}}^{\lVert \cdot \rVert_2}$ which is again contained in $\overline{P \cap (\mathscr{A}^{2}_{\psi})^{n}}^{\lVert \cdot \rVert_2}$ since $P$ is (algebraically) closed.
\end{comment}

\end{proof}

The next result is again just a restatement of part of \cite[Theorem 6.7]{Lu02}.

\begin{theorem}[(additivity of dimension)] \label{thm:projectivedimadditivity} \todo{projectivedimadditivity}
The dimension function $\dim_{\psi}$ is additive, in the sense that for every short exact sequence of right-$\mathscr{A}$-modules
\begin{equation}
0\rightarrow L \rightarrow M \rightarrow N \rightarrow 0, \nonumber
\end{equation}
we have
\begin{equation}
\dim_{\psi} M = \dim_{\psi} L + \dim_{\psi} N \nonumber
\end{equation}
with the usual convention regarding $+\infty$.
\end{theorem}
\begin{proof}
This is now verbatim as in \cite[Theorem 6.7]{Lu02}, noting that the statement is clear when all the modules are $\psi$-fg projective. 
\end{proof}

\begin{theorem}[(continuity of dimension)] \label{thm:dimensioncontinuityII} \todo{thm:dimensioncontinuityII}
The dimension function $\dim_{\psi}$ is continuous, in the sense that for any submodule $N$ of a $\psi$-fg module $M$,
\begin{equation}
\dim_{\psi} N = \dim_{\psi} \overline{N}^{(M)}. \nonumber
\end{equation}
\end{theorem}
\begin{proof}
Indeed, from the short exact sequence $0\rightarrow L \rightarrow p\mathscr{A}^{n}\xrightarrow{\kappa} M \rightarrow 0$ we get short exact sequences
\begin{equation}
0\rightarrow L \rightarrow \kappa^{-1}(N) \xrightarrow{\kappa} N \rightarrow 0 \nonumber
\end{equation}
and by Lemma \ref{lma:algebraiccontinuity},
\begin{equation}
0\rightarrow L \rightarrow \overline{\kappa^{-1}(N)}^{(p\mathscr{A}^{n})} \xrightarrow{\kappa} \overline{N}^{(M)} \rightarrow 0.
\end{equation}
Since all dimensions are finite (this is where we use the $\psi$-fg assumption), and the middle terms have the same dimension by \ref{prop:dimensioncontinuityI}, additivity finishes the proof.
\end{proof}

\begin{corollary}
For every $\psi$-fg module $M$, $\mathbf{T}M$ contains no projective submodules.
\end{corollary}

\begin{proof}
By the above, $\dim_{\psi} \mathbf{T}M = 0$, recalling that $\mathbf{T}M=\overline{\{0\}}^{(M)}$ by definition. This shows the claim since $\psi$ is faithful and $\mathscr{A}$ is semi-hereditary.
\end{proof}

\begin{proposition}[(See also {\cite[Theorem 6.24]{Lu02}})] \label{prop:dimelltwo} \todo{prop:dimelltwo}
Suppose that $M$ is a closed right-$\mathscr{A}$-invariant subspace of $L^2\psi^{n}$. Then
\begin{equation}
\dim_{\psi} M = (Tr_n \otimes \psi)(P_M) \nonumber
\end{equation}
where $P_M$ is the orthogonal projection onto $M$.
\end{proposition}

%\todo{referee-outcomment here}
%The proof is really a rank density type argument (see Definition \ref{def:rankdensity}) and may be formulated as such. However, we give a direct proof because we find it instructive to do so at this point.

\begin{proof}
Let $P$ be a $\psi$-fg projective submodule of $M$, and consider a splitting $p\mathscr{A}^{m} = P\oplus \ker q_P$. Clearly we may take $m=n$, and the inclusion $i$ of $P$ in $M$ then extends to a map $\theta=i\oplus0$ of $\mathscr{A}^{n}$ into $M\subseteq L^2 \psi^{n}$. Since $\mathscr{A}$ is unital, this has the form
\begin{equation}
\theta(a_1,\dots ,a_n) = \left( \begin{array}{ccc} \xi_{11} & \cdots & \xi_{1n} \\ \vdots & \ddots & \vdots \\ \xi_{n1} & \cdots & \xi_{nn} \end{array} \right) \left( \begin{array}{c} a_1 \\ \vdots \\ a_n \end{array} \right) . \nonumber
\end{equation}
where the $\xi_i\in L^2\psi^{n}$. We denote $\Xi = (\xi_{ij})$.

Now, $\Xi \in L^2(Tr_n\otimes \psi)$ so that it is an affiliated operator to $M_n(\mathscr{A})$. Thus we may take a spectral projection (of $\Xi^{*}\Xi$) $e\in M_n(\mathscr{A})$ such that $\Xi e$ is bounded. Further, clearly $\Xi = \Xi q_P$, and applying as above polar decomposition to the operator
\begin{equation}
\left( \begin{array}{cc} 0 & 0 \\ \Xi q_P e & 0 \end{array} \right) \in \left( \oplus_{i=1}^{2n}\mathscr{A}^{op}\right)' \cap \mathcal{B}(L^2\psi^n \oplus L^2\psi^n) \nonumber
\end{equation}
we get again as in the proof of Theorem \ref{thm:dimwellposed} that
\begin{eqnarray}
\dim_{\psi} P & = & (Tr_n\otimes \psi)([q_Pe]) + (Tr_n\otimes \psi)(q_P-[q_Pe]) \nonumber \\
 & \leq & (Tr_n\otimes \psi)(P_M) + (Tr_n\otimes \psi)(q_P - [q_Pe]). \nonumber
\end{eqnarray}
Finally we note that letting $e$ increase to the identity, the final term decreases to $0$ since it is always finite ($(Tr_n\otimes \psi)(q_P) < +\infty$ by assumption).

This shows the inequality '$\leq$' of the statement of the proposition, and since the other is true  by continuity of dimension in projective modules, Proposition \ref{prop:dimensioncontinuityI}, noting that $\overline{M\cap (\mathscr{A}^2_{\psi})^n}^{(\mathscr{A}^n)} = P_M(\mathscr{A}^n)$, we are done.
\end{proof}

\section{Rank density and projective limits} \label{sec:rankdensity} \todo{sec:rankdensity}

%In this section we want to discuss the basics of adapting the rank density techniques of Thom \cite{Thom06a, Thom06b} to our context.  As applications we prove a vanishing result for $L^2$-betti numbers of abelian groups, and several auxilliary results that we will need later.

In \cite{Thom06a,Thom06b} the key technical tool is the notion of rank completion of an $\mathscr{A}$-module. Whenever $\mathscr{A}$ is a finite tracial von Neumann algebra, the action on any module induces a pseudo-metric, the rank metric, with respect to which any $\mathscr{A}$-linear map is uniformly continuous. The (Hausdorff) completion of a module with respet to the rank metric is again an $\mathscr{A}$-module, in fact it is a module over the ring of operators affiliated with $\mathscr{A}$, and the construction is functorial. Then one can exploit nice properties of the category of rank complete modules, and the close connection with the category of all modules.

Given a semi-finite tracial von Neumann algebra $\mathscr{A}$, the situation seems a bit less direct. A heuristic reason for this can that the set of operators affiliated with $\mathscr{A}$ is no longer a ring; in order to get a ring, one has to fix the trace $\psi$ on $\mathscr{A}$ and consider only the $\psi$-measurable operators, see \cite[Chapter IX]{TakesakiII}. But in the finite case the rank completion is entirely "algebraic" in the sense that it does not depend on the choice of trace. Further, as we observe below, vanishing of dimension is also "algebraic" in this sense, even in the semi-finite case. Since the main point of studying the rank completion is that it reflects precisely vanishing of dimension, studying modules over the $\psi$-measurable operators is not sufficient then.

However, one can directly generalize to the semi-finite case the underlying technical observation, namely the local criterion for vanishing of dimension due to Sauer, appearing as \cite[Theorem 2.4]{Sau03}. In this section we do just that, and use it to prove several important properties of the dimension function (see \ref{thm:dimfinality} and \ref{lma:dimbyTr}).

We also prove a result relating the dimension function for modules over a tracial algebra $(\mathscr{A},\psi)$ to that for modules over the corner $(\mathscr{A}_p,\psi(p\cdot))$ when $p$ is a projection in $\mathscr{A}$. See Theorem \ref{thm:compression}. This allows in particular to pass to the finite setting, and in the remainder of the chapter we will use this correspondance and the properties of rank completion in the finite case to deduce "dimension exactness" results in the semi-finite case, which can to some extend make up for the lack of a suitable notion of rank completion.

\begin{lemma}[(Sauer's local criterion {\cite{Sau03}})] \label{lma:sauerslocalcriterion} \todo{lma:sauerslocalcriterion}
Let $(\mathscr{A},\psi)$ be a semi-finite, $\sigma$-finite tracial von Neumann algebra. 
\begin{enumerate}[(i)]
\item Let $M\subseteq N$ be (right-)modules over $\mathscr{A}$. Suppose that for every $x\in N$ there is a sequence $(p_n)$ of projections in $\mathscr{A}$ such that $p_n\nearrow \bbb$ and for all $n$, $x.p_n\in M$. Then
\begin{equation}
\dim_{\psi}N/M =0, \quad \mathrm{and} \; \dim_{\psi} M = \dim_{\psi} N. \nonumber
\end{equation}
\item Let $M$ be a (right-)module over $\mathscr{A}$ such that $\dim_{\psi}M=0$. Then for every $x\in M$ there is a sequence $(p_n)$ of projections in $\mathscr{A}$ such that $x.p_n=0$ for all $n\in \mathbb{N}$ and $p_n\nearrow \bbb$.
\end{enumerate}

\end{lemma}

\begin{proof}
The proof of $(i)$ is verbatim as in \cite[Theorem 2.4]{Sau03} so we leave it out. \todo{commented out in the code}
%%%%%%%%%%%%%%%%%%%%%%%%%%%%%%%%%%%%%%%%%%%%%%%%%%%%%%%%%%%%%%
\begin{comment}
Consider the short exact sequence
\begin{equation}
0\rightarrow M\rightarrow N \xrightarrow{\kappa} N/M \rightarrow 0. \nonumber
\end{equation}
Then by additivity $\dim_{\psi} N = \dim_{\psi} M + \dim_{\psi} N/M$ so we have to show that the final term has vanishing dimension. We show that it has no non-zero projective submodules.

Indeed, if $Q\subseteq N/M$ is f.g. projective there is an embedding $\pi :Q\rightarrow \mathscr{A}^{n}$ of $Q$ as a summand, whence a homomorphism $\overline{\pi}:Q\rightarrow \mathscr{A}$, by evaluation at some coordinate, say, which can be taken non-vanishing if $Q$ is. But now if $\kappa(x)\in Q$ there is by the hypothesis $(p_n)$ increasing to $\bbb$ with $\kappa(x).p_n=0$ for all $n$.

Then also $\overline{\pi}(\kappa(x)).p_n=0$ for all $n$, so that $\overline{\pi}(\kappa(x))=0$ since the $p_n$ increase to $\bbb$. Hence $\overline{\pi}$ is identically zero and since it was arbitrary, $Q = \{0\}$. This proves $(i)$.
\end{comment}
%%%%%%%%%%%%%%%%%%%%%%%%%%%%%%%%%%%%%%%%%%%%%%%%%%%%%%%%%%%%%

For $(ii)$ we consider for given $x$ the homomorphism $\phi:\mathscr{A}\rightarrow x.\mathscr{A}\subseteq M$ defined by $a\mapsto x.a$. Then $x.\mathscr{A} \simeq \mathscr{A}/\ker \phi$ so that in particular $\dim_{\psi} \mathscr{A}/\ker \phi =0$. It follows that $\overline{\ker \phi}^{(\mathscr{A})}=\mathscr{A}$ since otherwise $\mathscr{A}/\overline{\ker \phi}^{(\mathscr{A})}$ would be a f.g. projective module (by Theorem \ref{thm:algebraicclosuresplitting}) embedding into a zero-dimensional module, which is a contradiction.

From this we get that $\ker \phi \cap \mathscr{A}^{2}_{\psi}$ is dense in $L^2\psi$ in $2$-norm, hence dense in the weak topology. In particular, for every non-zero projection $q\in \mathscr{A}^{2}_{\psi}$ there is an $a\in \ker \phi$ such that $qa\neq 0$. This implies that $qaa^{*}\neq 0$ and then that $aa^{*}q\neq 0$.

Finally we may then consider the spectral projections of $aa^{*}$ corresponding to intervals $[\varepsilon, \infty)$ with $\varepsilon >0$. From the above, there is such a projection $e$ for which $eq\neq0$, and since $e=a(a^{*}f(aa^{*}))$ where $f(t)=t^{-1}$ for $t\geq \varepsilon$ and $f(t)=0$ for $t<\varepsilon$, we have $e\in \ker \phi$. Since $q$ was arbitrary in $\mathscr{A}_{\psi}^2$ the statement now follows by are standard maximality argument: by Zorn's lemma we can find a maximal family $(p_i)_{i\in I}$ of pairwise orthogonal projections $p_i'\in \mathscr{A}$ such that for any finite subset $I_0\subseteq I$ we have
\begin{equation}
x.\left( \sum_{i\in I_0}p_i' \right) = 0. \nonumber
\end{equation}

Indeed, the set of such families is non-empty by the argument above. One orders the set of such families by inclusion and it is then clear that any chain has an upper bound. Hence by Zorn's lemma we get a maximal element as claimed.

Then we have $\sum_{i\in I}p_i' = \bbb$ in $\mathscr{A}$ since otherwise we can apply the argument above with $q=\bbb-\sum_i p_i'$ to contradict maximality. The index set $I$ is countable by the $\sigma$-finiteness of $\mathscr{A}$. Taking $I=\mathbb{N}$ we are then done with $p_n = \sum_{i=1}^n p_i'$.
\end{proof}

\begin{definition} \label{def:rankdensity} \todo{def:rankdensity}
An inclusion satisfying the conditions of Lemma \ref{lma:sauerslocalcriterion}(i) is said to be rank dense. \todo{Perhaps modify this later.}
\end{definition}

As an application of Sauer's local criterion we show how the dimension function behaves under projective limits. The result is stated for finite traces in \cite[Theorem 6.18]{Lu02} but the proof is not given there. This will be an important technical tool as it will allow us to restrict cocycles in group cohomology to compact sets and that way work with Hilbert spaces instead of complete metrizable spaces.

But first we state the following lemma needed for the proof. It allows to generalize $\varepsilon / 2^n$-type arguments from the finite case to the semi-finite case. Specifically if $(\mathscr{A},\tau)$ is a finite tracial von Neumann algebra and $p_n$ are projections in $\mathscr{A}$ such that $\tau(p_n)\geq 1-\frac{\varepsilon}{2^n}$ then it is easy to see that $\tau( \wedge_n p_n ) \geq 1-\varepsilon$ so that we can get an approximation to the identity in this way. We state the analogous trick for semi-finite tracial algebras as follows.

\begin{lemma}[(Countable annihilation)] \label{lma:semifinitetwontrick} \todo{lma:semifinitetwontrick} \label{lma:countableannihilation} \todo{lma:countableannihilation}
Let $E$ be an $\mathscr{A}$-module and suppose that $(E_i)_{i\in \mathbb{N}}, (F_i)_{i\in \mathbb{N}}$ are sequences of submodules in $E$ such that for every $i$, $E_{i}\subseteq F_i$ is rank dense. Then $\cap_i E_i\subseteq \cap_i F_i$ is rank dense as well.
\end{lemma}

\begin{proof}
Let $x\in \cap_i F_i$ and denote by $S_i$ the set of projections $p$ in $\mathscr{A}$ such that $x.p\in E_i$. Note that $S_i$ is hereditary ($p\leq q$ and $q\in S_i$ implies $p\in S_i$).

It is sufficient, by our blanket assumptions that $\mathscr{A}$ be $\sigma$-finite, to show that whenever $p_1\in S_1$ is non-zero with finite trace, there is a non-zero subprojection of $p_1$ in $\cap_i S_i$.

To see this, fix some $0<\varepsilon <\psi(p)$ and suppose that we have found projections $p_2 \geq \dots \geq p_{n-1}$ such that $p_{n-1}\in \cap_{i=1}^{n-1}S_i$ and $\psi(p_{n-1})>\varepsilon$.

Choose a projection $p$ such that $x.p_{n-1}p\in E_i$ and $p_{n-1}p$ has range projection (i.e.~left support) with trace $> \varepsilon$. Then we can take $p_n\in S_n$ a sufficiently large spectral projection of $p_{n-1}pp_{n-1}$ and still have trace $\psi(p)>\varepsilon$.

Then $\wedge_n p_n$ works and is non-zero (the trace is $\psi(\wedge_n p_n)\geq \varepsilon$.
\end{proof}

We will give a second proof just below, which more directly shows how this extends the $\varepsilon / 2^n$ trick. First we single out a useful reformulation, which is the one most often used.

\begin{corollary}
Let $E\leq F$ be a rank dense inclusion of $\mathscr{A}$-modules. Then the inclusion $\mathcal{F}(\mathbb{N},E)\leq \mathcal{F}(\mathbb{N},F)$ is rank dense as well. \hfill \qedsymbol
\end{corollary}

Next we give the alternative
\begin{proof}[Proof of Lemma \ref{lma:countableannihilation}]
Let $q^{(j)}, j\in \mathbb{N}$ be a countable (by $\sigma$-finiteness) set of pairwise orthogonal projections in $\mathscr{A}$ such that
\begin{equation}
\bbb = \sum_{j\in \mathbb{N}} q^{(j)}, \quad \textrm{and}\quad  \psi(q^{(j)})<\infty \quad \textrm{for all } j. \nonumber
\end{equation}

Denote $E_i^{(j)}:=E_iq^{(j)}$ and $F_i^{(j)}:=F_iq^{(j)}$. Then for any $i,j$ the inclusion $E_i^{(j)}\leq F_i^{(j)}$ is a rank dense inclusion of modules over the finite tracial von Neumann algebra $(\mathscr{A}_{q^{(j)}},\psi(q^{(j)}\cdot))$. It is easy to see by an $\varepsilon/2^n$ argument that for every $j\in \mathbb{N}$, the inclusion $\cap_iE_i^{(j)}\subseteq \cap_iF_i^{(j)}$ is rank dense.

Thus, given $x\in \cap_iF_i$ we can find for each $j\in \mathbb{N}$ an increasing sequence $p_n^{(j)}\nearrow_n q^{(j)}$ of projections $p_n^{(j)}\in \mathscr{A}_{q^{(j)}}$ such that $x.p_n^{(j)}=x.q^{(j)}p_n^{(j)}\in \cap_iE_i^{(j)}$ for all $n\in \mathbb{N}$. It follows that, letting $p_n:=\sum_{k=1}^n p_n^{(k)}$, we have $p_n\nearrow_n \bbb$ in $\mathscr{A}$ and $x.p_n\in \cap_iF_i$ for all $n$, as had to be shown.
\end{proof}

\begin{theorem}[(Projective limits)] \label{thm:dimfinality} \todo{thm:dimfinality} \todo{This used to be called "finality," which is retarded}
Let $(\{E_i\}_{i\in I}, \{\phi_{ij}\}_{i,j\in I})$ be a projective system of (right-)$\mathscr{A}$-modules - where $(\mathscr{A},\psi)$ is a semi-finite, $\sigma$-finite, tracial von Neumann algebra, and denote the projective limit $E:=\lim_{\leftarrow} E_i$. Suppose that there is a subsequence $(i_k)_{k\in \mathbb{N}}$ of indices such that for all $i\in I$, $i\leq i_k$ for some $k$, and that $\dim_{\psi}E_{i_k} < \infty$ for all $k$. Then
\begin{equation}
\dim_{\psi} E = \sup_{i\in I} \inf_{j:j\geq i} \dim_{\psi} \phi_{ij}(E_{j}). \nonumber
\end{equation}
\end{theorem}

We first prove the following special case.

\begin{lemma} \label{lma:dimfinalityspecial} \todo{lma:dimfinalityspecial}
If $\{F_m\}_{m\in \mathbb{N}}$ are $\mathscr{A}$-modules, $F_m\supseteq F_{m+1}$, $F_m\searrow 0$ and $\dim_{\psi} F_{m_0}<\infty$ for some $m_0$, then $\dim_{\psi} F_m \rightarrow_m 0$.
\end{lemma}

\begin{proof}
We can assume without loss of generality that $m_0=1$ so that $\dim_{\psi}F_m < \infty$ for all $m$. Let $\varepsilon > 0$ be given. Choose $M_1 \subseteq F_1$ f.g. projective module such that $\dim_{\psi}M_1 > \dim_{\psi} F_1 - \frac{\varepsilon}{2}$.

Having chosen $M_1,M_2,\dots ,M_{n-1}$ f.g. projective modules such that $M_i\subseteq M_{i-1} \cap F_i$ and
\begin{equation}
\dim_{\psi} M_i > \dim_{\psi} F_i - \sum_{j=1}^{i} \frac{\varepsilon}{2^{j}} \nonumber
\end{equation}
we note that since the map $F_n/(M_{n-1}\cap F_n) \rightarrow F_{n-1}/M_{n-1}$ induced by the inclusion $F_n\subseteq F_{n-1}$ is injective we get by monotonicity and additivity
\begin{equation}
\dim_{\psi} M_{n-1}\cap F_n > \dim_{\psi}F_n - \sum_{j=1}^{n-1} \frac{\varepsilon}{2^j}. \nonumber
\end{equation}
Thus we can choose a f.g. projective module $M_n \subseteq M_{n-1}\cap F_n$ such that
\begin{equation}
\dim_{\psi} M_n > \dim_{\psi} F_n - \sum_{j=1}^{n} \frac{\varepsilon}{2^{j}}. \nonumber
\end{equation}
This way we get inductively a decreasing sequence of f.g. projective submodules $M_n$ satisfying this inequality for all $n$.

We claim that $\dim_{\psi} \cap_{n=1}^{\infty} \overline{M_n}^{(M_1)} = 0$. Given this the lemma easily follows since by Corollary \ref{cor:algebraicclosure}, $\overline{M_n}^{(M_1)} \simeq p_n\mathscr{A}^{n_1}$ with the $p_n$ a decreasing sequence of projections in $M_{n_1}(\mathscr{A})$. Hence $\dim_{\psi} \overline{M_n}^{(M_1)} \rightarrow_n 0$, so $\limsup_n \dim_{\psi} F_n \leq \varepsilon$. But $\varepsilon>0$ was arbitrary.

To see the claim just note that $M_n$ is rank dense in $\overline{M_n}^{(M_1)}$ for all $n$ whence by the countable annihilation lemma, \ref{lma:semifinitetwontrick}, the intersection $\cap_n M_n=0$ is rank dense in $\cap_n \overline{M_n}^{(M_1)}$. This is equivalent to the latter having dimension zero, as was claimed.
\end{proof}

\begin{proof}[Proof of the theorem.]
We consider first the inequality '$\leq$'. Let $d,\varepsilon >0$ be given and suppose that $N$ is a submodule of $E$ with $d\leq \dim_{\psi} N <\infty$. Then since $\ker \phi_i\vert_{N} \searrow 0$, where the $\phi_i:E\rightarrow E_i$ are the canonical maps, we can choose by Lemma \ref{lma:dimfinalityspecial} an $i_0$ such that $\dim_{\psi} \ker \phi_{i_0}\vert_{N} < \varepsilon$.

Then for all $j> i_0$ we get $\phi_{i_0j}(E_j) \supseteq (\phi_{i_0j}\circ \phi_j)(N) = \phi_{i_0}(N)$ so that by additivity
\begin{equation}
\inf_{j\geq i_0} \dim_{\psi} \phi_{i_0j}(E_j) \geq d-\varepsilon. \nonumber
\end{equation}
Since $d,\varepsilon$ were arbitrary the claim follows.

Next we prove the opposite inequality. We may assume that $\dim_{\psi} E$ is finite since otherwise the claim is trivially true. We have that $E\simeq \lim_{\leftarrow_k} E_{i_k}$ and it is easy to see that it is enough to prove the claim for the projective system $\{E_p\}_{p\in \{i_k\}}$.

Denote $E_p^q := \phi_{pq}(E_q)\subseteq E_p$ and write also $E_p^{\infty} := \cap_{q:q\geq p}E_p^q$ and $\phi_{pq}^{\infty} := \phi_{pq}\vert_{E_q^{\infty}}:E_q^{\infty} \rightarrow E_p^{\infty}$.

We claim that the $\phi_{pq}^{\infty}$ are $\dim_{\psi}$-surjective, i.e. that the cokernels are zero-dimensional. To see this, let $p<q \leq q_2$ and consider the map
\begin{equation}
\left( E_q^{q_2} \cap \phi_{pq}^{-1}(E_p^{\infty})\right) / E_q^{\infty} \xrightarrow{\phi{pq\vert}} E_p^{\infty}/\phi_{pq}^{\infty}(E_q^{\infty}). \nonumber
\end{equation}
This is surjective by construction for every $q_2\geq q$, and letting $q_2\rightarrow \infty$ the domains decrease to to zero and the claim follows by Lemma \ref{lma:dimfinalityspecial}.

Next we claim that the maps $\phi_p:E\rightarrow E_p^{\infty}$ have cokernels with vanishing $\psi$-dimension as well. Let $x\in E_p^{\infty}$ and $e_1$ be any $\psi$-finite projection, $0<\varepsilon<\psi(e_1)$, such that $x.e_1\in \phi_{p\; (p+1)}^{\infty}(E_{p+1}^{\infty})$ and choose $x_1=x_1.e_1\in E_{p+1}^{\infty}$ such that $x.e_1 = \phi_{p\; (p+1)}^{\infty}(x_1)$. By the same argument as in the proof of Lemma \ref{lma:semifinitetwontrick} we get a subprojection $e_2\leq e_1$ such that $x_1.e_2 \in \phi_{(p+1)\; (p+2)}^{\infty}(E_{p+2}^{\infty})$ and $\varepsilon < \psi(e_2)$.

Continuing in this fashion and putting $e=\wedge_q e_q$ we get $x.e = \phi_p((\dots , x.e, x_1.e, \dots)) \in \phi_p(E)$ and $\psi(e)\geq \varepsilon$. Now since $e_1$ was arbitrary the claim follows from this and Sauer's local criterion.

Finally the theorem follows now by Lemma \ref{lma:dimfinalityspecial} and the definition of $E_p^{\infty}$.
\end{proof}

\begin{remark}
\begin{itemize}
\item Note that the assumption of a sequence such that (...) was not used in the first part of the proof.
\item Also note we can replace the assumption that the $\dim_{\psi} E_{i_k} < \infty$ with the weaker assumption that for each $i$ there is an $i_0\geq i$ such that $\dim_{\psi} \phi_{ii_0}(E_{i_0}) <\infty$.
\end{itemize}
\end{remark}

Next we give some applications of Theorem \ref{thm:dimfinality} extending the result of Proposition \ref{prop:dimelltwo}.

\begin{lemma} \label{lma:diminfinitesum}  \todo{lma:diminfinitesum} \label{lma:dimbyTr} \todo{lma:dimbyTr}
Let $(\mathscr{A},\psi)$ be a semi-finite, $\sigma$-finite, tracial von Neumann algebra, and let $L$ be a right-$\mathscr{A}$ submodule of the countable (Hilbert space) sum $M:= \mathcal{H}\overline{\otimes} L^2\psi$ with $\mathcal{H}$ a separable Hilbert space. Then
\begin{equation} \label{eq:diminfinitesum}
\dim_{\psi}L = \dim_{\psi} \overline{L}^{\lVert \cdot \rVert_2} = (Tr\otimes \psi)(P)
\end{equation}
with $P\in \mathcal{B(H)}\overline{\otimes} \mathscr{A}$ the orthogonal projection onto $\overline{L}^{\lVert \cdot \rVert_2}$. 
\end{lemma}

\begin{proof}
Assume first that $\dim_{\psi}L<\infty$. Let $K$ be a $\psi$-fg projective submodule of $\overline{L}^{\lVert \cdot \rVert_2}$. Let $(p_n)$ be a sequence of projections in $\mathscr{A}_{\psi}^2$ increasing to the identity and let $q_n, n\in \mathbb{N}$ be the projection $p_n\oplus \cdots \oplus p_n\oplus 0\oplus \cdots$ with $n$ first summands equal to $p_n$ and the others zero. Then since $\overline{q_nL}^{\lVert \cdot \rVert_2} \supseteq \overline{q_nK}^{\lVert \cdot \rVert_2}$ it follows that
\begin{equation}
\dim_{\psi} q_nL = \dim_{\psi} \overline{q_nL}^{\lVert \cdot \rVert_2} \geq \dim_{\psi} \overline{q_nK}^{\lVert \cdot \rVert_2} = \dim_{\psi} q_nK. \nonumber
\end{equation}

The first equality here holds since $q_nL\cap (\mathscr{A}^{2}_{\psi})^n$, being a finite-dimensional submodule of $\mathscr{A}^n$ is rank dense in its algebraic closure in $\mathscr{A}^n$. By Corollary \ref{cor:algebraicclosure} this is isomorphic exactly to $p\mathscr{A}^n$ with $p$ the projection onto $\overline{q_nL}^{\lVert \cdot \rVert_2}$. Then the equality follows by Proposition \ref{prop:dimelltwo}. Similarly the final equality.

On the other hand, the kernels of $q_n\vert_{L}$ and $q_n\vert_{K}$ decrease to zero, so that by additivity and Theorem \ref{thm:dimfinality}, this proves the first equality of \eqref{eq:diminfinitesum}. 

Further, it is clear, using again Proposition \ref{prop:dimelltwo}, that
\begin{eqnarray}
\dim_{\psi} \overline{q_nL}^{\lVert \cdot \rVert_2} & = & (Tr_n\otimes \psi) ([q_nP])  \nonumber \\
 & = & (Tr_n\otimes \psi) ([Pq_n]) \nearrow_n (Tr\otimes \psi) (P), \nonumber
\end{eqnarray}
from which the second equality follows. This proves the lemma in case $\dim_{\psi}L<\infty$.

If $\dim_{\psi}L=\infty$ then the first equality is trivial. The second follows, e.g.~by the case just proved, since $L$ contains submodules with arbitrarily large, finite $\psi$-dimension whence $P$ contains subprojections with arbitrarily large trace.
\end{proof}

The next theorem shows that the semi-finite dimension function can in fact be treated entirely within the finite setting in many cases. This extends the fact noted at the end of the proof of \cite[Theorem 2.4]{CoSh03} that if $(\mathscr{A},\tau)$ is a $\mathrm{II}_1$-factor, $q$ a projection in $\mathscr{A}$ and $V$ a right-$\mathscr{A}$-module, then
\begin{equation}
\dim_{q\mathscr{A}q} Vq = \frac{1}{\tau(q)} \cdot \dim_{\mathscr{A}} V. \nonumber
\end{equation}

\begin{theorem}[(Compression)] \label{thm:compression} \todo{thm:compression}
Let $(\mathscr{A},\psi)$ be a $\sigma$-finite, semi-finite tracial von Neumann algebra and $p$ a projection in $\mathscr{A}$ with central support the identity. Consider the semi-finite tracial algebra $(\mathscr{A}_p,\psi_p) = (p\mathscr{A}p,\psi(p\cdot p))$. For any right-$\mathscr{A}$-module $M$ we have for the right-$\mathscr{A}_p$-module $Mp$
\begin{equation}
\dim_{(\mathscr{A},\psi)} M = \dim_{(\mathscr{A}_p,\psi_p)} Mp. \nonumber
\end{equation}
\end{theorem}

\begin{observation}
Suppose that $\mathscr{A}$ is a $\sigma$-finite type $\mathrm{II}_{\infty}$ algebra with faithful normal tracial weight $\psi$. Then there is a projection $p\in \mathscr{A}$ with central support the identity and such that $\psi(p) <\infty$.

Indeed by a standard maximality argument we find a set $\{p_i\}_{i\in \mathbb{N}}$ of projections, with pairwise orthogonal central supports summing to the identity, and such that $\psi(p_i)<\infty$ for all $i$. Then for each $i$ there is a sequence $(p_{i,n})_{n\in \mathbb{N}}$ of subprojections of $p_i$ decreasing to zero and such that the central supports  $C_{p_{i,n}}=C_{p_i}$.

The claim now follows since for each $i$ there is an $n(i)$ such that $\psi(p_{i,n(i)})<\frac{1}{2^i}$ and we may take $p=\sum_i p_{i,n(i)}$.
\end{observation}

The proof of Theorem \ref{thm:compression} relies on the following observation, which we single out. Let $N$ be a right-$\mathscr{A}$-module and let $M$ be a rank-dense submodule. Then $Mp$ is still rank dense in $Np$ when both are considered as $\mathscr{A}_p$-modules.

Indeed if $x\in Np$ we need to find a sequence of projections in $q_n\in \mathscr{A}_p$ such that $x.q_n\in Mp$ and $q_n\nearrow p=\bbb_{\mathscr{A}_p}$. Let $q_n^{\prime}\in \mathscr{A}$ be projections such that $(x.p).q_n^{\prime}\in M$ and $q_n^{\prime}\nearrow \bbb$. Then let the $q_n$ be sufficiently large spectral projections of $pq_n^{\prime}p$.

We are now ready for the proof.

\begin{proof}[Proof of Theorem \ref{thm:compression}.]
Let $P\simeq q\mathscr{A}^n$ be a $\psi$-fg. projective module. Considering $p_{11}=e_{11}\otimes p\in M_n\otimes \mathscr{A}$ where $e_{ij}$ are the matrix units, note that $p_{11}$ has central support the identity in $M_n\otimes \mathscr{A}$.

Then by the comparison theorem \cite[Theorem 6.27]{KRII} and a standard maximality argument we see that $q=\sum_{i\in \mathbb{N}} q_i$ with the $q_i$ pairwise orthogonal and $q_i\lesssim p_{11}$, say by $v_i^{*}q_iv_i \leq p_{11}$. We compute
\begin{eqnarray}
\dim_{\psi}(P) & = & (Tr_n\otimes \psi)(q) \nonumber \\
 & = & \sum_i (Tr_n\otimes \psi)(q_i) \nonumber \\
 & = & \sum_i (Tr_n\otimes \psi)(v_i^{*}q_iv_i) \nonumber \\
 & = & \sum_i \psi_p(v_i^{*}q_iv_i) \nonumber
\end{eqnarray}
and by Lemma \ref{lma:dimbyTr} it is easy to see that this is exactly $\dim_{\psi_p}P$. It follows from this and the remarks preceding the proof, using also additivity, that $\dim_{\psi} M = \dim_{\psi_p} M$ whenever $M$ is $\psi$-fg.

In particular $\dim_{\psi} M \leq \dim_{\psi_p} M$ for any $M$.

For the opposite inequality we may suppose that $\dim_{\psi} M <\infty$ since otherwise there is nothing to prove. Then $Q$, the inductive limit of the net of $\psi$-fg. submodules of $M$ is rank dense in $M$. By the remarks preceding the proof it is also rank dense with respect to the right-$\mathscr{A}_p$-module structure, so the equality follows by equality for $\psi$-fg.~modules and the inductive limit formula (see Theorem \ref{thm:dimensionsummary}).
\end{proof}

It was observed in \cite{KyPe12} that vanishing of dimension, being an algebraic property of the module cf.~Sauer's local criterion, is independent of the choice of faithful normal trace. This still holds true in the semi-finite case.

On the other hand, in the semi-finite case, at least for purposes of continuous cohomology, we have no good substitute for rank completion (see Section \ref{sec:rankhahnbanach}). \todo{should be a rant about measurable operators vs. tot disc grps somewhere.}

To stay within the finite setting, we will find the following simple observation useful.

\begin{proposition} \label{prop:finitetraceproj} \todo{prop:finitetraceproj}
Let $\mathscr{A}$ be a semi-finite, $\sigma$-finite von Neumann algebra. Then there is a finite projection $p_0\in \proj{\mathscr{A}}$ with central support the identity, and a faithful normal trace $\psi_0$ on $\mathscr{A}$ such that $\psi_0(p_0)=1$. \hfill \qedsymbol
\end{proposition}

In particular, we single out the follwoing 

\begin{corollary} \label{cor:dimvanishtrick} \todo{cor:dimvanishtrick}
Let $E$ be an $\mathscr{A}$-module. Then with $p_0,\psi_0$ as in the proposition, and denoting the corner $\mathscr{A}_{p_0}:=p_0\mathscr{A}p_0$, we have with $\dim_{\mathscr{A}}E = 0$ if and only if $\dim_{(\mathscr{A}_{p_0},\psi_0(p_0-))} Ep_0 = 0$. \hfill \qedsymbol
\end{corollary}

   % OKOKOK

\section{Some remarks on rank completion} \label{sec:rankhahnbanach} \todo{sec:rankhahnbanach}

In this section, unless explicitly stated otherwise, $(\mathscr{A},\tau)$ is a \emph{finite} tracial von Neumann algebra. We want to study the notion of rank-completion of modules over $\mathscr{A}$, with the aim of giving a simple, self-contained proof of the fact that the homological and cohomological $L^2$-Betti numbers agree for countable groups in the addendum to the prologue (Section \ref{sec:dualitydiscrete}).

The notion of rank ring was introduced by von Neumann, and the structure of complete rank rings first investigated in \cite{Ha65,Ge75}. These ideas were applied to the study of $L^2$-invariants by A. Thom in \cite{Thom06a,Thom06b}.

However, the proof that homological and cohomological $L^2$-Betti numbers agree alluded in \cite[p. 6]{PeTh06}, relying on Thom's powerful observations in homological algebra, is in some sense not very direct. (Unless the reader finds the Grothendiek spectral sequence very direct.)

In this section we prove a Hahn-Banach extension type theorem for rank-complete modules, which will immediately imply the result we are after, from a direct computation with inhomogeneous (co)chains. As a bonus, we also recover self-injectivity of the ring of affiliated operators by an argument which is in spirit very different from that in \cite[Corollary 15]{Ge75}.

\begin{definition} \label{def:rankfunction} \todo{def:rankfunction}
For any $\mathscr{A}$-module $E$ we define a function $\operatorname{rk}\colon E\rightarrow [0,1]$, called the rank function, by
\begin{equation}
\operatorname{rk} \xi := \sup \{ \tau(p) \in \proj{\mathscr{A}} \mid \xi.p = 0\}^{\perp}. \nonumber
\end{equation}
\end{definition}

We take for granted the following simple facts; see \cite{Thom06a,Thom06b}.
\begin{enumerate}[(i)]
\item The rank function induces a pseudo-metric $d$ on $E$ by $d(\xi,\eta) := \operatorname{rk}(\xi-\eta)$.
\item We denote by $\mathfrak{c}(E)$ the Hausdorff completion of $E$ with respect to this metric.
\item The rank completion $\mathfrak{c}(\mathscr{A})$ of $\mathscr{A}$ is a ring, and $\mathfrak{c}(E)$ is naturally a $\mathfrak{c}(\mathscr{A})$-module.
\item Rank completion is a covariant functor. Any $\mathscr{A}$-morphism $\varphi\colon E\rightarrow F$ is uniformly continuous, and the image $\mathfrak{c}(\varphi)$ of this under rank completion is the continuous extension.
\item For rank complete modules $E,F$, any $\mathscr{A}$-morphism is automatically a $\mathfrak{c}(\mathscr{A})$-morphism, and any such morphism has closed image.
\end{enumerate}

In this setting, Sauer's local criterion implies that a submodule $E_0\subseteq E$ is dense in rank topology if and only if $\dim_{\mathscr{A}}E/E_0 = 0$.

\begin{theorem} \label{thm:rankhahnbanach} \todo{thm:rankhahnbanach}
Let $(\mathscr{A},\tau)$ be a finite, $\sigma$-finite tracial von Neumann algebra. Let $E\subseteq F$, and $Y$ be rank-complete $\mathscr{A}$-modules. Then any $\mathscr{A}$-morphism $\varphi \in \operatorname{hom}_{\mathscr{A}}(E,Y)$ extends to an $\mathscr{A}$-morphism $\bar{\varphi}\in \operatorname{hom}_{\mathscr{A}}(F,Y)$.
\end{theorem}

The proof is a standard argument using Zorn's lemma:\footnote{It turns out this theorem was known, and appears in Goodearls's book, \textit{von Neumann Regular Rings}. I thank the thesis committee for pointing this out to me.}

\begin{proof}
Let $\varphi \in \operatorname{hom}_{\mathscr{A}}(E,Y)$ be given and consider
\begin{equation}
\Phi := \{ (L,\bar{\varphi}) \mid E\leq L\leq F, \mathfrak{c}(L)=L, \bar{varphi}\in \operatorname{hom}_{\mathscr{A}}(L,Y), \; \bar{\varphi} \; \mathrm{extends} \; \varphi \}. \nonumber
\end{equation}
We order $\Phi$ by declaring that $(L_1,\varphi_1')\leq (L_2\varphi_2')$ if $L_1\leq L_2$ and $\varphi_2'$ extends $\varphi_1'$.

Observe that $\Phi$ is nonempty since $(E,\varphi)\in \Phi$. If $\mathcal{J}\subseteq \Phi$ is a chain (i.e.~a totally ordered subset) we define a pair $(L,\bar{\varphi})\in \Phi$ as follows:

Put $L=\mathfrak{c}\left( \cup_{(L',\varphi')\in \mathcal{J}} L\right)$ and define a $\varphi_0$ on $\cup_{(L',\varphi')\in \mathcal{J}} L$ by $\varphi_0\vert_{L'}=\varphi'$ for all $L'$ occuring in pairs in $\mathcal{J}$. This is well-defined since $\mathcal{J}$ is totally ordered. Then let $\bar{\varphi}$ be the continuous extension of $\varphi_0$ to the rank-completion $L$. Note that this is indeed an $\mathscr{A}$-morphism by the functoriality of $\mathfrak{c}$, (iv) above.

Clearly $(L,\bar{\varphi})$ is an upper bound for $\mathcal{J}$ in $\Phi$. Hence by Zorn's lemma $\Phi$ contains a maximal element, which we denote $(L,\bar{\varphi})$ from here on. We have to show that $L=F$.

Suppose for a contradiction that this is not the case and choose $\xi\in F\setminus L$. Denote by $L'$ the submodule of $F$ generated (algebraically) by $L$ and $\xi$, and consider for $Q:=L'/L$ the right-$\mathscr{A}$-morphism $\pi\colon \mathscr{A} \rightarrow Q$ given by $\pi(a)=\xi.a+L$.

Since $L$ is complete in rank metric, it is in particular closed in $L'$, so that the kernel $\ker \pi\subseteq \mathscr{A}$ is closed in rank metric as well. It follows then by Sauer's local criterion that for any submodule $K$ of $\mathscr{A}$ such that $\ker \pi \subsetneq K$, we have
\begin{equation}
\dim_{\mathscr{A}} K/\ker \pi > 0. \nonumber
\end{equation}

By continuity of dimension, Theorem \ref{thm:dimensioncontinuityII}, we conclude that $\ker \pi$ is algebraically closed in $\mathscr{A}$, whence $Q\simeq q\mathscr{A}$ for some projection $q\in \mathscr{A}$.

In particular, $Q$ is a projective $\mathscr{A}$-module, so it follows that $L'$ splits as a direct sum of modules, $L'=L\oplus Q$. In particular we may extend $\bar{\varphi}$ to $\mathfrak{c}(L')$ by zero on $Q$, reaching the contradiction.
\end{proof}

\begin{remark}
%\begin{itemize}
%\item Notice that the argument where we conclude that $L$ is algebraically closed by using rank completeness uses crucially that $\dim_{\mathscr{A}}\mathscr{A}=1<\infty$ in the reference to Theorem \ref{thm:dimensioncontinuityII}.
The standard application of this theorem is the case where $Y=\mathfrak{c}(\mathscr{A})$.
%\end{itemize}
\end{remark}

\begin{corollary} \label{cor:selfinjective} \todo{cor:selfinjective}
On the category of rank complete modules, the functor $\operatorname{hom}_{\mathscr{A}}(-,\mathfrak{c}(\mathscr{A}))$ is exact, and for any rank complete module $E$,
\begin{equation}
\dim_{\mathscr{A}} E = \dim_{\mathscr{A}} \operatorname{hom}_{\mathscr{A}}(E,\mathfrak{c}(\mathscr{A})), \nonumber
\end{equation}
where the right-hand module is a left-$\mathscr{A}$-module by post-multiplication.
\end{corollary}

\begin{proof}
The first part is a special case of the theorem. To prove the equality of dimensions, see Lemma \ref{lma:moddualineq}.
\end{proof}

By spectral theory, $\mathscr{A}$ is rank-dense in its ring of affiliated operators, $\mathcal{U}(\mathscr{A})$. Conversely, $\mathfrak{c}(\mathscr{A})$ is a submodule of $\mathcal{U}(\mathscr{A})$ by \cite[Chapter IX, Theorem 2.5]{TakesakiII}.

Thus $\mathfrak{c}(\mathscr{A}) = \mathcal{U}(\mathscr{A})$, canonically via.~an isomorphism extending the identity map on $\mathscr{A}$.

\begin{corollary}[(Goodearl {\cite{Ge75}})]
The ring of affiliated operators $\mathcal{U}(\mathscr{A})$ is self-injective.
\end{corollary}

In the proof we use synonymously $\mathfrak{c}(\mathscr{A})$ and $\mathcal{U}(\mathscr{A})$ for emphasis.

\begin{proof}
Consider a short exact sequence of $\mathcal{U}(\mathscr{A})$-modules
\begin{displaymath}
\xymatrix{ 0\ar[r] & E \ar[r]^{\iota} & F \ar[r]^{\pi} & Q \ar[r] & 0} .
\end{displaymath}

We may consider these also as $\mathscr{A}$-modules and as such the rank completions make sense. It is obvious that $\operatorname{hom}_{\mathfrak{c}(\mathscr{A})}(-,\mathfrak{c}(\mathscr{A})) = \operatorname{hom}_{\mathscr{A}}(-,\mathfrak{c}(\mathscr{A}))$ on the category of rank complete modules. It then follows that the sequence
\begin{displaymath}
\xymatrix{ 0 & \operatorname{hom}_{\mathcal{U}(\mathscr{A})}(\mathfrak{c}(E),\mathcal{U}(\mathscr{A})) \ar[l] & \operatorname{hom}_{\mathcal{U}(\mathscr{A})}(\mathfrak{c}(F),\mathcal{U}(\mathscr{A})) \ar[l]_{\iota^*} & \operatorname{hom}_{\mathcal{U}(\mathscr{A})}(\mathfrak{c}(Q),\mathcal{U}(\mathscr{A})) \ar[l]_{\pi^*} & 0\ar[l] }
\end{displaymath}
is exact. It is also obvious that $\operatorname{hom}_{\mathfrak{c}(\mathscr{A})}(-,\mathfrak{c}(\mathscr{A})) = \operatorname{hom}_{\mathfrak{c}(\mathscr{A})}(\mathfrak{c}(-),\mathfrak{c}(\mathscr{A}))$. This completes the proof.
\end{proof}

\begin{remark}
In general, one should think of the rank metric as a generalization of the discrete metric, in particular when $\mathscr{A}=\mathbb{C}$, the rank metric is the discrete metric, and every vector space is a complete $\mathscr{A}$-module.
\end{remark}

\section{Dimension exactness result for semi-finite tracial algebras}

L{\"u}ck shows in \cite[Theorem 6.29]{Lu02} that given a (trace-preserving) inclusion of finite von Neumann algebras $\mathscr{A}\leq \mathscr{B}$, the induction functor $-\otimes_{\mathscr{A}} \mathscr{B}$ is a faithfully flat (dimension-preserving) functor from the category of $\mathscr{A}$-modules to the category of $\mathscr{B}$-modules.

The same is not necessarily true for the functor $\operatorname{hom}_{\mathscr{A}}( - , \mathscr{B})$. Indeed, it is not even true in the case $\mathscr{A}=\mathscr{B}$. (To see this, consider e.g.~the module $\operatorname{hom}_{\mathscr{A}}(Z,\mathscr{A})$ for any zero-dimensional $\mathscr{A}$-module $Z$.)

However, because the right-$\mathscr{A}$-module $\mathfrak{c}(\mathscr{A})$ is also a left-$\mathscr{A}$-module in a nice way, the Hahn-Banach type theorem \ref{thm:rankhahnbanach} implies that, restricting the domain to countably generated modules, $\operatorname{hom}_{\mathscr{A}}(-,\mathscr{A})$ is in fact dimension-exact. Let us formalize exactly meaning of this statement:

\begin{definition}
A short sequence $\xymatrix{ E\ar[r]^{\iota} & F \ar[r]^{\pi} & Q }$ of $\mathscr{A}$-modules such that $\pi \circ \iota = 0$ is called dimension exact at $F$ if $\dim_{(\mathscr{A},\psi)} \ker \pi / \operatorname{im}(\iota) = 0$.

Let $\mathfrak{E}'$ be a category of $\mathscr{A}$-modules. A functor (either contra- or covariant) $\mathfrak{f}\colon \mathfrak{E}' \rightarrow \mathfrak{E}_{\mathscr{A}}$ into the category of all $\mathscr{A}$-modules is called dimension exact if it maps short dimension exact sequences to short dimension exact sequences.

The functor $\mathfrak{f}$ is called dimension preserving if $\dim_{(\mathscr{A},\psi)} \mathfrak{f}E = \dim_{(\mathscr{A},\psi)} E$ for all objects $E\in \mathfrak{E}'$.
\end{definition}

The properties of dimension exact functors are best stated at the end of the chapter, after introducing the notion of a quasi-morphism (see Proposition \ref{prop:dimexactfunctor}). We now prove some results about dimension exactness using the extension theorem.

\begin{theorem} \label{thm:homdimexactness} \todo{thm:homdimexactness}
Let $\mathscr{A}$ be a finite von Neumann algebra and $\mathscr{B}$ a semi-finite von Neumann algebra, both $\sigma$-finite, and such that $\mathscr{A}\leq \mathscr{B}$.

Then on the category of countably generated modules, the functor $\operatorname{hom}_{\mathscr{A}}(-,\mathscr{B})$ (say, from right-$\mathscr{A}$-modules to left-$\mathscr{B}$-modules) is dimension-exact. More generally, the same is true on the category of modules with a countably generated rank dense submodule.
\end{theorem}

\begin{proof}
We already know from \ref{thm:rankhahnbanach} that the functor $\operatorname{hom}_{\mathscr{A}}(-,\mathfrak{c}_{\mathscr{A}}(\mathscr{B}))$ is exact on rank-complete modules. Hence the first part of the theorem will follow from the fact that for any countably generated left-$\mathscr{A}$-module $Z$ the inclusion of left-$\mathscr{B}$-modules
\begin{equation}
\operatorname{hom}_{\mathscr{A}}(Z,\mathscr{B}) \subseteq \operatorname{hom}_{\mathscr{A}}(Z,\mathfrak{c}_{\mathscr{A}}(\mathscr{B})) \nonumber
\end{equation}
is rank-dense.

By the countable annihilation lemma and the assumption on $Z$, it is sufficient to show that the inclusion $\mathscr{B} \subseteq \mathfrak{c}_{\mathscr{A}}(\mathscr{B})$ of left-$\mathscr{B}$-modules is rank-dense. Let $f\in \mathfrak{c}_{\mathscr{A}}(\mathscr{B})$. Then there is a sequence of projections $p_n\in \mathscr{A}$ such that $f_n:=f.p_n\in \mathscr{B}$ for all $n$, and $p_n\nearrow \bbb$.

For each $n$, let $q_n'\in \mathscr{B}$ be the range projection (left support) of $f_n$. Then we claim that for all $m$, actually $q_m'.f\in \mathscr{B}$. Indeed, we have for $n\geq m$ that $q_m'.f_n = f_m$ whence taking the limit, $q_m'.f=f_m$.

Clearly the $q_n'$ are increasing, and letting $q_n:=q_n'+(\vee_mq_m')^{\perp}$, we have $q_n.f\in \mathscr{B}$ for all $n$ and $q_n\nearrow \bbb$ as had to be shown. The final part of the theorem is now obvious, since if $Z\leq Z'$ is a rank dense inclusion, one considers the diagram
\begin{displaymath}
\xymatrix{ \operatorname{hom}_{\mathscr{A}}(Z,\mathscr{B}) \ar@{=}[d] & \subseteq & \operatorname{hom}_{\mathscr{A}}(Z,\mathfrak{c}_{\mathscr{A}}(\mathscr{B})) \ar@{=}[d] \\ \operatorname{hom}_{\mathscr{A}}(Z',\mathscr{B}) & \subseteq & \operatorname{hom}_{\mathscr{A}}(Z',\mathfrak{c}_{\mathscr{A}}(\mathscr{B})) } .
\end{displaymath}

This completes the proof.
\end{proof}

\begin{remark}
%\begin{enumerate}[(i)]
%\item L{\"u}ck's result \cite[Theorem 6.29]{Lu02} actually holds for any inclusion of von Neumann algebras. Similarly, the above holds in that setting as well, once we specify that, in the general case a "dimension-exact" sequence of modules is one where the homologies contain no finitely generated projective submodules.
The proof given above can be shortened a bit, and perhaps made conceptually more clear, by noting that $\mathfrak{c}_{\mathscr{A}}(\mathscr{B})$ is contained in the set of operators affiliated with $\mathscr{B}$, and then appealing to spectral theory.

On the other hand, the given proof is more general, at least formally, and should be compared with established notions of "dimension-compatibility" for bimodules \cite{Sau03,KyPe12}.
%\end{enumerate}
\end{remark}

The next result specializes to the case $\mathscr{A}=\mathscr{B}$ to get a more optimal result.

\begin{theorem} \label{thm:homdimexactpreserv} \todo{thm:homdimexactpreserv}
Let $(\mathscr{A},\psi)$ be a semi-finite, $\sigma$-finite tracial von Neumann algebra. Then the functor $\operatorname{hom}_{\mathscr{A}}(-,\mathscr{A})$ is dimension exact and -preserving when restricted to the category of countably generated modules.

Further, the same statement holds considering the category of modules which contain a rank dense countably generated submodule.
\end{theorem}

The point of the last part, which trivially follows from the first, is that the property of an $\mathscr{A}$-module being countably generated up to rank density is stable under passing to submodules. Hence the theorem implies, via the usual kernel-cokernel short exact sequences arising from a (long) complex of $\mathscr{A}$-modules, that passing to duals, the cohomology is rank-isomorphic to the dual of homology (and vice versa), so long as the terms in the complex are countably generated up to rank. See Proposition \ref{prop:dimexactfunctor}.

\begin{proof}
Let $\xymatrix{0\ar[r] & E \ar[r]^{\iota} & F \ar[r]^{\pi} & Q \ar[r] & 0}$ be a short dimension exact sequence of countably generated left-$\mathscr{A}$-modules. We can assume without loss of generality that $\mathscr{A}$ contains a projection $p$ with finite trace $\psi(p)=1<\infty$ (we normalize it for notational convenience only) and central support the identity; otherwise we could take an increasing sequence $p_n\nearrow \bbb$ with $\psi(p_n)<\infty$ and start out by considering the two sequences as follows, with $C_{p_n}$ the central support of $p_n$,
\begin{displaymath}
\xymatrix{ 0\ar[r] & C_{p_n}E \ar[r]^{\iota} & C_{p_n}F \ar[r]^{\pi} & C_{p_n}Q \ar[r] & 0 }
\end{displaymath}
and
\begin{displaymath}
\xymatrix{ 0 & \ar[l] \operatorname{hom}_{C_{p_n}\mathscr{A}}(C_{p_n}E,C_{p_n}\mathscr{A}) & \ar[l]_{\iota^*} \operatorname{hom}_{C_{p_n}\mathscr{A}}(C_{p_n}F,C_{p_n}\mathscr{A}) & \ar[l]_{\pi^*} \operatorname{hom}_{C_{p_n}\mathscr{A}}(C_{p_n}Q,C_{p_n}\mathscr{A}) & \ar[l] 0 }.
\end{displaymath}

Having proved the statement in the assumed case, it would follow that the lower sequence is dimension exact and with same dimensions as the upper one. Then one takes the limit in $n$ cf.~Proposition \ref{prop:dimcentralprojections}.

We recall also from the remarks preceding Theorem \ref{thm:homdimexactness} that the present claim is true in case $\psi$ is finite.

Denote by $\mathscr{A}_p:=p\mathscr{A}p$ the corner. Since the functor $X\mapsto pX$ from $\mathscr{A}$-modules to $\mathscr{A}_p$ modules is exact and dimension-preserving, and $\operatorname{hom}_{\mathscr{A}_p}(-,\mathscr{A}_p)$ is dimension exact and dimension preserving on countably generated modules, it is sufficient to compare the short dimension exact sequence
\begin{displaymath}
\xymatrix{ 0 & \ar[l] \operatorname{hom}_{\mathscr{A}_p}(pE,\mathscr{A}_p) & \ar[l]_{\iota^*} \operatorname{hom}_{\mathscr{A}_p}(pF,\mathscr{A}_p) & \ar[l]_{\pi^*} \operatorname{hom}_{\mathscr{A}_p}(pQ,\mathscr{A}_p) & \ar[l] 0 }
\end{displaymath}
with the complex
\begin{displaymath}
\xymatrix{ 0 & \ar[l] \operatorname{hom}_{\mathscr{A}}(E,\mathscr{A}p) & \ar[l]_{\iota^*} \operatorname{hom}_{\mathscr{A}}(F,\mathscr{A}p) & \ar[l]_{\pi^*} \operatorname{hom}_{\mathscr{A}}(Q,\mathscr{A}p) & \ar[l] 0 }
\end{displaymath}
of right-$\mathscr{A}_p$-modules.

For this, consider for $\varphi\in \operatorname{hom}_{\mathscr{A}}(X,\mathscr{A}p)$, $X$ a countably generated left-$\mathscr{A}$-module, the restriction $\varphi\vert_{pX}\in \operatorname{hom}_{\mathscr{A}_p}(pX,\mathscr{A}_p)$. The map $\varphi\mapsto \varphi\vert_{pX}$ is clearly $\mathscr{A}_p$-linear.

Clearly this commutes with $\iota^*,\pi^*$ whence we get a map of complexes
\begin{displaymath}
\xymatrix{ 0 & \ar[l] \operatorname{hom}_{\mathscr{A}}(E,\mathscr{A}p) \ar[d]^{-\vert_{pE}} & \ar[l]_{\iota^*} \operatorname{hom}_{\mathscr{A}}(F,\mathscr{A}p) \ar[d]^{-\vert_{pF}} & \ar[l]_{\pi^*} \operatorname{hom}_{\mathscr{A}}(Q,\mathscr{A}p) \ar[d]^{-\vert_{pQ}} & \ar[l] 0 \\ 0 & \ar[l] \operatorname{hom}_{\mathscr{A}_p}(pE,\mathscr{A}_p) & \ar[l]_{\iota^*} \operatorname{hom}_{\mathscr{A}_p}(pF,\mathscr{A}_p) & \ar[l]_{\pi^*} \operatorname{hom}_{\mathscr{A}_p}(pQ,\mathscr{A}_p) & \ar[l] 0 }
\end{displaymath}

Hence the claim is proved once we show that in general $\varphi\mapsto \varphi\vert_{pX}$ is a rank-isomorphism. It is clearly surjective. Suppose that $\varphi\vert_{pX} = \phi\vert_{pX}$. Letting $v_i$ be partial isometries in $\mathscr{A}$ such that $\sum_i v_iv_i^* = \bbb$ and $v_i^*v_i\leq p$ for all $i$, we may write each generator $x_k,k\in \mathbb{N}$ as a formal (i.e.~we don't really mean anything rigorously by this comment) infinite sum $x_k=\sum v_i.(v_i^*.x_k)$. In particular, by the countable annihilation lemma, the module $X^{\circ}$ generated by all finite sums $\sum_{k,i}^{fin} v_i(v_i^*.x_k)$ is rank-dense in $X$.

Further, we have $\varphi\vert_{X^{\circ}}=\phi\vert_{X^{\circ}}$. Since for every $x_k$, $(\varphi-\phi)(x_k)$ has range in $\mathscr{A}$ subequivalent to $p$, it follows then that $(\varphi-\phi)(x_k).p^{(k)}_n = 0$ for projections $p_n^{(k)}\nearrow_n p$ in $\mathscr{A}_p$. By the standard $\varepsilon/2^n$ argument, it follows that $(\varphi-\phi).p_n=0$ for $p_n\nearrow_n p$ in $\mathscr{A}_p$, as had to be shown.
\end{proof}

Finally, we shall come full circle and prove a version of \cite[Theorem 6.29]{Lu02} in our setting. For applications we only need a dimension exactness statement, so that is what we show. If the inclusion is trace-preserving, the functor will be dimension preserving as well.

\begin{theorem} \label{thm:tensordimexactness} \todo{thm:tensordimexactness}
Let $\mathscr{A}\leq \mathscr{B}$ be an inclusion of semi-finite, $\sigma$-finite tracial von Nemann algebras. We do not assume it to be trace-preserving.

Then the induction functor $- \otimes_{\mathscr{A}} \mathscr{B}$ is dimension exact.
\end{theorem}

It seems likely that one could use a direct generalization of L{\"u}ck's proof. However, we prefer to base the argument on Theorem \ref{thm:homdimexactness} whence indirectly on the Hahn-Banach type theorem \ref{thm:rankhahnbanach}.

\begin{proof}
Let $\xymatrix{0\ar[r] & E\ar[r]^{\iota} & F\ar[r]^{\pi} & Q \ar[r] & 0}$ be a short dimension exact sequence of right-$\mathscr{A}$-modules.

Considering $p_n\nearrow_n \bbb$ in $\mathscr{A}$ with finite trace $\psi_{\mathscr{A}}(p_n) <\infty$, we get a map of complexes
\begin{displaymath}
\xymatrix{ 0\ar[r] & E\otimes_{\mathscr{A}} \mathscr{B} \ar[r]^{\iota \otimes \bbb} & F\otimes_{\mathscr{A}} \mathscr{B} \ar[r]^{\pi\otimes \bbb} & Q\otimes_{\mathscr{A}} \mathscr{B} \ar[r] & 0 \\ 0\ar[r] & \lim_{\rightarrow} Ep_n\otimes_{\mathscr{A}_{p_n}} \mathscr{B}_{p_n} \ar[r]^{\iota \otimes \bbb} \ar[u] & \lim_{\rightarrow} Fp_n\otimes_{\mathscr{A}_{p_n}} \mathscr{B}_{p_n} \ar[r]^{\pi\otimes \bbb} \ar[u] & \lim_{\rightarrow} Qp_n\otimes_{\mathscr{A}_{p_n}} \mathscr{B}_{p_n} \ar[r] \ar[u] & 0 }
\end{displaymath}
where we note that the vector spaces in the lower complex are modules over every $\mathscr{B}_{p_n}$ but not necessarily over $\mathscr{B}$, and the vertical arrows are injective with rank-dense image, in the obvious sense.

Hence it is sufficient to show that the complex
\begin{displaymath}
\xymatrix{ 0\ar[r] & Ep_n\otimes_{\mathscr{A}_{p_n}} \mathscr{B}_{p_n} \ar[r]^{\iota \otimes \bbb} & Fp_n\otimes_{\mathscr{A}_{p_n}} \mathscr{B}_{p_n} \ar[r]^{\pi\otimes \bbb} & Qp_n\otimes_{\mathscr{A}_{p_n}} \mathscr{B}_{p_n} \ar[r] & 0 }
\end{displaymath}
is dimension-exact for every $n$.

Supposing that $E,F$ are countably generated as $\mathscr{A}$-modules, this follows now by the previous two theorems, since the tensor product is adjoint to $\operatorname{hom}$ (see e.g.~\cite[Proposition 2.6.3]{Weibelbook}), i.e. we have natural isomorphisms $\operatorname{hom}_{\mathscr{B}_{p_n}}(Ep_n\otimes_{\mathscr{A}_{p_n}} \mathscr{B}_{p_n}, \mathscr{B}_{p_n}) \xrightarrow{\simeq} \operatorname{hom}_{\mathscr{A}_{p_n}}(Ep_n,\mathscr{B}_{p_n})$ and similarly for $F,Q$.

To finish the proof then, it is sufficient to note that the tensor product "commutes" with colimits \cite[Theorem 2.6.10]{Weibelbook}
\end{proof}

\section{Quasi-morphisms and localization} \label{sec:quasimod} \todo{sec:quasimod}

In this section we discuss the rank completion from a category-theoretic point of view, on a somewhat informal level, in order to motivate the approach in Appendix \ref{chap:QC}. Let $(\mathscr{A},\tau)$ be a finite, $\sigma$-finite tracial von Neumann algebra.
\begin{enumerate}[(i)]
\item We may think of the category of rank-complete modules over $\mathscr{A}$ as a full subcategory of the category of all $\mathscr{A}$-modules, consisting of those modules that are Hausdorff complete in a canonical uniform structure. This is an abelian subcategory.

Further, there is a functor $\mathfrak{c}$ from the category of all modules to the category of rank-complete modules, which is idempotent ($\mathfrak{c}(\mathfrak{c}(E)) = \mathfrak{c}(E)$), and sends dimension-exact sequences to exact sequences.

\item Alternatively, we may think of the category of rank-complete modules as the localization of the category of all modules in the Serre subcategory consisting of all modules and morphisms into zero-dimensional modules \cite[Section 2.6]{SaTo10}.

The result is that one formally invert all dimension-isomorphisms, i.e.~morphisms $\varphi\colon E\rightarrow F$ such that $\dim_{(\mathscr{A},\tau)} \ker \varphi = \dim_{(\mathscr{A},\tau)} \operatorname{coker} \varphi = 0$.
\end{enumerate}

As we have already discussed above, for a semi-finite $(\mathscr{B},\psi)$ there seems to be no nice way to implement a rank-completion functor similarly to (i). The obstruction, in spirit, seems to be that here there is a difference between the space of operators affiliated with $\mathscr{B}$, and the ring of $\psi$-measurable operators. In particular, the latter depend on the choice of $\psi$, but rank-completion is an entirely algebraic process thanks to Sauer's local criterion, and is thus independent of $\psi$.

Possibly one could consider the approach in (ii) and just localize in the subcategory generated by morphisms into zero-dimensional modules. However, this is complicated by the fact that in general when the algebra is semi-finite, we like to restrict attention to modules that are also topological spaces and morphisms that are continuous, and we want to preserve some of that information. In particular, we might not want to invert all rank-isomorphisms, but only those also with nice bi-continuity properties. A more direct, category-theoretic way to say this is that we are in practice working with exact categories, suggesting the following

\begin{problem} \label{prob:rankcompletionsemifinite} \todo{prob:rankcompletionsemifinite}
Investigate, if possible, the notion of rank-completion wrt.~a semi-finite, $\sigma$-finite tracial von Neumann algebra $(\mathscr{B},\psi)$ via localization of exact categories.
\end{problem}

However, we will need some abstract setup to replace rank-completion in the setting of topological modules over $\mathscr{B}$. To motivate the developments in Appendix \ref{chap:QC}, we make the following observation, adding one point to our list of approaches to rank-completion in the finite case:

\begin{enumerate}[(i)]
\item[(iii)] Consider the category of all modules over $\mathscr{A}$ and let $f\colon E\rightarrow F$ be a morphism. Then $f$ is a rank-isomorphism, i.e.~$\mathfrak{c}(f)\colon \mathfrak{c}(E)\rightarrow \mathfrak{c}(F)$ is an isomorphism, if and only if there are submodules $E''\leq E'\leq E$ and $F''\leq F'\leq F$ such that
\begin{equation}
\dim_{\mathscr{A}} E/E' = \dim_{\mathscr{A}} E'' = \dim_{\mathscr{A}} F/F' = \dim_{\mathscr{A}} F'' = 0, \nonumber
\end{equation}
and an isomorphism $f_0\colon E'/E'' \rightarrow F'/F''$ fitting into a commutative diagram
\begin{displaymath}
\xymatrix{ E \ar[r]^f & F \\ E' \ar[r]^{f\vert_{E'}} \ar[u] \ar[d] & F' \ar[u] \ar[d] \\ E'/E'' \ar[r]^{f_0} & F'/F'' } .
\end{displaymath}

Hence an approach, which is in the spirit of Problem \ref{prob:rankcompletionsemifinite}, is to instead add morphisms to the category of $\mathscr{A}$-modules and consider them up to equivalence in the sense of the diagram above.
\end{enumerate}

\begin{definition}
Let $(\mathscr{B},\psi)$ be a semi-finite, $\sigma$-finite tracial von Neumann algebra. We consider the category with objects $\mathscr{B}$ modules and morphisms equivalence classes of quasi-morphisms, where a quasi-morphism $E\rightarrow F$ is a partially defined morphism $f_0\colon E_0'/E_0'' \rightarrow F_0'/F_0''$ where $E_0''\subseteq E_0'\subseteq E$, $F_0''\subseteq F_0'\subseteq F$,
\begin{equation}
\dim_{\mathscr{B}} E/E_0' = \dim_{\mathscr{B}} E_0'' = \dim_{\mathscr{B}} F/F_0' = \dim_{\mathscr{B}} F_0'' = 0, \nonumber
\end{equation}
and two quasi-morphisms $f_0,f_1$ are equivalent if there is a quasi-morphism $f_2$ and a commutative diagram
\begin{displaymath}
\xymatrix{ E_0'/E_0'' \ar[r]^{f_0} & F_0'/F_0'' \\ E_2'/E_2'' \ar[r]^{f_2} \ar[u] \ar[d] & F_2'/F_2'' \ar[u] \ar[d] \\ E_1'/E_1'' \ar[r]^{f_1} & F_1'/F_1'' } .
\end{displaymath}

Abusing terminology, we also call such an equivalence class of quasi-morphisms a quasi-morphism.

Also somewhat misleading, we call this category the category of quasi-modules.
\end{definition}

The following statements are then easy to check. We leave out the details here, but see Appendix \ref{chap:QC} for similar propositions in the topological case.

\begin{proposition}
The category of quasi-modules is abelian. The canonical functor $\mathfrak{q}$ from the category of $\mathscr{B}$-modules is dimension-exact.
\end{proposition}

\begin{proposition}
A morphism $f\colon E\rightarrow F$ of $\mathscr{B}$-modules is a dimension isomorphism if and only if it is invertible in the category of quasi-morphisms.
\end{proposition}

\begin{proposition}
If $P$ is a projective $\mathscr{A}$-module, then $\mathfrak{q}P$ is also a projective quasi $\mathscr{A}$-module.

If $E$ is an injective $\mathscr{B}$-module with no non-trivial zero-dimensional submodules, then $\mathfrak{q}E$ is also an injective $\mathscr{B}$-module.
\end{proposition}

The difference between finite and semi-finite cases here is curious. Even curiouser is that it seemingly disappears in the topological case: In the algebraic case it is not at all reasonable to require that $P$ has no non-trivial co-dimension zero submodules, but in the topological case it is reasonable to require that $P$ has no non-trivial, co-dimension zero, closed submodules. See Appendix \ref{chap:QC}.

Let us comment on the case of a discrete group $\Gamma$ and its group von Neumann algebra $L\Gamma$. Then we can think of the cohomology $H_n(\Gamma, E)$ as the $n$'th left-derived functor of of the co-invariants functor $\mathscr{C}_{\Gamma}$ on the category of $L\Gamma$-$\Gamma$-modules.

This construction also gives a homology, $H_n^{\mathfrak{Q}}(\Gamma,\mathfrak{q}E)$ if we consider instead the co-invariants functor on quasi $L\Gamma$-modules with commuting action of $\Gamma$. This gives a square of functors
\begin{displaymath}
\xymatrix{ E \ar@{~>}[r] \ar@{~>}[d] & \mathfrak{q}E \ar@{~>}[d] \\ H_n(\Gamma,E) \ar@{~>}[r]^{\mathfrak{q}} & ? }
\end{displaymath}
where in the question mark we have, going round one way $\mathfrak{q}H_n(\Gamma,E)$ and round the other $H_n^{\mathfrak{Q}}(\Gamma,\mathfrak{q}E)$. The standard way to formalize the commutativity of this diagram in an abstract argument, is to refer to a Grothendieck spectral sequence \cite[Section 5.8]{Weibelbook} (this argument is emplyoed for the rank completion in \cite{Thom06b}). In concrete case one may proceed by analyzing the bar resolution, as we do in Appendix \ref{chap:QC} for the continuous cohomology.

Finally, we list some properties of dimension exact functors for easy reference.

\begin{proposition} \label{prop:dimexactfunctor} \todo{prop:dimexactfunctor}
Let $(\mathscr{A},\psi)$ be a semi-finite, $\sigma$-finite, tracial von Neumann algebra.

Let $\mathfrak{E}'$ be a category of $\mathscr{A}$-modules and $\mathfrak{f}\colon \mathfrak{E}' \rightarrow \mathfrak{E}_{\mathscr{A}}$ a functor.
\begin{enumerate}[(i)]
\item The functor $\mathfrak{f}$ is dimension exact if and only if $\mathfrak{q}\circ \mathfrak{f}$ is exact.
\item If $\psi$ is finite then $\mathfrak{f}$ is dimension exact if and only if $\mathfrak{c}\circ \mathfrak{f}$ is exact.
\item Suppose $\mathfrak{f}$ is contravariant and dimension exact. Then for any complex $(E_*',d_*)\rightarrow 0$ in $\mathfrak{E}'$ such that the (kernels, images, and) homology is defined in $\mathfrak{E}'$, we have $\mathfrak{q}(H^n(\mathfrak{f}(E_*'))) \simeq q(\mathfrak{f}(H_n(E_*')))$ for all $n$. Similarly for covariance and interchanging homology and cohomology.
\item In case (iii) above, suppose that $\mathfrak{f}$ is also dimension preserving. Then
\begin{equation}
\dim_{(\mathscr{A},\psi)} H^n(\mathfrak{f}(E_*')) = \dim_{(\mathscr{A},\psi)} H_n(E_*'). \nonumber
\end{equation}
\end{enumerate}
\end{proposition}

\begin{flushright}
\qedsymbol
\end{flushright}

 % ALMOST OKOK

\chapter{Continuous (co)homology for locally compact groups} \label{chap:cohom}

\section{Continuous cohomology}

We recall here quite briefly the definition of continuous cohomology for locally compact groups. For exhaustive details we refer to the book by Guichardet \cite{Guichardetbook}, on which this section is based, or to \cite{BorelWallachBook}; the latter occasionally referred to as 'the orange book from hell'.

Let $G$ be a locally compact group, and suppose as a blanket assumption that it is $2$nd countable. A continuous (or topological) left-$G$-module is a topological vector space $E$ with an action of $G$ such that the map $G\times E \owns (g,e) \mapsto g.e\in E$ is continuous. 

Let $(\mathscr{A},\psi)$ be a semi-finite tracial von Neumann algebra.

We consider the category $\mathfrak{E}_{G,\mathscr{A}}$ of topological $G$-$\mathscr{A}$-modules, i.e.~continuous $G$-modules which are in addition locally convex Hausdorff spaces, and carry a commuting right-action of $\mathscr{A}$ by continuous maps.

The morphisms in this category are the continuous $G$-$\mathscr{A}$-linear maps.

\begin{definition}[(See {\cite[Appendix D.1]{Guichardetbook}})] \label{def:relinjective} \todo{def:relinjective}
A morphism $\pi:E\rightarrow F$ in the category of topological $G$-$\mathscr{A}$-modules is called strengthened if both maps $\ker \pi \rightarrow E$  and $\overline{\pi}:E/\ker \pi \rightarrow F$ have continuous, $\mathscr{A}$-linear left-inverses (not necessarily $G$-linear).
\end{definition}

The category of topological $G$-modules is not abelian, so the standard constructions of homological functors in such categories do not apply. One way to get around this is to restrict the class of short exact sequences under consideration in the definition of injective objects; this is the approach of "relative homological algebra" \cite{HochschildRelative}.

\begin{definition}[(See {\cite[Chapter III]{Guichardetbook}})]
A topological $G$-$\mathscr{A}$-module $E$ is (relatively) injective if, given any diagram

\begin{displaymath}
\xymatrix{  & E & & \\ U \ar@{-->}[ur]^{\exists ? w} & V \ar[l]_{u} \ar[u]_{v} & 0 \ar[l]  }
\end{displaymath}
in which the bottom row is exact with $u$ strengthened, there is a morphism $w:U\rightarrow E$ making the diagram commute. 

In the sequel we will just call these injective, with the tacit understanding that we really mean \emph{relatively} injective, in this sense.
\end{definition}

Given a topological $G$-$\mathscr{A}$-module $E$, the space of continuous maps $C(G,E)$ is endowed with the projective topology induced by the maps $C(G,E)\rightarrow C(K,E_q)$ where $K$ runs over the compact subsets of $G$ and $q$ over a separating family of semi-norms on $E$ (strictly speaking, we take a net of such), $E_q$ here denoting the Banach space completion of the semi-normed space $(E,q)$. This is a topological $G$-$\mathscr{A}$-module when given either of the $G$-actions
\begin{equation}
(g.f)(t)=g.f(g^{-1}t) \quad \mathrm{or}\quad (g.f)(t)=g.f(tg). \nonumber
\end{equation}

Unless explicitly stated otherwise, we always equip it with the former.

\begin{theorem}[(see {\cite[Chapter III, Proposition 1.2]{Guichardetbook}})]
For any topological $G$-$\mathscr{A}$-module $E$, the module $C(G,E)$ is an injective module in the category of topological $G$-$\mathscr{A}$-modules.
\end{theorem}

In the usual manner we thus observe that the category of topological $G$-$\mathscr{A}$-modules has sufficiently many injectives and proceed to construct, for any given $E$, an injective resolution

\begin{displaymath}
\xymatrix{ 0 \ar[r] & E \ar[r]^>>>>{\epsilon} & C(G,E) \ar[r]^{d_{'}^0} &  C(G^2,E) \ar[r]^<<<<{d_{'}^{1}} & \cdots },
\end{displaymath}
where $\epsilon(e)(g)=e$ and the cobooundary maps are given by
\begin{equation} \label{eq:coboundarydef}
(d_{'}^n\xi)(g_0,\dots ,g_{n+1}) = \sum_{i=0}^{n+1} (-1)^{i} f(g_0,\dots ,\hat{g_i},\dots ,g_{n+1}).
\end{equation}

\begin{comment} %%%%%%%%%%%%%%%%
These preserve $G$-invariant elements, and one defines the $n$'th continuous cohomology $H^n(G,E)$ as the $n$'th cohomology of the complex
\begin{equation}
0\rightarrow C(G,E)^{G} \xrightarrow{d_{'}^{0}} C(G^2,E)^{G} \xrightarrow{d_{'}^{1}} \cdots \nonumber
\end{equation}
\end{comment} %%%%%%%%%%%%%%%%%%

\begin{definition}
Let $E$ be a topological $G$-$\mathscr{A}$-module and
\begin{displaymath}
\xymatrix{ 0\ar[r] & E \ar[r]^{\epsilon} & E_0 \ar[r]^{d^0} & E_1 \ar[r] & \cdots \ar[r]^{d^{n-1}} &  E_n \ar[r]^{d^n} & \cdots }
\end{displaymath}
any injective strengthened resolution of $E$, we define the $n$'th cohomology of $G$ with coefficients in $E$ by
\begin{equation}
H^n(G,E) := \ker d^n\vert_{E_n^G} \left/ \operatorname{im} d^{n-1}\vert_{E_{n-1}^G}. \right. \nonumber
\end{equation}
This is a not necessarily Hausdorff topological vector space which carries a right-action of $\mathscr{A}$ by continuous maps, and this structure, including the topolgy, is independent of the choice of injective resolution.
\end{definition}

The independence of choice of injective resolution is a proposition. The proof follows e.g.~that in \cite[Chapter III, Section 1]{Guichardetbook}, with only a minimal amount of extra book keeping for the $\mathscr{A}$-action.

If $E$ is complete one may consider locally square integrable functions instead of continuous functions. Specifically, for $K\subseteq G$ compact, $L^{2}(K,E)$ is the projective limit of spaces $L^{2}(K,E_q)$ where the $q$ ranges over a separating net of semi-norms, ordered in the natural manner. Then $L^{2}_{loc}(G,E)$ is defined as the projective limit of the $L^{2}(K,E)$ over compact subsets. This is equiped with the action of $G$, defined again by $(g.\xi)(t)=g.\xi(g^{-1}t)$.

Note that this construction applies with any Borel space $X$, equiped with a Radon measure, in place of $G$ with its Haar measure, to yield a space $L^2_{loc}(X,E)$. In order to get a Fr{\'e}chet space, it is preferable to have a cofinal (in the set of compact subsets with its natural ordering by inclusion) sequence of compact sets $(X_n)$, with union $X$, cf.~the following.

\begin{observation} \label{obs:elltwoloccofinal} \todo{obs:elltwoloccofinal}
Since $G$ is $2$nd countable and $E$ is complete, then the second projective limit in the definition of $L^{2}_{loc}(G,E)$ has a cofinal subsequence. Indeed let $\{U_n\}$ be a relatively compact neighbourhood basis for $G$. The consider the family $K_N = \cup_{n=1}^{N} \overline{U_n}, N\in \mathbb{N}$.

In particular we note that if $E$ has a countable neighbourhood basis at the origin then so does each $L^2(K,E_q)$ whence so does $L^2_{loc}(G,E)$ so that if $E$ is complete metrizable, then so is $L^2_{loc}(G,E)$.
\end{observation}

\begin{theorem}[(see \cite{Guichardetbook})]
For any complete topological $G$-$\mathscr{A}$-module $E$, the topological $G$-$\mathscr{A}$-module $L^{2}_{loc}(G,E)$ is injective, and there is a strengthened resolution
\begin{displaymath}
\xymatrix{ 0 \ar[r] & E \ar[r]^>>>>{\epsilon} & ^2_{loc}(G,E) \ar[r]^{d_{'}^0} &  L^2_{loc}(G^2,E) \ar[r]^<<<<{d_{'}^{1}} & \cdots },
\end{displaymath}
where the coboundary maps are the same as in equation \eqref{eq:coboundarydef}.
\end{theorem}

\begin{remark}
Note that this all applies in particular to discrete groups $\Gamma$, endowing $\mathcal{F}(\Gamma, E)$ with the topology of pointwise convergence.

However, a special feature of the discrete group case is that one isn't forced to consider any structure on $E$, aside for the algebraic module assumptions. Thus one can embed the category of topological modules into an abelian category of all modules in a purely algebraic sense.

Of course, the bar resolution exists in both categories, with the same underlying modules, so the embedding into the larger category just allows one greater freedom in choosing an injective resolution, if one does not need to carry over the topological structure in a canonical manner.
\end{remark}

The next result gives explicitly the isomorphism allowing the usual description of cohomology by inhomogeneous cochains.

\begin{theorem}[Compare {\cite[Chapter III, Section 1.3]{Guichardetbook}}] \label{thm:cohominhomogeneous} \todo{thm:cohominhomogeneous}
Let $G$ be a locally compact, $2$nd countable group and $E$ a complete topological $G$-$\mathscr{A}$-module.

Define maps $T^n:L^{2}_{loc}(G^n,E)\rightarrow L^{2}_{loc}(G^{n+1},E)^{G}$, the latter a $G$-$\mathscr{A}$-module with the action $(g.\xi.a)(t)=g.\xi(g^{-1}t).a$, by
\begin{equation}
(T^n \xi)(g_0,\dots ,g_n) = g_0.\xi(g_0^{-1}g_1, \dots , g_{n-1}^{-1}g_n), \nonumber
\end{equation}
and coboundary maps $d^n:L^{2}_{loc}(G^n,E)\rightarrow L^{2}_{loc}(G^{n+1},E)$ by
\begin{eqnarray}
(d^n\xi)(g_1,\dots ,g_{n+1}) & = & g_1.\xi(g_2,\dots ,g_{n+1}) + \sum_{i=1}^{n} (-1)^{i} \xi(g_1,\dots ,g_ig_{i+1},\dots ,g_{n+1}) + \nonumber \\
 & & \; \; + (-1)^{n+1}\xi(g_1,\dots g_n). \nonumber
\end{eqnarray}
Then the $T^n$ are isomorphisms of topological $\mathscr{A}$-modules and the diagram
\begin{displaymath}
\xymatrix{ 0 \ar[r] & E \ar[d]^{T^0} \ar[r]^{d^0} & L^{2}_{loc}(G,E) \ar[d]^{T^1} \ar[r]^<<<<{d^1} & \cdots \\ 0 \ar[r] & L^{2}_{loc}(G,E)^{G} \ar[r]^{d^0_{'}} & L^{2}_{loc}(G^2,E)^{G} \ar[r]^<<<<{d^1_{'}} & \cdots }
\end{displaymath}
commutes, whence the $T^n$ induce $\mathscr{A}$-linear (homeomorphic) isomorphisms of cohomology $H^n(d^{*})\simeq H^n(G,E)$.
\end{theorem}

\begin{proof}
It is clear that $T^n$ is well-defined and continuous. A direct computation shows that the diagram does indeed commute.

To get an inverse, let $\xi\in L^{2}_{loc}(G^{n+1},E)$ be a $G$-invariant element and consider the element $\tilde{\xi}\in L^{2}_{loc}(G,L^{2}_{loc}(G^n,E))$ given by
\begin{equation}
\tilde{\xi}(g_0)(g_1,\dots , g_n) = g_0^{-1}. \xi(g_0,g_0g_1,\dots g_0\cdots g_n). \nonumber
\end{equation}
Then $\tilde{\xi}$ is invariant under the left-regular representation (i.e. $G$ acting on $L^{2}_{loc}(G,L^{2}_{loc}(G^n,E))$ by $(g.\eta)(g_0)(g_1,\dots ,g_n) = \eta(g^{-1}g_0)(g_1,\dots ,g_n)$), so there is an element $(T^n)^{-1}\xi\in L^{2}_{loc}(G^n,E)$ such that for a.e. $g_0\in G$,
\begin{equation}
((T^n)^{-1}\xi)(g_1,\dots ,g_n) = \tilde{\xi}(g_0)(g_1,\dots ,g_n). \nonumber
\end{equation}
One then checks directly that this does indeed provide an inverse. Notice in particular that continuity of the inverse is automatic by the closed graph theorem.
\end{proof}

Similarly, starting from modules of continuous functions instead of locally square integrable ones, one gets a cocycles description. This implies in particular that any locally square integrable cocycle is cohomologous (i.e.~represents the same class in cohomology) to a continuous one, a fact which is not at all abvious at a glance.

\subsection{The Shapiro lemma} 

For completeness we provide a proof of the Shapiro lemma in continuous cohomology. Compare e.g.~\cite{Guichardetbook}.

\begin{lemma}[({\cite[Chapter III, Lemma 4.3]{Guichardetbook}})]
Let $G$ be a locally compact, $2$nd countable group and $H$ be a closed subgroup of $G$. Let $E$ be a complete topological $H$-$\mathscr{A}$-module. Then the (complete, topological) $H$-$\mathscr{A}$-module $L^{2}_{loc}(G,E)$ with $H$-action $(h.\xi)(t)=h.\xi(h^{-1}t)$ and $\mathscr{A}$-action by post-multiplication, is injective. \hfill \qedsymbol
\end{lemma}

Thus we get an injective resolution of $E$ in the category $\mathfrak{E}_{H,\mathscr{A}}$ as follows
\begin{displaymath}
\xymatrix{ 0 \ar[r] & L^2_{loc}(G,E) \ar[r]^{d^0_{'}} & L^{2}_{loc}(G^2,E) \ar[r]^{d^1_{'}} & \cdots }
\end{displaymath}

The next lemma describes explicitly an isomorphism passing from the standard (bar) resolution of $E$ to this. In the statement, fix a section $s_r$ of $G\xrightarrow{\kappa} G/H$, and let for all $g\in G$ the element $r(g)\in H$ be the unique element in $H$ such that $g=s_r(\kappa (g)).r(g)$.

\begin{lemma} \label{lma:cohomveetwo} \todo{lma:cohomveetwo}
Let $G$ be a locally compact, $2$nd countable group, $H$ a closed subgroup, and $E$ a complete topological $H$-$\mathscr{A}$-module. Let $\mathcal{O}\in C_c(G)$ satisfy $\int \mathcal{O}\mathrm{d}\mu = 1$. Consider the maps $v_2^{n}:L^{2}_{loc}(H^{n+1},E)\rightarrow L^{2}_{loc}(G^{n+1},E)$, both left-$H$-modules as above, given by
\begin{equation}
(v^{n}_2\xi)(g_0,\dots ,g_n) = \int_{G^{n+1}} \mathcal{O}(t_0g_0)\cdots \mathcal{O}(t_ng_n)\xi(r(t_0)^{-1}, \dots ,r(t_n)^{-1}) \mathrm{d}\mu^{n+1}(t_i). \nonumber
\end{equation}
These are $H$-$\mathscr{A}$-maps, the diagram
\begin{displaymath}
\xymatrix{ 0\ar[r] & E \ar[r]^>>>>{\epsilon} \ar[d]^{id} & L^{2}_{loc}(H,E) \ar[r]^<<<<{d^1_{'}} \ar[d]^{v_2^0} & \cdots \\ 0 \ar[r] & E \ar[r]^>>>>{\epsilon} & L^{2}_{loc}(G,E) \ar[r]^<<<<{d^0_{'}} & \cdots }
\end{displaymath}
commutes, and the $v_2^n$ induce $\mathscr{A}$-linear (homeomorphic) isomorphisms on cohomology.
\end{lemma}

\begin{proof}
Let $K\subseteq G$ be a compact subset and denote $K'=\operatorname{supp}(\mathcal{O})\cdot K^{-1}$. Then the estimate
\begin{equation}
\lVert v_2^n(\xi)\vert_{K^{n+1}} \rVert_2^2 \leq \mu(K)\cdot \lVert\mathcal{O}\rVert_2^2 \cdot \lVert \xi\vert_{(r(K')^{-1})^{n+1}} \rVert_2^2, \quad \xi \in L^2_{loc}(H^{n+1},E) \nonumber
\end{equation}
shows that $v_2^n$ is well-defined, and upon substituting $\xi-\eta$ for $\xi$, that it is continuous.

It is clear that $v_2^n$ are $H$-$\mathscr{A}$-morphisms and that the diagram commutes.

For the final statement we consider similarly maps $w_2^{n}:L^{2}_{loc}(G^{n+1},E)\rightarrow L^{2}_{loc}(H^{n+1},E)$ defined by
\begin{equation}
(w_2^n\xi)(h_0,\dots ,h_n) = \int_{G^{n+1}} \mathcal{O}(h_0^{-1}t_0)\cdots \mathcal{O}(h_n^{-1}t_n)\xi(t_0,\dots ,t_n) \mathrm{d}\mu^{n+1}(t_i). \nonumber
\end{equation}
Then (these are continuous $H$-$\mathscr{A}$-morphisms and) the diagram
\begin{displaymath}
\xymatrix{ 0\ar[r] & E \ar[r]^{\epsilon} \ar[d]^{id} & L^{2}_{loc}(H,E) \ar[r]^{d^1_{'}} \ar[d]^{v_2^0} & \cdots \\ 0 \ar[r] & E \ar[r]^{\epsilon} \ar[u]^{id} & L^{2}_{loc}(G,E) \ar[r]^{d^0_{'}} \ar[u]^{w_2^0} & \cdots }
\end{displaymath}
commutes. Thus the compositions $v_2^{*}\circ w_2^{*}$ and $w_2^{*}\circ v_2^{*}$ are homotopic to the identity maps, whence induce the identity on cohomology from which the claim follows.
\end{proof}

\begin{definition}[(Coinduced module)] \label{def:coindmod} \todo{def:coindmod}
Let $G$ be a $2$nd countable locally compact group and $H$ a closed subgroup. Let $E$ be a complete topological $H$-$\mathscr{A}$-module.

Consider the space $L^{2}_{loc}(G,E)$. We endow this with the $H$-action given by $(h.\xi)(t)=h.\xi(th)$, as well as the left-$G$-action given by $(g.\xi)(t)=\xi(g^{-1}t)$. These commute, and both commute with the right-$\mathscr{A}$-action by post-multiplication, whence $\mathrm{Coind}_{H}^{G}E := L^{2}_{loc}(G,E)^{H}$ is a complete topological $G$-$\mathscr{A}$-module, called the coinduced module of $E$ wrt.~$G$.
\end{definition}

\begin{remark}
Suppose for simplicity that $H$ is discrete and $G$ is unimodular. Note then that the coinduced module identifies as a space with $L^{2}_{loc}(G/H,E)$. Letting $\alpha\colon G \times G/H \rightarrow H$ be a cocycle representative for the inclusion $H\subseteq G$, e.g.~$\alpha(g,x):=s_r(g.x)^{-1}gs_r(x)$, the $G$-action is given in this setting by
\begin{equation}
(g.\xi)(x) = \alpha(g,x).\xi(g^{-1}.x). \nonumber
\end{equation}
\end{remark}

\begin{lemma}[(Shapiro)] \label{lma:cohomshapiro} \todo{lma:cohomshapiro}
Let $G$ be a locally compact, $2$nd countable group, $H$ a closed subgroup, and $E$ a complete topological $H$-$\mathscr{A}$-module. Consider maps $\chi^n:L^{2}_{loc}(G^{n+1},E)^{H}\rightarrow L^{2}_{loc}(G^{n+1},\mathrm{Coind}_{H}^{G}E)^{G}$, where the $H$-action on the domain is $(h.\xi)(t)=h.\xi(h^{-1}t)$, given by
\begin{equation}
(\chi^n\xi)(g_0,\dots ,g_n)(g) = \xi(g^{-1}g_0,\dots ,g^{-1}g_n). \nonumber
\end{equation}
These are isomorphisms of topological $\mathscr{A}$-modules and the diagram
\begin{displaymath}
\xymatrix{ 0 \ar[r] & L^{2}_{loc}(G,E)^{H} \ar[r]^{d^0_{'}} \ar[d]^{\chi^0} & L^{2}_{loc}(G^2,E)^{H} \ar[r] \ar[d]^{\chi^1} & \cdots \\ 0 \ar[r] & L^{2}_{loc}(G,\mathrm{Coind}_{H}^{G}E)^{G} \ar[r]^{d^0_{'}} & L^{2}_{loc}(G^2,\mathrm{Coind}_{H}^{G}E)^{G} \ar[r] & \cdots }
\end{displaymath}
commutes.

In particular, the $\chi^n$ induce $\mathscr{A}$-linear (homeomorphic) isomorphisms on cohomology $H^n(H,E)\simeq H^n(G,\mathrm{Coind}_{H}^{G}E)$ for all $n$.
\end{lemma}

\begin{proof}
It is clear that $\chi^n$ is well-defined, continuous, and that the diagram commutes. Thus we need simply to produce an inverse. Note that this will automatically be continuous by the closed graph theorem.

Thus suppose we are given $\xi\in L^{2}_{loc}(G^{n+1},\mathrm{Coind}_{H}^{G}E)^{G}$. Then we consider much as above $\tilde{\xi}\in L^{2}_{loc}(G,L^{2}_{loc}(G^{n+1},E))$ defined by
\begin{equation}
\tilde{\xi}(g)(g_0,\dots ,g_n) = \xi(g,gg_0^{-1}g_1,\dots ,gg_0^{-1}g_n)(gg_0^{-1}). \nonumber
\end{equation}

Then again $\tilde{\xi}$ is invariant under the left-regular representation, whence there is $\chi^{-n}\xi\in L^{2}_{loc}(G^{n+1},E)$ such that for a.e. $g\in G$,
\begin{equation}
(\chi^{-n}\xi)(g_0,\dots ,g_n)= \tilde{\xi}(g)(g_0,\dots ,g_n). \nonumber
\end{equation}

A direct computation shows that in fact $\chi^{-n}\xi\in L^{2}_{loc}(G^{n+1},E)^{H}$, and we get for a.e. $g\in G$ that
\begin{eqnarray}
(\chi^{-n}\circ \chi^n)(\xi)(g_0,\dots ,g_n) & = & \tilde{(\chi^n\xi)}(g)(g_0,\dots ,g_n) \nonumber \\
 & = & (\chi^n\xi)(g,gg_0^{-1}g_1,\dots , gg_0^{-1}g_n)(gg_0^{-1}) \nonumber \\
 & = & \xi(g_0,\dots ,g_n). \nonumber
\end{eqnarray}
Similarly, for a.e. $g,g_0\in G$ we get
\begin{eqnarray}
(\chi^{n}\circ \chi^{-n})(\xi)(g_0,\dots ,g_n)(g) & = & \tilde{\xi}(g_0)(g^{-1}g_0,\dots ,g^{-1}g_n) \nonumber \\
 & = & \xi(g_0,\dots ,g_n)(g). \nonumber
\end{eqnarray}
Hence $\chi^{-n}$ is the inverse of $\chi^n$, proving the claim.
\end{proof}

\section{Continuous homology}
Unlike continuous cohomology, homology does not appear to have recieved much attention in the literature. In fact I have only managed to find two papers dealing with the subject, namely \cite{Blanc79, Pichaud83}. Since there appears to be no comprehensive account of continuous homology in the literature, we give the definitions and details of the basic results here.

\subsection{Definition of continuous homology}

Let $E$ be a complete, topological $\mathscr{A}$-$G$-module. Note that we have swapped the actions compared to the previous section, so that $E$ is a left-$\mathscr{A}$-right-$G$-module, and that this will be our standard for modules that appear as coefficients in homology. 

Consider again the subspaces $L^{2}(K,E)$ of $L^{2}(G,E)$, over $K\subseteq G$ compact. We let $L^{2}_{c}(G,E)$ be the inductive limit of these over compact subsets $K$, endowed with the inductive topology. This identifies with the space of compactly supported functions in $L^{2}(G,E)$ and the topology is the strongest topology such that the inclusion maps $L^{2}(K,E)\rightarrow L^{2}_{c}(G,E)$ are all continuous. Note that this is in general stronger than the subspace topology induced by the inclusion $L^2_c(G,E)\subseteq L^2(G,E)$.

We will in this section also require $G$ to be $2$nd countable. Then the inductive limit is strict\cite[Chapter II, 6.3]{TopVecbook}, the space $L^{2}_{c}(G,E)$ is complete\cite[Chapter II, 6.6]{TopVecbook}, and the subspace topologies on $L^{2}(K,E)\subseteq L^{2}_{c}(G,E)$, for every compact subset $K$, is the original Hilbert space topology. For more details on inductive limits see e.g.~\cite[Chapter II, Section 6]{TopVecbook}.

There are two ways in which we may make this a right-$G$-module, \todo{eq:rightactionX}
\begin{eqnarray}
\label{eq:rightaction1}  (f.g)(t) = f(gt).g \\
\label{eq:rightaction2}  (f.g)(t) = f(tg^{-1}).g,
\end{eqnarray}
and both will come into play. Further, $L^2_c(G,E)$ inherits a left-action of $A$ by post-multiplication on $E$, and this commutes with both right-actions of $G$ above. Hence $L^2_c(G,E)$ is a complete topological $\mathscr{A}$-$G$-module.

\begin{definition}
We consider the full subcategory of topological $\mathscr{A}$-$G$-modules consisting of those that are complete and all morphisms between these. Denote this $\bar{\mathfrak{E}}_{\mathscr{A},G}$

A complete topological $\mathscr{A}$-$G$-module $E$ is (relatively) projective if whenever we are given a diagram, in this category,

\begin{displaymath}
\xymatrix{  & E \ar[d]^{v} \ar@{-->}[dl]_{\exists ?w} & \\ W \ar[r]^{\pi} & V \ar[r] & 0 }
\end{displaymath}
where the row is exact with $\pi$ strengthened and $v$ is a continuous $\mathscr{A}$-$G$-map, there is a continuous $\mathscr{A}$-$G$-map $w:E\rightarrow W$ making the diagram commute.

As for injective modules, we always suppress the 'relative'. 
\end{definition}

\begin{theorem}[(Compare {\cite{Blanc79}})] \label{thm:elltwoprojective} \todo{thm:elltwoprojective}
For every complete $\mathscr{A}$-$G$-module $E$, the complete $\mathscr{A}$-$G$-module $L^{2}_{c}(G,E)$ is projective. In particular, the category of complete topological $\mathscr{A}$-$G$-modules has sufficiently many projectives.
\end{theorem}

\begin{proof}
We consider here the $G$-action (\ref{eq:rightaction1}). A completely analogous proof works for the other action. Let $\mathcal{O}\in C_c(G)$ be such that $\int_G \mathcal{O}\mathrm{d}\mu = 1$, for some fixed choice of (left-)Haar measure $\mu$ on $G$. Then for $f\in L^{2}_{c}(G,E)$ and $g\in G$ we define a new function $f_g\in L^{2}_{c}(G,E)$ by
\begin{equation}
f_g(t) = \mathcal{O}(tg^{-1})f(t), \quad t\in G. \nonumber
\end{equation}

Then clearly the map $g\mapsto f_g$ is continuous with compact support. We also note that for $\gamma \in G$, we have $(f.\gamma)_g = (f_{g\gamma^{-1}}).\gamma$ by a direct computation, and that
\begin{equation}
\int_G f_g\mathrm{d}\mu(g) = f. \nonumber
\end{equation}

Now suppose we are given a diagram as above, with $s$ a right-inverse of $\pi$. Then we define a map $w:L^{2}_{c}(G,E)\rightarrow W$ by
\begin{equation}
wf = \int_G s(v(f_g).g^{-1}).g\mathrm{d}\mu(g). \nonumber
\end{equation}

By the above, this is easily seen to be a $G$-equivariant continuous linear map, and further we get
\begin{eqnarray}
(\pi \circ w)(f) & = & \int_G (\pi\circ s)(v(f_g).g^{-1}).g \mathrm{d}\mu(g) \nonumber \\
 & = & \int_G v(f_g) \mathrm{d}\mu(g) \nonumber \\
 & = & v\left( \int_G f_g\mathrm{d}\mu(g) \right) = v(f). \nonumber
\end{eqnarray}

Also, note that $w$ is $\mathscr{A}$-linear by the trivial observation that, for any $g\in G$ and $f\in L^2_c(G,E)$, we have $(x.f)_g=x.f_g$ for all $x\in \mathscr{A}$

As for the final claim, recall that for any $K\subseteq G$ compact, the inclusion of $L^{2}(K,E)$ in $L^{1}(K,E)$ is continuous, whence so is the inclusion of $L^{2}_{c}(G,E)$ in $L^{1}_{c}(G,E)$ and we may define an evaluation map $\epsilon:L^{2}_{c}(G,E)\rightarrow E$ by
\begin{equation}
\epsilon (f) = \int_G f(g)\mathrm{d}\mu(g). \nonumber
\end{equation}
This is clearly surjective with an $\mathscr{A}$-linear right-inverse $e\mapsto \mathcal{O}(\cdot)e$.
\end{proof}

For arbitrary $n\in \mathbb{N}$ we consider the right-modules $L^{2}_{c}(G^n,E)$ with either of the two usual generalizations of the actions (\ref{eq:rightaction1}),(\ref{eq:rightaction2}). Then $L^{2}_{c}(G^{m+n},E)$ identifies with $L^{2}_{c}(G^{m},L^{2}_{c}(G^{n},E))$ in the usual manner and we get the following

\begin{corollary}
For all $n\in \mathbb{N}$, the modules $L^{2}_{c}(G^n,E)$ are projective. \hfill \qedsymbol
\end{corollary}

\begin{theorem}
For $E$ a complete topological $\mathscr{A}$-$G$-module we get a projective resolution
\begin{displaymath}
\xymatrix{ \cdots \ar[r]^>>>>{d_1^{'}} &  L^{2}_{c}(G^2,E) \ar[r]^{d_0^{'}} &  L^{2}_{c}(G,E) \ar[r]^<<<<{\epsilon} & E \ar[r] & 0 }
\end{displaymath}
where the boundary maps $d_n^{'}:L^{2}_{c}(G^{n+2},E)\rightarrow L^{2}_{c}(G^{n+1},E)$ are given by
\begin{eqnarray}
(d_n^{'}f)(g_1,\dots ,g_{n+1}) & = & \sum_{j=0}^{n+1} (-1)^{j} f^{j}(g_1,\dots , g_{n+1}), \; \mathrm{where} \nonumber \\
f^{j}(g_1,\dots ,g_{n+1}) & = & \int_G f(g_1,\dots ,g_j,g,g_{j+1},\dots ,g_{n+1}) \mathrm{d}\mu(g). \nonumber
\end{eqnarray}
\end{theorem}
\begin{proof}
The proof follows exactly the proof in the discrete case. The contraction here is given by
\begin{equation}
(sf)(g_1,\dots ,g_{n+1}) = \mathcal{O}(g_1)f(g_2,\dots ,g_{n+1}) \nonumber
\end{equation}
for some fixed $\mathcal{O}\in C_c(G)$ positive with integral one.
\end{proof}

Let $\underline{\mathscr{C}}_G$ denote the functor from complete topological $\mathscr{A}$-$G$-modules to complete topological $\mathscr{A}$-modules given by
\begin{equation}
\underline{\mathscr{C}}_G E = E/\overline{\span}\{ e.g-e\mid e\in E, g\in G\}. \nonumber
\end{equation}

\begin{definition}
Let $E$ be a complete topological $\mathscr{A}$-$G$-module. The $n$'th continuous homology $H_n(G,E)$ is the $n$'th homology of the complex

\begin{displaymath}
\xymatrix{ \cdots \ar[r]^>>>{d_1'} & \underline{\mathscr{C}}_G E_1 \ar[r]^{d_0'} & \underline{\mathscr{C}}_G E_0 \ar[r] & 0 }
\end{displaymath}
for any projective strengthened resolution $\xymatrix{ E_* \ar[r]^{\epsilon} & E \ar[r] & 0}$ of $E$.

\begin{comment} %%%%%%%%%%%%%%%
\begin{equation}
\cdots \xrightarrow{d_1^{'}} \underline{\mathscr{C}}_G L^{2}_{c}(G^2,E) \xrightarrow{d_0^{'}} \underline{\mathscr{C}}_G L^{2}_{c}(G,E) \rightarrow 0. \nonumber
\end{equation}
\end{comment}

\end{definition}

Although we can consider the quotient topology on $H_n$ this does not appear to be particularly relevant, and usually we consider $H_n(G,E)$ simply as an algebraic $\mathscr{A}$-module.

\begin{lemma}
Let $E$ be a complete topological $\mathscr{A}$-$G$-module and let $o:E\rightarrow E$ be the zero-map.

Suppose that $(E_i,d_i)$ and $(F_i,f_i)$ are strengthened projective resolutions of $E$ in $\bar{\mathfrak{E}}_{\mathscr{A},G}$, and $v_i:E_i\rightarrow F_i$ morphisms such that the diagram
\begin{displaymath}
\xymatrix{ \cdots \ar[r] & E_2 \ar[r]^{d_1} \ar[d]^{v_2} & E_1 \ar[r]^{d_0} \ar[d]^{v_1} & E_0 \ar[r]^{d_{-1}} \ar[d]^{v_0} & E \ar[r] \ar[d]^{o} & 0 \\ \cdots \ar[r] & F_2 \ar[r]^{f_1} & F_1 \ar[r]^{f_0} & F_0 \ar[r]^{f_{-1}} & E \ar[r] & 0 }
\end{displaymath}
commutes.

Then there is a (equivariant) contraction $(s_i)$, i.e. morphisms $s_i:E_i\rightarrow F_{i+1}$ such that $v_i=f_i\circ s_i +s_{i-1}\circ d_{i-1}$.

In particular, the $v_i$ induce the zero-map in homology $H_*(\underline{\mathscr{C}}_GE_*)\rightarrow H_*(\underline{\mathscr{C}}_GF_*)$.
\end{lemma}

\begin{proof}
We put $s_{-1}=0$. Since the $E_i$ are projective we may inductively construct the $s_i$ such that $f_i\circ s_i = v_i-s_{i-1}\circ d_{i-1}$ in the usual manner.

The rest is obvious.
\end{proof}

\begin{theorem}
Let $E$ be a complete topological $\mathscr{A}$-$G$-module.

Suppose that $(E_i,d_i)$ and $(F_i,f_i)$ are strenghtened projective resolutions of $E$. Then there are (equivariant) morphisms $u_i:E_i\rightarrow F_i$ and $v_i:F_i\rightarrow E_i$ such that the diagram
\begin{displaymath}
\xymatrix{ \cdots \ar[r] & E_2 \ar[r]^{d_1} \ar@/^/[d]^{u_2} & E_1 \ar[r]^{d_0} \ar@/^/[d]^{u_1} & E_0 \ar[r]^{d_{-1}} \ar@/^/[d]^{u_0} & E \ar@{=}[d] \ar[r] & 0 \\ \cdots \ar[r] & F_2 \ar[r]^{f_1} \ar@/^/[u]^{v_2} & F_1 \ar[r]^{f_0} \ar@/^/[u]^{v_1} & F_0 \ar[r]^{f_{-1}} \ar@/^/[u]^{v_0} & E \ar[r]  & 0 }
\end{displaymath}
commutes.

Further, the $u_i$ and $v_i$ induce mutually inverse isomorphisms on homology $H_*(\underline{\mathscr{C}}_GE_*)\simeq H_*(\underline{\mathscr{C}}_GF_*)$
\end{theorem}

\begin{proof}
Since $d_{-1}$ is surjective and $E_0$ is projective we get $u_0:E_0\rightarrow F_0$ such that $f_{-1}\circ u_0 = id\circ d_{-1}$. Suppose we have $u_0,\dots ,u_{k-1}$ such that the diagram
\begin{displaymath}
\xymatrix{ \cdots \ar[r] & E_{k-1} \ar[r]^{d_{k-2}} \ar[d]^{u_{k-1}} & E_{k-2} \ar[r]^{d_{k-3}} \ar[d]^{u_{k-2}} & \cdots \ar[r]^{d_{-1}} & E \ar[r] & 0 \\ \cdots \ar[r] & F_{k-1} \ar[r]^{f_{k-2}} & F_{k-2} \ar[r]^{f_{k-3}} & \cdots \ar[r]^{f_{-1}} & E \ar@{=}[u] \ar[r] & 0 }
\end{displaymath}
commutes.

Then by exactness $\image{ d_{k-1} }$ is closed and $u_{k-1}(\image{ d_{k-1} }) \subseteq \image{ f_{k-1} }$. Thus from the diagram
\begin{displaymath}
\xymatrix{  & E_k \ar[d]^{u_{k-1}\circ d_{k-1}} \\ F_k \ar[r]^{f_{k-1}} & \image{ f_{k-1} } \ar[r] & 0 }
\end{displaymath}
and projectivity of $E_k$ we get a morphism $u_k:E_k\rightarrow F_k$ such that $f_{k-1}\circ u_k = u_{k-1} \circ d_{k-1}$.

By recursion we get the $u_k$ as claimed. Similarly we construct the $v_k$ as claimed.

Then $id-v_k\circ u_k$ is zero on homology by the lemma, and so is $id-u_k\circ v_k$.
\end{proof}

\begin{remark}
It follows by the open mapping theorem, since we are working only with complete modules, that the functor $\underline{\mathscr{C}}(-)$ is right-exact, so that
\begin{equation}
H_0(G,E) \simeq \underline{\mathscr{C}}_G(E). \nonumber
\end{equation}
\end{remark}

\begin{remark}
If $\Gamma$ is a countable discrete group, then $L^{2}_{c}(\Gamma, E)$ is just the space of finitely supported functions on $\Gamma$, usually denoted $\mathcal{F}_0(\Gamma,E)$.

Then the space of functions spanned by the $f.g-f$ is in fact closed, as an element $f\in \mathcal{F}_0(\Gamma,E)$ is in this space if and only if $\sum_{g\in \Gamma} f(g).g^{\pm 1} = 0$ (Here the $\pm 1$ depends on which right-action we are considering.)

Thus the construction of continuous homology agrees with the usual definition in the case of countable discrete groups.
\end{remark}

\subsection{Inhomogeneous chains}
For $E$ a complete $G$-module we denote by $\tilde{L}^{2}_{c}(G^n,E)$ the right-$G$-module which has $L^{2}_{c}(G^n,E)$ as its underlying space, but with $G$ acting by $(f.g)(t)=f(gt)$, i.e. we forget about the action on $E$.

\begin{lemma}
The map $u\colon L^{2}_{c}(G,E)\rightarrow \tilde{L}^{2}_{c}(G,E)$ given by $(uf)(g)=f(g).g$ is an isomorphism of topological $\mathscr{A}$-$G$-modules, where the action on the domain is (\ref{eq:rightaction1}).
\end{lemma}

\begin{proof}
The only non-trivial part is that $u$ is well-defined and continuous.

To show that $u$ is well-defined it is enough to consider the restriction to compact sets. Further, to show then that $u$ is continuous, we have to show that for each inclusion $i_K:L^{2}(K,E)\rightarrow L^{2}_{c}(G,E)$, the composition $u\circ i_K$ is continuous. This has image in $L^{2}(K,E)$ as well, and since the subspace topology on this is just the original topology, this all comes down to showing that $u_K := u\vert_K$ is well-defined and continuous for each compact set $K$.

Now by \cite[Lemma D.8]{Guichardetbook} the set of linear maps $\{ e\mapsto e.g \mid g\in K\}$ is equicontinuous, and it follows by \cite[III.4.1]{TopVecbook} (essentially the definition of equicontinuity) that for a given semi-norm $p$ on $E$ there is a semi-norm $q$ on $E$ and a constant $C$ such that for all $e\in E,g\in K, p(e.g)\leq Cq(e)$. Further we can take $q$ in whatever family of separating semi-norms we wish.

Thus $u_K( \{f \mid q_K(f) \leq C \}) \subseteq \{f \mid p_K(f) \leq 1\}$ from which the claim follows since $p$ was arbitrary.
\end{proof}

Now consider the diagram

\begin{displaymath}
\xymatrix{  & & L^{2}_{c}(G,E) \ar[d]^{T} \ar[dl]_{u} & \\ \tilde{L}^{2}_{c}(G^2,E) \ar[r]^{\tilde{d}_0^{'}} & \tilde{L}^{2}_{c}(G,E) \ar[r]^{\epsilon} & E \ar[r] & 0 }
\end{displaymath}
in which $T$ is the map $Tf=\int_G f(g).g\mathrm{d}\mu(g)$, $\tilde{d}_0^{'}$ is the same map as $d_0^{'}$ on the underlying space $L^2_c(G^2,E)$, and similarly $\epsilon$ is integration over $G$. The triangle commutes, whence under the isomorphism $u$, the kernel of $T$ identifies with that of $\epsilon$.

\begin{proposition}
In the diagram above we have
\begin{equation}
\ker \epsilon = \overline{\span} \{ f.g-f\mid f\in \tilde{L}^{2}_{c}(G,E), g\in G\}. \nonumber
\end{equation}
\end{proposition}

For the proof we need to following auxilliary
\begin{lemma}
The set of characteristic functions of the form $\bbb_{A\times \gamma A}, A\subseteq G, \gamma\in G$ has dense linear span in $L^{2}_{c}(G\times G)$.
\end{lemma}
\begin{proof}
Let $\{g_n\}_{n\in \mathbb{N}}$ be a dense set in $G$ and $\{U_n\}_{n\in \mathbb{N}}$ a neighbourhood basis in $G$. Then it is easy to see that $\{ U_m\times (g_nU_m)\}$ is a neighbourhood basis in $G\times G$, from which the claim readily follows.
\end{proof}

\begin{proof}[Proof of the proposition.]
For this we just note that
\begin{equation}
\tilde{d}_0^{'} \bbb_{A\times \gamma A} = \mu(A) (\bbb_{\gamma A}-\bbb_{A}) = \mu(A)(\bbb_{A}.\gamma^{-1} - \bbb_{A}). \nonumber
\end{equation}
By the lemma $\image \tilde{d}_0^{'} \subseteq \overline{\span} \{ f.g-f\}$, and since $\span\{f.g-f\}\subseteq \ker \epsilon$ is obvious the proposition follows by exactness.
\end{proof}

\begin{theorem} \label{thm:homologyinhomogeneous} \todo{thm:homologyinhomogeneous}
Define maps $T_n:L^{2}_{c}(G^{n+1},E)\rightarrow L^{2}_{c}(G^{n},E)$ and $d_n:L^{2}_{c}(G^{n+1},E)\rightarrow L^{2}_{c}(G^{n},E)$ by
\begin{eqnarray}
(T_nf)(g_1,\dots ,g_n) & = & \int_G f(g,gg_1,\dots ,gg_1g_2\cdots g_n).g \mathrm{d}\mu(g), \nonumber \\
(d_nf)(g_1,\dots ,g_n) & = & \int_G [ f(g,g_1,\dots ,g_n).g + \sum_{j=1}^{n} (-1)^{j} f(g_1,\dots ,g,g^{-1}g_j,g_{j+1}, \dots ,g_n) + \nonumber \\
 & & \; + (-1)^{n+1}f(g_1,\dots ,g_n,g)]\mathrm{d}\mu(g). \nonumber
\end{eqnarray}
Then $T_n$ induces an isomorphism of topological vector spaces $T_n:\underline{\mathscr{C}}_G L^{2}_{c}(G^{n+1},E) \xrightarrow{\sim} L^{2}_{c}(G^{n},E)$ and the diagram

\begin{displaymath}
\xymatrix{ \cdots \ar[r] & L^{2}_{c}(G^{n+2},E) \ar[d]^{T_{n+1}} \ar[r]^{d_n^{'}} & L^{2}_{c}(G^{n+1},E) \ar[d]^{T_n} \ar[r] & \cdots \ar[r] & L^2_c(G,E) \ar[r] \ar[d]^{T_0} & 0 \\
\cdots \ar[r] & L^{2}_{c}(G^{n+1},E) \ar[r]^{d_n} & L^{2}_{c}(G^{n},E) \ar[r] & \cdots \ar[r] & E \ar[r] & 0 }
\end{displaymath}
\noindent
commutes, whence $H_n(G,E)\simeq H_n(d_{*})$.
\end{theorem}

\begin{proof}
For $n\geq 1$ denote by $\hat{L}^{2}_{c}(G^n,E)$ the right-$G$-module with underlying space $L^{2}_{c}(G^{n},E)$ and $G$ acting by post-multiplication. Then we have an isomorphism of modules $u:L^{2}_{c}(G^{n+1},E) \xrightarrow{\sim} L^{2}_{c}(G,\hat{L}^{2}_{c}(G^n,E))$ given by
\begin{equation}
(uf)(g)(g_1,\dots ,g_n) = f(g,gg_1,\dots ,gg_1\cdots g_n), \nonumber
\end{equation}
and with inverse $(u^{-1}f)(g,g_1,\dots ,g_n) = f(g)(g^{-1}g_1, g_1^{-1}g_2, \dots , g_{n-1}^{-1}g_n)$.

Then we note simply that $T_n$ is nothing but $T_0$ applied to this, and the first claim follows from the proposition above.

A direct computation shows that the diagram does indeed commute.
\end{proof}

\subsection{The Shapiro lemma}

We now prove a version of the Shapiro lemma for continuous homology.

\begin{proposition}
Let $H$ be a countable discrete subgroup, $E$ a complete $\mathscr{A}$-$H$-module. Then $L^{2}_{c}(G^n,E)$, equiped with the action
\begin{equation}
(f.h)(t) = f(th^{-1}).h, \quad h\in H, t\in G \nonumber
\end{equation}
is a projective $\mathscr{A}$-$H$-module.
\end{proposition}

Of course the same will hold with the action $(f.h)(t) = f(ht).h$.

\begin{proof}
This is much the same as the proof of theorem \ref{thm:elltwoprojective}. Choose $\mathcal{O}\in C_c(G)$ such that $\int_G \mathcal{O}\mathrm{d}\mu = 1$. Put for $f\in L^{2}_{c}(G,E), g\in G$
\begin{equation}
f_g(t) = \mathcal{O}(gt^{-1})f(t), t\in G. \nonumber
\end{equation}
Consider for $g\in G$ also the element $r(g)$ be the unique element in $H$ such that $g=s_r(\kappa(g))\cdot r(g)$, where $s_r$ is a section of the map $G\xrightarrow{\kappa} G/H$. Then given a diagram

\begin{displaymath}
\xymatrix{  & L^{2}_{c}(G,E) \ar[d]^{v} \ar@{-->}[dl]_{\exists ?w} & \\ W \ar[r]^{\pi} & V \ar[r] & 0 }
\end{displaymath}
and denote the right-inverse to $\pi$ by $s$. Then we define the map $w:L^{2}_{c}(G,E)\rightarrow W$ by
\begin{equation}
wf = \int_G s(v(f_g).r(g)^{-1}).r(g) \, \mathrm{d}\mu(g). \nonumber
\end{equation}
One sees then easily that this is a continuous $H$-map using $(f.h)_g = (f_{gh^{-1}}).h$, and that the diagram then commutes. 

This proves the statement for $n=1$, and one extends this to $L^{2}_{c}(G^n,E)$ as before.
\end{proof}

Now suppose that $E$ is a complete $\mathscr{A}$-$H$-module and consider $L^{2}_{c}(G,E)$ an $\mathscr{A}$-$H$-module as in the previous proposition. Then also $G$ acts on this from the right by $(f.g)(t) = f(gt)$ and this commutes with the $\mathscr{A}$-$H$-action.

\begin{definition}[(Induced module)] \label{def:Indmodule} \todo{def:Indmodule}
The induced module of $E$ wrt.~$G$ is defined as
\begin{equation}
\operatorname{Ind}_{H}^{G} E = \underline{\mathscr{C}}_H L^{2}_{c}(G,E) \nonumber
\end{equation}
with the $G$-action as above. Since that commutes with the $\mathscr{A}$-action, $\operatorname{Ind}_H^G E$ is a complete topological $\mathscr{A}$-$G$-module.
\end{definition}

\begin{lemma}[(Shapiro)]
For $n\in \mathbb{N}$ we define maps $\chi_n:L^{2}_{c}(G^n,\mathrm{Ind}_{H}^{G}E) \rightarrow \underline{\mathscr{C}}_H L^{2}_{c}(G^{n+1},E)$ by
\begin{equation}
(\chi_n f)(g_1,\dots ,g_{n+1}) = f(g_1g_2^{-1}, g_2g_3^{}, \dots ,g_ng_{n+1}^{-1})(g_1). \nonumber
\end{equation}

These are isomorphisms of topological $\mathscr{A}$-modules and the diagram

\begin{displaymath}
\xymatrix{ \cdots \ar[r] & L^{2}_{c}(G^n,\mathrm{Ind}_{H}^{G} E) \ar[d]^{\chi_n} \ar[r]^{d_{n-1}} & L^{2}_{c}(G^{n-1},\mathrm{Ind}_{H}^{G}E) \ar[d]^{\chi_{n-1}} \ar[r] & \cdots \\ \cdots \ar[r] & \underline{\mathscr{C}}_H L^{2}_{c}(G^{n+1},E) \ar[r]^{d_{n-1}^{'}} & \underline{\mathscr{C}}_H L^{2}_{c}(G^n,E) \ar[r] & \cdots  }
\end{displaymath}
commutes. In particular, the $\chi_n$ induce (topological, if we like) $\mathscr{A}$-linear isomophisms on homology $H_n(G,\mathrm{Ind}_{H}^{G}E)\xrightarrow{\sim} H_n(H,E)$.
\end{lemma}

\begin{proof}
We leave out the straight-forward proof.
\end{proof}

\section{Bar resolutions for totally disconnected groups} \label{sec:bartotdisc} \todo{sec:bartotdisc}

Let $G$ be a totally disconnected, locally compact $2$nd countable group, and fix a compact open subgroup $K$ in $G$. Normalize the Haar measure $\mu$ on $G$ such that $\mu(K)=1$. Let $E$ be a quasi-complete topological $G$-$\mathscr{A}$-module.

Define a right-action of $K^n$ on $G^n$ by
\begin{equation}
(g_1,\dots ,g_n).(k_1,\dots ,k_n) := (g_1k_1,k_1^{-1}g_2k_2,\dots ,k_{n-1}^{-1}g_nk_n) = (\bbb, k_1^{-1},\dots ,k_{n-1}^{-1})\cdot (g_1,\dots ,g_n)\cdot (k_1,\dots ,k_n). \nonumber
\end{equation}

Denote by $G^n_K$ the set of equivalence classes. Since each class is compact open in $G^n$, this is a countable set. Further, $G^n_K$ carries a left-action of $K$ by multiplication in the first variable.

We also note, that whenever $(g_1,\dots ,g_n)$ represents a class in $G^n_K$, we can form the multipliction $g_1\cdots g_n$, representing a class in $G/K$, independent of the choice of representative.

The following is a trivial generalization of \cite[Chapter III, Corollary 2.2]{Guichardetbook}.

\begin{proposition} \label{prop:totdiscbar} \todo{prop:totdiscbar}
For $G$ a totally disconnected, $2$nd countable, locally compact group, we have an isomorphism of $H^n(G,E)$ and the $n$'th homology of the complex
\begin{equation}
0\rightarrow E^K \xrightarrow{d^0} \mathcal{F}(G_K,E)^K \xrightarrow{d^1} \mathcal{F}(G^2_K,E)^K \rightarrow \cdots, \nonumber
\end{equation}
where the coboundary maps are the usual ones on inhomogeneous cochains, and the spaces of cochains are endowed with the topology of pointwise convergence. \hfill \qedsymbol
\end{proposition}

\plainbreak{2}

Next we construct a similar "bar" resolution in homology for totally disconnected, $2$nd countable $G$. Fix still a compact open subgroup $K$ of $G$. Let $E$ be a complete topological $\mathscr{A}$-$G$-module.

Considering the right-$K$-action on $L^2_c(G,E)$ defined by $(\xi.k)(g) = \xi(gk^{-1})$, this commutes with \eqref{eq:rightaction1}, and the module of fixed points identifies, algebraically, with $\mathbb{C}[G/K]\otimes E$, with the right-$G$-action \eqref{eq:rightaction1} given here by $(\delta_{gK}\otimes \xi).h = \delta_{h^{-1}gK}\otimes (\xi.h)$. The left-$\mathscr{A}$-action is $a.(\delta_{gK}\otimes \xi) = \delta_{gK}\otimes (a.\xi)$. We write $\mathcal{F}_c(G/K,E)$ for this module, endowed with the subspace topology from $L^2_c(G,E)$, which coincides with the natural inductive topology. This is a complete topologcial $\mathscr{A}$-$G$-module.

\begin{proposition}
For any complete topological $\mathscr{A}$-$G$-module $E$, the module $\mathcal{F}_c(G/K,E)$ is projective.

Hence so is $\mathcal{F}_c( (G/K)^{n+1},E )$ for any $n\geq 0$.
\end{proposition}

\begin{proof}
Given a diagram in $\bar{\mathfrak{E}}_{\mathscr{A},G}$
\begin{displaymath}
\xymatrix{  & \mathcal{F}_c(G/K,E) \ar@{-->}[dr]_{\exists ?w} \ar[d]^v \\ B \ar[r]^u & A \ar@{..>}@/^/[l]^s \ar[r] & 0 }
\end{displaymath}
with $s$ a continuous $\mathscr{A}$-linear right-inverse of $u$, the morphism $w\colon \mathcal{F}_c(G/K,E) \rightarrow B$ can be defined by
\begin{equation}
w(\delta_{gK}\otimes \xi) = \int_K s(v(\delta_{gK}\otimes \xi).gk).k^{-1}g^{-1}\mathrm{d}\mu(k). \nonumber
\end{equation}

One checks readily enough that this is an $\mathscr{A}$-$G$-morphism and makes the diagram commute.
\end{proof}

Similarly to above, we want to compute homology using inhomogeneous chains arising from this complex. We leave out the straight-forward proof.

\begin{theorem} \label{thm:totdiscbarhomology} \todo{thm:totdiscbarhomology}
Define maps $T_n\colon \underline{\mathscr{C}}_G\mathcal{F}_n( (G/K)^{n+1}, E) \rightarrow \underline{\mathscr{C}}_K\mathcal{F}_c(G^n_K,E)$ by
\begin{equation}
T_n \colon \delta_{(g_0K,\dots ,g_nK)}\otimes \xi \mapsto \delta_{(g_0^{-1}g_1,g_1^{-1}g_2, \dots ,g_{n-1}^{-1}g_n)}\otimes \xi.g_0.\nonumber
\end{equation}
Notice that, replacing all the $g_i$ with $g_ik_i$, this is well-defined in $\mathcal{F}_c(G^n_K,E)$ exactly up to the action of $k_0$.

Define also boundary maps $d_n\colon \underline{\mathscr{C}}_K\mathcal{F}_c(G^{n+1}_K,E)\rightarrow \underline{\mathscr{C}}_K\mathcal{F}_c(G^n_K,E)$ by
\begin{eqnarray}
d_n\colon \delta_{(g_0,\dots ,g_n)} \otimes \xi & \mapsto & \delta_{(g_1,\dots ,g_n)} \otimes \xi.g_0 + \nonumber \\ 
 & & + \sum_{i=1}^{n}(-1)^i\delta_{(g_0,\dots ,g_{i-1}g_i,g_{i+1},\dots g_n)} \otimes \xi + (-1)^{n+1}\delta_{(g_0,\dots ,g_{n-1})}\otimes \xi. \nonumber
\end{eqnarray}

The $T_n$ are isomorphisms of topological $\mathscr{A}$-modules, and the following diagram commutes.
\begin{displaymath}
\xymatrix{ \cdots \ar[r] & \underline{\mathscr{C}}_G\mathcal{F}_c( (G/K)^{n+2}, E ) \ar[r]^{d_{n}'} \ar[d]^{T_{n+1}} & \underline{\mathscr{C}}_G \mathcal{F}_c( (G/K)^{n+1}, E ) \ar[r] \ar[d]^{T_n} & \cdots \\ \cdots \ar[r] & \underline{\mathscr{C}}_K \mathcal{F}_c( G^{n+1}_K, E ) \ar[r]^{d_n} & \underline{\mathscr{C}}_K\mathcal{F}_c( G^n_K , E) \ar[r] & \cdots }
\end{displaymath}

In particular, the $T_n$ induce $\mathscr{A}$-isomorphisms on homology $H_n(G,E)\simeq H_n(d_*)$. \hfill \qedsymbol
\end{theorem}

%%%%%%%%%%%%%%%%%%%%%%%%%%%%%%%%%%%%%%%%%%%%%%%%%%%%%%%%%%%%%
\begin{comment}
\begin{proposition}
The module $\mathcal{F}_0(G/K,E):=\prod_{g\in G/K} E^{gKg^{-1}}$, $G$ acting by $(g.\xi)(hK)=g.\xi(g^{-1}hK)$ and $\mathscr{A}$ by post-multiplication, is injective.
\end{proposition}

\begin{proof}
$(wb)(gK) = \int_K v(gk.s(k^{-1}g^{-1}.b))(gK) \mathrm{d}\mu(k)$.
\end{proof}

More generally, we define
\begin{equation}
\mathcal{F}_n(G/K,E) := \prod_{g_0,\dots ,g_n\in G/K} E^{g_0Kg_0^{-1}\cap \cdots g_nKg_n^{-1}}. \nonumber
\end{equation}

Then by the same proof, letting
\begin{equation}
(wb)(g_0K,\dots ,g_nK) = \int_K v(g_0k.s(k^{-1}g^{-1}.b))(g_0K,\dots ,g_nK)\mathrm{d}\mu(k), \nonumber \\
\end{equation}

this is injective, and we get an injective resolution(?)

\begin{equation}
0\rightarrow E \xrightarrow{\epsilon} \mathcal{F}_0(G/K,E) \rightarrow \mathcal{F}_1(G/K,E) \rightarrow \dots \nonumber
\end{equation}
with
\begin{equation}
(\epsilon \xi)(gK) = \int_K gkg^{-1}.\xi \mathrm{d}\mu(k). \nonumber
\end{equation}
\end{comment}
%%%%%%%%%%%%%%%%%%%%%%%%%%%%%%%%%%%%%%%%%%%%%%%%%%%%%%%%%%%%%%%%%%

\section{Some remarks on the Hochschild-Serre spectral sequence}
The Hochschild-Serre spectral sequence in group cohomology computes the cohomology $H^n(G,E)$ of an extension
\begin{displaymath}
\xymatrix{ 0 \ar[r] & H \ar[r] & G \ar[r] & Q \ar[r] & 0 }
\end{displaymath}
of discrete groups, in terms of the groups $H,Q$. It was constructed in \cite{HochschildSerre}, and in \cite{Lyndon48}. (It is often also referred to as the Lyndon-Hochschild-Serre spectral sequence.)

Recall the end result: for such an extension there is a spectral sequence, abutting to $H^n(G,E)$, with $E_2$-term $E_2^{p,q}=H^p(Q,H^q(H,E))$. General references to spectral sequences are \cite{McClearybook,Weibelbook}. We recall just that the spectral sequence abutting to the cohomology means that at each $n$ there is a filtration $H^n(G,E) = H^n(G,E)_n \supseteq H^n(G,E)_{n-1} \supseteq \cdots H^n(G,E)_0=\{0\}$ such that
\begin{equation}
E^{p,q}_{\infty} \simeq H^{p+q}(G,E)_p / H^{p+q}(G,E)_{p+1} \nonumber
\end{equation}
for all $p,q$, where $E^{*,*}_{\infty}$ is the infinity page of the spectral sequence. See the references above for its exact construction, or just take to heart the following: in this text we only use spectral sequences for vanishing result, and the $E^{*,*}_{\infty}$ are \emph{subquotients} of the $E^{*,*}_2$. (In general, whenever constructing a spectral sequence to compute something, one hopes and prays that it "collapses" at the $E_2$-term, i.e.~that this coincides with the infinity page.)

The double complex argument used to construct the Hochschild-Serre spectral sequence for cohomology of discrete groups can be carried out in the continuous case as well, except that one has to show at one point that, considering the pointwise application $d''$ of the coboundary map $d_G\colon C(G^{q},E)^H \rightarrow C(G^{q+1},E)^H$ in the complex
\begin{displaymath}
\xymatrix{ \cdots \ar[r] & C(Q^p,C(G^q,E)^H)^Q \ar[r]^{d''} & C(Q^p,C(G^{q+1},E)^H)^Q \ar[r]^{d''} & C(Q^p,C(G^{q+2},E)^H)^Q \ar[r] & \cdots }
\end{displaymath}
one wants to show that the homology of the middle term is isomorphic to $C(Q^p(H^q(H,E)))$ for which one needs a continuous section of the coboundary map $d_G$ on $C(G^{q},E)^H$. (There is also an issue of definition when $H^q(H,E)$ is not Hausdorff.)

However, this is the only obstruction to carrying through the double complex argument, whence by the closed graph theorem we get the following version of the spectral sequence in continuous cohomology. As usual we are taking into consideration the action of a semi-finite tracial von Neumann algebra $\mathscr{A}$.

\begin{theorem}[(See e.g.~{\cite[Chapter III, Section 5]{Guichardetbook}})] \label{thm:HSgroups} \todo{thm:HSgroups}
Let $H$ be a closed normal subgroup of the locally compact (second countable) group $G$, and let $E$ be a topological $G$-$\mathscr{A}$-module. Suppose that $E$ is a Fr{\'e}chet space and that $H^n(H,E)$ is Hausdorff for all $n\geq 0$.

Then there is a spectral sequence with $E_2$-term $E_2^{p,q}=H^p(G/H,H^q(H,E))$, abutting to the continuous cohomology $H^n(G,E)$. \hfill \qedsymbol
\end{theorem}

In a similar direction we next show that one can compute the continuous homology by slightly more general resolutions than projective ones, i.e.~a type of flat resolution lemma.

\begin{definition} \label{def:tensoracyclic} \todo{def:tensoracyclic}
Let $E$ be a complete topological $\mathscr{A}$-$G$-module. We call $E$ $\otimes$-acyclic if
\begin{equation}
H_n(G,E) = 0, \quad \textrm{ for all } \; n\geq 1. \nonumber
\end{equation}
\end{definition}

\begin{proposition} \label{prop:acyclicreshom} \todo{prop:acyclichom}
Let $G$ be a {\lcsu} group and $E$ a complete topological $\mathscr{A}$-$G$-module. Suppose that $\xymatrix{ (P_*,\partial_*)\ar[r]^{\epsilon} & E \ar[r] & 0 }$ is a strengthened resolution in $\bar{\mathfrak{E}}_{\mathscr{A},G}$ of $E$ by $\otimes$-acyclic modules. Then for all $n\geq 0$
\begin{equation}
H_n(\underline{\mathscr{C}}_GP_*,\partial_*) \simeq H_n(G,E) \nonumber
\end{equation}
as $\mathscr{A}$-modules.
\end{proposition}

We note that, as formulated, and as the proof goes below, the isomorphism here is in the algebraic sense, i.e.~an isomorphism of $\mathscr{A}$-modules in the category $\mathfrak{E}^{(alg)}_{\mathscr{A}}$. As remarked upon earlier (e.g.~in the introduction) we are generally not so interested in the (quotient) topology on continuous homology, but we also remark that checking all the maps appearing should give an straight-forward verification that the isomorphism above is indeed a homeomorphism as well, when both sides are endowed with the relevant quotient topologies.

\begin{proof}
Write a double complex $K^{p,q}:= L^2_{c}(G^{p},P_q)$ for $p,q\geq 0$ and boundary maps $d'\colon L^2_c(G^{p},P_q) \rightarrow L^2_c(G^{p-1},P_q)$ induced by the usual boundary maps on the complex $L^2_c(G^*,P_q)$ of $P_q$-valued inhomogeneous chains, for every $q$, and $d''$ being pointwise application of $\partial$.

Considering the $K^{p,q}$ as just (algebraic) $\mathscr{A}$-modules, there are two (homology) spectral sequences in this context (i.e.~in the category $\mathfrak{E}_{\mathscr{A}}^{(alg)}$) abutting to the total homology. Considering complexes

\begin{displaymath}
\xymatrix{ \mathfrak{M}_{(p)} : & \cdots \ar[r]^>>>>{d''} & H_p( K^{*,q}, d' ) \ar[r]^<<<<{d''} & H_p( K^{*,q-1}, d' ) \ar[r]^<<<<{d''}  & \cdots \ar[r] & 0 \\ \mathfrak{N}_{(q)} : & \cdots \ar[r]^>>>>{d'} & H_q( K^{p,*}, d'' ) \ar[r]^<<<<{d'} & H_q( K^{p-1,*}, d'' ) \ar[r]^<<<<{d'} & \cdots \ar[r] & 0 },
\end{displaymath}
the spectral sequences have $E^2$-terms as follows.

For the first spectral sequence the $E^2$-term is ${}'E^2_{p,q} = H_q(\mathfrak{M}_{(p)})$. Then we observe that since $P_q$ is $\otimes$-acyclic for all $q$, we have (the first equality is clear)
\begin{equation}
H_p( K^{*,q}, d' ) = H_p( G, P_q ) = \left\{ \begin{array}{cc} 0 & , p>0 \\ \underline{\mathscr{C}}_GP_q & ,p=0 \end{array} \right. \nonumber
\end{equation}

Hence this spectral sequence collapses at the second page and $H_n^{tot}(K^{*,*}, d'+d'') = E^2_{0,n} = H_n( \underline{\mathscr{C}}_GP_* )$.

For the other spectral sequence we have ${}''E^2_{p,q} = H_p( \mathfrak{N}_{(q)} )$. We want to show that this collapses and computes the homology $H_*(G,E)$ since we already know that it abuts to the total homology. We claim that
\begin{equation}
(\mathfrak{N}_{(q)})_p \simeq \left\{ \begin{array}{cc} 0 &, q>0 \\ L^2_c(G^p,E) & , q=0 \end{array} \right. \nonumber
\end{equation}

Indeed it is clear that the kernel of the boundary map $d''\colon L^2_c(G^p,P_q) \rightarrow L^2_c(G^p,P_{q-1})$ (with $P_{-1}:=0$) is exactly $L^2_c(G^p,\ker (\partial_{q-1} P_q\rightarrow P_{q-1}))$. Further, since the resolution of $E$ by $(P_*)$ is strenghtened, there are continuous $\mathscr{A}$-linear sections $s_q\colon \operatorname{im} \partial_q \rightarrow P_q$, and these extend by pointwise application then to sections of $d''$. This shows the claim for $q>0$. For $q=0$ we are looking at
\begin{displaymath}
\xymatrix{ \cdots \ar[r] & L^2_c(G^p,P_1) \ar[r]^{d''} & L^2_c(G^p,P_0) \ar[r] & 0 }.
\end{displaymath}

By the same argument, the quotient identifies with $L^2_c(G^p,P_0/\operatorname{im}(\partial_0))$ and indeed since we also have $P_0/\operatorname{im}(\partial_0)\simeq E$ ($\mathscr{A}$-linearly and homeomorphically) the claim follows.

Then one checks readily enough that $d'$ induces the canonical boundary map on inhomogeneous chains, completing the proof of the proposition.
\end{proof}

\begin{remark}
We mention for completeness that, of course, there is a Hochschild-Serre spectral sequence in continuous homology as well, and that one can compute the continuous cohomology using (the corresponding notion of) acyclic resolutions too.
\end{remark}

\chapter{Quasi-continuous cohomology for locally compact groups} \label{chap:QC} \todo{chap:QC}
%\epigraph{You know, I wish that there was some way that I could be outside playing basketball, in the rain, and not get wet. Now wouldn't that be great?}{Beastie Boys}

In this chapter we make some further remarks on rank completion, along the lines of Section \ref{sec:quasimod}. That is, we consider a construction of a localization (at least in principle) of the category og topological $G$-$\mathscr{A}$-modules with respect to the rank (dimension) isomorphisms, where we explicitly represent the inverses as "quasi-morphisms" in some sense (Definition \ref{def:QCmod}).

The end result is that when $H\unlhd G$ is a closed normal subgroup of the {\lcsu} group $G$, then if the quotient is totally disconnected, we can in fact construct a Hochschild-Serre spectral sequence "up to dimension". We emphasize that this is a structural approach, in the sense that once we specify which morphisms to invert and in what sense, the rest falls out quite naturally.

I also remark that, even though from a technical point of view, this chapter is superfluous in order to prove the results in the main text, it should be viewed within the context of the slogan that totally disconnected groups should be placed on the same footing as discrete groups, whence it is intrinsically useful to develop tools that behave exactly as they do for discrete groups whenever possible.

This chapter is intended as a casual digression; a spark, not a flame so to speak. As such, I refrain from giving exhaustive details.

\section{Localizing the category of topological modules}

In this section $(\mathscr{A},\psi)$ is a semi-finite, $\sigma$-finite tracial von Neumann algebra. We consider throughout this chapter vector spaces with topologies that are not necessarily Hausdorff. It is understood that the topologies are still vector topologies in the sense that all the vector operations are continuous.

\begin{proposition} \label{prop:quasitopconditions} \todo{prop:quasitopconditions}
Let $E$ be a \emph{\textbf{not necessarily Hausdorff}} topological vector space, and in addition an $\mathscr{A}$-module such that each $a\in \mathscr{A}$ acts as a continuous linear map.

Suppose $E$ is has a countable neighbourhood basis at each point. Then:

\begin{enumerate}[(i)]
\item If $E$ is Hausdorff, the maximal zero-dimensional submodule, $\{ \xi\in E \mid \exists p_n\nearrow \bbb \forall n: \xi.p_n=0\}$ is closed in $E$.
\item In general, if $\overline{\{0\}}$, the topological closure of $\{0\}$ in $E$ has $\mathscr{A}$-dimension zero, then the topological closure of any zero-dimensional submodule is zero-dimensional.
\end{enumerate}
\end{proposition}

\begin{proof}
To prove (i) suppose that $(e_k)_k$ is a sequence of points in $E$, each of which are annihilated by appropriate, arbitrarily large projections in $\mathscr{A}$. Denote $E_k:= e_k.\mathscr{A}$ the $\mathscr{A}$-module algebraically generated by $e_k$. Then consider $\Xi := \prod_k E_k$ as a module with diagonal action and
\begin{equation}
\Xi_k := \{ (\xi_i)\in \Xi \mid \forall i=1,\dots ,k : \xi_i = 0 \}. \nonumber
\end{equation}

By additivity of $\mathscr{A}$-dimension, $\Xi_k\subseteq \Xi$ is rank-dense for all $k$. Hence by the countable annihilation lemma, $\{0\} = \cap_k \Xi_k$ is rank-dense in $\Xi$.

In particular, there exist $p_n\nearrow \bbb$ such that for all $k,n$ we have $e_k.p_n=0$. Hence if $e_k\rightarrow_k e$, then $e.p_n=0$ for all $n$. This proves (i).

Part (ii) is the same except we get just $e.p_n\in \overline{\{0\}}$, from which the claim then follows.
\end{proof}

\begin{definition} \label{def:QCmod} \todo{def:QCmod}
A quasi-topological $G$-$\mathscr{A}$-module is a vector space $E$ with a not necessarily Hausdorff vector topology, such that:
\begin{enumerate}[(i)]
\item $E$ carries a continuous action of $G$, and a commuting action of $\mathscr{A}$ by continuous linear maps,
\item the topological closure of any zero-dimensional submodule is itself zero-dimensional.
\end{enumerate}
A quasi-morphism $f\colon E\rightarrow F$ of quasi-topological modules is an equivalence class of $G$-$\mathscr{A}$-linear continuous maps defined on subquotients,
\begin{equation}
f_0 \colon E_0'/E_0'' \rightarrow F_0'/F_0'', \nonumber
\end{equation}
where all spaces appearing are submodules, the $E_0',F_0'$ are closed in the respective ambient modules, and such that
\begin{equation}
\dim_{\mathscr{A}} E/E_0' = \dim_{\mathscr{A}} E_0'' = \dim_{\mathscr{A}} F/F_0' = \dim_{\mathscr{A}} F_0'' = 0 \nonumber
\end{equation}
and two such maps $f_0,f_1$ are equivalent if there is an $f_2$ and a commutative diagram
\begin{displaymath}
\xymatrix{ E_0'/E_0'' \ar[r]^{f_0} & F_0'/F_0'' \\ E_2'/E_2'' \ar[d] \ar[u] \ar[r]^{f_2} & F_2'/F_2'' \ar[d] \ar[u] \\ E_1'/E_1'' \ar[r]^{f_1} & F_1'/F_1'' }
\end{displaymath}
where
\begin{equation}
\dim_{\mathscr{A}} E/E_2'' = \dim_{\mathscr{A}} E_2'' = 0. \nonumber
\end{equation} 
\end{definition}

\begin{remark}
By Proposition \ref{prop:quasitopconditions}, condition (ii) in the definition above is satisfied if:
\begin{itemize}
\item Each point in $E$ has a countable neighbourhood basis,
\item The topological closure of $\{0\}$ in $E$ has $\mathscr{A}$-dimension zero.
\end{itemize}
\end{remark}

\begin{proposition}
There is a canonical composition of morphisms between quasi-topological $G$-$\mathscr{A}$-modules, and this induces the structure of a category with objects quasi-topological modules and morphisms the quasi-morphisms in the sense of Definition \ref{def:QCmod}.
\end{proposition}

\begin{proof}
Let $f_0$ respectively $g_1$ be representatives of morphisms $f\colon E\rightarrow F$ respectively $g\colon F\rightarrow Q$. Then $f_0$ is equivalent to the morphism $f_2\colon (E_0'\cap f_0^{-1}(F_1'))/E_0'' \rightarrow (F_0'\cap F_1')\left/ (F_0''\cap F_1'+F_1''\cap F_0')\right.$ and $g_1$ to the morphism $g_3\colon (F_0'\cap F_1')\left/ (F_0''\cap F_1'+F_1'' \cap F_0')\right. \rightarrow Q_1'\left/ (Q_1''+g_1(F_0''\cap F_1'))\right.$. Then the composition
\begin{equation}
g_3\circ f_2 \colon (E_0'\cap f_0^{-1}(F_1'))/E_0'' \rightarrow Q_1'\left/ (Q_1''+g_1(F_0''\cap F_1'))\right. \nonumber
\end{equation}
represents a quasi-morphism $g\circ f\colon E\rightarrow Q$.
\end{proof}

Hence we talk about the category of quasi-topological modules. If need be we denote it $\mathfrak{E}^{\mathfrak{Q}}_{G,\mathscr{A}}$. Denote by $\mathfrak{Q}$ the canonical embedding functor from topological modules satisfying the conditions in \ref{def:QCmod} to the category of quasi-topological modules, and denote by $\mathfrak{P}$ functor from quasi-topological modules to topological modules of forming the quotient by the maximal zero-dimensional submodule, which we denote by $\mathfrak{T}_{max}(-)$.

Whenever $f\colon E\rightarrow F$ is a quasi-morphism and $f_0\colon E_0'/E_0'' \rightarrow F_0'/F_0''$ a representative, we can construct a new representative $f_1\colon E_0'/\mathfrak{T}_{max}(E_0') \rightarrow F_0'/\mathfrak{T}_{max}(F_0')$ which is even a morphism of topological modules since the $\mathfrak{T}_{max}(-)$ are closed submodules. We call any such representative a \emph{canonical representative}.

\begin{lemma} \label{lma:quasimorphismeq} \todo{lma:quasimorphismeq}
Let $f,g\colon E\rightarrow F$ be quasi-morphisms. Fix representatives $f_0,g_1$. Denote $E'':=E_0'\cap E_1'$ and let $E'':=\mathfrak{T}_{max}(E')$ be the maximal zero-dimensional submodule of $E'$. Let $F'':=\mathfrak{T}_{max}(F)$ be the maximal zero-dimensional submodule of $F$. 

Then $f$ and $g$ are equivalent if and only if the induced morphism $(f_0-g_1)\colon E'/E'' \rightarrow F/F''$ has zero-dimensional image.
\end{lemma}

\begin{proof}
Obvious.
\end{proof}

\begin{definition} \label{def:quasistrengthened} \todo{def:quasistrengthened}
A morphism $f\colon E\rightarrow F$ of quasi-topological modules is injective if every (equivalently any) representive has zero-dimensional kernel, surjective if every (equivalently any) representative has zero-dimensional cokernel.

We say that an injective morphism $f\colon E \rightarrow F$ is strengthened if it has a left-inverse (continuous) $\mathscr{A}$-quasi-morphism (i.e.~not necessarily $G$-linear) $s\colon F\rightarrow E$, in the sense that the composition of quasi-morphisms $s\circ f$ is equivalent to the identity on $E$.

In general, we say that $f$ is strengthened if both injective morphisms $\iota \colon \ker f \rightarrow E$ and $\bar{f}\colon E/\ker f\rightarrow F$ are strengthened.
\end{definition}

Note that the modules $\ker f$ and $E/\ker f$ are only defined up to dimension zero, i.e.~up to isomorphism as quasi-topological modules. They are only explicit, as modules, once we choose a representative. Hence we tend to think of kernels in the categorical sense, i.e.~as (isomorphism classes of) embeddings of the "kernels". Observe that the category of quasi-topological modules has all kernels, but the not every morphism need have a well-defined cokernel inside the category.

\begin{lemma}[(rearrangement)] \label{lma:QCrearr} \todo{lma:QCrearr}
Let $f\colon E \rightarrow F$ be a quasi-morphism and $f_0$ a representative.
\begin{enumerate}[(i)]
\item For every pair of submodules $F_1''\subseteq F_1'$ with $F_1'$ closed in $F$ and such that $\dim_{\mathscr{A}}F_1''=\dim_{\mathscr{A}}F/F_1' = 0$, there is an equivalent representative $f_1\sim f_0$ such that $f_1\colon E_1'/\mathfrak{T}_{max}(E_1')\rightarrow F_1'/\mathfrak{T}_{max}(F_1')$ and $E_1'\subseteq E_0'$ is a closed submodule with codimension zero.
\item For every pair of submodules $E_1''\subseteq E_1'$ with $E_1'$ closed in $E$ and such that $\dim_{\mathscr{A}}E_1''=\dim_{\mathscr{A}}E/E_1' = 0$, there is an equivalent representative $f_1\sim f_0$ such that $f_1\colon E_1'/\mathfrak{T}_{max}(E_1')\rightarrow F_1'/\mathfrak{T}_{max}(F_1')$ and $F_1'\subseteq F_0'$ is a closed submodule with codimension zero.
\end{enumerate}
\end{lemma}

We leave out the straight-forward proof. \hfill \qedsymbol

\begin{lemma}
The definition of 'strengthened' in \ref{def:quasistrengthened} is meaningful, i.e.~if we fix representatives $f_0,f_1$ of $f$ and denote by $\pi_i\colon E_i'\rightarrow E_i'/E_i''$ the projection, then the induced quasimorphisms $\iota_0\colon \pi_0^{-1}(\ker f_0)\rightarrow E$ and $\bar{f}_0\colon E_0'/\ker f_0 \rightarrow F$ are strengthened if and only if the analogous quasi-morphisms $\iota_1,\bar{f}_1$ are. Thus the definition is independent of the specific realization of the kernel of $f$.

A morphism $f\colon E\rightarrow F$ of weak $G$-$\mathscr{A}$-modules is strengthened if and only if it has a (canonical) representative which is a strengthened morphism of topological modules in the sense of Definition \ref{def:relinjective}.
\end{lemma}

\begin{proof}
The first part is clear. Indeed, it expresses just the fact the a quasi-morphism has a well-defined kernel in the sense of category theory.

The 'if' in the second part is obvious.

%%%%%%%%%%%%%%%%%%%%%%%%%%%%%%%%%%%%%%
\begin{comment}
Suppose on the other hand that $f$ is strengthened in the sense of the previous definition. It is sufficient to prove the claim in the case where $f$ is injective, also in the sense of the previous definition.

Let $s\colon F\rightarrow E$ be a left-inverse continuous quasi-$\mathscr{A}$-morphism. We can choose canonical representatives $f_0$ and $s_1$. Since $f$ is injective in the sense of quasi-morphisms, $f_0$ is injective as a map. One then checks that the restriction of $f_0$ to the intersection of images of $f_0^{-1}(\operatorname{im}(F_0'\cap F_1' \rightarrow F_0'/F_0''))$ and $E_1'\cap E_0'$ in $E_0'/E_0''$ is an injective strengthened morphism of topological modules.
\end{comment}
%%%%%%%%%%%%%%%%%%%%%%%%%%%%%%%%%%%%%%%%%
The 'only if' part follows by choosing canonical representatives and repeatedly applying Lemmas \ref{lma:quasimorphismeq} and \ref{lma:QCrearr}.
\end{proof}

\begin{definition} \label{def:quasiinjective} \todo{def:quasiinjective}
A quasi-topological $G$-$\mathscr{A}$-module $E$ is (relatively) injective if whenever we have a diagram
\begin{displaymath}
\xymatrix{  & E \\ 0\ar[r] & A \ar[u]^v \ar[r]^u & B \ar@{-->}[ul]_{\exists ?w} }
\end{displaymath}
where $u\colon A \rightarrow B$ is a strengthened morphism in the sense of Definition \ref{def:quasistrengthened}, there is a morphism $w\colon B\rightarrow E$ making the diagram commute.
\end{definition}

\begin{proposition}
The canonical 'embedding' functor $\mathfrak{Q}$ from topological $G$-$\mathscr{A}$-modules to the category of quasi-topological $G$-$\mathscr{A}$-modules sends injectives with no non-trivial, closed, zero-dimensional $G$-$\mathscr{A}$-submodules, to injectives.
\end{proposition}

\begin{proof}
Obvious.\todo{WRITE}
\end{proof}

\begin{proposition}
The category of quasi-topological modules has sufficiently many injectives.
\end{proposition}

\begin{proof}
Let $E$ be a quasi-topological module. Denote by $E''$ the maximal zero-dimensional submodule. Then $\mathfrak{P}(E)=E/E''$ is a topological $G$-$\mathscr{A}$-module whence there is an injective topological $G$-$\mathscr{A}$-module $F$ such that $\mathfrak{P}(E)$ embeds in $F$.

Since $\mathfrak{P}(E)$ satisfies the conditions of the previous proposition, clearly we can choose $F$ such that it does as well, e.g.~$F=C(G,\mathfrak{P}(E))$ works.

By the previous proposition, $F$ is injective the the category of quasi-topological modules, and $E$ embeds in $F$ in this category.
\end{proof}

\begin{definition} \label{def:quasiexact} \todo{def:quasiexact}
A diagram $\xymatrix{E\ar[r]^f & F\ar[r]^{g} & Q}$ of quasi-topological modules is exact (or, more explicitly, quasi-exact or dimension exact) at $F$ provided that there are composable (as maps) choices of representatives $f_0,g_1$ such that $\operatorname{im} f_0\subseteq \ker g_1$, and that this inclusion is rank dense.
\end{definition}

\begin{theorem}
Let $E,F$ be quasi-topological $G$-$\mathscr{A}$-modules, $\xymatrix{ 0\ar[r] & E\ar[r]^>>>{\epsilon} & (E_*,d^*) }$ a strengthened resolution of $E$, and $\xymatrix{ 0\ar[r] & F\ar[r]^>>>{\epsilon} & (F_*,f^*) }$ a complex with each $F_i$ an injective quasi-topological module.

Then any quasi-morphism $u\colon E\rightarrow F$ lifts to a sequence of quasi-morphisms $u_*\colon E_*\rightarrow F_*$ compatible with the coboundary maps, and any such lift is unique up to $G$-$\mathscr{A}$-homotopy.
\end{theorem}

Note that all statements in the theorem ('resolution', 'injective', 'unique', etc.) refer to the category of quasi-topological modules.

\begin{proof}
One can prove this explicitly by choosing canonical representatives along all of $(E_*,d^*)$ and then constructing the lifting inductively, all the time keeping track of domains "to the left" of the current step, taking intersections (possible by the countable annihilation lemma) each time to make sure everything is nicely defined.

Alternatively, the statement is true generally in exact categories \cite[Section 12]{BuhlerExact}, and we show below in Theorem \ref{thm:QCexactcat} that $\mathfrak{E}^{\mathfrak{Q}}_{G,\mathscr{A}}$ is exact.
\end{proof}

\begin{definition}
For a quasi-topological module $E$ we define the quasi-continuous cohomology $H_{\mathfrak{Q}}^n(G,E)$ as the $n$'th cohomology of the complex
\begin{displaymath}
\xymatrix{ 0\ar[r] & E_0^{G} \ar[r]^{d^0} & E_1^{G} \ar[r]^{d^1} & \cdots }
\end{displaymath}
where $\xymatrix{ 0\ar[r] & E\ar[r]^{\epsilon} & (E_*,d^*) }$ is any injective strengthened resolution of $E$.

This is an $\mathscr{A}$-module defined up to quasisomorphism.\todo{make this terminology consistent :)}
\end{definition}

That is, computing $H^n_{\mathfrak{Q}}$ using two different injective resolutions, the two cohomology spaces are "algebraically" quasi-isomorphic as $\mathscr{A}$-modules. In particular, if $\mathscr{A}$ is finite, they have isomorphic rank-completions.

If $\mathscr{A}$ is semi-finite, we could fix a trace $\psi_0$ on $\mathscr{A}$ and $p_0$ a projection with $\psi_0(p_0)<\infty$ and central support the identity, and then proceed to define $H^n_{(\mathfrak{Q},p_0)}(G,E.p_0)$ as the rank-completion of the cohomology space $H^n_{\mathfrak{Q}}(G,E.p_0)$. This is then unique.

\begin{theorem}
For any topological $G$-$\mathscr{A}$-module $E$, we have
\begin{equation}
\dim_{\mathscr{A}} H^n(G,E) = \dim_{\mathscr{A}} H_{\mathfrak{Q}}^n(G,\mathfrak{Q}(E)). \nonumber
\end{equation}
More generally, for any quasi-topological $G$-$\mathscr{A}$-module $F$,
\begin{equation}
\dim_{\mathscr{A}} H^n_{\mathfrak{Q}}(G,F) = \dim_{\mathscr{A}} H^n(G,\mathfrak{P}(F)). \nonumber
\end{equation}
\end{theorem}

\begin{flushright}
\qedsymbol
\end{flushright}

\section{Exact structure}

As mentioned in the introduction, it should be possible to put the relative homological algebra constructions in the present text entirely within the framework of exact categories. Our reference for exact categories is \cite{BuhlerExact}.

Let us consider the class $\mathscr{E}$ of kernel-cokernel pairs in the category of quasi-topological $G$-$\mathscr{A}$-modules as follows. Let a kernel-cokernel pair
\begin{displaymath}
\xymatrix{ E \ar[r]^{\iota} & F \ar[r]^{\pi} & Q }
\end{displaymath}
be in $\mathscr{E}$ if both quasi-morphisms $\iota,\pi$ are strengthened.

\begin{lemma}
Let $\pi\colon E\rightarrow Q$ be a surjective quasi-morphism. Then $\pi$ is strengthened if and only if it has a right-inverse in $\mathfrak{E}^{\mathfrak{Q}}_{\bbb,\mathscr{A}}$. \hfill \qedsymbol
\end{lemma}

See \cite[Definition 2.1]{BuhlerExact} for the definition of an exact caegory.

\begin{theorem} \label{thm:QCexactcat}  \todo{thm:QCexactcat}
The class $\mathscr{E}$ defined above induces an exact structure on the category of quasi-topological modules.
\end{theorem}

\begin{proof}
Axioms $[E0]$ and $[E0^{op}]$ are trivial.

By construction the admissible monics and epics are precisely the strengthened injective respectively surjective quasi-morphisms. Axioms $[E1]$ and $[E1^{op}]$ follow directly from this.

To see that $[E2]$ holds consider a push-out diagram
\begin{displaymath}
\xymatrix{ E \ar[r]^{u} \ar[d]_{v} & U \ar@{-->}[d]^{\bar{v}} \\ V \ar@{-->}[r]_{\bar{u}} & F }
\end{displaymath}
with $u$ strengthened injective. Choose canonical representatives $u_0,v_0$ defined on the same domain $E_0'/\mathfrak{T}_{max}(E_0')$ and such that $u_0\colon E_0'/\mathfrak{T}_{max}(E_0') \rightarrow U_0'/\mathfrak{T}_{max}(U_0')$ is a strengthened morphism in the category of topological modules, with left-inverse $s\colon U_0'/\mathfrak{T}_{max}(U_0') \rightarrow E_0'/\mathfrak{T}_{max}(E_0')$. Define 

\begin{equation}
F:= \left( (U_0'/\mathfrak{T}_{max}(U_0'))\oplus (V_0'/\mathfrak{T}_{max}(V_0')) \right) / \overline{\operatorname{im}}\; (u_0\oplus (-v_0)). \nonumber
\end{equation}
and consider the natural morphisms $\bar{u}:=\bbb\oplus 0$ and $\bar{v}:=0\oplus \bbb$ into this in the diagram.

Now one checks easily enough that this defines a push-out diagram in $\mathfrak{E}^{\mathfrak{Q}}_{G,\mathscr{A}}$ and that the morphism $\bar{u}$ is strengthened injective with left-inverse $s\colon (x,y) \mapsto x+(v\circ s)(y)$.

The final axiom $[E2^{op}]$ is entirely analogous to this.
\end{proof}

\section{The Hochschild-Serre spectral sequence}

In this section we construct a Hochschild-Serre spectral sequence in quasi-continuous cohomolgy, under suitable assumptions. In order to do this, we first need to justify our use of the spectral sequence associated with a filtration / double complex in the category of quasi-topological $\mathscr{A}$-modules.

This is straight-forward: one just proceeds as usual, and everytime one encounters an countable intersection of closed submodules, one appeals to the countable annihilation lemma, glance furtively left and right, and then proceed as if everything was OK. \todo{Obviously this paragraph gets replaced with the details later...}

Compare the following e.g.~with \cite[Proposition A.9]{Guichardetbook}.

\begin{proposition} \label{prop:CQdoublecomplexI} \todo{prop:CQdoublecomplexI}
Let $K^{p,q}$ be a first quadrant double complex of quasi-topological $\mathscr{A}$-modules with coboundary maps
\begin{equation}
d'\colon K^{*,q}\rightarrow K^{*+1,q}, \quad d''\colon K^{p,*}\rightarrow K^{p,*+1}, \nonumber
\end{equation}
and denote
\begin{equation}
{}'H^{p,q}(K) := (\ker d' \cap K^{p,q})/d'(K^{p-1,q}), \quad {}''H^{p,q}(K) := (\ker d''\cap K^{p,q})/d''(K^{p,q-1}). \nonumber
\end{equation}

Suppose that ${}'H^{p,q}(K), {}''H^{p,q}(K)$ are quasi-topological $\mathscr{A}$-modules (in their quotient topology) for all $p,q\geq 0$. Then, considering the complexes
\begin{displaymath}
\xymatrix{ {}'\mathfrak{H}^{p} : & 0\ar[r] & {}'H^{p,0}(K) \ar[r]^{d''} & {}'H^{p,1}(K) \ar[r]^{d''} & \cdots \\ {}''\mathfrak{H}^{q} : & 0\ar[r] & {}''H^{0,q}(K) \ar[r]^{d'} & {}''H^{1,q}(K) \ar[r]^{d''} & \cdots },
\end{displaymath}
there are two spectral sequences ${}'E_2^{p,q} = H^q({}'\mathfrak{H}^p)$ respectively ${}''E_2^{p,q} = H^p({}''\mathfrak{H}^q)$, both abutting to $H^*_{tot}(K)$ in the category of algebraic quasi-$\mathscr{A}$-modules.
\end{proposition}

Here is an application. (Compare \cite[Lemma 3.9]{KyPe12}.)

\begin{definition}
A quasi-topological $G$-$\mathscr{A}$-module $E$ is called acyclic if $H^n_{\mathfrak{Q}}(G,E)$ is equivalent to zero, i.e.~$\dim_{\mathscr{A}} H^n_{\mathfrak{Q}}(G,E) = 0$, for all $n\geq 1$.
\end{definition}

\begin{proposition} \label{prop:CQacyclic} \todo{prop:CQacyclic}
Let $\xymatrix{ 0\ar[r] & E \ar[r]^{\epsilon} & (E_*,d_E^*)}$ be a strengthened resolution, in the category of quasi-topological modules, of $E$ by acyclic modules $E_*$. Then
\begin{equation}
H^n_{\mathfrak{Q}}(G,E) = H^n(((E_*)^G,d_E^*)), \nonumber
\end{equation}
in the category of (algebraic) quasi-$\mathscr{A}$-modules.
\end{proposition}

One way to prove this is again to appeal to exact categories. However, here is a direct proof.

\begin{proof}
We can assume without loss of generality that $E_i$ is in fact Hausdorff for all $i$, with no closed submodules of dimension zero. Define a double complex $K^{p,q}:= C(G^p,E_q)^G$ for all $p,q\geq 0$ with the coboundary maps $d'$ the canonical coboundary map on the bar resolution $C(G^*,E_q)$ of $E_q$, respectively $d''$ the pointwise application of $d_E$.

Then by hypothesis already we have ${}'H^{p,q}(K) = H^p_{\mathfrak{Q}}(G,E_q)$ which is zero in the category of quasi-topological $\mathscr{A}$-modules since $E_q$ is acyclic, unless $p=0$ in which case we get $E_q^G$. By Proposition \ref{prop:CQdoublecomplexI} there is a spectral sequence abutting to $H^*_{tot}(K)$ and this has $E_2$-term, using notation from that proposition,
\begin{equation}
{}'E_2^{p,q} = H^q({}'\mathfrak{H}^{p}) = \left\{ \begin{array}{ll} 0 & , p>0 \\ H^q((E_*^{G},d_E^*)) & , p=0 \end{array} \right. . \nonumber
\end{equation}
Here the equalities are in the sense of (algebraic) quasi-$\mathscr{A}$-modules.

Hence the spectral sequence collapses at the $E_2$-term, and $H^n_{tot}(K) = H^n((E_*,d_E^*))$. (Again, "up to dimension".)

Since the resolution of $E$ by $E_*$ is strengthened, there is for each $i\geq 0$ a closed $\mathscr{A}$-submodule $F_i\leq E_i$ of codimension zero, such that $d_E^i\vert_{F_i}$ is strengthened in the category of (proper) topological $G$-$\mathscr{A}$-modules. It follows directly from this and the countable annihilation lemma that
\begin{equation}
{}''H^{p,q}(K) = \left\{ \begin{array}{ll} C(G^p,\ker d_E^{0}) \simeq C(G^p, E) & , q=0 \\ 0 & , q>0 \end{array} \right. .\nonumber
\end{equation}

Hence we get the spectral sequence with $E_2$-term ${}''E_2^{p,0} = H^p(G,E)$ (up to dimension), and zero-dimensional elsewhere; the spectral sequence collapses whence $H^n_{tot}(K) = H^n(G,E)$ (up to dimension).

This finishes the proof.
\end{proof}

\begin{theorem}[(Hochschild-Serre)] \label{thm:QCHSSS} \todo{thm:QCHSSS}
Let $G$ be a {\lcsu} group and $H$ a closed normal subgroup such that $G/H$ is totally disconnected. Suppose that $H^n(H,L^2G)$ is a quasi-topological $LG$-module for all $n$. Then there is a Hochschild-Serre spectral sequence abutting to $H^*(G,L^2G)$ with $E_2$-term
\begin{equation}
E_2^{p,q} = H^p_{\mathfrak{Q}}(G/H,H^q(H,L^2G)) \simeq H^p_{\mathfrak{Q}}(G/H,\underline{G}^q(H,L^2G)). \nonumber
\end{equation}
\end{theorem}

\begin{proof}
Write a double complex
\begin{equation}
K^{p,q} := \mathcal{F}( (G/H)^{p+1}, C(G^{q+1},E)^H)^G \nonumber
\end{equation}
in $\mathfrak{E}_{\bbb,\mathscr{A}}$, the category of topological $\mathscr{A}$-modules. Then do the usual diagram chase and show the claim explicitly that, for $d'' \colon K^{p,*}\rightarrow K^{p,*+1}$ pointwise application of the coboundary map on $C(G^{q+1},E)^H$ one gets an isomorphism of topological quasi-modules
\begin{equation}
H^q(K^{p,*},d'') \simeq C( (G/H)^{p+1}, \underline{H}^{q}(H,E))^G. \nonumber
\end{equation}

This is where the assumption that $G/H$ be totally disconnected is used, so that there are no obstacles to lifting a pointwise coboundary to a global one. We leave out the details.
\end{proof}

Of course, the "right" proof should go through a Grothendieck spectral sequence for derived functors in exact categories. I leave it for future work.

\chapter{Homological algebra for Lie groups and Lie algebras} \label{app:cohomLie}

\section{Cohomology of Lie groups; a technical lemma} \label{sec:cohomLiegroup} \todo{sec:cohomLiegroup}

Let $X$ be a (smooth, $2$nd countable) manifold and $E$ a quasi-complete topological vector space. We denote by $C^{\infty}(X,E)$ the space of smooth functions $X\rightarrow E$, in the "strong" sense that a function is $f\colon X \rightarrow E$ differentiable in a point $x\in X$ if, choosing a chart $U$ around $x$, the induced function $f_{\vert U}$ has all derivatives in local coordinates. This is equivalent to $f$ being weakly differentiable, i.e.~that $x\mapsto \langle e',f(x) \rangle$ being differentiable for all $e'$ in the topological dual of $E$. Equivalently, $C^{\infty}(X,E) \simeq C^{\infty}(X)\bar{\otimes}E$, the projective tensor product, when $E$ is complete. Both equivalent characterizations are due to Grothendieck. See \cite[Chapter 4.4 and Appendix 2]{Warnerbook} for more details.

We endow $C^{\infty}(X,E)$ with the topology of uniform convergence, on compact sets, of all derivatives; as such it is again a quasi-complete space.

If $G$ is a Lie group acting on $X$ by diffeomorphisms, and acting continuously on $E$, and $\mathscr{A}$ is a semi-finite, $\sigma$-finite, tracial von Neumann algebra with a commuting right-action on $E$, then $C^{\infty}(X,E)$ is a topological $G$-$\mathscr{A}$-module when endowed with the actions
\begin{equation}
(g.\xi.a)(x) = g.\xi(g^{-1}.x).a. \nonumber
\end{equation}

As in \cite[Chapter III, Proposition 1.3]{Guichardetbook}, the modules $C^{\infty}(G^n,E)$ are all injective as topological $G$-$\mathscr{A}$-modules, whence we may compute the cohomology $H^*(G,E)$ using the complex
\begin{displaymath}
0\rightarrow C^{\infty}(G,E)^{G} \rightarrow C^{\infty}(G^2,E)^{G}\rightarrow \cdots
\end{displaymath}
and the usual coboundary maps. For more details we refer to \cite[Chapter III, Proposition 1.5]{Guichardetbook}.

Next we give the definitions needed to state \cite[Chapter III, Proposition 1.6]{Guichardetbook}. Recall that the space of smooth vectors in $E$ is the set of vectors $e\in E$ such that the function $g\mapsto g.\xi$ from $G$ to $E$ is smooth. We denote this space $E^{(\infty,G)}$

This is a subspace of $E$, and since the latter is quasi-complete it is dense in $E$ \cite{Car74}. Clearly it is also stable under the $\mathscr{A}$-action.

\begin{proposition}[(~{\cite[Chapter III, Proposition 1.6]{Guichardetbook}})]
For any Lie group $G$ and any quasi-complete topological $G$-$\mathscr{A}$-module $E$,
\begin{equation}
H^n(G,E) = H^n(G,E^{(\infty,G)}), \quad n\geq 0. \nonumber
\end{equation}
\end{proposition}
\begin{flushright}
\qedsymbol
\end{flushright}

The point of considering the space of smooth vectors is that the $G$-module structure induces an action of $\mathfrak{g}$, the Lie algebra of $G$, on $E^{(\infty,G)}$ as follows. If we denote by $\pi$ the action of $G$, then one defines the derived action $\mathrm{d}\pi$ of $\mathfrak{g}$ by
\begin{equation}
\mathrm{d}\pi(x).e = \left. \frac{\mathrm{d}}{\mathrm{d}t}\right\vert_{t=0} \left( \mathbb{R}\owns t\mapsto \exp(tx).e \right). \nonumber
\end{equation}

Further, $E^{(\infty,G)}$ is a quasi-complete topological vector space where the topology is given by the inclusion in $C^{\infty}(G,E)$ mapping $e\in E^{(\infty,G)}$ to the smooth function $g\mapsto g.e$.

In Section \ref{sec:cohomLiealgebra} we set up cohomology theory for Lie algebras and explain the connection with group cohomology. Presently we prove a technical lemma for later use.

Consider an extension of Lie groups $N\unlhd G$. A priori, for a $G$-module $E$, the inclusion of spaces of smooth vectors $E^{(\infty,G)}\subseteq E^{(\infty,N)}$ might be strict. However, in certain arguments involving the Hochschild-Serre spectral sequence, it will be necessary for us to consider the cohomology spaces $H^n(N,L^2G^{(\infty,G)})$, similarly to the previous proposition.

\begin{lemma} \label{lma:technicalsmooth} \todo{lma:technicalsmooth}
Let $G$ be a {\lcsu} group such that $G_0$, the connected component of the identity, is a Lie group. Then the inclusion $L^2G^{(\infty,G_0)}\subseteq L^2G$ is rank dense.

In particular, for any connected subgroup $H$ of $G_0$, the map $H^n(H,L^2G^{(\infty,G_0)})\rightarrow H^n(H,L^2G^{(\infty,H)})$ in (smooth) cohomology induced by the inclusion $L^2G^{(\infty,G_0)}\subseteq L^2G^{(\infty,H)}$ is a rank isomorphism.
\end{lemma}

\begin{proof}
By the initial remarks of \cite[p. 347]{Guichardetbook} we have for any (quasi-complete) topological right-$G_0$-module $E$ and any $f\in C_c(G_0), \xi\in E$,
\begin{equation}
\xi*f := \int_{G_0} f(g)\xi.g \mathrm{d}\mu_{G_0} \in E^{(\infty,G_0)}, \nonumber
\end{equation}
where $\mu_{G_0}$ is Haar measure on $G_0$.

Now we just identify $L^2G \simeq L^2(G_0, L^2(G_0\backslash G))$ and note that for $\xi\in L^2(G_0, L^2(G_0\backslash G))$ we have $\xi*f$ in the sense above identified with the usual convolution product of $\xi\in L^2G$ with the "smooth" measure $f\mathrm{d}\mu_{G_0}$ supported on $G_0$, since the action of $G_0$ on $L^2(G/G_0)$ is trivial.
\end{proof}

\begin{remark}
There should be a more natural approach to this, which would also give directly a more general result in Lemma \ref{lma:abelianbycompact}. Namely, given any locally compact $2$nd countable group $G$, one can define the smooth functions into a continuous module as suggested in \cite[Chapter III, Section 9]{Guichardetbook}. Then the following should be true:
\begin{enumerate}[(i)]
\item For any subgroup $H\leq G$ and any $\xi$ in any topological $G$-module $E$, if $G\owns g\mapsto g.\xi$ is smooth, then so is $H\owns h\mapsto h.\xi$.
\item Convolution by any compactly supported smooth function has image in the smooth vectors.
\item There is a smooth approximate unit in $L^1G$.
\end{enumerate}
\end{remark}

\section{Cohomology of Lie algebras} \label{sec:cohomLiealgebra} \todo{sec:cohomLiealgebra}

In this section we briefly recall the definition of (relative) Lie algebra cohomology. Our approach will be based on that of \cite{Guichardetbook}, since it allows slightly more general coefficient modules than that of \cite{BorelWallachBook}, though we will borrow some structural ideas from the latter.

Throughout this section, $\mathfrak{g}$ will be a real Lie algebra and $\mathfrak{k}$ a Lie subalgebra which is reductive in $\mathfrak{g}$, in the sense that the adjoint representation of $\mathfrak{k}$ in $\mathfrak{g}$ is semi-simple. In particular, this is always the case when $\mathfrak{g}$ is the Lie algebra if a real Lie group and $\mathfrak{k}$ of a compact subgroup.

Also, $(\mathscr{A},\psi)$ will be a fixed semi-finite, $\sigma$-finite tracial von Neumann algebra.

\begin{definition}
A $\mathfrak{g}$-module is a real (or, by restriction, complex) vector space $E$ with a Lie algebra action $\mathfrak{g}\rightarrow \operatorname{End}(E)$, i.e.~a linear map into the codomain, mapping brackets in $\mathfrak{g}$ to the natural bracket in $\operatorname{End}(E)$.

Equivalently, $E$ is a module over the universal enveloping algebra $\mathfrak{A}(\mathfrak{g})$ of $\mathfrak{g}$. If $E$ is a complex vector space, then this is also equivalent to being an $\mathfrak{A}_{\mathbb{C}}(\mathfrak{g})$-module where $\mathfrak{A}_{\mathbb{C}}(\mathfrak{g})$ is the complexification of the universal enveloping algebra.

We say that $E$ is a $\mathfrak{g}$-$\mathscr{A}$-module if $E$ (is a complex vector space and) also carries a commuting action of $\mathscr{A}$.

A morphism $\varphi\colon E\rightarrow F$ of $\mathfrak{g}$-$\mathscr{A}$-modules is a (complex-)linear map such that $\varphi(X.e.T) = X.\varphi(e).T$ for all $X\in \mathfrak{g}, e\in E, T\in \mathscr{A}$.
\end{definition}

\begin{definition}
An injective morphism of $\mathfrak{g}$-$\mathscr{A}$-modules is $\mathfrak{k}$-strengthened if it has a $\mathfrak{k}$-$\mathscr{A}$-linear left-inverse.

A morphism $\varphi\colon E\rightarrow F$ is $\mathfrak{k}$-strengthened if both morphisms $\iota \colon \ker \varphi \rightarrow E$ and $\bar{\varphi} \colon E/\ker\varphi \rightarrow F$ are $\mathfrak{k}$-strengthened.

A $\mathfrak{g}$-$\mathscr{A}$-module $E$ is $\mathfrak{k}$-injective if given any diagram
\begin{displaymath}
\xymatrix{ & E \\ B \ar@{-->}[ur]^{\exists ? w} & A \ar[l]_u \ar[u]^v & 0 \ar[l] }
\end{displaymath}
where the bottom row is $\mathfrak{k}$-strengthened exact, there is a morphism $w\colon B\rightarrow E$ making the diagram commute.
\end{definition}

\begin{proposition} \label{prop:liealginjectiveres} \todo{prop:liealginjectiveres}
Let $E$ be a $\mathfrak{g}$-$\mathscr{A}$-module and $F$ a $\mathfrak{k}$-module. The $\mathfrak{g}$-$\mathscr{A}$-module $\operatorname{hom}_{\mathfrak{k}}(\mathfrak{A}(\mathfrak{g}), \operatorname{hom}_{\mathbb{R}}(F, E))$ is $\mathfrak{k}$-injective, where $\mathfrak{k}$ acts by
\begin{equation}
k.a = -ak, a\in \mathfrak{A}(\mathfrak{g}) , \quad \textrm{respectively} \quad (k.\eta)(f) = -\eta(k.f), \eta \in \operatorname{hom}_{\mathbb{R}}(F,E) \nonumber
\end{equation}
and $\mathfrak{g}$ by
\begin{equation}
(g.\xi)(a)(f) = g.(\xi(a)(f)) - \xi(ga)(f). \nonumber
\end{equation}

Further, considering the case $F=\extprod{n}(\mathfrak{g}/\mathfrak{k})$ with action
\begin{equation}
k.\left( X_1\wedge \cdots \wedge X_n \right) =  \sum_{i=1}^nX_1\wedge \cdots \wedge [k,X_i]\wedge \cdots \wedge X_n, \nonumber
\end{equation}

there is a $\mathfrak{k}$-strengthened exact resolution of $E$
\begin{displaymath}
\xymatrix{ 0\ar[r] & E \ar[r]^>>>>>{\bar{\epsilon}} & \operatorname{hom}_{\mathfrak{k}}(\mathfrak{A}(\mathfrak{g}), E) \ar[r]^<<<<<{d^0} & \cdots \ar[r] &  \operatorname{hom}_{\mathfrak{k}}(\mathfrak{A}(\mathfrak{g})\otimes \extprod{n}(\mathfrak{g}/\mathfrak{k}), E) \ar[r]^<<<<<<{d^n} & \cdots }
\end{displaymath}
where $\bar{\epsilon}(e)(a) = \epsilon(a)e$ with $\epsilon\colon \mathfrak{A}(\mathfrak{g}) \rightarrow \mathbb{R}$ the trivial representation, and
\begin{eqnarray}
(d^n\xi)(a\otimes (X_1\wedge \cdots \wedge X_{n+1})) = \sum_{i=1}^{n+1}(-1)^{i+1} \xi( (aX_i)\otimes (X_1\wedge \cdots \wedge \hat{X}_i\wedge \cdots \wedge X_{n+1}) ) + \nonumber \\
  + \sum_{1\leq i<j\leq n+1} (-1)^{i+j}\xi ( a\otimes ([X_i,X_j]\wedge X_1\wedge \cdots \wedge \hat{X}_i\wedge \cdots \wedge \hat{X}_j\wedge \cdots \wedge X_{n+1})). \nonumber
\end{eqnarray}
\end{proposition}

\begin{proof}
This is essentially proved in \cite[Chapter II, Lemma 2.5-7]{Guichardetbook}, so we leave out the details.
\end{proof}

Notice that we have an obvious identification \todo{eq:liealgmultilin}
\begin{equation} \label{eq:liealgmultilin}
\operatorname{hom}_{\mathfrak{k}}(\mathfrak{A}(\mathfrak{g}), \operatorname{hom}_{\mathbb{R}}( \extprod{n}(\mathfrak{g}/\mathfrak{k}) , E)) \simeq \operatorname{hom}_{\mathfrak{k}}( \mathfrak{A}(\mathfrak{g}) \otimes \left( \extprod{n}(\mathfrak{g}/\mathfrak{k}) \right) , E )
\end{equation}
where on the right-hand side, $\mathfrak{k}$ acts trivially on $E$ and on the domain by

\begin{equation}
k.\left( a\otimes (X_1\wedge \cdots \wedge X_n) \right) = a\otimes \left( \sum_{i=1}^nX_1\wedge \cdots \wedge [k,X_i]\wedge \cdots \wedge X_n\right) -(ak)\otimes (X_1\wedge \cdots \wedge X_n). \nonumber
\end{equation}

\begin{remark} \label{rmk:liealgmultilin} \todo{rmk:liealgmultilin}
More generally, equation \eqref{eq:liealgmultilin} holds for any $F$ in place of $\extprod{n} \mathfrak{g}/\mathfrak{k}$, and we can then further identify this with bilinear maps $\mathfrak{A}(\mathfrak{g})\times F\rightarrow E$ satisfying the obvious condition
\begin{equation}
f((ak)\otimes (X_1\wedge \cdots \wedge X_n)) = \sum_{i=1}^n f\left( a\otimes \left( X_1\wedge \cdots \wedge [k,X_i]\wedge \cdots \wedge X_n\right) \right). \nonumber
\end{equation}
\end{remark}

\begin{definition}
The relative cohomology of the inclusion $\mathfrak{k}\subseteq \mathfrak{g}$ with coefficients in the $\mathfrak{g}$-$\mathscr{A}$-module $E$ is the $\mathscr{A}$-module defined, up to isomorphism by
\begin{equation}
H^n( \mathfrak{k}\subseteq \mathfrak{g}, E ) := \ker \left( d \colon (E^{n})^{\mathfrak{g}}\rightarrow (E^{n+1})^{\mathfrak{g}}\right) / \operatorname{im} \left( d \colon (E^{n-1})^{\mathfrak{g}} \rightarrow (E^{n})^{\mathfrak{g}}\right), \nonumber
\end{equation}
for any $\mathfrak{k}$-injective resolution $0\rightarrow E \rightarrow E^*$ of $E$.
\end{definition}
For the the proof of uniqueness up to isomorphism, see \cite[Chapter II, Section 1]{Guichardetbook}.

\begin{theorem}[(van Est {\cite{vanEst53,vanEst55}})] \label{thm:vanEst} \todo{thm:vanEst}
Let $G$ be a connected real Lie group and $K$ a maximal compact subgroup. Let $\mathfrak{g}$ respectively $\mathfrak{k}$ be their (real) Lie algebras.

Let $E$ be a quasi-complete topological $G$-$\mathscr{A}$-module. Then for any $n\geq 0$ there is an isomorphism of (algebraic) $\mathscr{A}$-modules
\begin{equation}
H^n(G,E) \simeq H^n(\mathfrak{k}\subseteq \mathfrak{g}, E^{\infty}) = H^n(\mathfrak{k}\subseteq \mathfrak{g}, (E^{\infty})_{(K)}). \nonumber
\end{equation}
\end{theorem}

\begin{proof}
By \cite[Chapter III, Proposition 1.6]{Guichardetbook} and van Est's theorem \cite[Chapter III, Corollary 7.2]{Guichardetbook} (See also \cite[Remark 3.5, Chapter II]{Guichardetbook}) we have $H^n( G, E) \simeq H^n(\mathfrak{g}, K, E_{(K)}) \simeq H^n(\mathfrak{g},\mathfrak{k},E_{(K)})$, and checking Guichardet's explicit formula for the van Est isomorphism (p.227 in \cite{Guichardetbook}) this is an isomorphism of right-$\mathscr{A}$-modules.
\end{proof}

\section{The Hochschild-Serre spectral sequence} \label{sec:HSLiealgebra} \todo{sec:HSLiealgebra}

In this section we construct the Hochschild-Serre spectral sequence to compute the relative cohomology for an extension of Lie algebras. There is a construction of the spectral sequence in the "absolute" case in \cite{Guichardetbook}, that is, the case where $\mathfrak{k}=\{0\}$. There is another construction in \cite[Section I.6]{BorelWallachBook}, in a slightly different setting than what we are after here. For convenience we will give a more streamlined construction in the setup we need.

As in the previous section, $\mathscr{A}$ is a fixed semi-finite, $\sigma$-finite, tracial von Neumann algebra, $\mathfrak{k}\subseteq \mathfrak{g}$ an inclusion of real Lie algebras with $\mathfrak{k}$ reductive in $\mathfrak{g}$. We study extensions, so fix also an ideal $\mathfrak{h}$ in $\mathfrak{g}$ and denote $\mathfrak{l} := \mathfrak{h}\cap \mathfrak{k}$. Then $\mathfrak{l}$ is an ideal in $\mathfrak{k}$.

Denote the quotients $\mathfrak{g}':=\mathfrak{g}/\mathfrak{h}$ and $\mathfrak{k}'=\mathfrak{k}/\mathfrak{l}$. For reference we state the following easy facts:

\begin{lemma} \label{lma:HSLiealgebrafacts} \todo{lma:HSLiealgebrasfacts}
With the setup above:
\begin{enumerate}[(i)]
\item $\mathfrak{k}$ is a direct sum of Lie algebras $\mathfrak{k}=\mathfrak{l}\oplus \mathfrak{k}'$,
\item $\mathfrak{l}$ is reductive in $\mathfrak{h}$ and $\mathfrak{k}'$ is reductive in $\mathfrak{g}'$.
\item Every $\mathfrak{A}(\mathfrak{h})$-invariant subspace of $\mathfrak{g}$ has an $\mathfrak{A}(\mathfrak{h})$-invariant complement.
\end{enumerate}
\end{lemma}

\begin{flushright}
\qedsymbol
\end{flushright}

\noindent
Given a $\mathfrak{g}$-$\mathscr{A}$-module $E$, this gives rise to two new modules:
\begin{itemize}
\item The $\mathfrak{h}$-$\mathscr{A}$-module obtained by restricting the $\mathfrak{g}$-action,
\item and the $\mathfrak{g}'$-$\mathscr{A}$-module $E^{\mathfrak{h}}$.
\end{itemize}
The next two lemmas show that both constructions preserve relative injectivity.

\begin{lemma} \label{lma:HSLiealgebralemma1} \todo{lma:HSLiealgebralemma1}
Let $E$ be a $\mathfrak{g}$-$\mathscr{A}$-module and $F$ a $\mathfrak{k}$-module. Then the module $\operatorname{hom}_{\mathfrak{k}}(\mathfrak{A}(\mathfrak{g}), \operatorname{hom}_{\mathbb{R}}(F,E))$ as in Proposition \ref{prop:liealginjectiveres} is $\mathfrak{l}$-injective as an $\mathfrak{h}$-module.
\end{lemma}

\begin{proof}
The idea is to write $\operatorname{hom}_{\mathfrak{k}}(\mathfrak{A}(\mathfrak{g}), \operatorname{hom}_{\mathbb{R}}(F,E)) \simeq \operatorname{hom}_{\mathfrak{l}}(\mathfrak{A}(\mathfrak{h}), \operatorname{hom}_{\mathbb{R}}(F',E))$, as $\mathfrak{h}$-$\mathscr{A}$-modules, for some $\mathfrak{l}$-module $F'$, and then appeal to the first part of Proposition \ref{prop:liealginjectiveres}.

First, since $\mathfrak{k}'$ is an ideal in $\mathfrak{k}$, we have an identity of $\mathfrak{k}$-modules $\mathfrak{g} = \mathfrak{h}\oplus V\oplus \mathfrak{k}'$ with $V\subseteq \mathfrak{g}$ some $\mathfrak{k}$-invariant subspace, by the previous lemma.

Choose a linear basis $\{X_i\}_{i=1}^n$ of $\mathfrak{g}$ such that $X_i\in \mathfrak{h}$ for $1\leq i \leq n_1$, $X_i\in V$ for $n_1+1\leq i\leq n_2$, and $X_i\in \mathfrak{k}'$ for $n_2+1\leq i\leq n$.

Then the Poincar{\'e}-Birkhoff-Witt Theorem implies directly that $\mathfrak{A}(\mathfrak{g}) \simeq \mathfrak{A}(\mathfrak{h})\otimes U \otimes \mathfrak{A}^*(\mathfrak{k}')$ where in general we denote by $\mathfrak{A}^*(-)$ the kernel of the augmentation map (trivial representation). Here $U$ is the span of simple tensors of $X_i$'s with $n_1+1\leq i\leq n_2$.

Since $V\oplus \mathfrak{k}'$ is stable under the adjoint representation of $\mathfrak{k}'$, it follows that $U':=U\otimes \mathfrak{A}(\mathfrak{k}')$ is a module under right-multiplication by $\mathfrak{A}(\mathfrak{k}')$.

Then, inspired by \cite[Eq. (2), p. 35]{BorelWallachBook} we want to show that

\begin{equation}
\operatorname{hom}_{\mathfrak{k}}(\mathfrak{A}(\mathfrak{g}), \operatorname{hom}_{\mathbb{R}}(F,E)) \simeq \operatorname{hom}_{\mathfrak{l}}(\mathfrak{A}(\mathfrak{h}) , \operatorname{hom}_{\mathbb{R}} ( U'\otimes_{\mathfrak{A}(\mathfrak{k}')} F, E)). \nonumber
\end{equation}
Here $\mathfrak{l}$ acts on $U'\otimes_{\mathfrak{A}(\mathfrak{k}')} F$ by\todo{What this really does is run the $k$ through the leg on which it doesn't act by mult but should, and then through to F.}

\begin{equation}
k.(X_{i_1}\otimes \cdots \otimes X_{i_s} \otimes f) = \left(\sum_{j=1}^{s} X_{i_1}\otimes \cdots \otimes [k,X_{i_j}]\otimes \cdots \otimes X_{i_S}\right)\otimes f + \left(X_{i_1}\otimes \cdots \otimes X_{i_s}\right) \otimes (k.f). \nonumber
\end{equation}
This makes sense since $[\mathfrak{l},\mathfrak{k}'] = 0$, i.e.~the enveloping algebras commute.

The exact same formula defines also an action on $U\otimes F$, and we have an obvious\todo{CHECKCHECKCHECK!} isomorphism of $\mathfrak{l}$-modules 
\begin{equation}
\operatorname{hom}_{\mathbb{R}}(U'\otimes_{\mathfrak{A}(\mathfrak{k}')} F,E)  \simeq \operatorname{hom}_{\mathbb{R}}(U\otimes F,E). \nonumber
\end{equation}

Then the desired isomorphism is simply a matter of using the identification as in remark \ref{rmk:liealgmultilin} and checking that the obvious conditions on multilinear maps $\mathfrak{A}(\mathfrak{g})\times F \rightarrow E$ versus $\mathfrak{A}(\mathfrak{h})\times U \times F \rightarrow E$ coincide.
\end{proof}

\begin{lemma} \label{lma:HSLiealgebralemma2} \todo{lma:HSLiealgebralemma2}
If $E$ is a $\mathfrak{k}$-injective $\mathfrak{g}$-$\mathscr{A}$-module, then the $\mathfrak{g}'$-$\mathscr{A}$-module $E^{\mathfrak{h}}$ is $\mathfrak{k}'$-injective.
\end{lemma}

\begin{proof}
This is straight-forward: suppose we have a diagram of $\mathfrak{g}'$-$\mathscr{A}$-modules
\begin{displaymath}
\xymatrix{  & E^{\mathfrak{h}} \\ B \ar@{-->}[ur]^{\exists ? w} & \ar[l]_u A \ar[u]^v & 0 \ar[l] }
\end{displaymath}
with the bottom row $\mathfrak{k}'$-strengthened exact, and denote by $s\colon B\rightarrow A$ a $\mathfrak{k}'$-$\mathscr{A}$-linear left-inverse of $u$.

We may consider $A,B,E$ is $\mathfrak{g}$-$\mathscr{A}$-modules via.~the quotient homomorphism $\mathfrak{g}\rightarrow \mathfrak{g}'$, and then the maps $u,v$ are $\mathfrak{g}$-linear and $s$ is $\mathfrak{k}$-linear.

Hence by hypothesis, since $E^{\mathfrak{h}}$ is a $\mathfrak{g}$-submodule of $E$, there is a $\mathfrak{g}$-$\mathscr{A}$-morphism $w\colon B \rightarrow E$ such that $w\circ u = v$. Further, for any $Y\in \mathfrak{h}, b\in B$ we have
\begin{equation}
Y.w(b) = w(Y.b) = w(0) = 0, \nonumber
\end{equation}
so in fact $w(B) \subseteq E^{\mathfrak{h}}$, which completes the proof.
\end{proof}

Next we need to define a $\mathfrak{g}'$-action on the cohomology $H^n(\mathfrak{l}\subseteq \mathfrak{h}, E)$ when $E$ is a $\mathfrak{g}$-$\mathscr{A}$-module. We have two natural ways to construct such an action, arising from two different injective resolutions of $E$:

\begin{enumerate}[(i)]
\item First note that the $\mathfrak{h}$-action on $\operatorname{hom}_{\mathfrak{l}}( \mathfrak{A}(\mathfrak{h})\otimes \extprod{n}(\mathfrak{h}/\mathfrak{l}), E)$ extends directly to a $\mathfrak{g}$-action, since $\mathfrak{h}$ is an ideal, whence we have an induced action of $\mathfrak{g}'$ on the $\mathfrak{h}$-fixpoints.

\item Similarly we have the natural $\mathfrak{g}$ action on the $\mathfrak{l}$-injective $\mathfrak{h}$-module $\operatorname{hom}_{\mathfrak{k}}( \mathfrak{A}(\mathfrak{g})\otimes \extprod{n}(\mathfrak{g}/\mathfrak{k}), E)$, extending the $\mathfrak{h}$-action whence inducing a $\mathfrak{g}'$-action on the fixpoints.
\end{enumerate}

\begin{lemma} \label{lma:HSLiealgebralemma3} \todo{lma:HSLiealgebralemma3}
The actions of $\mathfrak{g}'$ on $H^*(\mathfrak{l}\subseteq \mathfrak{h}, E)$ induced by (i) and (ii) above coincide.
\end{lemma}

\begin{proof}
We have canonical $\mathfrak{g}$-$\mathscr{A}$-linear embeddings
\begin{equation}
\iota_n\colon \operatorname{hom}_{\mathfrak{l}}( \mathfrak{A}(\mathfrak{h})\otimes \extprod{n}(\mathfrak{h}/\mathfrak{l}), E) \rightarrow \operatorname{hom}_{\mathfrak{k}}( \mathfrak{A}(\mathfrak{g})\otimes \extprod{n}(\mathfrak{g}/\mathfrak{k}), E). \nonumber
\end{equation}
One checks easily enough that these commute with the coboundary maps, and extend the identity on $E$, whence it follows that they induce the identity map on $H^n(\mathfrak{l}\subseteq \mathfrak{h}, E)$, and these conjugate the two $\mathfrak{g}'$-actions by the $\mathfrak{g}$-linearity of $\iota_n$.
\end{proof}
%%%%%%%%%%%%%%%%%%%%%%%%%%%%%%%%%%%%%%%%%%%%%%%%%%% I think this is superfluous
\begin{comment}
To see this, define linear a map $\mathfrak{A}(\mathfrak{g}) \rightarrow \mathfrak{A}(\mathfrak{h})$ via.~the Poincar{\'e}-Birkhoff-Witt theorem, mapping all simple tensors containing factors not in $h$ to zero. This is clearly an $\mathfrak{h}$-mmorphism.

Next, define similarly an linear map $\extprod{n}\mathfrak{g}/\mathfrak{k} \rightarrow \extprod{n} \mathfrak{h}/\mathfrak{l}$ using the decomposition $\mathfrak{g}/\mathfrak{k} = \mathfrak{h}/\mathfrak{l} \oplus \mathfrak{g}'/\mathfrak{k}'$.

This yields then maps
\begin{equation}
\pi_n\colon \operatorname{hom}_{\mathfrak{k}}( \mathfrak{A}(\mathfrak{g})\otimes \extprod{n}(\mathfrak{g}/\mathfrak{k}), E) \rightarrow \operatorname{hom}_{\mathfrak{l}}( \mathfrak{A}(\mathfrak{h})\otimes \extprod{n}(\mathfrak{h}/\mathfrak{l}), E). \nonumber
\end{equation}
Then one checks readily enough that the $\iota_n$ and $\pi_n$ commute with coboundary maps, and both $\iota_{*}$ and $\pi_{*}$ extend the identity
\end{comment}
%%%%%%%%%%%%%%%%%%%%%%%%%%%%%%%%%%%%%%%%%%%%%%%

Now we can construct the Hochschild-Serre spectral sequence as follows. First define the bi-complex

\begin{equation}
K^{p,q} := \operatorname{hom}_{\mathfrak{k}'} \left( \mathfrak{A}(\mathfrak{g}')\otimes \extprod{p}(\mathfrak{g}'/\mathfrak{k}'), \operatorname{hom}_{\mathfrak{k}} ( \mathfrak{A}(\mathfrak{g})\otimes \extprod{q}(\mathfrak{g}/\mathfrak{k}), E)^{\mathfrak{h}} \right)^{\mathfrak{g}'} , \nonumber
\end{equation}
where the second $\operatorname{hom}$-space is a $\mathfrak{k}'$-module with trivial action.

This carries two canonical coboundary maps $d'$ of degree $(1,0)$ repsectively $d''$ of degree $(0,1)$. Indeed, we take $d'$ to be the coboundary map of Proposition \ref{prop:liealginjectiveres} with $\operatorname{hom}_{\mathfrak{k}} ( \mathfrak{A}(\mathfrak{g})\otimes \extprod{q}(\mathfrak{g}/\mathfrak{k}), E)^{\mathfrak{h}}$ in place of $E$, and $d''$ to be the pointwise application of the coboundary map on the complex $\operatorname{hom}_{\mathfrak{k}} ( \mathfrak{A}(\mathfrak{g})\otimes \extprod{*}(\mathfrak{g}/\mathfrak{k}), E)^{\mathfrak{h}}$.

Then we have two filtrations on the cohomology of the total complex $K^*$ with $K^n:=\oplus_{p+q=n} K^{p,q}$, carrying the coboundary map $d:=d'+d''$ - one by restricting the $p$ index and one by restricting the $q$ index, which gives rise to two spectral sequences to compute $H^n(K^*)$. To describe these, denote

\begin{equation}
H^{p,q}(d') := \ker (d'\colon K^{p,q}\rightarrow K^{p+1,q}) / \operatorname{im} (d'\colon K^{p-1,q}\rightarrow K^{p,q}),
\end{equation}
and similarly $H^{p,q}(d'')$.

\begin{enumerate}[(i)]
\item The coboundary $d''$ induces a coboundary map to form for any $p$ a complex
\begin{displaymath}
\xymatrix{ \mathfrak{M}^{(p)}: & 0 \ar[r] & H^{p,0}(d') \ar[r] & H^{p,1}(d') \ar[r] & \cdots }
\end{displaymath}
and there is a spectral sequence ${}'E_2^{p,q} = H^q(\mathfrak{M}^{(p)})$ abutting to $H^{p+q}(K^*)$.

\item The coboundary $d'$ induced a coboundary to form for any $q$ the complex
\begin{displaymath}
\xymatrix{ \mathfrak{N}^{(q)}: & 0 \ar[r] & H^{0,q}(d'') \ar[r] & H^{1,q}(d'') \ar[r] & \cdots }
\end{displaymath}
and there is a spectral sequence ${}''E_2^{p,q} = H^p(\mathfrak{N}^{(q)})$ abutting to $H^{p+q}(K^*)$.
\end{enumerate}

Using these two spectral sequences, we are ready to show the following.

\begin{theorem}[(Hochschild-Serre)] \label{thm:HSLiealgebra} \todo{thm:HSLiealgebra}
Let $\mathfrak{g}$ be a Lie algebra, $\mathfrak{h}$ an ideal in $\mathfrak{g}$, and $\mathfrak{k}$ a subalgebra of $\mathfrak{g}$ which is reductive in $\mathfrak{g}$. Let $E$ be a $\mathfrak{g}$-$\mathscr{A}$-module.

Then with notation as in the beginning of this section (see also Lemma \ref{lma:HSLiealgebrafacts}), there is a spectral sequence $E_2^{p,q}=H^p(\mathfrak{k}'\subseteq \mathfrak{g}' , H^q( \mathfrak{l}\subseteq \mathfrak{h}, E))$ abutting to $H^{p+q}(\mathfrak{k}\subseteq \mathfrak{g},E)$.
\end{theorem}

\begin{proof}
We have already seen above that there is a spectral sequence ${}''E$ abutting to the total cohomology of the double complex $K^{p,q}$. The proof then consists of two parts: first we show that the total cohomology coincides with the relative cohomology $H^{n}(\mathfrak{k}\subseteq \mathfrak{g}, E)$; second that the spectral sequence ${}''E$ in part (ii) has $E_2$-term as claimed in the statement.

First we observe that
\begin{equation}
H^{p,q}(d') = H^p(\mathfrak{k}'\subseteq \mathfrak{g}', \operatorname{hom}_{\mathfrak{k}}( \mathfrak{A}(\mathfrak{g})\otimes \extprod{q}(\mathfrak{g}/\mathfrak{k}), E)^{\mathfrak{h}} ). \nonumber
\end{equation}
By Lemma \ref{lma:HSLiealgebralemma2}, the coefficient module is $\mathfrak{k}'$-injective whence this vanishes when $p>0$. When $p=0$, we get

\begin{equation}
H^{0,q}(d') = \left( \operatorname{hom}_{\mathfrak{k}}( \mathfrak{A}(\mathfrak{g})\otimes \extprod{q}(\mathfrak{g}/\mathfrak{k}), E)^{\mathfrak{h}} \right)^{\mathfrak{g}'} = \operatorname{hom}_{\mathfrak{k}}( \mathfrak{A}(\mathfrak{g})\otimes \extprod{q}(\mathfrak{g}/\mathfrak{k}), E)^{\mathfrak{g}}. \nonumber
\end{equation}
Hence it follows that the spectral sequence ${}'E$ has ${}'E_2^{0,q}= H^q(\mathfrak{k}\subseteq \mathfrak{g}, E)$ and ${}'E_2^{p,q} = 0$, for all $p>0$. Thus it collapses at ${}'E_2$ and since it abuts to the total cohomology we have shown that $H^{n}_{tot}(K^{*,*}) = H^n(\mathfrak{k}\subseteq \mathfrak{g}, E)$.

Next, to show that the spectral sequence ${}''E$ in part (ii) above has $E_2$-term as stated, we need to show that
\begin{equation}
H^{p,q}(d'') \simeq \operatorname{hom}_{\mathfrak{k}'}( \mathfrak{A}(\mathfrak{g}')\otimes \extprod{p}(\mathfrak{g}'/\mathfrak{k}'), H^q(\mathfrak{l}\subseteq \mathfrak{h}, E)). \nonumber
\end{equation}
Recall that the coboundary $d''$ acts by the usual coboundary applied pointwise in the complex 
\begin{displaymath}
\operatorname{hom}_{\mathfrak{k}'} \left( \mathfrak{A}(\mathfrak{g}')\otimes \extprod{p}(\mathfrak{g}'/\mathfrak{k}'), \operatorname{hom}_{\mathfrak{k}} ( \mathfrak{A}(\mathfrak{g})\otimes \extprod{*}(\mathfrak{g}/\mathfrak{k}), E)^{\mathfrak{h}} \right)^{\mathfrak{g}'}
\end{displaymath}
and that the coefficient complex itself, $\mathfrak{L}^*: \operatorname{hom}_{\mathfrak{k}} ( \mathfrak{A}(\mathfrak{g})\otimes \extprod{*}(\mathfrak{g}/\mathfrak{k}), E)^{\mathfrak{h}}$ computes the cohomology $H^*(\mathfrak{l}\subseteq \mathfrak{h}, E)$ by Lemma \ref{lma:HSLiealgebralemma1}, and does so consistently with the canonical $\mathfrak{g}'$ action on the latter by Lemma \ref{lma:HSLiealgebralemma3}.

It is of course clear that if $\phi\in \ker d'' \cap K^{p,q}$ then $\phi(a\otimes X_1\otimes \cdots \otimes X_p)\in \ker(d_{\mathfrak{L}}:\mathfrak{L}^q\rightarrow \mathfrak{L}^{q+1})$ for all $a\in\mathfrak{A}(\mathfrak{g}'), X_1,\dots, X_p\in \mathfrak{g}'/\mathfrak{k}'$.

We need to show then, that if $\phi$ is a coboundary pointwise, then it is globally a coboundary. That is, we need to lift a choice of linear section of $d_{\mathfrak{L}}$ to a global, $\mathfrak{g}'$-equivariant section of $d''$; fix such a section $s$ and note that we can and will choose a $\mathfrak{k}'$-linear section, since the relevant complex is $\mathfrak{k}'$-strengthened. We indicate a lifting to finish the proof:

Let $\{Y_i\}_{i=1}^{m_1}$ be a linear basis of $\mathfrak{k}'$ and extend this to a linear basis $\{Y_i\}_{i=1}^m$ of $\mathfrak{g}'$. For every (finite) multi-index $\lambda=(\lambda_i)_{i=1,\dots ,I}$ with $m_1\geq \lambda_I\geq \cdots \geq \lambda_1\geq 1$ (the $\lambda_i$ not necessarily distinct), we put

\begin{equation}
e_{\lambda} := Y_{\lambda_I}\otimes \cdots Y_{\lambda_1} \in \mathfrak{A}(\mathfrak{k}')\subseteq \mathfrak{A}(\mathfrak{g}'). \nonumber
\end{equation}
Then we define idempotent maps $\pi_{\lambda}\colon \mathfrak{A}(\mathfrak{g}')\rightarrow \mathfrak{A}(\mathfrak{g}')$ for $m\geq i_{j}\geq m_1+1$ by

\begin{equation}
\pi_{\lambda}(Y_{i_J}\otimes \cdots \otimes Y_{i_1} \otimes e_{\lambda'}) = \left\{ \begin{array}{rl} 0 & ,if \; \lambda'\neq \lambda \\ Y_{i_J}\otimes \cdots \otimes Y_{i_1} & ,if \; \lambda'=\lambda \end{array} \right. . \nonumber
\end{equation}
Observe that every $\pi_{\lambda}$ is right-$\mathfrak{A}(\mathfrak{k}')$-linear by the Poincar{\'e}-Birkhoff-Witt Theorem, and that

\begin{equation}
a = \sum_{\lambda} \pi_{\lambda}(a)\otimes e_{\lambda}, \quad a\in \mathfrak{A}(\mathfrak{g}'). \nonumber
\end{equation}
(For the lawyers: $e_{\emptyset} := \bbb$ and $\pi_{\emptyset}$ is the augmentation map.)

Then the section we're looking for is given by

\begin{equation}
(\bar{s}\phi)(a\otimes X_1\otimes \cdots \otimes X_p) = \sum_{\lambda} \pi_{\lambda}(a).s(e_{\lambda}\otimes X_1\otimes \cdots X_p). \nonumber
\end{equation}
\end{proof}

\section{The mixed case} \label{sec:HSLiemixed} \todo{sec:HSLiemixed}

Theorem \ref{thm:HSLiealgebra} allows us in particular to apply the Hochschild-Serre spectral sequence to inclusions of connected Lie groups $H\leq G$. However, in many cases we also want to consider inclusions where $H$ is not necessarily connected; a particular inclusion of interest is $Z(G)\leq G$ when the center is infinite discrete. But in this case, the cohomology space $H^1(Z(G),L^2G)$ is non-Hausdorff, so the Hochschild-Serre spectral sequence for groups does not apply directly.

To get around this we work out a version of Theorem \ref{thm:HSLiealgebra} which is directly applicable. Let $G$ be a connected Lie group, $E$ a quasi-complete topological $G$-$\mathscr{A}$-module, and $H$ a discrete, normal subgroup of $G$. Denote $G':=G/H$, again a connected Lie group, and let $\mathfrak{g}'$ be its Lie algebra and $\mathfrak{k}'$ the Lie subalgebra of a maximal compact subgroup. Let $E$ be a quasi-complete topological $G$-$\mathscr{A}$-module.

Since $C^{\infty}(G^n,E^{(\infty,G)})$ is an injective topological $H$-$\mathscr{A}$-module for each $n$ (similarly to \cite[Chapter III, Lemma 4.2]{Guichardetbook}, noting that there is a locally smooth section of the quotient map), the continuous cohomology $H^n(H,E^{(\infty,G)})$ is the cohomology of the complex
\begin{displaymath}
\xymatrix{ (\mathfrak{L}^*)^H : & 0\ar[r] & C^{\infty}(G,E^{(\infty,G)})^{H} \ar[r]^{\partial^0} & C^{\infty}(G^2,E^{(\infty,G)})^{H} \ar[r]^>>>{\partial^1} & \cdots }
\end{displaymath}

The spaces $\mathfrak{L}^n$ are smooth for the $G$-action (See \cite[Section D.4.2]{Guichardetbook}) whence $(\mathfrak{L}^n)^H$ is smooth for the induced $G'$-action (again this follows since the quotient map $G\rightarrow G'$ is a local diffeomorphism). Hence this has the structure of a $\mathfrak{g}'$-module, and the coboundary maps commute with this whence we get a $\mathfrak{g}'$-module structure on $H^n(H,E^{(\infty,G)})$. In this setup we can now show:

\begin{theorem} \label{thm:HSmixed} \todo{thm:HSmixed}
Keep notations as above. There is a spectral sequence (of $\mathscr{A}$-modules, i.e.~in $\mathfrak{E}^{(alg)}_{\mathscr{A}}$) $E_2^{p,q} = H^p(\mathfrak{k}'\subseteq \mathfrak{g}', H^q(H,E^{(\infty,G)}))$ abutting to $H^{*}(\mathfrak{k}\subseteq \mathfrak{g},E^{(\infty,G)})$.
\end{theorem}

\begin{proof}
Write again the double complex of $\mathscr{A}$-modules.
\begin{equation}
K^{p,q} := \operatorname{hom}_{\mathfrak{k}'}( \mathfrak{A}(\mathfrak{g}')\otimes \extprod{p}(\mathfrak{g}'/\mathfrak{k}'), (\mathfrak{L}^q)^H )^{\mathfrak{g}'}. \nonumber
\end{equation}
The coboundary maps here are $d':K^{*,q}\rightarrow K^{*+1,q}$ given for each fixed $q$ by the coboundary maps in Proposition \ref{prop:liealginjectiveres} with $(\mathfrak{L}^q)^{H}$ in place of $E$. The vertical coboundary maps $d''\colon K^{p,*}\rightarrow K^{p,*+1}$ given by pointwise application of the coboundary maps $\partial^*$.

Then we have two spectral sequences of $\mathscr{A}$-modules, both abutting to the total cohomology of $K^{*,*}$. As usual we analyze the $E_2$-terms. First, we have a complex
\begin{displaymath}
\xymatrix{ \mathfrak{M}^{(p)} : & \cdots \ar[r] & H^p(K^{*,q},d') \ar[r]^{d''} & H^p(K^{*,q+1},d') \ar[r]^>>{d''} & \cdots }
\end{displaymath}
and the $E_2$-term is the first spectral sequence is ${}'E_2^{p,q}= H_q(\mathfrak{M}^{(p)})$. But clearly, by the van Est theorem

\begin{equation}
(\mathfrak{M}^{(p)})_q = H^p( \mathfrak{k}'\subseteq \mathfrak{g}', (\mathfrak{L}^q)^{H} ) = H^p( G', (\mathfrak{L}^q)^H). \nonumber
\end{equation}

Since $(\mathfrak{L}^q)^H$ is an injective $G'$-module, this vanishes except for $p=0$ where $(\mathfrak{M}^{(0)})_q \simeq (\mathfrak{L}^q)^{G'} \simeq C^{\infty}(G^{q+1},E^{(\infty,G)})^G$. The coboundary maps $d''$ induce the $\partial^*$ on $\mathfrak{M}^{(0)}$ whence
\begin{equation}
{}'E_2^{p,q} = \left\{ \begin{array}{cc} 0 &, p> 0 \\ H^q(\mathfrak{k}\subseteq \mathfrak{g},E^{(\infty,G)}) & , p=0 \end{array} \right. \nonumber
\end{equation}

Thus the total cohomology coincides with $H^*(\mathfrak{k}\subseteq \mathfrak{g}, E^{(\infty,G)})$. For the second spectral sequence we have complexes

\begin{displaymath}
\xymatrix{ \mathfrak{N}^{(q)} : & \cdots \ar[r]^{d'} & H^q(K^{p,*},d'') \ar[r]^{d'} & H^q(K^{p+1,*},d'') \ar[r]^>>{d'} & \cdots }
\end{displaymath}
and ${}''E_2^{p,q} = H^p(\mathfrak{N}^{(q)})$. As in the proof of Theorem \ref{thm:HSLiealgebra} we need to show that

\begin{equation}
(\mathfrak{N}^{(q)})_p \simeq \operatorname{hom}_{\mathfrak{k}'}( \mathfrak{A}(\mathfrak{g}')\otimes \extprod{p}(\mathfrak{g}'/\mathfrak{k}'), H^q(H,E) )^{\mathfrak{g}'}. \nonumber
\end{equation}

In fact, this is shown just as in that proof, once we note that the in complex $(\mathfrak{L}^*)^H$, we can lift an element in $\operatorname{im}\partial^q$ $K'$-invariantly to an element in $(\mathfrak{L}^q)^H$ element, whence ditto $\mathfrak{k}'$-invariantly (where $K'$ is a maximal compact subgroup). In fact this is not a priori true, but we note that any element in $K^{p,q}$ takes its values in $((\mathfrak{L}^q)^H)_{(K')}$, the space of $K'$-finite vectors. Thus, since linear maps are automatically continuous on finite-dimensional spaces we can choose a linear section $s_q$ of $\partial^q\vert_{((\mathfrak{L}^q)^H)_{(K')}}$ and replace it with $\int_{K'} k'.s_q(k'^{-1}\cdot)$.

This finishes the proof.
\end{proof}

\begin{remark}
Note that in fact Theorem \ref{thm:HSLiealgebra} follows by (an extension of) the previous theorem (to general closed subgroups $H$), by the van Est theorem.
\end{remark}

% doesn't work for some reason, the list...
%\listoftodos

\backmatter

%\settocdepth{subsection}
%\tableofcontents

%\bibliographystyle{workingbib}
\bibliographystyle{plain}
\bibliography{ganesh}

\end{document}